\documentclass[a4paper, 11pt, usenames, dvipsnames]{article}

\usepackage[utf8]{inputenc}
\usepackage[top=1.0in, bottom=1.0in, left=1.0in, right=1.0in]{geometry}
\usepackage[all]{xy}
\usepackage{amsfonts}
\usepackage{amsmath}
\usepackage{amssymb}
\usepackage{amsthm}
\usepackage{calrsfs}
\usepackage{enumerate}
\usepackage{enumitem}
\usepackage{xparse, etoolbox}
\usepackage{mathtools, nccmath, textcomp}
\usepackage{relsize}
\usepackage{rsfso}
\usepackage{stmaryrd}
\usepackage{sectsty}
\usepackage[nottoc]{tocbibind}
\usepackage{tensor}
\usepackage{tikz}
\usepackage{tikz-cd}
\usepackage[sf, bf]{titlesec}	
\usepackage[titles]{tocloft}
\usepackage{xcolor}
\usepackage[backend=biber, style=alphabetic, sorting=nyt, doi=false, url=false, maxbibnames=99, maxalphanames=99]{biblatex}

\usepackage[T1]{fontenc}
\usepackage{lmodern}

\usepackage{hyperref}
\hypersetup{
	colorlinks=true,
	linkcolor=WildStrawberry,
	citecolor=Emerald,
	urlcolor=black,
}

\usepackage[titletoc]{appendix}

\usepackage{imakeidx}
\makeindex

\usepackage{fancyhdr}

\usepackage{etoolbox}
\makeatletter
\patchcmd{\ttlh@hang}{\parindent\z@}{\parindent\z@\leavevmode}{}{}
\patchcmd{\ttlh@hang}{\noindent}{}{}{}
\makeatother
\newcommand{\addperiod}[1]{#1.}
\titleformat{\section}[block]{\scshape\Large\bfseries\filcenter}{\thesection.}{1em}{}
\titleformat{\subsection}[runin]{\normalfont\large\bfseries}{\thesubsection.}{1em}{\addperiod}
\titleformat{\subsubsection}[runin]{\normalfont\bfseries}{\thesubsubsection.}{1em}{\addperiod}

\titlelabel{\thetitle.\quad}

\renewcommand{\qedsymbol}{\rule{0.7em}{0.7em}}

\numberwithin{equation}{section}

\theoremstyle{plain}
\newtheorem{thm}{Theorem}[section]
\newtheorem{cor}[thm]{Corollary}
\newtheorem{lem}[thm]{Lemma}
\newtheorem{prop}[thm]{Proposition}

\theoremstyle{definition}
\newtheorem{defi}[thm]{Definition}
\newtheorem{cond}[thm]{Condition}

\AtEndEnvironment{const}{\qed}

\theoremstyle{remark}
\newtheorem{rem}[thm]{Remark}

\newtheorem*{conv}{Convention}

\newtheorem*{nota}{Notation}

\newcommand{\blocktheorem}[1]{%
  \csletcs{old#1}{#1}
  \csletcs{endold#1}{end#1}
  \RenewDocumentEnvironment{#1}{o}
    {\par\addvspace{1.5ex}
     \noindent\begin{minipage}{\textwidth}
     \IfNoValueTF{##1}
       {\csuse{old#1}}
       {\csuse{old#1}[##1]}}
    {\csuse{endold#1}
     \end{minipage}
     \par\addvspace{1.5ex}}
}

\newenvironment{enumromanup}
{\begin{enumerate}[font=\upshape, labelindent=\parindent, label=(\roman*)]}
{\end{enumerate}}

\DeclareMathOperator*{\bmoplus}{\text{\raisebox{0.25ex}{\scalebox{0.7}{$\bigoplus$}}}}

\DeclareMathOperator*{\bmwedge}{\text{\raisebox{0.25ex}{\scalebox{0.7}{$\bigwedge$}}}}

\DeclareMathOperator*{\bmstar}{\text{\raisebox{0.25ex}{\scalebox{0.8}{$\bigstar$}}}}
\DeclareMathOperator*{\smstar}{\text{\raisebox{0.25ex}{\scalebox{0.65}{$\bigstar$}}}}

\DeclareMathAlphabet{\pazocal}{OMS}{zplm}{m}{n}

\newcommand{\isomorphic}{\xrightarrow{\hspace{0.5mm} \sim \hspace{0.5mm}}}
\newcommand{\lisomorphic}{\xleftarrow{\hspace{0.5mm} \sim \hspace{0.5mm}}}

\newcommand{\paza}{\pazocal{A}}

\newcommand{\pazf}{\pazocal{F}}
\newcommand{\pazg}{\pazocal{G}}

\newcommand{\pazi}{\pazocal{I}}

\newcommand{\pazo}{\pazocal{O}}

\def\CC{\mathbb{C}}
\def\FF{\mathbb{F}}

\def\NN{\mathbb{N}}
\def\QQ{\mathbb{Q}}

\def\ZZ{\mathbb{Z}}

\newcommand{\mbfa}{\mathbf{A}}
\newcommand{\mbfb}{\mathbf{B}}
\newcommand{\mbfd}{\mathbf{D}}
\newcommand{\mbfe}{\mathbf{E}}

\newcommand{\mbfn}{\mathbf{N}}

\newcommand{\smbfk}{\mathbf{k}}

\newcommand{\crys}{\textup{cris}}

\newcommand{\CRYS}{\textup{CRIS}}
\newcommand{\cycl}{\textup{cycl}}
\newcommand{\dlog}{\hspace{1mm}d\textup{\hspace{0.5mm}log\hspace{0.3mm}}}

\newcommand{\dprm}{^{\prime\prime}}

\newcommand{\dR}{\textup{dR}}

\newcommand{\etale}{\textup{\'et}}
\newcommand{\Fil}{\textup{Fil}}

\newcommand{\Fr}{\textup{Fr}\hspace{0.5mm}}
\newcommand{\fini}{\textup{f}}
\newcommand{\free}{\hspace{0.5mm}\textup{free}}
\newcommand{\GL}{\textup{GL}}
\newcommand{\Gal}{\textup{Gal}}
\newcommand{\gr}{\textup{gr}}

\newcommand{\Hom}{\textup{Hom}}

\newcommand{\kert}{\textup{Ker }}

\newcommand{\Lie}{\textup{Lie }}

\newcommand{\MF}{\textup{MF}}

\newcommand{\Mat}{\textup{Mat}}
\newcommand{\unrami}{\textup{ur}}
\newcommand{\padic}{p\textrm{-adic}}
\newcommand{\PD}{\textup{PD}}

\newcommand{\prm}{^{\prime}}

\newcommand{\Rep}{\textup{Rep}}
\newcommand{\rig}{\hspace{0.2mm}\textup{rig}}

\newcommand{\Spec}{\textup{Spec}\hspace{0.5mm}}
\newcommand{\Spf}{\textup{Spf }}
\newcommand{\Sym}{\textup{Sym}}
\newcommand{\sep}{\textup{sep}}

\newcommand{\textdr}{\textup{dR}}

\newcommand{\textpd}{\textup{PD}}

\title{\textsc{Crystalline representations and Wach modules \\ in the relative case}}
\author{\textsc{Abhinandan}}

\newcommand{\Addresses}{{
	\bigskip
	\footnotesize
	
	\rule{2cm}{0.4pt}\vspace{4mm}

	\hspace{-2.1mm}\textsc{Abhinandan}\par\nopagebreak
	\vspace{-2mm}
	\begin{enumerate}
	        \item[(1)] \textsc{Université de Bordeaux, France}\par\nopagebreak\vspace{-0.7mm}
	        \item[(2)] \textsc{Université de Lille, France}\par\nopagebreak\vspace{-0.7mm}
		\item[(3)] \textsc{University of Tokyo, Japan}\par\nopagebreak\vspace{-0.9mm}
	\end{enumerate}\vspace{-1.2mm}
	\hspace{4mm}\textit{E-mail} :  \footnotesize{\href{abhinandan@math.cnrs.fr}{abhinandan@math.cnrs.fr}}, \footnotesize{\href{abhi@ms.u-tokyo.ac.jp}{abhi@ms.u-tokyo.ac.jp}}
}}

\date{ }

\begin{document}

\pagenumbering{arabic}

\sloppy

\goodbreak

\maketitle
{
	\textbf{Abstract.} We study the notion of Wach modules in relative setting, generalizing the arithmetic case.
		Over an unramified base, for a $p$-adic representation admitting such structure, we examine the relationship between its relative Wach module and filtered $(\varphi, \partial)$-module.
		Moreover, we show that such a representation is crystalline (in the sense of Brinon), and one can recover its filtered $(\varphi, \partial)$-module from the relative Wach module.
		Conversely, for low Hodge-Tate weights $[0, p-2]$, we construct relative Wach modules from free relative Fontaine-Laffaille modules (in the sense of Faltings).

}
{\hypersetup{linkcolor=black}
\tableofcontents}


\section{Introduction}

The theory of Wach modules for $\padic$ crystalline representations of the absolute Galois group of a finite unramified extension of $\QQ_p$ was introduced in the paper of  Fontaine \cite{fontaine-festschrift}.
This notion was further developed by Wach \cite{wach-pot-crys, wach-cristallines-torsion} and Berger \cite{berger-limites-cristallines}.
Over the years, this theory has found many applications, for example, to the Iwasawa theory of crystalline representations in \cite{benois-iwasawa, benois-berger-iwasawa}, and in the study of the $\padic$ local Langlands program \cite{berger-breuil-gl2qp}.
Wach modules were also among one of the motivations for Scholze's idea of $q\textrm{-deformations}$ \cite{scholze-q-deformations}, which in turn paved the way for the theory of prisms and prismatic cohomology of Bhatt and Scholze developed in \cite{bhatt-scholze-prismatic}.

Our goal in this article is to upgrade the notion of Wach modules to the relative case by which we mean certain étale algebras over a formal torus (see \S \ref{subsec:setup_nota} for precise setup).
But before examining the relative case, let us recall the relation between Wach modules and crystalline representations in the arithmetic case.

\subsection{The arithmetic case}

Let $p$ be a fixed prime number and let $\kappa$ denote a finite field of characteristic $p$; set $O_F = W(\kappa)$ to be the ring of $p\textrm{-typical}$ Witt vectors with coefficients in $\kappa$ and $F = \Fr(O_F)$.
Let $\overline{F}$ denote a fixed algebraic closure of $F$, $\CC_p := \widehat{\overline{F}}$ the $\padic$ completion, and $G_F = \Gal(\overline{F}/F)$ the absolute Galois group of $F$.
Further, let $F_{\infty} = \cup_n F(\mu_{p^n})$ with $\Gamma_F := \Gal(F_{\infty}/F)$ and $H_F := \Gal(\overline{F}/F_{\infty})$.
Finally, let $\CC_p^{\flat}$ denote the tilt of $\CC_p$.

\subsubsection{\texorpdfstring{$(\varphi, \Gamma_F)$}{-}-modules}

Using a certain period ring $\mbfa \subset W(\CC_p^{\flat})$ stable under the Frobenius on Witt vectors and the $G_F\textrm{-action}$ (see \S \ref{subsec:relative_phi_gamma_mod} for precise definition), Fontaine functorially attached to any $\ZZ_p\textrm{-representation}$ $T$ of $G_F$ (i.e. finitely generated $\ZZ_p\textrm{-modules}$ equipped with a linear and continuous $G_F\textrm{-action}$), the module $\mbfd(T) = (\mbfa \otimes_{\ZZ_p} T)^{H_F}$ over the two dimensional local ring $\mbfa_F = \mbfa^{H_F}$.
The module $\mbfd(T)$ is equipped with a (induced from $\mbfa$) Frobenius-semilinear operator $\varphi$ such that the image of $\varphi$ generates $\mbfd(T)$, i.e. $\mbfd(T)$ is étale.
Moreover, $\mbfd(T)$ is equipped with a continuous and semilinear action of $\Gamma_F$ and if $T$ is free the $\mbfa_F\textrm{-rank}$ of $\mbfd(T)$ equals the $\ZZ_p\textrm{-rank}$ of $T$.
In \cite{fontaine-festschrift} Fontaine estalished an equivalence of categories between $\ZZ_p\textrm{-representation}$s of $G_F$ and étale $(\varphi, \Gamma_F)\textrm{-modules}$ over $\mbfa_F$.
Furthermore, this construction naturally extends to $\padic$ representations of $G_F$.
Namely, using the period ring $\mbfb = \mbfa\big[\frac{1}{p}\big]$, Fontaine functorially attached to any $\padic$ representation $V$ of $G_F$ an étale $(\varphi, \Gamma_F)\textrm{-module}$ $\mbfd(V) = (\mbfb \otimes_{\QQ_p} V)^{H_F}$ over $\mbfb_F = \mbfb^{H_F}$ (i.e. there exists a $\ZZ_p\textrm{-lattice}$ $T \subset V$ such that $\mbfd(T)$ is an étale $(\varphi, \Gamma_F)\textrm{-module}$ over $\mbfa_F$).
Moreover, he showed that this induces an equivalence between $\padic$ representations of $G_F$ and étale $(\varphi, \Gamma_F)\textrm{-modules}$ over $\mbfb_F$.

\subsubsection{Crystalline representations of \texorpdfstring{$G_F$}{-}}

Using another period ring $\mbfb_{\crys}$ also equipped with a Frobenius and continuous $G_F\textrm{-action}$ (see \S \ref{subsec:relative_crystalline_period_rings} for precise definition), Fontaine functorially attached to any $\padic$ representation $V$ of $G_F$ an $F\textrm{-vector}$ space $\mbfd_{\crys}(V) = (\mbfb_{\crys} \otimes_{\QQ_p} V)^{G_F}$.
The $F\textrm{-vector}$ space $\mbfd_{\crys}(V)$ is a filtered $\varphi\textrm{-module}$, i.e. it is equipped with a (induced from $\mbfb_{\crys}$) Frobenius-semilinear operator $\varphi$ and a filtration.
In case $\dim_F \mbfd_{\crys}(V) = \dim_{\QQ_p} V$, such a representation is said to be crystalline (the terminology \textit{crystalline} comes from the fact that for a smooth proper scheme $X/O_F$ and $i \in \NN$ the $\padic$ étale cohomology group of the generic fiber $V_i = H^i_{\etale}(X_{\overline{F}}, \QQ_p)$ is crystalline as a $G_F\textrm{-representation}$ and the crystalline cohomology group of the special fiber $H^i_{\crys}(X_{\kappa}/F)$ is naturally isomorphic to $\mbfd_{\crys}(V_i)$).
Restricting the functor $\mbfd_{\crys}$ to the subcategory of crystalline representations, in \cite{fontaine-annals} Fontaine observed that the associated filtered $\varphi\textrm{-modules}$ are weakly admissible (a property relating the endomorphism $\varphi$ and filtration on $\mbfd_{\crys}(V)$ in a non-trivial manner).
In fact, in \cite{fontaine-colmez-weakadmis} Colmez and Fontaine showed that crystalline representations of $G_F$ are equivalent to weakly admissible filtered $\varphi\textrm{-modules}$.

\subsubsection{Arithmetic Wach modules}

From the discussion above, it is a natural question to ask: Does there exist some direct relation between the \'etale $(\varphi, \Gamma)\textrm{-module}$ of a crystalline representation and its associated weakly admissible filtered $\varphi\textrm{-module}$?
For a fixed representation, this question could be rephrased in terms of comparing certain elements of the period rings $\mbfb$ and $\mbfb_{\crys}$.
However, the rings $\mbfb$ and $\mbfb_{\crys}$ are not comparable.
So to answer this question, Fontaine considered a smaller period ring $\mbfb^+ \subset \mbfb$ stable under Frobenius and $G_F\textrm{-action}$ and such that $\mbfb^+ \rightarrowtail \mbfb_{\crys}$ stable under Frobenius and $G_F\textrm{-action}$.
Using $\mbfb^+$ he defined: a $\padic$ representation $V$ of $G_F$ is said to be of finite height if the associated $(\varphi, \Gamma_F)\textrm{-module}$ $\mbfd(V)$ admits a $(\varphi, \Gamma_F)\textrm{-stable}$ lattice over the subring $\mbfb_F^+ = (\mbfb^+)^{H_F} \subset \mbfb_F$ (see \S \ref{subsec:arithmetic_wach} for precise definitions).

In \cite{fontaine-festschrift} Fontaine conjectured that for a crystalline representation $V$ of $G_F$ there exist lattices inside $\mbfd(V)$ over which the action of $\Gamma_F$ admits a simpler form.
More precisely, finite height and crystalline representations of $G_F$ are related as follows: 

\begin{thm}[{\cite[Wach]{wach-pot-crys}, \cite[Colmez]{colmez-finite-height}, \cite[Berger]{berger-differentielles}}]\label{thm:wach_crys_arith}
	Let $V$ be a $\padic$ representation of $G_F$.
	Then $V$ is crystalline if and only if it is of finite height and there exists $r \in \ZZ$ and a free $\mbfb_F^+\textrm{-submodule}$ $N \subset \mbfd(V)$ of rank $= \dim_{\QQ_p} V$, stable under the action of $\Gamma_F$ and such that $\Gamma_F$ acts trivially over $(N/\pi N)(-r)$.
\end{thm}

Here $(-r)$ denotes the Tate twist.
Note that in the situation of Theorem \ref{thm:wach_crys_arith}, the module $N$ is not unique.
A functorial construction was given by Berger in \cite{berger-limites-cristallines}, i.e. to any $\padic$ crystalline representation $V$ of $G_F$ he attached a canonical $\mbfb_F^+\textrm{-submodule}$ $\mbfn(V) \subset \mbfd(V)$ which he called the Wach module of $V$.
Moreover, Berger established an equivalence of categories between crystalline representations of $G_F$ and Wach modules over $\mbfb_F^+$.
Furthermore, Berger obtained an integral version of his result by considering the period ring $\mbfa^+ = \mbfa \cap \mbfb^+ \subset \mbfb$ stable under Frobenius and $G_F\textrm{-action}$.
He showed that for a crystalline representation $V$ of $G_F$, there exists a bijection between $G_F\textrm{-stable}$ $\ZZ_p\textrm{-lattices}$ $T \subset V$ and integral Wach modules $\mbfn(T) \subset \mbfn(V)$ where $\mbfn(T)$ is defined over the integral subring $\mbfa_F^+ = (\mbfa^+)^{H_F}$.
Finally, given $\mbfn(V)$ one can canonically recover the other linear algebraic object attached to $V$, i.e. $\mbfd_{\crys}(V)$ (see \cite[Propositions II.2.1 \& III.4.4]{berger-limites-cristallines}).

\subsection{The relative case}

The motivation for defining Wach modules in the relative case and exploring its relation with $\pazo \mbfd_{\crys}(V)$ (see \S \ref{sec:relative_padic_Hodge_theory} for notations) comes from the hope of computing Galois cohomology of $\padic$ representations using syntomic complexes with coefficients in $\pazo \mbfd_{\crys}(V)$.
Using syntomic complexes and techniques from the theory of $(\varphi, \Gamma)\textrm{-modules}$, this was done for the trivial representation by Colmez and Nizio{\l} \cite{colmez-niziol-nearby-cycles}.
A generalization of these complexes to non-trivial coefficients can be found in \cite{abhinandan-syntomic} and \cite[Chapter 5]{abhinandan-thesis}.

In this article, we are interested in the $\padic$ Hodge theory of an \'etale algebra over a formal torus defined over $O_F$.
More precisely, let $d \in \NN$ and $X = (X_1, X_2, \ldots, X_d)$ be some indeterminates, $O_F\{X, X^{-1}\}$ the $\padic$ completion of a $d\textrm{-dimensional}$ torus over $O_F$ and let $R$ denote the $\padic$ completion of an \'etale algebra over $O_F\{X, X^{-1}\}$ with non-empty and geometrically integral special fiber.
Next, let $G_{R}$ denote the \'etale fundamental group of $R\big[\frac{1}{p}\big]$ and $\Gamma_{R}$ the Galois group of the cyclotomic tower over $R$ and $H_{R} = \kert(G_{R} \rightarrow \Gamma_{R})$ (see \S \ref{subsec:relative_phi_gamma_mod} for precise definitions).
In the relative setting, on one hand Brinon has developed the theory of crystalline representations of $G_R$ \cite{brinon-padicrep-relatif}, while on the other hand Andreatta, Brinon and Iovita have developed the theory of $(\varphi, \Gamma_R)\textrm{-modules}$ in \cite{andreatta-generalized-phiGamma, andreatta-brinon-surconvergence, andreatta-iovita-relative-phiGamma}.

\begin{rem}
	Note that in Theorem \ref{thm:wach_crys_arith} it is important to restrict to an unramified extension $F/\QQ_p$.
	For ramified extensions, such a statement does not hold in general.
	Therefore, in the relative setting we consider an analogue of ``unramified extension of $\QQ_p$'' (indeed, by removing the geometric coordinates one obtains $R = O_F$).
\end{rem}

\subsubsection{\texorpdfstring{$(\varphi, \Gamma_R)$}{-}-modules}

Analogous to the arithmetic case, we have relative period rings $\mbfa \subset \mbfb \supset \mbfb^+$ and $\mbfa^+ = \mbfa \cap \mbfb^+ \subset \mbfb$ (see \S \ref{subsec:relative_phi_gamma_mod} for precise definition) equipped with Frobenius and a continuous action of $G_R$.
Let $V$ be a $\padic$ representation of $G_{R}$, then one can functorially attach to $V$ a projective and \'etale $(\varphi, \Gamma_{R})\textrm{-module}$ $\mbfd(V) = (\mbfb \otimes_{\QQ_p} V)^{H_R}$ over $\mbfb_{R} = \mbfb^{H_R}$ of rank $=\dim_{\QQ_p} V$ equipped with a Frobenius-semilinear operator $\varphi$ and a semilinear and continuous action of $\Gamma_{R}$.
This induces an equivalence of categories between $\padic$ representations of $G_R$ and étale $(\varphi, \Gamma_R)\textrm{-modules}$ over $\mbfb_R$.
Similarly, using the period ring $\mbfa$ one can functorially attach to any $\ZZ_p\textrm{-representation}$ $T$ of $G_R$ an étale $(\varphi, \Gamma_R)\textrm{-module}$ $\mbfd(T) = (\mbfa \otimes_{\ZZ_p} T)^{H_R}$ over the period ring $\mbfa_{R} = \mbfa^{H_R}$.
Again, this induces an equivalence between $\ZZ_p\textrm{-representation}$s of $G_R$ and étale $(\varphi, \Gamma_R)\textrm{-modules}$ over $\mbfa_R$.

\subsubsection{Relative Wach modules}

Using the period ring $\mbfa^+$ we set $\mbfd^+(T) = (\mbfa^+ \otimes_{\ZZ_p} T)^{H_{R}}$, which is a $(\varphi, \Gamma_{R})\textrm{-module}$ over $\mbfa_R^+ = (\mbfa^+)^{H_R}$ and let $q = \frac{\varphi(\pi)}{\pi}$, where $\pi$ is the usual element in Fontaine's constructions (see \S \ref{subsec:rel_deRham_ring} for notations).
Note that for a finite free $\ZZ_p\textrm{-representation}$ $T$ of $G_R$ the $\mbfa_R\textrm{-module}$ $\mbfd(T)$ is finite projective, however it is not known whether $\mbfd^+(T)$ is projective.
So, we introduce the following definition:
\begin{defi}\label{defi:posfinq_height_reps}
	A \textit{positive finite $q\textrm{-height}$} representation is a $\padic$ representation $V$ of $G_{R}$ admitting a $\ZZ_p\textrm{-lattice}$ $T \subset V$ such that there exists a finite projective $\mbfa_R^+\textrm{-submodule}$ $\mbfn(T) \subset \mbfd^+(T)$ of rank $=\dim_{\QQ_p} V$ satisfying the following conditions:
	\begin{enumromanup}
		\item $\mbfn(T)$ is stable under the action of $\varphi$ and $\Gamma_{R}$ and $\mbfa_{R} \otimes_{\mbfa_R^+} \mbfn(T) \isomorphic \mbfd(T)$;
		
		\item The $\mbfa_R^+\textrm{-module}$ $\mbfn(T) / \varphi^{\ast}(\mbfn(T))$ is killed by $q^{s}$ for some $s \in \NN$;
		
		\item The action of $\Gamma_{R}$ is trivial on $\mbfn(T) / \pi \mbfn(T)$;
		
		\item There exists $R\prm \subset \overline{R}$ finite étale over $R$ such that the $\mbfa_{R\prm}^+\textrm{-module}$ $\mbfa_{R\prm}^+ \otimes_{\mbfa_R^+} \mbfn(T)$ is free.
	\end{enumromanup}
	The module $\mbfn(T)$ is a \textit{Wach module} associated to $T$ and we set $\mbfn(V) := \mbfn(T)\big[\frac{1}{p}\big]$ which satisfies analogous properties.
	The \textit{height} of $V$ is the smallest $s \in \NN$ satisfying (ii) above.
\end{defi}

\begin{rem}
	\begin{enumromanup}
	\item A finite $q\textrm{-height}$ representation is twist of a positive one by some power of the $\padic$ cyclotomic character (see Definition \ref{defi:wach_reps} for details).
		The terminology ``positive'' refers to the fact that the Wach module $\mbfn(T)$ is stable under the Frobenius-semilinear operator $\varphi$.
		It is motivated by the fact (and as we will show) that $V$ is positive crystalline (see Theorem \ref{thm:intro_wachtocrys}).

	\item In the arithmetic case, i.e. $R = O_F$, the notion of finite height representations in Theorem \ref{thm:wach_crys_arith} and finite $q\textrm{-height}$ representations in Definition \ref{defi:posfinq_height_reps} are related.
		In fact, in the arithmetic case using Definition \ref{defi:posfinq_height_reps} one obtains the functorial Wach module of Berger mentioned above (see \cite[Proposition II.1.1]{berger-limites-cristallines}).
	\end{enumromanup}
\end{rem}

\subsubsection{Crystalline representations of \texorpdfstring{$G_R$}{-}}

Using the period ring $\pazo \mbfb_{\crys}(\overline{R})$ Brinon functorially attaches to any $\padic$ representation $V$ of $G_R$ an $R\big[\frac{1}{p}\big]\textrm{-module}$
\begin{equation*}
	\pazo \mbfd_{\crys}(V) := \big(\pazo \mbfb_{\crys}(\overline{R}) \otimes_{\QQ_p} V\big)^{G_R}.
\end{equation*}
The module $\pazo \mbfd_{\crys}(V)$ is called a filtered $(\varphi, \partial)\textrm{-module}$, i.e. it is equipped with a filtration, a Frobenius-semilinear endomorphism $\varphi$ and a quasi-nilpotent integrable connection $\partial$ satisfying Griffiths transversality with respect to the filtration (see \S \ref{subsec:relative_padic_reps} for precise definitons).
The representation $V$ is said to be crystalline if the natural map is an isomorphism
\begin{equation*}
	\pazo \mbfb_{\crys}(\overline{R}) \otimes_{R[1/p]} \pazo \mbfd_{\crys}(V) \isomorphic \pazo \mbfb_{\crys}(\overline{R}) \otimes_{\QQ_p} V,
\end{equation*}
compatible with Frobenius, filtration, connection and the action of $G_R$ on each side.
Moreover, Brinon also defined the notion of weak admissibility in the relative case and showed that $\pazo \mbfd_{\crys}(V)$ is weakly admissible for crystalline representations (see \cite[Chapitre 8]{brinon-padicrep-relatif} for more details). 

\begin{nota}
	We use period rings such as $\pazo \mbfb_{\crys}(\overline{R})$ which is a modified version of Fontaine's relative period ring $\mbfb_{\crys}(\overline{R})$ (see \S \ref{subsec:relative_crystalline_period_rings} for details).
	The notation $\pazo$ here indicates that apart from Frobenius, filtration and $G_R\textrm{-action}$, we have a connection over $\pazo \mbfb_{\crys}(\overline{R})$ and we will call such rings fat relative period rings.
	However, note that in \cite{brinon-padicrep-relatif} Brinon denotes these rings as $B_{\crys}(R)$ and $B_{\crys}^{\nabla}(R)$, respectively.
	Similarly, we will use the notation $\pazo \mbfd_{\crys}(V)$ and $\mbfd_{\crys}(V)$ for modules instead of Brinon's $D_{\crys}(V)$ and $D_{\crys}^{\nabla}(V)$, respectively.
	We hope it is not too confusing for the reader.
\end{nota}

\subsubsection{Main result}

Our aim is to show that for positive finite $q\textrm{-height}$ representations, the $\mbfb_{R}^+\textrm{-module}$ $\mbfn(V)$ and the $R\big[\frac{1}{p}\big]\textrm{-module}$ $\pazo \mbfd_{\crys}(V)$ are related in a precise manner and the latter can be recovered from the former.
To relate these objects we consider the ring $R[\varpi]$ where $\varpi = \zeta_{p}-1$ for a primitive $p\textrm{-th}$ root of unity $\zeta_{p}$ (take $\varpi = \zeta_{p^2}-1$ if $p=2$ for a primitve $p^2\textrm{-th}$ root of unity $\zeta_{p^2}$), and using this ring we construct a fat relative period ring $\pazo \mbfa_{R, \varpi}^{\textpd} \subset \pazo \mbfb_{\crys}(\overline{R})$ equipped with compatible Frobenius, filtration, connection and the action of $\Gamma_{R}$ (see \S \ref{subsec:main_result} for precise definitions).
The main result of this article is as follows:
\begin{thm}[see Theorem \ref{thm:crys_wach_comparison}]\label{thm:intro_wachtocrys}
	Let $V$ be a positive finite $q\textrm{-height}$ representation of $G_{R}$, then
	\begin{enumromanup}
	\item $V$ is a positive crystalline representation.

	\item Let $M := \big(\pazo\mbfa_{R, \varpi}^{\textpd} \otimes_{\mbfa_R^+} \mbfn(T)\big)^{\Gamma_{R}}$, then after extending scalars to $\pazo \mbfa_{R, \varpi}^{\textpd}$ and inverting $p$, we obtain a natural isomorphism
		\begin{equation*}
			\pazo \mbfa_{R, \varpi}^{\textpd} \otimes_R M\big[\tfrac{1}{p}\big] \isomorphic \pazo \mbfa_{R, \varpi}^{\textpd} \otimes_{\mbfa_R^+} \mbfn(V),
		\end{equation*}
		compatible with Frobenius, filtration, connection and the action of $\Gamma_{R}$ on each side.

	\item We have an isomorphism of $R\big[\frac{1}{p}\big]\textrm{-modules}$ 
		\begin{equation*}
			\pazo \mbfd_{\crys}(V) \lisomorphic \big(\pazo\mbfa_{R, \varpi}^{\textpd} \otimes_{\mbfa_R^+} \mbfn(T)\big)^{\Gamma_{R}}\big[\tfrac{1}{p}\big],
		\end{equation*}
		compatible with Frobenius, filtration, and connection on each side.
		Therefore, we obtain a comparison isomorphism
		\begin{equation*}
			\pazo \mbfa_{R, \varpi}^{\textpd} \otimes_{\mbfa_R^+} \mbfn(V) \isomorphic \pazo \mbfa_{R, \varpi}^{\textpd} \otimes_R \pazo \mbfd_{\crys}(V),
		\end{equation*}
		compatible with Frobenius, filtration, connection and the action of $\Gamma_{R}$ on each side.
	\end{enumromanup}
\end{thm}

Let us mention the idea of the proof.
In case $\mbfn(T)$ is free, we proceed in two steps: 
First, we describe a process (see Proposition \ref{prop:crys_from_wach_mod} for details) by which we can recover a submodule of $\pazo \mbfd_{\crys}(V)$ starting with the Wach module $\mbfn(T)$, establishing a comparison over $\pazo \mbfa_{R, \varpi}^{\textpd}$ between the submodule obtained and the Wach module.
Next, the claims made in the theorem are shown by exploiting properties of Wach modules and the comparison obtained in the first step.
In the first step, one can take two approaches to obtain generators of the promised submodule of $\pazo \mbfd_{\crys}(V)$: either by taking $\Gamma_R\textrm{-fixed}$ points of $\pazo \mbfa_{R, \varpi}^{\textpd} \otimes_{\mbfa_R^+} \mbfn(T)$ (by successively approximating for $\Gamma_R\textrm{-action}$ on a basis of $\mbfn(T)$); or by taking elements killed by differential operators defined using topological generators of $\Gamma_R$ (see Lemma \ref{lem:horizontal_elems} for details).
In this paper, we take the latter approach whereas the former approach is detailed in author's thesis (see \cite[Chapter 3]{abhinandan-thesis}).
In the general case when $\mbfn(T)$ is projective, using property (iv) in Definition \ref{defi:posfinq_height_reps} one can pass to an extension $\mbfa_R^+ \subset \mbfa_{R\prm}^+$ to obtain a free Wach module, then use the preceding argument and finally apply Galois descent to obtain the theorem (see Proposition \ref{prop:crys_from_wach_mod} for details).
Finally, we also show that all one-dimensional crystalline representations are of finite $q\textrm{-height}$ and for such representations one can directly compare $\pazo \mbfd_{\crys}(V)$ and the Wach module $\mbfn(V)$.

\subsection{Relative Fontaine-Laffaille modules}

After obtaining Theorem \ref{thm:intro_wachtocrys} above, it is natural to wonder if a converse statement could be true, i.e. starting with a lattice $T \subset V$ of a crystalline representation $G_R$, is it possible to construct the Wach module $\mbfn(T)$?
In the arithmetic setting, for $\padic$ crystalline representations of $G_F$, this was shown to be true by Wach \cite{wach-pot-crys}, and the statement was refined by Berger \cite{berger-limites-cristallines}.
In the relative case, the picture is quite encouraging when we restrict to Hodge-Tate length $\leq p-2$ (also see Remark \ref{rem:converse_result}).

For a $\padic$ crystalline representation of $G_F$ with Hodge-Tate length $\leq p-1$, there exists a canonical $O_F\textrm{-lattice}$ inside $\mbfd_{\crys}(V)$ called the Fontaine-Laffaille module defined in \cite{fontaine-laffaille}.
In this case, Wach constructed Wach modules out of Fontaine-Laffaille data in \cite{wach-cristallines-torsion}.
In the relative setting, Faltings studied relative Fontaine-Laffaille modules in \cite{faltings-crystalline} and used them to functorially recover $\ZZ_p\textrm{-lattices}$ inside crystalline representations of $G_{R}$.
Recently, for free relative Fontaine-Laffaille modules of filtration length $\leq p-2$, adapting techniques from Wach's computations, Tsuji has constructed generalized representations of $G_{R}$ over $\mbfa_{\inf}(\overline{R})$ (see \cite{tsuji-ainf-genrep}).
In fact, it is possible to show that starting with a free relative Fontaine-Laffaille module, one can obtain a free relative Wach module over $\mbfa_R^+$.

\begin{thm}[see Theorem \ref{thm:fl_to_wach}]
	Let $M$ be a free relative Fontaine-Laffaille module over $R$ of level $[0, p-2]$, and let $T_{\crys}(M)$ denote the associated $\ZZ_p\textrm{-representation}$ of $G_{R}$.
	Then, the $\padic$ representation $V_{\crys}(M) := \QQ_p \otimes_{\ZZ_p} T_{\crys}(M)$ is a positive finite $q\textrm{-height}$ representation.
\end{thm}
Twisting the representation thus obtained by powers of the cyclotomic character, generalizes the statement to all free Fontaine-Laffaille modules with filtration length $\leq p-2$.

The proof of the theorem crucially exploits the computation of Fontaine \cite{fontaine-corps-des-periodes}, Wach \cite{wach-cristallines-torsion} and Tsuji \cite{tsuji-ainf-genrep}.
It follows in three steps:
First, starting with a Fontaine-Laffaille module, we obtain an $\mbfa_{R, \varpi}^{\textpd}\textrm{-module}$ using formal properties of crystalline site for maps $\theta : \mbfa_{R, \varpi}^{\textpd} \twoheadrightarrow R$ and $\theta_{R} : \pazo \mbfa_{R, \varpi}^{\textpd} \twoheadrightarrow R$ (see \S \ref{subsubsec:fl_to_arplus_mod} for details).
Next, we exploit equivalence of categories in Theorem \ref{thm:arplus_arpd_cat_equiv} obtained by scalar extension along the maps $\mbfa_{R, \varpi}^{\textpd} \twoheadrightarrow \mbfa_{R, \varpi}^{\textpd} / I^{(p-1)} \mbfa_{R, \varpi}^{\textpd} \lisomorphic \mbfa_{R, \varpi}^+ / I^{(p-1)} \mbfa_{R, \varpi}^+ \twoheadleftarrow \mbfa_{R, \varpi}^+$ (see Proposition \ref{prop:arplus_mod_arpd_mod_iso} for explanations).
This gives us an $\mbfa_{R, \varpi}^+\textrm{-module}$ with precise description of the Frobenius and the action of $\Gamma_{R}$.
Finally, we descend over to the ring $\mbfa_R^+$ by exploiting the Frobenius and $\Gamma_{R}\textrm{-action}$, thus obtaining a Wach module over $\mbfa_R^+$ and proving the theorem (see \S \ref{subsubsec:obtain_wach_mod}).

\begin{rem}
	In a recent work, Morrow and Tsuji have developed a theory of coefficients for integral $\padic$ Hodge theory in \cite{morrow-tsuji-coeff}.
	Extending scalars of relative Wach modules along $O_F[[\pi]] \rightarrow \mbfa_{\inf}(O_{\overline{F}})$ would yield generalized representions over $\mbfa_{\inf}^{\square}(R)$ in the sense of Morrow-Tsuji.
\end{rem}

\begin{rem}\label{rem:converse_result}
	Recent developments in the theory of prismatic crystals \cite{bhatt-scholze-prismatic-crystals, du-liu-moon-shimizu, guo-reinecke-analytic}, indicate that to obtain a full converse statement, i.e. to construct Wach modules from lattices inside crystalline representations, one needs to generalize Definition \ref{defi:posfinq_height_reps} slightly.
	This is a work in progress and we will report further on this line of investigation in future.
\end{rem}

\subsection{Setup and notations}\label{subsec:setup_nota}

In this section we will describe the setup for the rest of the text and fix some notations.

\begin{conv}
	We will work under the convention that $0 \in \NN$, the set of natural numbers.
\end{conv}

Let $p$ be a fixed prime number, $\kappa$ a finite field of characteristic $p$, $W := W(\kappa)$ the ring of $p\textrm{-typical}$ Witt vectors with coefficients in $\kappa$ and $F := W\big[\frac{1}{p}\big]$, the fraction field of $W$.
In particular, $F$ is an unramified extension of $\QQ_p$ with ring of integers $O_F = W$.
Let $\overline{F}$ be a fixed algebraic closure of $F$ so that its residue field, denoted as $\overline{\kappa}$, is an algebraic closure of $\kappa$.
Further, we denote by $G_F = \Gal(\overline{F}/F)$, the absolute Galois group of $F$.

Let $Z = (Z_1, \ldots, Z_s)$ denote a set of indeterminates and $\smbfk = (k_1, \ldots, k_s) \in \NN^s$ be a multi-index, then we write $Z^{\smbfk} := Z_1^{k_1} \cdots Z_s^{k_s}$.
For $\smbfk \rightarrow +\infty$ we will mean that $\sum k_i \rightarrow +\infty$.
Now for a topological algebra $\Lambda$ we define 
\begin{equation*}
	\Lambda\{Z\} := \Big\{\sum_{\smbfk \in \NN^s} a_{\smbfk} Z^{\smbfk}, \hspace{1mm} \textrm{where} \hspace{1mm} a_{\smbfk} \in \Lambda \hspace{1mm} \textrm{and} \hspace{1mm} a_{\smbfk} \rightarrow 0 \hspace{1mm} \textrm{as} \hspace{1mm} \smbfk \rightarrow +\infty\Big\}.
\end{equation*}
We fix $d \in \NN$ and let $X = (X_1, X_2, \ldots, X_d)$ be some indeterminates.
Let $R$ be the $\padic$ completion of an \'etale algebra over $O_F\{X, X^{-1}\}$ with non-empty geometrically integral special fiber. 
In particular, we have a presentation
\begin{equation*}
	R = O_F\{X, X^{-1}\}\{Z_1, \ldots, Z_s\} / \hspace{0.5mm} (Q_1, \ldots, Q_s),
\end{equation*}
where $Q_i(Z_1, \ldots, Z_s) \in O_F\{X, X^{-1}\}[Z_1, \ldots, Z_s]$ for $1 \leq i \leq s$ are multivariate polynomials such that $\det \big(\frac{\partial Q_i}{\partial Z_j}\big)_{1 \leq i,j \leq s}$ is invertible in $R$.
The algebra $R\big[\frac{1}{p}\big]$ is the relative analogue of ``finite unramified extension of $\QQ_p$'' (indeed, by removing the geometric coordinates we will obtain $R\big[\frac{1}{p}\big] = F$).

\begin{rem}
	Note that Theorem \ref{thm:wach_crys_arith} serves as our main motivation for the theory developed in this article.
	The assumptions we put on $R$ generalizes the fact that ``$F$ is unramified over $\QQ_p$''.
\end{rem}

The $\padic$ Hodge theory over $R$ is the study of $\padic$ representations of the \'etale fundamental group of $R\big[\frac{1}{p}\big]$, which we introduce next.
We fix an algebraic closure $\overline{\Fr(R)}$ of $\Fr(R)$ containing $\overline{F}$.
Let $\overline{R}$ denote the union of finite $R\textrm{-subalgebras}$ $S \subset \overline{\Fr(R)}$, such that $S\big[\frac{1}{p}\big]$ is \'etale over $R\big[\frac{1}{p}\big]$.
Let $\overline{\eta}$ denote a geometric point of the generic fiber $\Spec R\big[\frac{1}{p}\big]$ and let $G_{R} := \pi_1^{\etale}\big(\Spec R\big[\frac{1}{p}\big], \overline{\eta}\big)$ denote the \'etale fundamental group.
By \cite[Expos\'e V, \S 8]{sga1}, we can write this étale fundamental group as the Galois group (of the fraction field of $\overline{R}\big[\frac{1}{p}\big]$ over the fraction field of $R\big[\frac{1}{p}\big]$)
\begin{equation*}
	G_{R} = \pi_1^{\etale}\big(\Spec R\big[\tfrac{1}{p}\big], \overline{\eta}\big) = \Gal\big(\overline{R}\big[\tfrac{1}{p}\big] / R\big[\tfrac{1}{p}\big]\big).
\end{equation*}

For $n \in \NN$, let $F_n := F(\mu_{p^n})$.
From now onwards, we will fix some $m \in \NN_{\geq 1}$ (take $m \in \NN_{\geq 2}$ if $p=2$) and set $K := F_m$, with its ring of integers $O_K$.
The element $\varpi = \zeta_{p^m}-1 \in O_K$ is a uniformizer of $K$, and its minimal polynomial $P_{\varpi}(X) = \frac{(1+X)^{p^m}-1}{(1+X)^{p^{m-1}}-1}$ is an Eisenstein polynomial in $W[X]$ of degree $e := [K:F] = p^{m-1}(p-1)$.
Finally, for $S = R[\varpi] = O_K \otimes_{O_F} R$ we have that $R[\varpi]$ is totally ramified at the prime ideal $(p) \subset R[\varpi]$.
And similar to above, we obtain Galois groups $G_K \triangleleft G_F$ and $G_S \triangleleft G_R$ respectively, such that $G_R / G_S = G_F / G_K = \Gal(K/F)$.
Finally, we have that $R$ and $R[\varpi]$ are \textit{small} algebras in the sense of Faltings (see \cite[\S II 1(a)]{faltings-padic-hodge-theory}).

For $k \in \NN$, let $\Omega^k_R$ denote the $\padic$ completion of the module of $k\textrm{-differentials}$ of $R$ relative to $\ZZ$.
Then, we have
\begin{equation*}
	\Omega^1_{R} = \bmoplus_{i=1}^d R \dlog X_i, \hspace{2mm} \textrm{and} \hspace{2mm} \Omega^k_{R} = \bmwedge_R^k \Omega^1_{R}.
\end{equation*}
For $S = R[\varpi]$, the natural map $\Omega^k_R \otimes_{R} S \rightarrow \Omega^k_{S}$ is bijective.
In particular, we get that 
\begin{equation*}
	\Omega^k_{S} = \bmwedge_R^k \big(\bmoplus_{i=1}^d S \dlog X_i\big).
\end{equation*}

We also have that $R/pR \isomorphic S/\varpi S$ and for all $n \in \NN$, $R / p^nR$ is a smooth $\ZZ / p^n\ZZ\textrm{-algebra}$.
Finally, we have a unique lift $\varphi : R \rightarrow R$ of the absolute Frobenius $x \mapsto x^p$ over $R/p R$ such that $\varphi(X_i) = X_i^p$, for all $1 \leq i \leq d$ (in general, a lift of Frobenius modulo $p$ need not be unique, see \cite[p.9]{brinon-padicrep-relatif}).

\begin{conv}
	Let $A$ be a ring and $I \subsetneq A$ an ideal.
	We say that an $A\textrm{-module}$ $M$ is $I\textrm{-adically}$ complete if and only if $M \isomorphic \lim_n M / I^n M$.
\end{conv}

\vspace{3mm}

\noindent \textbf{Acknowledgements.} The work presented here was part of my doctoral thesis.
I would like to thank my supervisor Denis Benois for encouraging me to think about this question, having many helpful discussions, reading the article and suggesting improvements.
I would also like to thank Takeshi Tsuji for pointing out connections between our work and his results in \cite{tsuji-ainf-genrep} and helpful discussions concerning Fontaine-Laffaille modules which resulted in a simplified exposition of the proof of Theorem \ref{thm:fl_to_wach}.
Moreover, I would like to thank Lorenzo Ramero for discussions around Remark \ref{rem:spectral_norm} and suggesting improvements in writing and Nicola Mazzari for suggesting many improvements in the exposition.
Finally, many thanks to the referee for carefully reading the article and suggesting several improvements, in particular, proofs of Lemma \ref{lem:reg_frob_finht} and Lemma \ref{lem:horizontal_elems} could be shortened.

\cleardoublepage

\section{\texorpdfstring{$p$}{-}-adic Hodge theory}\label{sec:relative_padic_Hodge_theory}

In this section we will recall some constructions and results in relative $\padic$ Hodge theory developed in \cite{brinon-padicrep-relatif}, albeit in a simpler setting compared to Brinon's book.
As we will be using different notations compared to Brinon, we will make most of the definitions explicit.

We are interested in exploring the relationship between $\padic$ crystalline representations and finite height representations of $G_{R}$.
This will be detailed in \S \ref{sec:relative_finite_height} and \S \ref{sec:fontaine_laffaile_to_wach}.
To carry out some computations in the aforementioned sections, we will need to extend our base field (hence the base ring) by adjoining some $p\textrm{-power}$ roots of unity (see the field $K$ and the ring $S = R[\varpi]$ in \S \ref{subsec:setup_nota}).
As a consequence, we will also require the corresponding period rings defined for such rings.
However, in \S \ref{subsec:rel_deRham_ring}, \S \ref{subsec:relative_crystalline_period_rings} \& \S \ref{subsec:relative_padic_reps} we will only recall results from \cite{brinon-padicrep-relatif} by fixing our base as $R$.
As we shall see the period rings will only depend on $\overline{R}$ and we have $\overline{S} = \overline{R} \subset \overline{\Fr(R)} = \overline{\Fr(S)}$, therefore  fixing our base as $R$ is sufficient (see \cite{brinon-padicrep-relatif} for general constructions).

\subsection{The de Rham period ring}\label{subsec:rel_deRham_ring}

We will recall definitions and properties of the relative version of Fontaine's period ring $\mbfb_{\dR}$ (see \cite{fontaine-corps-des-periodes} for classical case).

\subsubsection{The ring \texorpdfstring{$\CC^+(\overline{R})$}{-} and its tilt}

Let $\CC_p$ denote the $\padic$ completion of $\overline{F}$.
Recall that $\overline{R}$ is the union of finite $R\textrm{-subalgebras}$ $S \subset \overline{\Fr(R)} = \overline{\Fr(R[\varpi])}$, such that $S\big[\frac{1}{p}\big]$ is étale over $R\big[\frac{1}{p}\big]$.
Let $\CC^+(\overline{R})$ denote the $\padic$ completion of $\overline{R}$ and $\CC(\overline{R}) = \CC^+(\overline{R})\big[\frac{1}{p}\big]$.
We define the tilt $\CC^+(\overline{R})^{\flat} := \lim_{x \mapsto x^p} \CC^+(\overline{R}) / p = \lim_{x \mapsto x^p} \overline{R}/p$ and equip it with the inverse limit topology (where we equip $\overline{R}/p$ with the discrete topology) and let $\CC(\overline{R})^{\flat} = \CC^+(\overline{R})^{\flat}\big[\frac{1}{p^{\flat}}\big]$ for $p^{\flat} := (p, p^{1/p}, p^{1/p^2}, \ldots) \in \CC^+(\overline{R})^{\flat}$ and equipped with the coarsest ring topology such that $\CC^+(\overline{R})$ is an open subring.
Note that an element $x \in \CC(\overline{R})^{\flat}$ can be described as a sequence $(x_n)_{n \in \NN}$, with $x_n \in \CC(\overline{R})$ and $x_{n+1}^p = x_n$ for all $n \in \NN$.
These rings admit a continuous $G_R\textrm{-action}$ for the topology described.

We will fix some choices of compatible $p\textrm{-power}$ roots which will appear throughout the text.
Let $\varepsilon := (1, \zeta_p, \zeta_{p^2}, \ldots) \in \CC_p^{\flat}$, $X_i^{\flat} := \big(X_i, X_i^{1/p}, X_i^{1/p^2}, \ldots\big) \in \CC(\overline{R})^{\flat}$ for $1 \leq i \leq d$.
We set $\mbfa_{\inf}(\overline{R}) := W(\CC^+(\overline{R})^{\flat})$, the ring of $p\textrm{-typical}$ Witt vectors with coefficients in $\CC^+(\overline{R})^{\flat}$ equipped with weak topology (see \cite[\S 2.10]{andreatta-iovita-relative-phiGamma}).
The absolute Frobenius on $\CC^+(\overline{R})^{\flat}$ lifts to an endomorphism $\varphi : \mbfa_{\inf}(\overline{R}) \rightarrow \mbfa_{\inf}(\overline{R})$ and the $G_R\textrm{-action}$ extends to $\mbfa_{\inf}(\overline{R})$ such that the action is continuous for the weak topology.
For $x \in \CC^+(\overline{R})^{\flat}$, let $[x] = (x, 0, 0, \ldots) \in \mbfa_{\inf}(\overline{R})$ denote its Teichm\"uller representative.
So we have $[\varepsilon] \in \mbfa_{\inf}(\overline{R})$ with $g[\varepsilon] = [\varepsilon]^{\chi(g)}$ for $g \in G_R$ and $\chi : G_R \rightarrow \ZZ_p^{\times}$ the $\padic$ cyclotomic character and $\varphi([\varepsilon]) = [\varepsilon]^p$.
Now any element $x \in \mbfa_{\inf}(\overline{R})$ can be uniquely written as $x = \sum_{k \in \NN} p^k [x_k]$ for $x_k \in \CC^+(\overline{R})^{\flat}$.
So we set $\pi := [\varepsilon] - 1, \hspace{2mm} \pi_1 := \varphi^{-1}(\pi) = [\varepsilon^{1/p}] - 1$, and $\xi := \tfrac{\pi}{\pi_1}$.
Clearly we have $g(\pi) = (1 + \pi)^{\chi(g)} - 1$ for $g \in G_R$ and $\varphi(\pi) = (1+\pi)^p - 1$.

\subsubsection{Definition of \texorpdfstring{$\pazo \mbfb_{\textdr}(\overline{R})$}{-}}\label{subsubsec:deRham_defi}

We have Fontaine's $\theta\textrm{-map}$ defined as $\theta : \mbfa_{\inf}(\overline{R}) \rightarrow \CC^+(\overline{R})$ sending $\sum_{k \in \NN} p^k [x_k] \mapsto \sum_{k \in \NN} p^k x_k^{\sharp}$, it is a $G_R\textrm{-equivariant}$ surjective ring homomorphism whose kernel is principal and generated by, for example, $p - [p^{\flat}]$ or $\xi$ (see \cite[Proposition 2.4 (ii)]{fontaine-annals}).
By $\QQ_p\textrm{-linearity}$, the map $\theta$ can be extended to $\theta : \mbfa_{\inf}(\overline{R})\big[\frac{1}{p}\big] \rightarrow \CC(\overline{R})$ and we define
\begin{equation*}
		\mbfb_{\dR}^+(\overline{R}) := \lim_n \mbfa_{\inf}(\overline{R})\big[\tfrac{1}{p}\big] / \hspace{0.5mm} (\kert \theta)^n,
\end{equation*}
as the $(\kert \theta)\textrm{-adic}$ completion of $\mbfa_{\inf}(\overline{R})\big[\frac{1}{p}\big]$.
The ring $\mbfb_{\dR}^+(\overline{R})$ is an $F\textrm{-algebra}$ and admits a $G_R\textrm{-action}$.
The map $\theta$ further extends to a $G_R\textrm{-equivariant}$ surjective ring homomorphism $\theta : \mbfb_{\dR}^+(\overline{R}) \rightarrow \CC(\overline{R})$ with $\kert \theta = t \mbfb_{\dR}^+(\overline{R})$, where $t := \log [\varepsilon] = \log (1+\pi) = \sum_{k \in \NN} (-1)^k \frac{\pi^{k+1}}{k+1}  \in \mbfb_{\dR}^+(\overline{R})$ such that $g \in G_R$ acts by $g(t) = \chi(g) t$.
By functoriality of the construction of $\mbfb_{\dR}^+(\overline{R})$, the homomorphism $O_{\overline{F}} \rightarrow \overline{R}$ induces an injection $\mbfb_{\dR}^+(O_{\overline{F}}) \rightarrow \mbfb_{\dR}^+(\overline{R})$.
The ring $\mbfb_{\dR}^+(\overline{R})$ is $t\textrm{-torsion}$ free, so we set $\mbfb_{\dR}(\overline{R}) := \mbfb_{\dR}^+(\overline{R})\big[\tfrac{1}{t}\big]$.
The $G_R\textrm{-action}$ extends to $\mbfb_{\dR}(\overline{R})$ and the ring $\mbfb_{\dR}^+(\overline{R})$ admits a natural $G_R\textrm{-stable}$ filtration given as $\Fil^r \mbfb_{\dR}(\overline{R}) := t^r \mbfb_{\dR}^+(\overline{R})$ for $r \in \ZZ$ and we equip $\mbfb_{\dR}^+(\overline{R})$ with the induced filtration (see \cite[\S 5.1]{brinon-padicrep-relatif} for details).

We can extend the map $\theta : \mbfa_{\inf}(\overline{R}) \rightarrow \CC^+(\overline{R})$ by $R\textrm{-linearity}$ to obtain a ring homomorphism $\theta_R : R \otimes_{\ZZ} \mbfa_{\inf}(\overline{R}) \rightarrow \CC^+(\overline{R})$.
Let $\pazo\mbfa_{\inf}(\overline{R})$ denote the $\theta_R^{-1}(p\CC^+(\overline{R}))\textrm{-adic}$ completion of $R \otimes_{\ZZ} \mbfa_{\inf}(\overline{R})$ (the ideal $\theta_R^{-1}(p\CC^+(\overline{R}))$ is generated by $p$ and $\kert \theta_R$), then $\theta_R$ extends to a surjective homomorphism $\theta_R : \pazo\mbfa_{\inf}(\overline{R})\big[\tfrac{1}{p}\big] \rightarrow \CC(\overline{R})$.
Define
\begin{equation*}
	\pazo \mbfb_{\dR}^+(\overline{R}) := \lim_n \pazo\mbfa_{\inf}(\overline{R})\big[\tfrac{1}{p}\big] / \hspace{0.5mm} (\kert \theta_R)^n,
\end{equation*}
as the $(\kert \theta_R)\textrm{-adic}$ completion of $\pazo\mbfa_{\inf}(\overline{R})\big[\frac{1}{p}\big]$.
The ring $\pazo\mbfb_{\dR}^+(\overline{R})$ is an $R\big[\frac{1}{p}\big]\textrm{-algebra}$ and admits a $G_R\textrm{-action}$.
The homomorphism $\theta_R$ extends to a $G_R\textrm{-equivariant}$ surjective homomorphism $\theta_R : \pazo\mbfb_{\dR}^+(\overline{R}) \rightarrow \CC(\overline{R})$.
The ring $\mbfb_{\dR}^+(\overline{R})$ is $t\textrm{-torsion}$ free and we set $\pazo \mbfb_{\dR}(\overline{R}) := \pazo\mbfb_{\dR}^+(\overline{R})\big[\tfrac{1}{t}\big]$.
Moreover, the $G_R\textrm{-action}$ extends to $\pazo \mbfb_{\dR}(\overline{R})$.

\subsubsection{Structure and properties of \texorpdfstring{$\pazo \mbfb_{\textdr}(\overline{R})$}{-}}\label{subsubsec:deRham_prop}

A more explicit description of the ring $\pazo\mbfb_{\dR}^+(\overline{R})$ can be given.
Note that $X_i \otimes 1 - 1 \otimes [X_i^{\flat}] \in \kert \theta_R \subset R \otimes_{\ZZ} \mbfa_{\inf}(\overline{R})$ for $1 \leq i \leq d$ and let $z_i$ denote its image in $\pazo\mbfa_{\inf}(\overline{R}) \subset \pazo\mbfb_{\dR}^+(\overline{R})$.
Since $\pazo\mbfb_{\dR}^+(\overline{R})$ is complete for the $(\kert \theta_R)\textrm{-adic}$ topology, the homomorphism $\mbfb_{\dR}^+(\overline{R}) \rightarrow \pazo\mbfb_{\dR}^+(\overline{R})$ extends to a homomorphism
\begin{align*}
	f : \mbfb_{\dR}^+(\overline{R})[[T_1, \ldots, T_d]] &\longrightarrow \pazo\mbfb_{\dR}^+(\overline{R})\\
	T_i &\longmapsto z_i, \hspace{5mm} \textrm{for} \hspace{1mm} 1 \leq i \leq d.
\end{align*}
In fact, $f$ is an isomorphism and $\kert \theta_R = (t, z_1, \ldots, z_d) \subset \pazo \mbfb_{\dR}^+(\overline{R})$.
Therefore, one can identify $\mbfb_{\dR}^+(\overline{R})$ with a subring of $\pazo \mbfb_{\dR}^+(\overline{R})$.
There is a natural $G_R\textrm{-stable}$ filtration on $\pazo\mbfb_{\dR}^+(\overline{R})$ given by $\Fil^r \pazo\mbfb_{\dR}^+(\overline{R}) = (\kert \theta_R)^r$ for $r \in \NN$.
We set $\Fil^0 \pazo\mbfb_{\dR}(\overline{R}) := \sum_{n=0}^{+\infty} t^{-n} \Fil^n \pazo\mbfb_{\dR}^+(\overline{R}) = \pazo\mbfb_{\dR}^+(\overline{R})\big[\tfrac{z_1}{t}, \ldots, \tfrac{z_d}{t}\big]$ and $\Fil^r \pazo\mbfb_{\dR}(\overline{R}) := t^r \Fil^0 \pazo\mbfb_{\dR}(\overline{R})$ for $r \in \ZZ$, satisfying the same conditions.
Moreover, the induced filtrations on $\pazo\mbfb_{\dR}^+(\overline{R})$, $\mbfb_{\dR}^+(\overline{R})$ and $\mbfb_{\dR}(\overline{R})$ match with the ones defined before.
Finally, we have $\big(\pazo\mbfb_{\dR}(\overline{R})\big)^{G_R} = R\big[\frac{1}{p}\big]$ (see \cite[\S 5.2]{brinon-padicrep-relatif} for details).

We can equip the rings $\pazo \mbfb_{\dR}^+(\overline{R})$ and $\pazo \mbfb_{\dR}(\overline{R})$ with a connection.
Let $N_i$ denote the unique $(\kert \theta_R)\textrm{-adically}$ continuous, $\mbfb_{\dR}^+(\overline{R})\textrm{-linear}$ derivation on $\pazo\mbfb_{\dR}^+(\overline{R})$ given as $N_i(z_j) = \delta_{ij} X_j$ for $1 \leq i, j \leq d$, where $\delta_{ij}$ denotes the Kronecker delta symbol.
The derivation $N_i$ extends to a $\mbfb_{\dR}(\overline{R})\textrm{-linear}$ derivation on $\pazo\mbfb_{\dR}(\overline{R})$, since $N_i(t) = 0$.
Define a connection
\begin{align*}
	\partial : \pazo\mbfb_{\dR}(\overline{R}) &\longrightarrow \pazo\mbfb_{\dR}(\overline{R}) \otimes_{R\big[\frac{1}{p}\big]} \Omega^1_R\big[\tfrac{1}{p}\big]\\
			x &\longmapsto \sum_{i=1}^d N_i(x) \otimes \dlog X_i.
\end{align*}
The connection $\partial$ is $G_R\textrm{-equivariant}$ and satisfies Griffiths transversality with respect to the filtration, i.e. $\partial\big(\Fil^r \pazo\mbfb_{\dR}(\overline{R})\big) \subset \Fil^{r-1} \pazo\mbfb_{\dR}(\overline{R}) \otimes_{R\big[\frac{1}{p}\big]} \Omega^1_R\big[\tfrac{1}{p}\big]$.
Its restriction to $R\big[\frac{1}{p}\big]$ is the canonical differential operator.
Moreover, we have $\big(\pazo\mbfb_{\dR}^+(\overline{R})\big)^{\partial=0} = \mbfb_{\dR}^+(\overline{R})$ and $\big(\pazo\mbfb_{\dR}(\overline{R})\big)^{\partial=0} = \mbfb_{\dR}(\overline{R})$ (see \cite[\S 5.3]{brinon-padicrep-relatif} for details).

\subsection{The crystalline period ring}\label{subsec:relative_crystalline_period_rings}

In this section, we will recall the definition and properties of crystalline period rings following \cite{brinon-padicrep-relatif}.
Note that Brinon defines these rings under a certain assumption on his base ring (see condition (BR) on \cite[p. 9]{brinon-padicrep-relatif}) which is always true in our setting.

\subsubsection{Definition of \texorpdfstring{$\pazo \mbfb_{\crys}(\overline{R})$}{-}}\label{subsubsec:crystalline_defi}

Let us consider the map $\theta : \mbfa_{\inf}(\overline{R}) \rightarrow \CC^+(\overline{R})$ whose kernel is a principal ideal generated by $\xi$ or $p-[p^{\flat}]$.
Let us denote $x^{[k]} := \frac{x^k}{k!}$ for $x \in \kert \theta \subset \mbfa_{\inf}(\overline{R})$ and $k \in \NN$.
The divided power envelope of $\mbfa_{\inf}(\overline{R})$ with respect to $\kert \theta$ is given as $\mbfa_{\inf}(\overline{R})\big[x^{[k]}, \hspace{1mm} x \in \kert \theta\big]_{k \in \NN} = \mbfa_{\inf}(\overline{R})\big[\xi^{[k]}\big]_{k \in \NN}$.
We define 
\begin{equation*}
	\mbfa_{\crys}(\overline{R}) := p\textrm{-adic completion of} \hspace{1mm} \mbfa_{\inf}(\overline{R})\big[\xi^{[k]}\big]_{k \in \NN}.
\end{equation*}
This is a $W(\kappa)\textrm{-algebra}$ equipped with a continuous action of $G_R$.
The ring $\mbfa_{\crys}(\overline{R})$ is $p\textrm{-torsion}$ free and the Frobenius on $\mbfa_{\inf}(\overline{R})$ extends to $\mbfa_{\crys}(\overline{R})$.
The homomorphism $\theta$ in \S \ref{subsubsec:deRham_defi} extends to a surjective homomorphism $\theta : \mbfa_{\crys}(\overline{R}) \rightarrow \CC^+(\overline{R})$.
Also, we have $t = \log (1+\pi) \in \kert \theta \subset \mbfa_{\crys}(\overline{R})$ and the Frobenius $\varphi$ on this element is given as $\varphi(t) = pt$.
Moreover, $\kert \theta \subset \mbfa_{\crys}(\overline{R})$ is a divided power ideal.
Further, the ring $\mbfa_{\crys}(\overline{R})$ is $t\textrm{-torsion}$ free, so we set $\varphi\big(\frac{1}{t}\big) = \frac{1}{pt}$ and define $\mbfb_{\crys}^+(\overline{R}) := \mbfa_{\crys}(\overline{R})\big[\tfrac{1}{p}\big]$ and $\mbfb_{\crys}(\overline{R}) := \mbfb_{\crys}^+(\overline{R})\big[\tfrac{1}{t}\big]$.
These are $F\textrm{-algebra}$s, equipped with a continuous action of $G_R$ and the Frobenius $\varphi$ (see \cite[\S 6.1 and \S 6.2]{brinon-padicrep-relatif} for details).

Next, let us consider the map $\theta_{R} : R \otimes_{\ZZ} \mbfa_{\inf}(\overline{R}) \rightarrow \CC^+(\overline{R})$.
The kernel of this map is an ideal generated by $\{1 \otimes \xi, z_1, \ldots, z_d\}$, where $z_i = X_i \otimes 1 - 1 \otimes [X_i^{\flat}]$ for $1 \leq i \leq d$.
The divided power envelope of $R \otimes_{\ZZ} \mbfa_{\inf}(\overline{R})$ with respect to $\kert \theta_{R}$ is given as $R \otimes_{\ZZ} \mbfa_{\inf}(\overline{R})\big[x^{[k]}, \hspace{1mm} x \in \kert \theta_{R}\big]_{k \in \NN}$.
We define
\begin{equation*}
	\pazo\mbfa_{\crys}(\overline{R}) := p\textrm{-adic completion of} \hspace{1mm} R \otimes_{\ZZ} \mbfa_{\inf}(\overline{R})\big[x^{[k]}, \hspace{1mm} x \in \kert \theta_{R}\big]_{k \in \NN}.
\end{equation*}
This is an $R\textrm{-algebra}$ equipped with a continuous action of $G_R$.
Taking the diagonal action of the Frobenius on $R \otimes_{\ZZ} \mbfa_{\inf}(\overline{R})$ it can be shown that the Frobenius extends to $\pazo\mbfa_{\crys}(\overline{R})$ and we denote this extension again by $\varphi$.
The homomorphism $\theta_{R}$ from \S \ref{subsec:rel_deRham_ring} extends to surjective homomorphism $\theta_{R} : \pazo\mbfa_{\crys}(\overline{R}) \rightarrow \CC^+(\overline{R})$ (see \cite[p. 64]{brinon-padicrep-relatif} for details).

\subsubsection{Structure and properties of \texorpdfstring{$\pazo \mbfb_{\crys}(\overline{R})$}{-}}\label{subsubsec:crystalline_prop}

Let $T = (T_1, \ldots, T_d)$ be some indeterminates as in \S \ref{subsubsec:deRham_prop} and let $\mbfa_{\crys}(\overline{R})\langle T \rangle^{\wedge}$ denote the $\padic$ completion of the divided power polynomial algebra in indeterminates $T$ and coefficients in $\mbfa_{\crys}(\overline{R})$.
Then we obtain an isomorphism of $\mbfa_{\crys}(\overline{R})\textrm{-algebra}$s (see \cite[Proposition 6.1.5]{brinon-padicrep-relatif})
\begin{align}
	f_{\crys} : \mbfa_{\crys}(\overline{R})\langle T \rangle^{\wedge} &\longrightarrow \pazo\mbfa_{\crys}(\overline{R})\\
	T_i &\longmapsto z_i \hspace{4mm} \textrm{for} \hspace{2mm} 1 \leq i \leq d.
\end{align}

The ring $\pazo \mbfa_{\crys}(\overline{R})$ is $p\textrm{-torsion}$ free as well as $t\textrm{-torsion}$ free, so we set $\pazo\mbfb_{\crys}^+(\overline{R}) := \pazo\mbfa_{\crys}(\overline{R})\big[\tfrac{1}{p}\big]$ and $\pazo\mbfb_{\crys}(\overline{R}) := \pazo\mbfb_{\crys}^+(\overline{R})\big[\tfrac{1}{t}\big]$.
These $R\big[\frac{1}{p}\big]\textrm{-algebra}$s are equipped with a continuous action of $G_R$ and the action of Frobenius extends to these rings and we denote this extension again by $\varphi$ (see \cite[\S 6.1 and \S 6.2]{brinon-padicrep-relatif} for details).

Note that there exist natural morphisms of rings $\mbfa_{\crys}(\overline{R}) \rightarrow \mbfb_{\dR}^+(\overline{R})$ and $\pazo\mbfa_{\crys}(\overline{R}) \rightarrow \pazo\mbfb_{\dR}^+(\overline{R})$.
So we obtain induced homomorphisms $\mbfb_{\crys}^+(\overline{R}) \rightarrow \mbfb_{\dR}^+(\overline{R})$, $\pazo\mbfb_{\crys}^+(\overline{R}) \rightarrow \pazo\mbfb_{\dR}^+(\overline{R})$, $\mbfb_{\crys}(\overline{R}) \rightarrow \mbfb_{\dR}(\overline{R})$ and $\pazo\mbfb_{\crys}(\overline{R}) \rightarrow \pazo\mbfb_{\dR}(\overline{R})$, which are injective and $G_R\textrm{-equivariant}$.
Using this, we get induced filtrations on crystalline period rings as $\Fil^r \mbfb_{\crys}(\overline{R}) := \mbfb_{\crys}(\overline{R}) \cap \Fil^r \mbfb_{\dR}(\overline{R})$ and $\Fil^r \pazo\mbfb_{\crys}(\overline{R}) := \pazo\mbfb_{\crys}(\overline{R}) \cap \Fil^r \pazo\mbfb_{\dR}(\overline{R})$ for $r \in \ZZ$ (see \cite[\S 6.2]{brinon-padicrep-relatif} for details).

Next, we will consider a connection on $\pazo\mbfb_{\crys}(\overline{R})$ induced from the connection on $\pazo \mbfb_{\dR}(\overline{R})$.
For $n \in \NN$, we have $\partial(z_i^{[n]}) = z_i^{[n-1]} \otimes \textup{d}\hspace{0.3mm} X_i$ for $1 \leq i \leq d$, so we get that for any $x \in \pazo\mbfa_{\crys}(\overline{R}) = \mbfa_{\crys}(\overline{R})\langle T \rangle^{\wedge}$, we have $\partial(x) \in \pazo\mbfa_{\crys}(\overline{R}) \otimes_{R} \Omega^1_{R}$.
This gives us an induced connection
\begin{equation*}
	\partial : \pazo\mbfb_{\crys}(\overline{R}) \longrightarrow \pazo \mbfb_{\crys}(\overline{R}) \otimes_{R[\frac{1}{p}]} \Omega^1_{R}\big[\tfrac{1}{p}\big].
\end{equation*}
The connection $\partial$ is $G_R\textrm{-equivariant}$ and satisfies Griffiths transversality with respect to the filtration, since the same is true over $\pazo \mbfb_{\dR}(\overline{R})$.
Its restriction to $R\big[\frac{1}{p}\big]$ is the canonical differential operator.
Moreover, $\big(\pazo\mbfa_{\crys}^+(\overline{R})\big)^{\partial=0} = \mbfa_{\crys}(\overline{R})$, $\big(\pazo\mbfb_{\crys}^+(\overline{R})\big)^{\partial=0} = \mbfb_{\crys}^+(\overline{R})$ and $\big(\pazo\mbfb_{\crys}(\overline{R})\big)^{\partial=0} = \mbfb_{\crys}(\overline{R})$.
We equip $\Omega^1_R\big[\frac{1}{p}\big]$ with the unique Frobenius-linear map $\varphi$ satisfying $\varphi(dx) = d \varphi(x)$ for $x \in R$.
Then, over $\pazo\mbfb_{\crys}(\overline{R})$ the Frobenius operator commutes with the connection, i.e. $\varphi \partial = \partial \varphi$ (see \cite[Proposition 6.2.5]{brinon-padicrep-relatif}).
Furthermore, we have $\big(\pazo\mbfb_{\crys}(\overline{R})\big)^{G_R} = R\big[\frac{1}{p}\big]$.
Finally, the natural map $R\big[\frac{1}{p}\big] \rightarrow \pazo \mbfb_{\crys}(\overline{R})$ is faithfully flat (see \cite[\S 6.2 and \S 6.3]{brinon-padicrep-relatif} for details).

\subsection{\texorpdfstring{$p$}{-}-adic representations}\label{subsec:relative_padic_reps}

In this section we will recall results on linear algebra data associated to $\padic$ de Rham and crystalline representations of the Galois group $G_R$.
We will use the $G_R\textrm{-regularity}$ of a topological $\QQ_p\textrm{-algebra}$ $B$ in the sense of \cite[p. 106]{brinon-padicrep-relatif}.
If $V$ is a $\padic$ representation of $G_R$, we set
\begin{equation*}
	\mbfd_B(V) := (B \otimes_{\QQ_p} V)^{G_R}.
\end{equation*}
This is a $B^{G_R}\textrm{-module}$ and we have a natural morphism of $B\textrm{-module}$s, functorial in $V$
\begin{align*}
	\alpha_B(V) : B \otimes_{B^{G_R}} \mbfd_B(V) &\longrightarrow B \otimes_{\QQ_p} V\\
			b \otimes d &\longmapsto bd.
\end{align*}
The representation $V$ is said to be \textit{$B\textrm{-admissible}$} if $\alpha_B$ is an isomorphism.

\subsubsection{Unramified representations}

Let $R^{\unrami}$ denote the union of finite \'etale $R\textrm{-subalgebras}$ $S \subset \overline{R}$, and let $\widehat{R^{\unrami}}$ denote its $\padic$ completion.
It is an $R\textrm{-subalgebra}$ of $\CC(\overline{R})$ equipped with a continuous action of $G_R$.
Further, we have $\big(\widehat{R^{\unrami}}\big[\frac{1}{p}\big]\big)^{G_R} = R\big[\frac{1}{p}\big]$ and $\widehat{R^{\unrami}}\big[\frac{1}{p}\big]$ is $G_R\textrm{-regular}$.
Let us set $G_R^{\unrami} := \Gal(R^{\unrami} / R)$ which is a quotient of $G_R$.
A $\padic$ representation $\rho : G_R \rightarrow \GL(V)$ is said to be \textit{unramified}, if $\rho$ factorizes through $G_R \rightarrow G_R^{\unrami}$.

Let $V$ be a $\padic$ representation of $G_R$ and we set
\begin{equation*}
	\mbfd_{\unrami}(V) := \big(\widehat{R^{\unrami}}\big[\tfrac{1}{p}\big] \otimes_{\QQ_p} V\big)^{G_R}.
\end{equation*}
It is an $R\big[\frac{1}{p}\big]\textrm{-module}$ and $V$ is unramified if and only if $V$ is $\widehat{R^{\unrami}}\big[\frac{1}{p}\big]\textrm{-admissible}$ (see \cite[\S 8.1]{brinon-padicrep-relatif}).

\begin{rem}\label{rem:unrami_trivialised}
	Let $V$ be an $h\textrm{-dimensional}$ $\padic$ representation of $G_R$ and $T \subset V$ a $\ZZ_p\textrm{-lattice}$ stable under the action of $G_R$ such that the action is trivial modulo $p$.
	Consider the associated continuous cocycle $f : G_R^{\unrami} \rightarrow \textup{GL}_h(\widehat{R^{\unrami}})$ describing the action of $G_R^{\unrami}$ over $\widehat{R^{\unrami}} \otimes_{\ZZ_p} T$.
	Since $V$ is unramified, $f$ is cohomologous to the trivial cocycle and from \cite[proof of Proposition 8.1.2]{brinon-padicrep-relatif}, there exists $b \in 1 + p \cdot \Mat(h, \widehat{R^{\unrami}})$ such that $f$ is given as $g \mapsto f(g) = g(b) b^{-1}$ for $g \in G_R$.
	In this case, we say that $f$ is \textit{trivialised} by $b \in 1 + p \cdot \Mat(h, \widehat{R^{\unrami}})$.
\end{rem}

\subsubsection{de Rham representations}

Note that $\pazo \mbfb_{\dR}(\overline{R})$ is a $G_R\textrm{-regular}$ $R\big[\frac{1}{p}\big]\textrm{-algebra}$.
We set
\begin{equation*}
	\pazo \mbfd_{\dR}(V) := \big(\pazo \mbfb_{\dR}(\overline{R}) \otimes_{\QQ_p} V\big)^{G_R}.
\end{equation*}
The representation $V$ is said to be de Rham representations if it is $\pazo \mbfb_{\dR}(\overline{R})\textrm{-admissible}$.
The $R\big[\frac{1}{p}\big]\textrm{-module}$ $\pazo\mbfd_{\dR}(V)$ is equipped with a decreasing, separated and exhaustive filtration induced from the filtration on $\pazo\mbfb_{\dR}(\overline{R}) \otimes_{\QQ_p} V$ where we consider the $G_R\textrm{-stable}$ filtration on $\pazo\mbfb_{\dR}(\overline{R})$ from \S \ref{subsubsec:deRham_prop}.
Moreover, the module $\pazo\mbfd_{\dR}(V)$ is equipped with an integrable connection, induced from the $G_R\textrm{-equivariant}$ integrable connection
\begin{align*}
	\partial : V \otimes_{\QQ_p} \pazo \mbfb_{\dR}(\overline{R}) &\longrightarrow V \otimes_{\QQ_p} \pazo\mbfb_{\dR}(\overline{R}) \otimes_{R[\frac{1}{p}]} \Omega^1_R\big[\tfrac{1}{p}\big]\\
	v \otimes b &\longmapsto v \otimes \partial(b).
\end{align*}
We denote the induced connection on $\pazo\mbfd_{\dR}(V)$ again by $\partial$.
Since the connection $\partial$ on $\pazo\mbfb_{\dR}(\overline{R})$ satisfies Griffiths transversality, the same is true for $\pazo\mbfd_{\dR}(V)$, i.e. $\partial(\Fil^r \pazo\mbfd_{\dR}(V)) \subset \Fil^{r-1} \pazo\mbfd_{\dR}(V) \otimes_{R[\frac{1}{p}]} \Omega^1_R\big[\tfrac{1}{p}\big]$.
Further, $\pazo\mbfd_{\dR}(V)$ is projective of rank $\leq \dim(V)$ over $(\pazo \mbfb_{\dR}(\overline{R}))^{G_R} = R\big[\frac{1}{p}\big]$.
If $V$ is de Rham then for all $r \in \ZZ$, the $R\big[\frac{1}{p}\big]\textrm{-module}$s $\Fil^r \pazo\mbfd_{\dR}(V)$ and $\gr^r \pazo\mbfd_{\dR}(V)$ are projective of finite type and for such a representation the collection of integers $r_i$ for $1 \leq i \leq \dim_{\QQ_p}(V)$ such that $\gr^{-r_i} \pazo\mbfd_{\dR}(V) \neq 0$ are called \textit{Hodge-Tate weights} of $V$.
Moreover, we say that $V$ is positive if and only if $r_i \leq 0$ for all $1 \leq i \leq \dim_{\QQ_p}(V)$ (see \cite[\S 8.3]{brinon-padicrep-relatif} for details).

\subsubsection{Crystalline representations}

Note that $\pazo \mbfb_{\crys}(\overline{R})$ is a $G_R\textrm{-regular}$ $R\big[\frac{1}{p}\big]\textrm{-algebra}$.
We set
\begin{equation*}
	\pazo \mbfd_{\crys}(V) := \big(\pazo \mbfb_{\crys}(\overline{R}) \otimes_{\QQ_p} V\big)^{G_R}.
\end{equation*}
We will denote the category of crystalline representations (i.e. $\pazo \mbfb_{\crys}(\overline{R})\textrm{-admissible}$) as $\Rep_{\QQ_p}^{\pazo\crys}(G_R)$.
The $R\big[\frac{1}{p}\big]\textrm{-module}$ $\pazo\mbfd_{\crys}(V)$ is equipped with a Frobenius-semilinear operator $\varphi$ induced from the Frobenius on $\pazo\mbfb_{\crys}(\overline{R}) \otimes_{\QQ_p} V$, where we consider the $G_R\textrm{-equivariant}$ Frobenius on $\pazo\mbfb_{\crys}(\overline{R})$.
Further, $\pazo\mbfd_{\crys}(V)$ is an $R\big[\frac{1}{p}\big]\textrm{-submodule}$ of $\pazo\mbfd_{\dR}(V)$, and we equip the former with induced filtration and connection which satisfies Griffiths transversality with respect to the filtration.
Additionally, we have $\partial \varphi = \varphi\partial$ over $\pazo\mbfd_{\crys}(V)$ (see \cite[\S 8.3]{brinon-padicrep-relatif} for details).

The $R\big[\frac{1}{p}\big]\textrm{-module}$ $\pazo\mbfd_{\crys}(V)$ is projective of rank $\leq \dim(V)$.
If $V$ is crystalline, then the $R\big[\frac{1}{p}\big]\textrm{-linear}$ homomorphism $1 \otimes \varphi : R\big[\tfrac{1}{p}\big] \otimes_{R[\frac{1}{p}], \varphi} \pazo \mbfd_{\crys}(V) \rightarrow \pazo\mbfd_{\crys}(V)$ is an isomorphism and $\pazo\mbfd_{\crys}(V)$ is called a filtered $(\varphi, \partial)\textrm{-module}$.
The inclusion $\pazo\mbfb_{\crys}(\overline{R}) \rightarrowtail \pazo\mbfb_{\dR}(\overline{R})$ induces the inclusion $\pazo\mbfd_{\crys}(V) \rightarrowtail \pazo\mbfd_{\dR}(V)$.
Let $V$ be a non-trivial de Rham representation of $G_R$, then the inclusion $\pazo\mbfd_{\crys}(V) \rightarrowtail \pazo\mbfd_{\dR}(V) \neq 0$ is surjective if and only if $V$ is crystalline (see \cite[\S 8.2 and \S 8.3]{brinon-padicrep-relatif} for details).

In conclusion, we have a functor
\begin{equation*}
	\pazo\mbfd_{\crys} : \Rep_{\QQ_p}^{\pazo\crys}(G_R) \longrightarrow \textrm{filtered} \hspace{1mm} (\varphi, \partial)\textrm{-modules over} \hspace{1mm} R\big[\tfrac{1}{p}\big].
\end{equation*}
The objects in the essential image are called \textit{admissible} filtered $(\varphi, \partial)\textrm{-module}$s and the functor induces an equivalence of categories with the essential image (see \cite[Th\'eor\`emes 8.4.2, 8.5.1]{brinon-padicrep-relatif}).

\begin{rem}
	In the arithmetic case, the essential image of $\mbfd_{\crys}$, i.e.\ admissible filtered $\varphi\textrm{-modules}$ can be described more explicitly.
	In particular, using certain invariants attached to filtered $\varphi\textrm{-modules}$ one considers the full subcategory of \textit{weakly admissible} filtered $\varphi\textrm{-modules}$ and it is a result of Colmez and Fontaine that weakly admissible filtered $\varphi\textrm{-modules}$ are admissible (in the sense above, see \cite[Théorème A]{fontaine-colmez-weakadmis}).
	In the relative case, Brinon gave a definition of weakly admissible filtered $(\varphi, \partial)\textrm{-modules}$ (see \cite[p. 136]{brinon-padicrep-relatif}).
	However, the notion is not completely satisfactory as one does not obtain an equivalence between admissible and weakly admissible filtered $(\varphi, \partial)\textrm{-modules}$ (see \cite[Theorem 1.3]{moon-weakly-admissible}).
\end{rem}

\subsubsection{One dimensional de Rham and crystalline representations}

In the 1-dimensional case, it is possible to classify all crystalline representations:

\begin{prop}[{\cite[Propositions 8.4.1, 8.6.1]{brinon-padicrep-relatif}}]\label{prop:onedim_unramrep_struct}
	Let $\eta : G_{R} \rightarrow \ZZ_p^{\times}$ be a continuous character.
	\begin{enumromanup}
	\item  $\eta$ is de Rham if and only if we can write $\eta = \eta_{\fini}\eta_{\unrami} \chi^n$ where $\eta_{\fini}$ is a finite character, $\eta_{\unrami}$ is an unramified character taking values in $1 + p\ZZ_p$ (therefore trivialized $\alpha \in 1 + p\widehat{R^{\unrami}}$, see Remark \ref{rem:unrami_trivialised}) and $\chi$ is the $\padic$ cyclotomic character and $n \in \ZZ$.

	\item $\eta$ is crystalline if and only if we can write $\eta = \eta_{\fini}\eta_{\unrami} \chi^n$ where $\eta_{\fini}$ is a finite unramified character, $\eta_{\unrami}$ is an unramified character taking values in $1 + p\ZZ_p$ (therefore trivialized by some $\alpha \in 1 + p\widehat{R^{\unrami}}$, see Remark \ref{rem:unrami_trivialised}) and $\chi$ is the $\padic$ cyclotomic character and $n \in \ZZ$.
	\end{enumromanup}
	In particular, a 1-dimensional de Rham representation is potentially crystalline.

	\begin{enumromanup}
		\setcounter{enumi}{2}
	\item Let $V = \QQ_p(\eta)$ be a one-dimensional crystalline representation.
		Then there exists a finite \'etale extension $R \rightarrow R\prm$ such that the $R\prm\big[\frac{1}{p}\big]\textrm{-module}$ $R\prm\big[\frac{1}{p}\big] \otimes_{R[\frac{1}{p}]} \pazo \mbfd_{\crys}(V)$ is free.
		In particular, if $\eta_{\fini}$ is trivial then $\pazo\mbfd_{\crys}(V)$ is a free $R\big[\frac{1}{p}\big]\textrm{-module}$ of rank 1.
	\end{enumromanup}
\end{prop}

\cleardoublepage

\section{\texorpdfstring{$(\varphi, \Gamma)$}{-}-modules and crystalline coordinates}

We will keep the setting and notations of \S \ref{sec:relative_padic_Hodge_theory}.
In particular, we have that $F$ is a finite unramified extension of $\QQ_p$ and $K = F(\mu_{p^m})$ for a fixed $m \in \NN_{\geq 1}$ (fix $m \in \NN_{\geq 2}$ if $p=2$).
Recall that $R$ is étale over $O_F\{X, X^{-1}\}$ and we have multivariate polynomials $Q_i(Z_1, \ldots, Z_s) \in O_F\{X, X^{-1}\}[Z_1, \ldots, Z_s]$ for $1 \leq i \leq s$ such that $\det \big(\frac{\partial Q_i}{\partial Z_j}\big)$ is invertible in $R$.
In particular, the ring $O_F\{X, X^{-1}\}$ provides a system of coordinates for $R$.

\subsection{\texorpdfstring{$(\varphi, \Gamma)$}{-}-modules}\label{subsec:relative_phi_gamma_mod}

In this section, we briefly recall the theory of relative $(\varphi, \Gamma)\textrm{-modules}$ from \cite{andreatta-generalized-phiGamma, andreatta-brinon-surconvergence, andreatta-iovita-relative-phiGamma}.

Let $F_n = F(\mu_{p^n})$ for $n \in \NN$ and $F_{\infty} = \cup_n F_n$.
We take $R_n$ to be the integral closure of $R \otimes_{O_F[X^{\pm 1}]} O_{F_n}\big[X_1^{p^{-n}}, \ldots X_d^{p^{-n}}\big]$ inside $\overline{R}\big[\frac{1}{p}\big]$ and set $R_{\infty} := \cup_{n \geq m} R_n$ noting that $F_{\infty} \subset R_{\infty}\big[\frac{1}{p}\big]$.
From \S \ref{subsubsec:deRham_defi} recall that $\CC(\overline{R}) = \CC^+(\overline{R})\big[\frac{1}{p}\big]$ and $\CC(\overline{R})^{\flat}$ denotes its tilt.
The ring $\CC(\overline{R})^{\flat}$ is perfect of characteristic $p$ and we set $\mbfa_{\overline{R}} := W(\CC(\overline{R})^{\flat})$, the ring of $p\textrm{-typical}$ Witt vectors with coefficients in $\CC(\overline{R})^{\flat}$ and endowed with the weak topology (see \cite[\S 2.10]{andreatta-iovita-relative-phiGamma}).
The absolute Frobenius over $\CC(\overline{R})^{\flat}$ lifts to an endomorphism $\varphi : \mbfa_{\overline{R}} \rightarrow \mbfa_{\overline{R}}$, which we again call the Frobenius.
The action of $G_{R}$ on $\CC(\overline{R})^{\flat}$ extends to a continuous action on $\mbfa_{\overline{R}}$ commuting with the Frobenius.
The inclusion $\overline{F} \subset \overline{R}\big[\tfrac{1}{p}\big]$ induces inclusions $\CC_p^{\flat} \subset \CC(\overline{R})^{\flat}$ and $\mbfa_{\overline{F}} \subset \mbfa_{\overline{R}}$.
Recall that we set $\mbfa_{\inf}(\overline{R}) := W(\CC^+(\overline{R})^{\flat})$.
The inclusion $O_{\overline{F}} \subset \overline{R}$ induces inclusions $O_{\CC_p}^{\flat} \subset \CC^+(\overline{R})^{\flat} \hspace{2mm} \textrm{and} \hspace{2mm} \mbfa_{\inf}(O_{\overline{F}}) \subset \mbfa_{\inf}(\overline{R})$.

\subsubsection{The group \texorpdfstring{$\Gamma_R$}{-}}

The ring $R_{\infty}\big[\frac{1}{p}\big]$ is a Galois extension of $R\big[\frac{1}{p}\big]$ with Galois group $\Gamma_{R} := \Gal\big(R_{\infty}\big[\frac{1}{p}\big] / R\big[\frac{1}{p}\big]\big)$ isomorphic to the semidirect product of $\Gamma_F$ and $\Gamma_{R}\prm$, where $\Gamma_F = \Gal(F_{\infty} / F)$ and $\Gamma_{R}\prm = \Gal\big(R_{\infty}\big[\frac{1}{p}\big] / F_{\infty}R\big[\frac{1}{p}\big]\big)$.
In particular, we have an exact sequence
\begin{equation}\label{eq:gammar_semidirect_product}
	1 \longrightarrow \Gamma_{R}\prm \longrightarrow \Gamma_{R} \longrightarrow \Gamma_F \longrightarrow 1,
\end{equation}
where (see \cite[p. 9]{brinon-padicrep-relatif} and \cite[\S 2.4]{andreatta-generalized-phiGamma})
\begin{align*}
	\Gamma_{R}\prm &= \Gal\big(R_{\infty}\big[\tfrac{1}{p}\big] / F_{\infty}R\big[\tfrac{1}{p}\big]\big) \isomorphic \ZZ_p^d, \\
	\chi : \Gamma_F &= \Gal(F_{\infty}/F) \isomorphic \ZZ_p^{\times}.
\end{align*}
The group $\Gamma_F$ can be viewed as a subgroup of $\Gamma_{R}$, i.e. we can take a section of the projection map in \eqref{eq:gammar_semidirect_product} such that for $\gamma \in \Gamma_F$ and $g \in \Gamma_{R}\prm$, we have $\gamma g \gamma^{-1} = g^{\chi(\gamma)}$.
So we can choose topological generators $\{\gamma, \gamma_1, \ldots, \gamma_d\}$ of $\Gamma_{R}$ such that
\begin{align*}
	\gamma(\varepsilon) &= \varepsilon^{\chi(\gamma)}, \hspace{2mm} \gamma_i(\varepsilon) = \varepsilon \hspace{13mm} \textrm{for} \hspace{1mm} 1 \leq i \leq d, \\
	\gamma_i(X_i^{\flat}) &= \varepsilon X_i^{\flat}, \hspace{2mm} \gamma_i(X_j^{\flat}) = X_j^{\flat} \hspace{7.5mm} \textrm{for} \hspace{1mm} i \neq j \hspace{1mm} \textrm{and} \hspace{1mm} 1 \leq j \leq d,
\end{align*}
and that $\gamma_0 = \gamma^e$ with $\chi(\gamma_0) = \exp(p^m)$, is a topological generator of $\Gamma_K = \Gal(K_{\infty} / K)$, where $K_{\infty} = F_{\infty}$ and $e = [K:F]$.
It follows that $\{\gamma_1, \ldots, \gamma_d\}$ are topological generators of $ \Gamma_{R}\prm $, $\gamma$ is a lift of a topological generator of $\Gamma_F$, and $\gamma_0$ is a topological generator of $\Gamma_K$.
In particular, 
\begin{equation*}
	\chi : \Gamma_K = \Gal(F_{\infty}/K) \isomorphic 1 + p^{m}\ZZ_p.
\end{equation*}

\subsubsection{Setup}

In \cite{fontaine-wintenberger-corpsNormes-i, fontaine-wintenberger-corpsNormes-ii, wintenberger-corpsNormes}, using the field-of-norms functor, Fontaine and Wintenberger constructed a non-archimedean complete discrete valuation field $\mbfe_K \subset \widehat{K}_{\infty}^{\flat}$ of characteristic $p$ with residue field $\kappa$ and admitting a continuous action of $\Gamma_K$ (notation is a bit unfortunate as $\mbfe_K$ depends only on $K_{\infty}$).
Utilizing the isomorphism of Galois groups $\Gal(\overline{F}/K_{\infty}) \isomorphic \Gal(\mbfe_K^{\sep}/\mbfe_K)$ (also see tilting correspondence in \cite{scholze-perfectoid-spaces} for a modern treatment), Fontaine classified $\textup{mod-} p$ representations of $G_K$ in terms of \'etale $(\varphi, \Gamma_K)\textrm{-modules}$ over $\mbfe_K$.
By some technical considerations one can then lift this to the classification of $\ZZ_p\textrm{-representations}$ of $G_F$ in terms of \'etale $(\varphi, \Gamma_K)\textrm{-modules}$ over a certain two dimensional regular local ring $\mbfa_K \subset W(\widehat{K}_{\infty}^{\flat})$ (see \cite{fontaine-festschrift} for details).

We have an analogous theory in the relative setting, to describe which we need to consider generically étale algebras over finite extensions of $R$ in the cyclotomic tower $R_{\infty}/R$.
More precisely, let $S \subset \overline{R}$ be a finite $R_n\textrm{-algebra}$ with $S\big[\frac{1}{p}\big]$ étale over $R_n\big[\frac{1}{p}\big]$.
For $k \geq n$ denote by $S_k$ the integral closure of $S \otimes_{R_n} R_k$ in $\overline{R}\big[\frac{1}{p}\big]$ and set $S_{\infty} := \cup_{k \geq n}S_k$.
We have that $S_{\infty}$ is a normal $R_{\infty}\textrm{-algebra}$ and an integral domain as a subring of $\overline{R}$.
As in the case of $R$, for $S$ we define $G_S := \Gal\big(\overline{R}\big[\tfrac{1}{p}\big] / S\big[\tfrac{1}{p}\big]\big)$, $\Gamma_S := \Gal\big(S_{\infty}\big[\tfrac{1}{p}\big] / S\big[\tfrac{1}{p}\big]\big)$ and $H_S := \kert(G_S \rightarrow \Gamma_S)$.
Again, $\Gamma_S$ is isomorphic to the semidirect product of $\Gamma_{F_n}$ and $\Gamma_S\prm$, where $\Gamma_S\prm = \Gal\big(S_{\infty}\big[\frac{1}{p}\big] / F_{\infty}S\big[\frac{1}{p}\big]\big)$ is a finite index subgroup of $\Gamma_R\prm \isomorphic \ZZ_p^d$.

\subsubsection{Rings in characteristic \texorpdfstring{$p$}{-}}

In the relative setting, Andreatta in \cite{andreatta-generalized-phiGamma} constructed an analogue of the subfield $\mbfe_K \subset \widehat{K}_{\infty}^{\flat}$, i.e. to any $S$ as above, he associated a ring $\mbfe_S \subset \Fr \widehat{S}_{\infty}^{\flat}$ functorial in $S_{\infty}$.
Let us recall his definition:
Let $\mbfe_F^+$ denote the valuation ring of $\mbfe_F$ and we have $\pi \in W\big(\widehat{F}_{\infty}^{\flat}\big)$ such that its reduction modulo $p$, denoted as $\overline{\pi} = \varepsilon-1$, is a uniformizer of $\mbfe_F^+$.
Depending on $S$, let $\delta \in \QQ \cap [0, 1]$ small enough and $N \in \NN$ large enough (see \cite[Definition 4.2]{andreatta-generalized-phiGamma} for precise formulations of $\delta$ and $N$), and define the ring 
\begin{equation*}
	\mbfe_S^+ := \big\{(a_0, \ldots, a_k, \ldots) \in \widehat{S}_{\infty}^{\flat}, \hspace{1mm} \textrm{such that} \hspace{1mm} a_k \in S_k / p^{\delta}S_k \hspace{1mm} \textrm{for all} \hspace{1mm} k \geq N \big\}.
\end{equation*}
The ring $\mbfe_S^+$ is finite and torsion free as an $\mbfe_R^+\textrm{-module}$.
It is a reduced Noetherian ring which is $\overline{\pi}\textrm{-adically}$ complete.
By construction, it is endowed with a $\overline{\pi}\textrm{-adically}$ continuous action of $\Gamma_S$ and a Frobenius endomorphism $\varphi$, commuting with each other and compatible with respective structures on $\widehat{S}_{\infty}^{\flat}$.
Moreover, $\mbfe_S^+$ is a normal extension of $\mbfe_R^+$, \'etale after inverting $\overline{\pi}$ and of degree equal to the generic degree of $R_m \subset S$.
Further, the set of elements $\{\overline{\pi}, X_1^{\flat}, \ldots, X_d^{\flat}\}$ form an absolute $p\textrm{-basis}$ of $\mbfe_R^+$ (see \cite[Proposition 4.5, Corollaries 5.3 \& 5.4]{andreatta-generalized-phiGamma}).
The ring $\widehat{S}_{\infty}^{\flat}$ coincides with the $\overline{\pi}\textrm{-adic}$ completion of the perfect closure of $\mbfe_S^+$ and the extension $\mbfe_S^+ \rightarrow \widehat{S}_{\infty}^{\flat}$ is faithfully flat.
Finally, set $\mbfe_S := \mbfe_S^+\big[\tfrac{1}{\overline{\pi}}\big]$.

\begin{defi}\label{defi:phi_gamma_modp_ring}
	Define $\mbfe^+ := \cup_{S} \mbfe_S^+$, where the union runs over $R_n\textrm{-subalgebras}$ $S \subset \overline{R}$ for some $n \in \NN$ such that $S$ is normal and finite as an $R_n\textrm{-module}$ and $S\big[\tfrac{1}{p}\big]$ is \'etale over $R_n\big[\tfrac{1}{p}\big]$.
	Also, we set $\mbfe := \mbfe^+\big[\tfrac{1}{\overline{\pi}}\big]$.
	These rings are $\overline{\pi}\textrm{-adically}$ complete and equipped with a Frobenius and a continuous $G_R\textrm{-action}$.
\end{defi}

\begin{rem}
	From \cite[Proposition 2.9]{andreatta-iovita-relative-phiGamma}, we have $\big(\CC^+(\overline{R})\big)^{H_R} = \widehat{R}_{\infty}$, $\big(\CC^+(\overline{R})^{\flat}\big)^{H_R} = \widehat{R}_{\infty}^{\flat}$, $\big(\CC(\overline{R})^{\flat}\big)^{H_R} = \widehat{R}_{\infty}^{\flat}\big[\tfrac{1}{\overline{\pi}}\big]$, $(\mbfe^+)^{H_R} = \mbfe_R^+$ and $\mbfe^{H_R} = \mbfe_R$.
\end{rem}

\begin{rem}\label{rem:spectral_norm}
	We will describe $\CC^+(\overline{R})^{\flat}$ as the ring of power-bounded elements inside $\CC(\overline{R})^{\flat}$ (for the spectral norm).
	Recall that $\overline{R}$ is the union of finite $R\textrm{-subalgebras}$ $S \subset \overline{\Fr(R)}$ such that $S\big[\frac{1}{p}\big]$ is étale over $R\big[\frac{1}{p}\big]$.
	Since $\overline{R}$ is an integral domain and $p\textrm{-adically}$ separated, i.e. $\cap_{k \in \NN} p^k \overline{R} = 0$, we obtain that the filtration by powers of the ideal $p\overline{R} \subset \overline{R}$ induces a sub-multiplicative norm (see \cite[\S 1.3.3, Proposition 1]{bosch-guntzer-remmert}) which extends to $\overline{R}\big[\frac{1}{p}\big]$.
	A further ``smoothening'' of the aforementioned norm yields a power-multiplicative norm on $\overline{R}\big[\frac{1}{p}\big]$ (see \cite[\S 1.3.2]{bosch-guntzer-remmert}) which we call the \textit{spectral norm} on $\overline{R}\big[\frac{1}{p}\big]$.
	Let $C$ denote the completion of $\overline{R}\big[\frac{1}{p}\big]$ for the spectral norm and $C^{\circ}$ its power-bounded elements.

	Next, one can show that under the spectral norm the power-bounded elements (or equivalently, the closed unit ball) of $\overline{R}\big[\frac{1}{p}\big]$ written as $\big(\overline{R}\big[\frac{1}{p}\big]\big)^{\circ}$ is exactly $\overline{R}$.
	Indeed, we have the obvious inclusion $\overline{R} \subset \big(\overline{R}\big[\frac{1}{p}\big]\big)^{\circ}$ and for the converse taking $x \in \big(\overline{R}\big[\frac{1}{p}\big]\big)^{\circ}$, one can reduce the claim to a finite $R\textrm{-subalgebra}$ $S \subset \overline{R}$ integrally closed in $\overline{R}\big[\frac{1}{p}\big]$ and such that $x \in S\big[\frac{1}{p}\big]$.
	Then it easily follows that $S = \big(S\big[\frac{1}{p}\big]\big)^{\circ} = S\big[\frac{1}{p}\big] \cap \big(\overline{R}\big[\frac{1}{p}\big]\big)^{\circ} \subset \overline{R}\big[\frac{1}{p}\big]$.
	So we obtain that the topology induced by the spectral norm is equivalent to the $\padic$ topology on $\overline{R}\big[\frac{1}{p}\big]$, therefore $C = \CC(\overline{R})$ and $C^{\circ} = \CC^+(\overline{R})$ and $(\CC(\overline{R}), \CC^+(\overline{R}))$ is a uniform adic Banach $\QQ_p\textrm{-algebra}$ (see \cite[Definitions 2.4.1 and 2.8.1]{kedlaya-liu-relative}).

	Finally, by the perfectoid correspondence of uniform adic Banach algebras in \cite[Theorem 3.6.5]{kedlaya-liu-relative}, we obtain that $(\CC(\overline{R})^{\flat}, \CC^+(\overline{R})^{\flat})$ is a uniform adic Banach $\FF_p\textrm{-algebra}$ such that the topology induced by the spectral norm (arising from the sub-multiplicative norm induced by the ideal $p^{\flat} \CC^+(\overline{R})^{\flat} \subset \CC^+(\overline{R})^{\flat}$) is equivalent to the topology on $(\CC(\overline{R})^{\flat}, \CC^+(\overline{R})^{\flat})$ described in \S \ref{subsubsec:deRham_defi}.
	Finally, since $\CC^+(\overline{R})$ is the ring of power-bounded elements in $\CC(\overline{R})$ we obtain that the its tilt $\CC^+(\overline{R})^{\flat}$ is the ring of power-bounded elements in $\CC(\overline{R})^{\flat}$.
\end{rem}

\begin{rem}\label{rem:ring_intersect_modp}
	Let us denote the natural valuation on $\CC_p^{\flat}$ by $\upsilon^{\flat}$.
	Then one can show that $\upsilon^{\flat}(\overline{\pi}) = \frac{p}{p-1} > 0$, i.e. $\overline{\pi}$ is not invertible in $O_{\CC_p}^{\flat}$.
	Since $O_{\CC_p}^{\flat} = \CC_p^{\flat} \cap \CC^+(\overline{R})^{\flat} \subset \CC(\overline{R})^{\flat}$, we obtain that $\overline{\pi}$ is not invertible in $\CC^+(\overline{R})^{\flat}$.
	Moreover, as $\CC^+(\overline{R})^{\flat}$ is the ring of power-bounded elements in $\CC^+(\overline{R})^{\flat}$ (see Remark \ref{rem:spectral_norm}) we conclude that $\mbfe^+ = \mbfe \cap \CC^+(\overline{R})^{\flat} \subset \CC(\overline{R})^{\flat}$.
\end{rem}

\subsubsection{Rings in characteristic \texorpdfstring{$0$}{-}}

We have liftings of the rings discussed above to characteristic 0.
In other words, there exists a Noetherian regular domain $\mbfa_R \subset W\big(\widehat{R}_{\infty}^{\flat}\big[\tfrac{1}{\overline{\pi}}\big]\big)$, complete for the weak topology and endowed with a continuous action of $\Gamma_R$ and a Frobenius such that $\mbfa_R / p \mbfa_R = \mbfe_R$.
Moreover, $\mbfa_R$ contains a subring $\mbfa_R^+$ lifting $\mbfe_R^+$ complete for the weak topology with $\pi, [X_1^{\flat}], \ldots, [X_d^{\flat}] \in \mbfa_R^+$ (see \cite[Appendix C]{andreatta-generalized-phiGamma}).
Furthermore, for $S$ as in Definition \ref{defi:phi_gamma_modp_ring} let $\mbfa_S$ denote the unique finite \'etale $\mbfa_R\textrm{-algebra}$ lifting the finite \'etale extension $\mbfe_R \subset \mbfe_S$.
It is a Noetherian regular domain, complete for the weak topology and endowed with a continuous action of $\Gamma_S$ and a Frobenius, lifting the ones defined on $\mbfe_S$.
Moreover, it contains a subring $\mbfa_S^+$ lifting $\mbfe_S^+$ so that the former is complete for the weak topology.
In characteristic 0, we set $\mbfb_{\overline{R}} := \mbfa_{\overline{R}}\big[\tfrac{1}{p}\big] = \cup_{j \in \NN} p^{-j} \mbfa_{\overline{R}}$ equipped with the direct limit topology (see \cite[\S 7]{andreatta-generalized-phiGamma} for details).

\begin{defi}\label{defi:phi_gamma_ring}
	Define $\mbfa :=$ completion of $\cup_{S} \mbfa_S \subset \mbfa_{\overline{R}}$ for the $\padic$ topology, where the union runs over all $R_n\textrm{-subalgebras}$ $S \subset \overline{R}$ as in Definition \ref{defi:phi_gamma_modp_ring}.
	Equip $\mbfa$ with the weak topology induced by the inclusion $\mbfa \subset \mbfa_{\overline{R}}$.
	Moreover, we set $\mbfa^+ := \mbfa \cap \mbfa_{\inf}(\overline{R})$, $\mbfb^+ := \mbfa^+\big[\tfrac{1}{p}\big]$ and $\mbfb := \mbfa\big[\tfrac{1}{p}\big]$ equipped with induced weak topology.
	These rings are stable under $\varphi$ and admit a continuous $G_R\textrm{-action}$.
\end{defi}

\begin{rem}\label{rem:a_varpi}
	In Definition \ref{defi:phi_gamma_ring} one can take the base ring as $R[\varpi]$ instead of $R$ to obtain period rings $\mbfa_{\varpi}^+ \subset \mbfa_{\varpi}$ (instead of $\mbfa^+ \subset \mbfa$).
	In particular, one has that $\pi_m = \varphi^{-m}(\pi) \in \mbfa_{\varpi}^+$ and it easily follows that $\mbfa^+ \subset \mbfa_{\varpi}^+ \subset \mbfa_{\inf}(\overline{R})$ compatible with Frobenius and $G_R\textrm{-action}$.
\end{rem}

\begin{rem}\phantomsection\label{rem:ring_intersect}
	\begin{enumromanup}
	\item It follows from definitions that $p\mbfa^+ = p\mbfa \cap \mbfa_{\inf}(\overline{R}) = \mbfa \cap p\mbfa_{\inf}(\overline{R}) = p(\mbfa \cap \mbfa_{\inf}(\overline{R}))$.
		Therefore, from Remark \ref{rem:ring_intersect_modp} it easily follows that $\mbfa^+ / p \mbfa^+ = \mbfe^+$.

	\item From \cite[Lemma 2.11]{andreatta-iovita-relative-phiGamma} we have $\mbfa^{H_R} = \mbfa_R$ and $(\mbfa^+)^{H_R} = \mbfa_R^+$.
	\end{enumromanup}
\end{rem}

\subsubsection{Some lemmas on matrices}

Let us note some results which will be useful in the proof of Proposition \ref{prop:wach_approx_aplus_admis}.
\begin{lem}\label{lem:reg_frob_finht_modp}
	Let $h \in \NN$ and matrices $Y \in \Mat(h, \mbfe)$ and $X, Z, W \in \Mat(h, \mbfe^+)$ such that $\varphi(Y) = XYZ + W$, then $Y \in \Mat(h, \mbfe^+)$.
\end{lem}
\begin{proof}
	From Remark \ref{rem:ring_intersect_modp} we have $\mbfe^+ = \mbfe \cap \CC^+(\overline{R})^{\flat}$.
	So it is enough to show that $Y \in \Mat(h, \CC^+(\overline{R})^{\flat})$.
	Recall that we have $\CC(\overline{R})^{\flat} = \CC^+(\overline{R})^{\flat}\big[\frac{1}{p^{\flat}}\big]$.
	Therefore, for some smallest $k \in \NN$, we can write $Y = \frac{1}{(p^{\flat})^k} Y_1$ with $Y_1 \in \Mat(h, \CC^+(\overline{R})^{\flat})$.
	Now, applying $\varphi$ we get that $\varphi\big(\frac{1}{(p^{\flat})^k} Y_1\big) = \frac{1}{(p^{\flat})^k} X Y_1 Z + W$, which can be rewritten as $\frac{(p_1^{\flat})^k}{(p^{\flat})^k} Y_1 = \varphi^{-1}(X Y_1 Z + (p^{\flat})^k W)$, where $p_1^{\flat} = \varphi^{-1}(p^{\flat})$.
	In the last equality, note that the expression on the left $\frac{(p_1^{\flat})^k}{(p^{\flat})^k} Y_1 \in \Mat(h, \CC(\overline{R})^{\flat})$, whereas the expression on the right $\varphi^{-1}(X Y_1 Z + (p^{\flat})^k W) \in \Mat(h, \varphi^{-1}(\CC^+(\overline{R})^{\flat})) = \Mat(h, \CC^+(\overline{R})^{\flat})$ since $\CC^+(\overline{R})^{\flat}$ is perfect.
	So we obtain that $\frac{(p_1^{\flat})^k}{(p^{\flat})^k} Y_1 \in \Mat(h, \CC^+(\overline{R})^{\flat})$, i.e. $Y = \frac{1}{(p^{\flat})^k} Y_1 \in \Mat\big(h, \frac{1}{(p_1^{\flat})^k}\CC^+(\overline{R})^{\flat}\big)$.
	Next, we write $Y = \frac{1}{(p_1^{\flat})^k} Y_2$ with $Y_2 \in \Mat(h, \CC^+(\overline{R})^{\flat})$.
	Again, applying $\varphi$ and arguing as above, one obtains that $Y \in \Mat\big(h, \frac{1}{(p_2^{\flat})^k}\CC^+(\overline{R})^{\flat}\big)$, where $p_2^{\flat} = \varphi^{-2}(p^{\flat})$.
	Now, it easily follows by induction on $n \in \NN$ that $Y \in \Mat\big(h, \frac{1}{(p_n^{\flat})^k}\CC^+(\overline{R})^{\flat}\big)$, where $p_n^{\flat} = \varphi^{-n}(p^{\flat})$.
	Therefore, $Y \in \Mat\big(h, \cap_{n \in \NN} \frac{1}{(p_n^{\flat})^k}\CC^+(\overline{R})^{\flat}\big) \subset \Mat(h, \CC(\overline{R})^{\flat})$.
	But since $\CC^+(\overline{R})^{\flat}$ is the ring of power-bounded elements in $\CC(\overline{R})^{\flat}$, we obtain that $\cap_{n \in \NN} \frac{1}{(p_n^{\flat})^k}\CC^+(\overline{R})^{\flat} = \CC^+(\overline{R})^{\flat}$.
	Hence, we get $Y \in \Mat(h, \CC^+(\overline{R})^{\flat})$ as desired.
\end{proof}

\begin{lem}\label{lem:reg_frob_finht}
	Let $h \in \NN$ and matrices $Y \in \Mat(h, \mbfa)$ and $X, Z, W \in \Mat(h, \mbfa^+)$ such that $\varphi(Y) = XYZ + W$, then $Y \in \Mat(h, \mbfa^+)$.
\end{lem}
\begin{proof}
	Reducing the equation modulo $p$ we have $\varphi(\overline{Y}) = \overline{X} \hspace{1mm} \overline{Y} \hspace{1mm} \overline{Z} + \overline{W}$, with $\overline{Y} \in \Mat(h, \mbfe)$ and $\overline{X}, \overline{Z}, \overline{W} \in \Mat(h, \mbfe^+)$.
	Therefore, from Lemma \ref{lem:reg_frob_finht_modp} we obtain that $\overline{Y} \in \Mat(h, \mbfe^+)$.
	As we have $\mbfa^+ / p\mbfa^+ = \mbfe^+$ (see Remark \ref{rem:ring_intersect} (ii)), let $V_0 \in \Mat(h, \mbfa^+)$ such that $\overline{Y} = \overline{V}_0$ and $\varphi(\overline{V}_0) = \overline{X} \hspace{1mm} \overline{V}_0 \hspace{1mm} \overline{Z} + \overline{W}$.
	So we can write $Y = V_0 + p Y_1$ with $Y_1 \in \Mat(h, \mbfa)$, and obtain that $\varphi(V_0 + p Y_1) = X(V_0 + pY_1)Z + W$.
	Simplifying the latter expression, we have $\varphi(V_0) - (X V_0 Z + W) = p(X Y_1 Z - \varphi(Y_1))$.
	Since $\varphi(V_0) - (X V_0 Z + W) \in \Mat(h, p\mbfa^+)$, we conclude that $\varphi(Y_1) - X Y_1 Z \in \Mat(h, \mbfa^+)$.
	In other words, we have an equality $\varphi(Y_1) = X Y_1 Z + W_1$ with $Y_1 \in \Mat(h, \mbfa)$ and $X, Z, W_1 \in \Mat(h, \mbfa^+)$.
	Repeating the argument as above, we get that $\overline{Y}_1 \in \Mat(h, \mbfe^+)$ and we can take a lift to write $Y_1 = V_1 + pY_2$ with $V_1 \in \Mat(h, \mbfa^+)$ and $Y_2 \in \Mat(h, \mbfa)$.
	This gives us that $Y = V_0 + p V_1 + p^2 Y_2$.
	Now, it easily follows by induction on $n \in \NN$ that $Y = V_0 + p V_1 + \cdots + p^{n-1} V_{n-1} + p^n Y_n$ with $V_i \in \Mat(h, \mbfa^+)$ for $0 \leq i \leq n-1$ and $Y_n \in \Mat(h, \mbfa)$.
	Letting $n \rightarrow +\infty$ and noting that $\mbfa^+$ is $\padic$ally complete, we obtain that $Y \in \Mat(h, \mbfa^+)$ as desired.
\end{proof}

\subsubsection{Étale \texorpdfstring{$(\varphi, \Gamma_R)\textrm{-modules}$}{-}}

\begin{defi}\label{defi:phigamma_modules}
	A $(\varphi, \Gamma_R)\textrm{-module}$ $D$ over $\mbfa_R$ is a finitely generated module equipped with
	\begin{enumromanup}
	\item A semilinear action of $\Gamma_R$, continuous for the weak topology, 

	\item A $\Gamma_R\textrm{-equivariant}$ Frobenius-semilinear endomorphism $\varphi$.
	\end{enumromanup}
	We say that $D$ is \textit{\'etale} if the natural map $1 \otimes \varphi : \mbfa_R \otimes_{\mbfa_R, \varphi} D \rightarrow D$ is an isomorphism of $\mbfa_R\textrm{-modules}$.
\end{defi}

Denote by $(\varphi, \Gamma_R)\textup{-Mod}_{\mbfa_R}^{\etale}$ the category of \'etale $(\varphi, \Gamma_R)\textrm{-modules}$ over $\mbfa_R$ with morphisms between objects being continuous, $(\varphi, \Gamma_R)\textrm{-equivariant}$ morphisms of $\mbfa_R\textrm{-modules}$.
Next, denote by $\Rep_{\ZZ_p}(G_R)$ the category of finitely generated $\ZZ_p\textrm{-modules}$ equipped with a linear and continuous action of $G_R$, with morphisms between objects being continuous and $G_R\textrm{-equivariant}$ morphisms of $\ZZ_p\textrm{-modules}$.

Let $T$ be a $\ZZ_p\textrm{-representation}$ of $G_R$.
The $\mbfa_R\textrm{-module}$ $\mbfd(T) := (\mbfa \otimes_{\ZZ_p} T)^{H_R}$ is equipped with a semilinear operator $\varphi$ and a continuous (for the weak topology) and semilinear action of $\Gamma_R$, commuting with each other.
Moreover, $\mbfd(T)$ is an \'etale $(\varphi, \Gamma_R)\textrm{-module}$.
Furthermore, if $T$ is free of finite rank, then $\mbfd(T)$ is a projective module of rank $= \textup{rk}_{\ZZ_p} T$ (see \cite[Theorem 7.11]{andreatta-generalized-phiGamma}).
The functor
\begin{equation}\label{eq:phi_gamma_equiv}
	\mbfd : \Rep_{\ZZ_p}(G_R) \longrightarrow (\varphi, \Gamma_R)\textup{-Mod}_{\mbfa_R}^{\etale},
\end{equation}
induces an equivalence of categories (see \cite[Theorem 7.11]{andreatta-generalized-phiGamma} and \cite[Th\'eor\`eme 4.35]{andreatta-brinon-surconvergence}), and the natural map $\mbfa \otimes_{\mbfa_R} \mbfd(T) \isomorphic \mbfa \otimes_{\ZZ_p} T$ is an isomorphism of $\mbfa\textrm{-modules}$ compatible with Frobenius and the action of $G_R$ on each side.

\subsection{Crystalline coordinates}\label{subsec:pd_envelope}

In this section we will introduce certain ``coordinate'' rings.
As we shall see in the next section, these rings are related to period rings appearing in \S \ref{sec:relative_padic_Hodge_theory} and \S \ref{subsec:relative_phi_gamma_mod}.

Let $r_{\varpi}^+$ and $r_{\varpi}$ denote the algebras $O_F[[X_0]]$ and $O_F[[X_0]]\{X_0^{-1}\}$.
Sending $X_0$ to $\varpi$ induces a surjective homomorphism $r_{\varpi}^+ \twoheadrightarrow O_K$.
Let $R^+_{\varpi, \square}$ denote the completion of $O_F[X_0, X, X^{-1}]$ for the $(p, X_0)\textrm{-adic}$ topology.
Sending $X_0$ to $\varpi$ induces a surjective homomorphism $R^+_{\varpi, \square} \twoheadrightarrow O_K\{X, X^{-1}\}$, whose kernel is generated by $P = P_{\varpi}(X_0)$.
This provides a closed embedding of $\Spf O_K\{X, X^{-1}\}$ into a formal scheme $\Spf R_{\varpi, \square}^+$, which is smooth over $O_F$.
Recall that $R$ is \'etale over $O_F\{X, X^{-1}\}$ and we have multivariate polynomials $Q_i(Z_1, \ldots, Z_s) \in O_F\{X, X^{-1}\}[Z_1, \ldots, Z_s]$ for $1 \leq i \leq s$ such that $\det \big(\frac{\partial Q_i}{\partial Z_j}\big)$ is invertible in $R$.
So we can set $R_{\varpi}^+$ to be the quotient by $(Q_1, \ldots, Q_s)$ of the completion of $R_{\varpi, \square}^+[Z_1, \ldots, Z_s]$ for $(p, X_0)\textrm{-adic}$ topology.
Again, we have that $\det \big(\frac{\partial Q_i}{\partial Z_j}\big)$ is invertible in $R_{\varpi}^+$ (since $R \rightarrowtail R_{\varpi}^+$).
Hence, $R_{\varpi}^+$ is \'etale over $R_{\varpi, \square}^+$ and smooth over $O_F$.
Sending $X_0$ to $\varpi$ induces a surjective homomorphism $R^+_{\varpi} \twoheadrightarrow R[\varpi]$ whose kernel is generated by $P = P_{\varpi}(X_0)$.
This can be summarized by the commutative diagram
\begin{center}
	\begin{tikzcd}
		\Spf R[\varpi] \arrow[rr, rightarrowtail] \arrow[rd, rightarrow] \arrow[ddd] & & \Spf R_{\varpi}^+ \arrow[ddd] \arrow[ld, rightarrow]\\
		& \Spf R \arrow[d]\\
		& \Spf O_F\{X, X^{-1}\} \\
		\Spf O_K\{X, X^{-1}\} \arrow[rr, rightarrowtail] \arrow[ru, rightarrow] & & \Spf R_{\varpi, \square}^+ \arrow[lu, rightarrow],
	\end{tikzcd}
\end{center}
where the vertical arrows are \'etale extensions and the horizontal maps are obtained by sending $X_0 \mapsto \varpi$, and the rest are natural maps.
Finally, we set $R_{\varpi} = \padic \textrm{ completion of } R_{\varpi}^+\big[\frac{1}{X_0}\big]$.

Next, since $P \equiv X_0^e \mod p$, we have $R^+_{\varpi}\big[\tfrac{P^k}{k!}\big]_{k \in \NN} = R^+_{\varpi}\big[\tfrac{X_0^k}{[k/e]!}\big]_{k \in \NN}$.
So, we set $R_{\varpi}^{\PD} := p\textrm{-adic}$ completion of $R^+_{\varpi}\big[\tfrac{P^k}{k!}\big]_{k \in \NN}$.
In summary, we have a diagram of formal schemes where the horizontal arrows are closed embeddings into formal schemes smooth over $O_F$, obtained by sending $X_0 \mapsto \varpi$ on the level of algebras,
\begin{center}
	\begin{tikzcd}
		& \Spf R_{\varpi}^{\PD} \arrow[rd]\\
		\Spf R[\varpi] \arrow[ru, rightarrowtail] \arrow[rr, rightarrowtail] \arrow[d] & & \Spf R_{\varpi}^+ \arrow[d]\\
		\Spf O_K\{X, X^{-1}\} \arrow[rr, rightarrowtail] \arrow[d] & & \Spf R_{\varpi, \square}^+ \arrow[d]\\
		\Spf O_K \arrow[rr, rightarrowtail] \arrow[d] & & \Spf r_{\varpi}^+ \arrow[lld]\\
		\Spf O_F.
	\end{tikzcd}
\end{center}

Recall that $P$ generates the kernel of the surjective map $R_{\varpi}^+ \twoheadrightarrow R[\varpi]$ and divided powers of $P$ generate the kernel of the surjective map $R_{\varpi}^{\textpd} \twoheadrightarrow R[\varpi]$.
\begin{defi}\label{defi:filtration_vanishing_varpi}
	Endow the ring $R_{\varpi}^{\textpd}$ with a filtration by divided power ideals as
	\begin{equation*}
		\Fil^k R_{\varpi}^{\textpd} = (P^{[n]}, \hspace{1mm} n \geq k) \subset R_{\varpi}^{\textpd} \hspace{2mm} \textrm{for} \hspace{2mm} k \in \NN.
	\end{equation*}
	In other words, the filtration on $R_{\varpi}^{\textpd}$ is given by divided powers of the kernel of $R_{\varpi}^{\textpd} \twoheadrightarrow R[\varpi]$.
	Furthermore, the ring $R_{\varpi}^+$ is endowed with the induced filtration
	\begin{equation*}
		\Fil^k R_{\varpi}^+ := R_{\varpi}^+ \cap \Fil^k R_{\varpi}^{\textpd} = P^k R_{\varpi}^+ \hspace{2mm} \textrm{for} \hspace{2mm} k \in \NN,
	\end{equation*}
	where the last equality follows since $P$ generates the kernel of $R_{\varpi}^+ \twoheadrightarrow R[\varpi]$.
\end{defi}

\subsection{Cyclotomic embedding}\label{subsec:cyclotomic_embeddings}

In this section, we will describe the relationship between $R_{\varpi}^{\bmstar}$ for $\smstar \in \{\hspace{1mm}, +, \textpd\}$ and the period rings discussed in \S \ref{sec:relative_padic_Hodge_theory} and \S \ref{subsec:relative_phi_gamma_mod}.
We start by defining the (cyclotomic) Frobenius endomorphism on the former rings.
Over $R_{\varpi, \square}^+$ define a lift of the absolute Frobenius on $R_{\varpi, \square}^+ / p$ by
\begin{align*}
	\varphi : R_{\varpi, \square}^+ &\longrightarrow R_{\varpi, \square}^+\\
			X_0 &\longmapsto (1+X_0)^p - 1\\
			X_i &\longmapsto X_i^p, \hspace{2mm} \textrm{for} \hspace{2mm} 1 \leq i \leq d,
\end{align*}
which we will call the (cyclotomic) Frobenius.
Clearly, $\varphi(x) - x^p \in p R_{\varpi, \square}^+$ for $x \in R_{\varpi, \square}^+$.
Using the implicit function theorem for topological rings \cite[Proposition 2.1]{colmez-niziol-nearby-cycles}, we can extend the Frobenius homomorphism to $\varphi : R_{\varpi}^+ \rightarrow R_{\varpi}^+$.
By continuity, the Frobenius endomorphism $\varphi$ admits unique extensions $\varphi : R_{\varpi}^{\textpd} \rightarrow R_{\varpi}^{\textpd}$ and $\varphi : R_{\varpi} \rightarrow R_{\varpi}$.

\subsubsection{The rings \texorpdfstring{$\mbfa_{R, \varpi}^{\bmstar}$}{-}}

We will describe the (cyclotomic) embeddings of $R_{\varpi}^{\bmstar}$ into various period rings discussed in \S \ref{sec:relative_padic_Hodge_theory} and \S \ref{subsec:relative_phi_gamma_mod}.
Define an embedding
\begin{align*}
	\iota_{\cycl} : R_{\varpi, \square}^+ &\longrightarrow \mbfa_{\inf}(\overline{R})\\
			X_0 &\longmapsto \pi_{m} = \varphi^{-m}(\pi),\\
			X_i &\longmapsto [X_i^{\flat}], \hspace{2mm} \textrm{for} \hspace{2mm} 1 \leq i \leq d.
\end{align*}
\begin{lem}
	The map $\iota_{\cycl}$ has a unique extension to an embedding $R_{\varpi}^+ \rightarrow \mbfa_{\inf}(\overline{R})$ such that $\theta \circ \iota_{\cycl}$ is the projection $R_{\varpi}^+ \rightarrow R[\varpi]$.
\end{lem}
\begin{proof}
	We can use the implicit function theorem \cite[Proposition 2.1]{colmez-niziol-nearby-cycles} to extend the embedding to $\iota_{\cycl} : R_{\varpi}^+ \rightarrow \mbfa_{\inf}(\overline{R})$.
	Next, from defintions we already have that $\theta \circ \iota_{\cycl} : R_{\varpi, \square}^+ \twoheadrightarrow O_K\{X, X^{-1}\}$ coincides with the canonical projection and $R_{\varpi}^+$ is \'etale over $R_{\varpi, \square}^+$, hence the second claim follows.
\end{proof}

This embedding commutes with Frobenius on either side, i.e. $ \iota_{\cycl} \circ \varphi = \varphi \circ \iota_{\cycl}$.
By continuity, the morphism $\iota_{\cycl}$ extends to embeddings $\iota_{\cycl}  : R_{\varpi}^{\textpd} \rightarrowtail \mbfa_{\crys}(\overline{R})$ and $\iota_{\cycl} : R_{\varpi} \rightarrowtail \mbfa_{\overline{R}}$.
Denote by $\mbfa_{R, \varpi}^+$ and $\mbfa_{R, \varpi}$ the image in $\mbfa_{\overline{R}}$ of $R_{\varpi}^+$ and $R_{\varpi}$ respectively, under the map $\iota_{\cycl}$.
Similarly, let $\mbfa_{R, \varpi}^{\textpd} := \iota_{\cycl}\big(R_{\varpi}^{\textpd}\big) \subset \mbfa_{\crys}(\overline{R})$.
These rings are stable under the action of $\Gamma_{R}$ (see \cite[\S 2.5.3]{colmez-niziol-nearby-cycles}).
Moreover, these embeddings induce a filtration on $\mbfa_{R, \varpi}^{\bmstar}$ for $\smstar \in \{+, \textpd\} $ and $r \in \ZZ$ (use Definition \ref{defi:filtration_vanishing_varpi}).

\begin{rem}
	Note that we write $\mbfa_{R, \varpi}^+$ and so on instead of slightly cumbersome notation $\mbfa_{R[\varpi]}^+$ or simpler notation $\mbfa_S^+$ for $S = R[\varpi]$, in order to emphasize the choice of root of unity in the definition.
\end{rem}

We note a simple lemma that will be useful later.
\begin{lem}\label{lem:t_over_pi_unit}
	$\frac{t}{\pi}$ is a unit in $\mbfa_{F, \varpi}^{\textpd} \subset \mbfa_{R, \varpi}^{\textpd}$.
\end{lem}
\begin{proof}
	We can write the fraction
	\begin{equation*}
		\frac{t}{\pi} = \frac{\log (1+\pi)}{\pi} = \sum_{k\geq 0} (-1)^k \tfrac{\pi^k}{k+1}.
	\end{equation*}
	Formally, we can write
	\begin{equation*}
		\frac{\pi}{t} = \frac{\pi}{\log(1+\pi)} = 1 + b_1 \pi + b_2 \pi^2 + b_3 \pi^3 + \cdots,
	\end{equation*}
	where $\upsilon_p(b_k) \geq -\frac{k}{p-1}$ for all $k \geq 1$.
	Since $\pi = (1 + \pi_m)^{p^m} - 1$, we get that $\pi \in (p, \pi_m^{p^m}) \mbfa_{F, \varpi}^+$ (as $m \geq 1$).
	By induction over $k$, we can easily conclude that $\pi^k \in \big(p, \pi_m^{p^m}\big)^k\mbfa_{F, \varpi}^{\textpd}$.
	Using this, we can re-express the series $\sum_k b_k \pi^k$ as a power series in $\pi_m$, written as $\sum_i c_i \pi_m^i$.
	We need to check that this re-expressed series converges in $\mbfa_{F, \varpi}^{\textpd}$.
	To do this, we collect the terms with coefficients having the smallest $\padic$ valuation for each power of $\pi_m^{p^m}$ in the re-expressed series.
	For $k \geq 1$, $b_k$ has the smallest $\padic$ valuation among the coefficients of $\pi_m^{p^m k}$, and therefore it has the least $\padic$ valuation among coefficients of $\pi_m^i$ for $p^m k \leq i < p^m(k+1)$.
	We write the collection of these terms as
	\begin{equation}\label{eq:smallest_valuation_t_over_pi}
		\sum_{k \geq 1} (-1)^{k+1} b_k \pi_m^{p^m k} = \sum_{k \geq 1} (-1)^{k+1} b_k \big\lfloor \tfrac{p^m k}{e} \big\rfloor! \tfrac{\pi_m^{p^m k}}{\lfloor p^m k/e \rfloor!},
	\end{equation}
	and by the preceding discussion it is sufficient to show that these coefficients go to $0$ as $k \rightarrow +\infty$.
	Moreover, for \eqref{eq:smallest_valuation_t_over_pi} it would suffice to check the estimate for $k = (p-1)j$ as $j \rightarrow +\infty$ (this gets rid of the floor function above).
	With the observation in Remark \ref{rem:factorial_padic_estimate}, we have
	\begin{equation*}
		\upsilon_p\Big(b_k \big\lfloor \tfrac{p^mk}{e} \big\rfloor!\Big) = \upsilon_p(b_k) + \upsilon_p((pj)!) \geq -\tfrac{(p-1)j}{p-1} + \tfrac{pj - s_p(pj)}{p-1}  = \tfrac{j - s_p(j)}{p-1} = \upsilon_p(j!),
	\end{equation*}
	which goes to $+\infty$ as $j \rightarrow +\infty$.
	Hence, $\frac{\pi}{t}$ converges in $\mbfa_{F, \varpi}^{\textpd}$ and is an inverse to $\frac{t}{\pi}$.
\end{proof}

The following elementary observation was used above,
\begin{rem}\label{rem:factorial_padic_estimate}
	Let $n \in \NN$, so we can write $n = \sum_{i=0}^k n_i p^i$ for some $k \in \NN$, where $0 \leq n_i \leq p-1$ for $0 \leq i \leq k$.
	Let us set $s_p(n) = \sum_{i=0}^k n_i$.
	Then we have
	\begin{align*}
		\upsilon_p(n!) &= \sum_{j \geq 1} \big\lfloor\tfrac{n}{p^j}\big\rfloor = \sum_{j \geq 0} \big\lfloor\tfrac{\sum_{i=0}^k n_ip^i}{p^j}\big\rfloor = \sum_{j=1}^k \sum_{i=j}^k n_i p^{i-j}\\
		&= \sum_{i=1}^k n_i \sum_{j=1}^i p^j = \sum_{i=1}^k n_i \tfrac{p^i-1}{p-1} = \tfrac{n-s_p(n)}{p-1}.
	\end{align*}
	Also, note that we have $s_p(pn) = s_p(n)$ for any $n \in \NN$.
\end{rem}

\begin{lem}\label{lem:gamma_minus_1_pd}
	Let $i \in \{0, 1, \ldots, d\}$.
	Then $(\gamma_i-1) \mbfa_{R, \varpi}^{\bmstar} \subset \pi \mbfa_{R, \varpi}^{\bmstar}$ for $\bmstar \in \{+, \textpd\}$;
\end{lem}
\begin{proof}
	First, let $i = 0$.
	Then we have
	\begin{align*}
		(\gamma_0-1)\pi_m &= (1 + \pi_m)\big((1 + \pi_m)^{\chi(\gamma_0)-1} - 1\big) = (1 + \pi_m)\big((1 + \pi_m)^{p^ma} - 1\big) \\
				&= (1 + \pi_m)((1 + \pi)^a - 1) = (1 + \pi_m)\big(a\pi + \tfrac{a(a-1)}{2!}\pi^2 + \tfrac{a(a-1)(a-2)}{3!} \pi^3 + \cdots\big) = \pi x,
	\end{align*}
	for some $x \in \mbfa_{F, \varpi}^+$, i.e. $(\gamma_0-1)\pi_m \in \pi \mbfa_{F, \varpi}^+$.
	Then it follows that $(\gamma_0-1)\mbfa_{F, \varpi}^{\bmstar} \subset \pi \mbfa_{F, \varpi}^{\bmstar}$ for $\bmstar \in \{+, \textpd\}$

	Next, for $i \in \{1, \ldots, d\}$ we have $(\gamma_i-1)[X_i^{\flat}] = \pi [X_i^{\flat}] \in \pi \mbfa_{R, \varpi}^+$ and $(\gamma_i-1)\big([X_i^{\flat}]^{-1}\big) = -\pi(1+\pi)^{-1}[X_i^{\flat}]^{-1} \in \pi \mbfa_{R, \varpi}^+$.
	Therefore, we get the claim.
\end{proof}

\subsubsection{The ring \texorpdfstring{$\mbfa_R^+$}{-}}

The preceding discussion works well for $R[\varpi]$ where $\varpi = \zeta_{p^m}-1$ for $m \in \NN_{\geq 1}$ ($m \in \NN_{\geq 2}$ if $p=2$).
For $R$ one can repeat the construction above to obtain the period ring $\mbfa_R^+ \subset \mbfa_{R, \varpi}^+$ (the embedding $R_{\varpi}^+ \rightarrowtail \mbfa_{\inf}(\overline{R})$ for $R$ sends $X_0 \mapsto \pi$).
Moreover, restriction of the map $\theta$ gives us a surjective map $\theta : \mbfa_R^+ \twoheadrightarrow R$ whose kernel is principal and generated by $\pi$ (since $\theta \circ \iota_{\cycl} = id$ on $R$).
Next, over $\mbfa_{R, \varpi}^+$ the filtration is given as $\Fil^k \mbfa_{R, \varpi}^+ = \xi^k \mbfa_{R, \varpi}^+$, where $\xi = \frac{\pi}{\pi_1}$.
However, $\xi \not\in \mbfa_{R}^+$.
Therefore, we equip $\mbfa_{R}^+$ with the induced filtration $\Fil^k \mbfa_{R}^+ = \mbfa_{R}^+ \cap \Fil^k \mbfa_{R, \varpi}^+$.
Then describing the filtration as kernel of the $\theta$ map, we obtain

\begin{lem}\label{lem:fil_ar0plus}
	$\Fil^k \mbfa_{R}^+ = \pi^k \mbfa_{R}^+$.
\end{lem}

\begin{rem}\label{rem:a_varpi_iso}
	Let $\mbfa^+$ be the ring from Definition \ref{defi:phi_gamma_ring} and $\mbfa_{\varpi}^+$ be the ring defined in Remark \ref{rem:a_varpi}.
	From the definitions it follows that $\mbfa_{R, \varpi}^+ \otimes_{\mbfa_R^+} \mbfa^+ \isomorphic \mbfa_{\varpi}^+$ compatible with Frobenius and $G_R\textrm{-action}$.
	Moreover, we have $\mbfa_R^+ = (\mbfa^+)^{H_R}$ and $\mbfa_{R, \varpi}^+ = (\mbfa_{\varpi}^+)^{H_{R, \varpi}}$ where $H_{R, \varpi} = H_R$.
	Now, if we equip $\mbfa^+ \subset \mbfa_{\varpi}^+ \subset \mbfa_{\inf}(\overline{R})$ with the induced filtration, then we see that the isomorphism $\mbfa_{R, \varpi}^+ \otimes_{\mbfa_R^+} \mbfa^+ \isomorphic \mbfa_{\varpi}^+$ is compatible with filtrations as well (where on the left we consider the tensor product filtration).
\end{rem}

\subsection{Fat period rings}\label{subsec:fat_period_rings}

In this section we will introduce an alternative construction of fat period rings.
This will be helpful in constructing some auxiliary rings in the proof of Proposition \ref{prop:crys_from_wach_mod}.
Let $S$ and $\Lambda$ be $\padic$ally complete filtered $O_F\textrm{-algebras}$.
Let $\iota : S \rightarrow \Lambda$ be a continuous injective morphism of filtered $O_F\textrm{-algebras}$ and let $f : S \otimes \Lambda \rightarrow \Lambda$ be the morphism sending $x \otimes y \mapsto \iota(x)y$.

\begin{defi}\label{defi:fat_ring_const}
	Define $S\Lambda$ to be the $\padic$ completion of the divided power envelope of $S \otimes \Lambda$ with respect to $\kert f$.
\end{defi}

Now, let $S = R, R_{\varpi}^{\textpd}$, where over $R$ we consider the trivial filtration, whereas over $R_{\varpi}^{\textpd}$ we consider the filtration described in Definition \ref{defi:filtration_vanishing_varpi}.
Then we have,
\begin{rem}
	\begin{enumromanup}
	\item The ring $S\Lambda$ is the $\padic$ completion of $S \otimes \Lambda$ adjoined $(x \otimes 1 - 1 \otimes \iota(x))^{[k]}$, for $x \in S$ and $n \in \NN$ and $(V_i - 1)^{[k]}$ for $1 \leq i \leq d$ and $k \in \NN$, where $V_i = \frac{X_i \otimes 1}{1 \otimes \iota(X_i)}$ for $1 \leq i \leq d$.

	\item The morphism $f : S \otimes \Lambda \rightarrow \Lambda$ extends uniquely to a continuous morphism $f : S \Lambda \rightarrow \Lambda$.

	\item There is a natural filtration over $S \Lambda$ where we define $\Fil^r S \Lambda$ to be the topological closure of the ideal generated by the products of the form $x_1 x_2 \prod(V_i - 1)^{[k_i]}$, with $x_1 \in \Fil^{r_1}S$, $x_2 \in \Fil^{r_2}\Lambda$ and $r_1 + r_2 + \sum k_i \geq r$.
		
	\item From \cite[Lemma 2.36]{colmez-niziol-nearby-cycles}, we have that any element $x \in S\Lambda$ can be uniquely written as $x = \sum_{\smbfk \in \NN^{d}} x_{\smbfk}(1-V_1)^{[k_1]} \cdots (1-V_d)^{[k_d]}$ with $x_{\smbfk} \in \Lambda$ for all $\smbfk = (k_1, \ldots, k_d) \in \NN^{d}$ and $x_{\smbfk} \rightarrow 0$ as $|\smbfk| = \sum_{i=1}^d k_i \rightarrow +\infty$.
		Moreover, an element $x \in \Fil^r S \Lambda$ if and only if $x_{\smbfk} \in \Fil^{r - |\smbfk|} \Lambda$ for all $\smbfk \in \NN^{d}$.
	\end{enumromanup}
\end{rem}

\cleardoublepage

\section{Finite height representations}\label{sec:relative_finite_height}

In this section we will study Wach modules and their relationship with crystalline modules for crystalline representations.

\subsection{The arithmetic case}\label{subsec:arithmetic_wach}

Recall that we have $G_F = \Gal(\overline{F}/F)$ as the absolute Galois group of $F$, $\Gamma_F := \Gal(F_{\infty}/F)$ and $H_F := \Gal(\overline{F}/F_{\infty})$, where $F_{\infty} = \cup_n F(\mu_{p^n})$.
From the theory of $(\varphi, \Gamma_F)\textrm{-modules}$, we have a two dimensional local ring $\mbfa_F$ given as the $\padic$ completion of $O_F[[\pi]]\big[\frac{1}{\pi}\big]$ and $\mbfb_F := \mbfa_F\big[\frac{1}{p}\big]$ is a complete discrete valuation field with uniformizer $p$ and residue field $\kappa((\overline{\pi}))$, the field of Laurent series with uniformizer $\overline{\pi}$ (the reduction of $\pi$ modulo $p$).

Next, we have certain subrings $\mbfa_F^+ := O_F[[\pi]] \subset \mbfa_F$ and $\mbfb_F^+ = \mbfa_F^+\big[\frac{1}{p}\big] \subset \mbfb_F$, stable under the action of $\varphi$ and $\Gamma_F$.
Let $V$ be a $\padic$ representation of $G_F$, then $\mbfd^+(V) = (\mbfb^+ \otimes_{\QQ_p} V)^{H_F}$ is a free module over the principal domain $\mbfb_F^+$ of rank $\leq \dim_{\QQ_p} V$, equipped with a Frobenius-semilinear endomorphism $\varphi$ and a continuous and semilinear action of $\Gamma_F$.
Further, let $\mbfd(V) = (\mbfb \otimes_{\QQ_p} V)^{H_F}$ be the associated $(\varphi, \Gamma_F)\textrm{-module}$ which is a $\mbfb_F\textrm{-vector}$ space of dimension $= \dim_{\QQ_p} V$, equipped with a Frobenius-semilinear endomorphism $\varphi$ and a continuous and semilinear action of $\Gamma_F$.
We have a $\mbfb_F^+\textrm{-linear}$ inclusion $\mbfd^+(V) \subset \mbfd(V)$ compatible with the action of $\varphi$ and $\Gamma_F$.
We say that $V$ is of \textit{finite height} if $\mbfd^+(V)$ is a $\mbfb_F^+\textrm{-lattice}$ inside $\mbfd(V)$.

Similarly, if $T \subset V$ is a free $\ZZ_p\textrm{-lattice}$, stable under the action of $G_F$, then $\mbfd^+(T) = (\mbfa^+ \otimes_{\ZZ_p} T)^{H_F}$ is a free $\mbfa_F^+\textrm{-module}$ of rank $\leq \dim_{\QQ_p} V$, stable under the action of $\varphi$ and $\Gamma_F$ (see \cite[\S B.1.2]{fontaine-festschrift}).
Moreover, $\mbfd(T) = (\mbfa \otimes_{\ZZ_p} T)^{H_F}$ is a free $\mbfa_F\textrm{-module}$ of rank $=\dim_{\QQ_p} V$ equipped with a Frobenius-semilinear operator $\varphi$ and a continuous and semilinear action of $\Gamma_F$, and we have $\mbfd^+(T) \subset \mbfd(T)$.

Fontaine showed that $V$ is of finite height if and only if there exists a finite free $\mbfb_F^+\textrm{-submodule}$ of $\mbfd(V)$ of rank $= \dim_{\QQ_p} V$, stable under the operator $\varphi$ (see \cite[\S B.2.1]{fontaine-festschrift} and \cite[\S III.2]{colmez-finite-height}).
Moreover, if $T \subset V$ is a free $\ZZ_p\textrm{-lattice}$ as above and $V$ of finite height, then $\mbfd^+(T)$ is a free $\mbfa_F^+\textrm{-module}$ of rank $= \dim_{\QQ_p} V$ such that $\mbfa_F \otimes_{\mbfa_F^+} \mbfd^+(T) \isomorphic \mbfd(T)$ (see \cite[Th\'eor\`eme B.1.4.2]{fontaine-festschrift}).

For crystalline representations there exist submodules of $\mbfd^+(V)$ admitting a simpler action of $\Gamma_F$.
Finite height and crystalline representations of $G_F$ are related by the following result:
\begin{thm}[{\cite{wach-pot-crys}, \cite{colmez-finite-height}, \cite{berger-differentielles}}]\label{thm:crystalline_finite_height_unrami}
	Let $V$ be a $\padic$ representation of $G_F$.
	Then $V$ is crystalline if and only if it is of finite height and there exists $r \in \ZZ$ and a $\mbfb_F^+\textrm{-submodule}$ $N \subset \mbfd^+(V)$ of rank $= \dim_{\QQ_p} V$, stable under the action of $\Gamma_F$, such that $\Gamma_F$ acts trivially over $(N/\pi N)(-r)$.
\end{thm}

In the situation of Theorem \ref{thm:crystalline_finite_height_unrami}, the module $N$ is not unique.
A functorial construction was given by Berger:

\begin{prop}[{\cite[Proposition II.1.1]{berger-limites-cristallines}}]\label{prop:wach_module_existence}
	Let $V$ be a positive crystalline representation of $G_F$, i.e. all Hodge-Tate weights of $V$ are $\leq 0$.
	Let $T \subset V$ be a free $\ZZ_p\textrm{-lattice}$, stable under the action of $G_F$.
	Then there exists a unique $\mbfa_F^+\textrm{-module}$ $\mbfn(T) \subset \mbfd(T)$, which is free of rank $=\dim_{\QQ_p} V$, stable under the action of $\varphi$ and $\Gamma_F$, and the action of $\Gamma_F$ is trivial over $\mbfn(T)/\pi \mbfn(T)$.
	Moreover, there exists $s \in \NN$ such that $\pi^s\mbfd^+(T) \subset \mbfn(T)$.
	Finally, set $\mbfn(V) := \mbfb_F^+ \otimes_{\mbfa_F^+} \mbfn(T)$, then $\mbfn(V)$ is a unique $\mbfb^+_F\textrm{-submodule}$ of $\mbfd^+(V)$ satisfying analogous properties.
\end{prop}

\begin{nota}
	For an algebra $S$ admitting an action of the Frobenius and an $S\textrm{-module}$ $M$ admitting a Frobenius-semilinear endomorphism $\varphi : M \rightarrow M$, we denote by $\varphi^{\ast}(M) \subset M$ the $S\textrm{-submodule}$ generated by the image of $\varphi$.
\end{nota}

\begin{rem}\phantomsection\label{rem:wach_berger_module_const}
	\begin{enumromanup}
	\item In Proposition \ref{prop:wach_module_existence} for positive crystalline representations, Berger applies Theorem \ref{thm:crystalline_finite_height_unrami} with $r=0$ to define $\mbfn(V) := \mbfd^+(V) \cap N\big[\frac{1}{\varphi^{n-1}(q)}\big]_{n \geq 1}$, where $q = \frac{\varphi(\pi)}{\pi}$.
		Using this one can take $\mbfn(T) := \mbfn(V) \cap \mbfd(T)$ and it can be shown to satisfy the desired properties.

	\item Berger further showed that in the setup of Proposition \ref{prop:wach_module_existence}, if we take $s$ to be the maximum among the absolute values of Hodge-Tate weights of $V$, then $\mbfn(T)/\varphi^{\ast}(\mbfn(T))$ is killed by $q^s$ and we have that $\pi^s \mbfa^+ \otimes_{\ZZ_p} T \subset \mbfa^+ \otimes_{\mbfa_F^+} \mbfn(T)$ (see \cite[Th\'eor\`eme III.3.1]{berger-limites-cristallines}).
		The former observation can be thought of as a finite $q\textrm{-height}$ property of Wach modules.
		We will impose it as one of the main conditions for defining finite $q\textrm{-height}$ representations in the relative case (see \ref{defi:wach_reps}).
	\end{enumromanup}
\end{rem}

\begin{defi}\label{defi:wach_modules}
	Let $a$, $b \in \ZZ$ with $b \geq a$.
	A \textit{Wach module} with weights in the interval $[a, b]$ is a finite free $\mbfa_F^+\textrm{-module}$ or a $\mbfb_F^+\textrm{-module}$ $N$, equipped with a continuous and semilinear action of $\Gamma_F$ such that the action of $\Gamma_F$ is trivial on $N/\pi N$ and a Frobenius-semilinear operator $\varphi: N\big[\frac{1}{\pi}\big] \rightarrow N\big[\frac{1}{\varphi(\pi)}\big]$ commuting with the action of $\Gamma_F$, $\varphi(\pi^b N) \subset \pi^b N$ and $\pi^bN / \varphi^{\ast}(\pi^b N)$ is killed by $q^{b-a}$.
\end{defi}

\begin{rem}
	The definition of the functor $\mbfn$ can be extended to crystalline representations of arbitrary Hodge-Tate weights quite easily.
	Indeed, let $V \in \Rep_{\QQ_p}^{\crys}(G_F)$ with Hodge-Tate weights in the interval $[a, b]$ and let $T \subset V$ a free $\ZZ_p\textrm{-lattice}$, stable under the action of $G_F$.
	Then $\mbfn(T) = \pi^{-b} \mbfn(T(-b)) \otimes_{\ZZ_p} \ZZ_p(b)$ is a Wach module over $\mbfa_F^+$ with weights in the interval $[a, b]$.
\end{rem}

As it turns out, one can recover the crystalline representation from a given Wach module:
\begin{prop}[{\cite[Proposition III.4.2]{berger-limites-cristallines}}]\label{prop:wach_modules_equivalent_categories}
	The functor
	\begin{align*}
		\mbfn : \Rep_{\QQ_p}^{\crys}(G_F) &\longrightarrow \textup{ Wach modules over } \mbfb_F^+\\
			V &\longmapsto \mbfn(V),
	\end{align*}
	establishes an equivalence of categories with a quasi-inverse given by $N \mapsto (\mbfb \otimes_{\mbfb_F^+} N)^{\varphi=1}$.
	These functors are compatible with tensor products, duality and preserve exact sequences.
	Moreover, for a crystalline representation $V$, the map $T \mapsto \mbfn(T)$ induces a bijection between $\ZZ_p\textrm{-lattices}$ inside $V$ and Wach modules over $\mbfa_F^+$ contained in $\mbfn(V)$.
\end{prop}

We have a natural filtration on Wach modules given as
\begin{equation*}
	\Fil^k \mbfn(V) = \{x \in \mbfn(V) \hspace{1mm} \textrm{such that} \hspace{1mm} \varphi(x) \in q^k \mbfn(V)\} \hspace{2mm} \textrm{for} \hspace{2mm} k \in \ZZ.
\end{equation*}
If $V$ is positive crystalline, i.e. all its Hodge-Tate weights are $\leq 0$, then for $r \in \NN$ we have
\begin{equation*}
	\Fil^k \mbfn(V(r)) = \Fil^k \pi^{-r} \mbfn(V)(r) = \pi^{-r} \Fil^{k+r} \mbfn(V)(r).
\end{equation*}
Using this filtration on $\mbfn(V)$, one can also recover the other linear algebraic object associated to $V$, i.e. the filtered $\varphi\textrm{-module}$ $\mbfd_{\crys}(V)$:
Let $\mbfb_{\rig, F}^+ \subset F[[\pi]]$ denote the subring of convergent power series over the open unit disc.
Then we have $\mbfd_{\crys}(V) \subset \mbfb_{\rig, F}^+ \otimes_{\mbfb_F^+} \mbfn(V)$ and this gives $\mbfd_{\crys}(V) = \big(\mbfb_{\rig, F}^+ \otimes_{\mbfb_F^+} \mbfn(V)\big)^{\Gamma_F}$ (see \cite[Proposition II.2.1]{berger-limites-cristallines}).
Moreover, the induced map 
\begin{equation*}
	\mbfd_{\crys}(V) \longrightarrow \big(\mbfb_{\rig, F}^+ \otimes_{\mbfb_F^+} \mbfn(V)\big) / \pi\big(\mbfb_{\rig, F}^+ \otimes_{\mbfb_F^+} \mbfn(V)\big) = \mbfn(V)/\pi\mbfn(V),
\end{equation*}
is an isomorphism of filtered $\varphi\textrm{-modules}$ (see \cite[Proposition III.4.4]{berger-limites-cristallines}).

\subsection{The relative case}\label{subsec:rel_wach_crys}

In this section, we will introduce the notion of relative Wach modules and study representations of finite height.
Recall that we fixed $m \in \NN_{\geq 1}$ (fix $m \in \NN_{\geq 2}$ if $p=2$) and we have $K = F_m = F(\zeta_{p^m})$.
The element $\varpi = \zeta_{p^m}-1$ is a uniformizer of $K$.
We have $X = (X_1, \ldots, X_d)$ a set of indeterminates and we defined $R$ to be the $\padic$ completion of an \'etale algebra over $O_F[X, X^{-1}]$ having non-empty and geometrically integral special fiber and $R[\varpi] = O_K \otimes_{O_F} R$.
For $R$ and $R[\varpi]$, we can use the $(\varphi, \Gamma)\textrm{-module}$ theory discussed in \S \ref{subsec:relative_phi_gamma_mod}, as well as the constructions in \S \ref{subsec:pd_envelope} and \S \ref{subsec:cyclotomic_embeddings}.

Setting $q = \frac{\varphi(\pi)}{\pi}$ and using the formulation in Definition \ref{defi:wach_modules}, we define relative Wach modules:
\begin{defi}\label{defi:rel_wach_mods}
	Let $a$, $b \in \ZZ$ with $b \geq a$.
	A \textit{Wach module} over $\mbfa_R^+$ (resp. $\mbfb_{R}^+$) with weights in the interval $[a, b]$ is a finite projective $\mbfa_R^+\textrm{-module}$ (resp. $\mbfb_{R}^+\textrm{-module}$) $N$, equipped with a continuous and semilinear action of $\Gamma_{R}$ such that the action of $\Gamma_R$ is trivial on $N/\pi N$.
	Further, there is a Frobenius-semilinear operator $\varphi: N\big[\frac{1}{\pi}\big] \rightarrow N\big[\frac{1}{\varphi(\pi)}\big]$ commuting with the action of $\Gamma_{R}$ such that $\varphi(\pi^b N) \subset \pi^b N$ and $\pi^bN / \varphi^{\ast}(\pi^b N)$ is killed by $q^{b-a}$.
\end{defi}

Let $V$ be a $\padic$ representation of the Galois group $G_{R}$ admitting a $\ZZ_p\textrm{-lattice}$ $T \subset V$ stable under the action of $G_{R}$.
Then we have the finitely generated $\mbfa_R^+\textrm{-module}$ $\mbfd^+(T) := (\mbfa^+ \otimes_{\QQ_p} T)^{H_{R}}$.
We introduce the following definition:

\begin{defi}\label{defi:wach_reps}
	A \textit{positive finite $q\textrm{-height}$} representation is a $\padic$ representation $V$ of $G_{R}$ admitting a $\ZZ_p\textrm{-lattice}$ $T \subset V$ such that there exists a finite projective $\mbfa_R^+\textrm{-submodule}$ $\mbfn(T) \subset \mbfd^+(T)$ of rank $=\dim_{\QQ_p} V$ satisfying the following conditions:
	\begin{enumromanup}
		\item $\mbfn(T)$ is stable under the action of $\varphi$ and $\Gamma_{R}$, and $\mbfa_{R} \otimes_{\mbfa_R^+} \mbfn(T) \isomorphic \mbfd(T)$;
		
		\item The $\mbfa_R^+\textrm{-module}$ $\mbfn(T) / \varphi^{\ast}(\mbfn(T))$ is killed by $q^{s}$ for some $s \in \NN$;
		
		\item The action of $\Gamma_{R}$ is trivial on $\mbfn(T) / \pi \mbfn(T)$;
		
		\item There exists a $R\prm \subset \overline{R}$ finite étale over $R$ such that the $\mbfa_{R\prm}^+\textrm{-module}$ $\mbfa_{R\prm}^+ \otimes_{\mbfa_R^+} \mbfn(T)$ is free.
	\end{enumromanup}
	The module $\mbfn(T)$ is a \textit{Wach module} associated to $T$ with weights in the interval $[-s, 0]$ and we set $\mbfn(V) := \mbfn(T)\big[\frac{1}{p}\big]$ satisfying properties analogous to (i)-(iv) above.
	The \textit{height} of $V$ is defined to be the smallest $s \in \NN$ satisfying (ii) above.
\end{defi}

For $r \in \ZZ$, we set $V(r) := V \otimes_{\QQ_p} \QQ_p(r)$ and $T(r) := T \otimes_{\ZZ_p} \ZZ_p(r)$.
We will call these twists as representations of \textit{finite $q\textrm{-height}$} and define
\begin{equation*}
	\mbfn(T(r)) := \tfrac{1}{\pi^r}\mbfn(T)(r) \hspace{2mm} \textrm{and} \hspace{2mm} \mbfn(V(r)) := \tfrac{1}{\pi^r} \mbfn(V)(r).
\end{equation*}
Since $\mbfn(V)$ and $\mbfn(T)$ are Wach modules with weights in the interval $[-s, 0]$, twisting by $r$ gives us Wach modules in the sense of Definition \ref{defi:rel_wach_mods} with weights in the interval $[r-s, r]$.
We will say that \textit{height} of $V(r) = $ height of $V$.

\begin{rem}
	\begin{enumromanup}
	\item In the arithmetic case, i.e. $R = O_F$, the notion of finite height representations in Theorem \ref{thm:crystalline_finite_height_unrami} and finite $q\textrm{-height}$ representations in Definition \ref{defi:wach_reps} are related.
		In fact, in the arithmetic case using Definition \ref{defi:wach_reps} one obtains the functorial object of Berger mentioned above (see \cite[Proposition II.1.1]{berger-limites-cristallines}).

	\item In Definition \ref{defi:wach_reps} conditions (i), (ii) and (iii) are motivated from the definition of finite height representations of $G_F$ admitting a Wach module structure.
		The last condition, i.e. (iv) is inspired by Brinon's definition of weak admissibility in the relative case (see \cite[p. 136]{brinon-padicrep-relatif}). 

	\item In Definition \ref{defi:wach_reps} following Remark \ref{rem:wach_berger_module_const} (i), one can first define Wach module for the representation $V$ and then consider the module $\mbfn(T) = \mbfn(V) \cap \mbfd(T)$ associated to $T$.
		However, it is not clear whether the latter module, defined in this fashion, is a projective $\mbfa_R^+\textrm{-module}$.
		Therefore, we impose the condition on $\mbfn(T)$ to be projective, which is required in establishing several results in this section.
	\end{enumromanup}
\end{rem}

\subsubsection{Some properties of Wach modules}

Let us note some important properties of Wach modules associated to finite $q\textrm{-height}$ representations
\begin{prop}\label{prop:wach_approx_aplus_admis}
	Let $V$ be a positive finite $q\textrm{-height}$ representation and $T \subset V$ a $G_{R}\textrm{-stable}$ $\ZZ_p\textrm{-lattice}$.
	Then we have $\pi^s \mbfa^+ \otimes_{\ZZ_p} T \subset \mbfa^+ \otimes_{\mbfa_R^+} \mbfn(T)$, where $s \in \NN$ is the height of the representation $V$.
\end{prop}
\begin{proof}
	To show the claim, we can assume that $\mbfn(T)$ is free by base changing to the period ring corresponding to the finite \'etale extension $R\prm$ of $R$.
	Then $\mbfa^+ \otimes_{\mbfa_{R\prm}^+} \big(\mbfa_{R\prm}^+ \otimes_{\mbfa_R^+} \mbfn(T)\big) = \mbfa^+ \otimes_{\mbfa_R^+} \mbfn(T)$ is free.
	Since the discussion of previous chapters hold for the $\padic$ completion of a finite \'etale extension of $R$ (see \cite[Chapitre 2]{brinon-padicrep-relatif} and \cite[\S 2]{andreatta-iovita-relative-phiGamma} for more on this), base changing to $R\prm$ is harmless.
	So with a slight abuse of notation, below we will replace $R\prm$ obtained in this manner by $R$ and assume $\mbfn(T)$ to be free of rank $h = \dim_{\QQ_p} V$ over $\mbfa_R^+$.

	Note that by definition we have $\mbfn(T) \subset \mbfd^+(T) = (\mbfa^+ \otimes_{\ZZ_p} T)^{H_{R}} \subset \mbfa^+ \otimes_{\ZZ_p} T$.
	So let $A \in \Mat(h, \mbfa^+)$ be the matrix obtained by expressing a basis of $\mbfn(T)$ in a chosen basis of $T$.
	Also, let $P \in \Mat(h, \mbfa_R^+)$ be the matrix of $\varphi$ in the basis of $\mbfn(T)$.
	Then we have $\varphi(A) = AP$ and therefore $\varphi(\pi^s A^{-1}) = (q^s P^{-1})(\pi^s A^{-1})$.
	The fact that $\mbfn(T) / \varphi^{\ast}(\mbfn(T))$ is killed by $q^s$ implies that $q^s P^{-1} \in \Mat(h, \mbfa_R^+)$, therefore from Lemma \ref{lem:reg_frob_finht} we obtain that $\pi^s A^{-1} \in \Mat(h, \mbfa^+)$.
	Hence, we conclude that $\pi^s \mbfa^+ \otimes_{\ZZ_p} T \subset \mbfa^+ \otimes_{\mbfa_R^+} \mbfn(T)$.
\end{proof}

\begin{cor}\label{cor:wach_uniqueness_crit}
	By taking $H_{R}\textrm{-invariants}$ in Proposition \ref{prop:wach_approx_aplus_admis} it follows that $\pi^s \mbfd^+(T) \subset \mbfn(T)$.
\end{cor}

\begin{prop}\label{prop:wach_module_uniqueness}
	Let $V$ be a finite $q\textrm{-height}$ representation $G_{R}$.
	The Wach module $\mbfn(V)$ over $\mbfb_{R}^+$ is unique.
	Same holds true for the $\mbfa_R^+\textrm{-module}$ $\mbfn(T)$.
\end{prop}
\begin{proof}
	The argument carries over from the classical case \cite[p. 13]{berger-limites-cristallines}.
	First note that we can assume that $V$ is positive, since by definition the uniquess of Wach module for such a representation is equivalent to uniqueness for all its Tate twists.
	In this case, let $N_1$ and $N_2$ be two $\mbfa_R^+\textrm{-modules}$ satisfying the conditions of Definition \ref{defi:wach_reps} (the proof stays the same for $\mbfn(V)$).
	By symmetry, it is enough to show that $N_1 \subset N_2$.
	Since we have $\pi^s N_1 \subset \pi^s \mbfd^+(T) \subset N_2$ (see Corollary \ref{cor:wach_uniqueness_crit}) and $N_2$ is $\pi\textrm{-torsion}$ free, therefore for any $x \in N_1$ there exists $k \leq s$ such that $\pi^k x \in N_2$ but $\pi^k x \not\in \pi N_2$.
	Varying over all $x \in N_1 \setminus \pi N_1$, we can take $k \leq s$ to be the minimal integer such that $\pi^k N_1 \subset N_2$.
	Since $\pi^k x \in N_2$ and $\Gamma_{R}$ acts trivially on $N_2 /\pi N_2$, we have that $(\gamma_0 - 1)(\pi^k x) \in \pi N_2$.
	So we can write
	\begin{equation*}
		(\gamma_0 - 1)(\pi^k x) = \gamma_0(\pi^k)(\gamma_0(x) - x) + (\gamma_0(\pi^k) - \pi^k)x.
	\end{equation*}
	Since $\Gamma_{R}$ also acts trivially on $N_1/\pi N_1$ and $\pi^k N_1 \subset N_2$, we see that $\gamma_0(\pi^k)(\gamma_0(x)-x) \in \pi N_2$, therefore $(\gamma_0(\pi^k) - \pi^k)x \in \pi N_2$, which means that $(\chi(\gamma_0)^k-1)\pi^k x \in \pi N_2$.
	But $\pi \nmid (\chi(\gamma_0)^k-1)$ if $k \geq 1$, and $\pi^k x \not\in \pi N_2$.
	Hence, we must have $k = 0$, i.e. $N_1 \subset N_2$. 
\end{proof}

The uniqueness of Wach modules helps us in establishing compatibility with usual operations:
\begin{prop}\label{prop:wach_module_sum_tensor}
	Let $V$ and $V\prm$ be two finite $q\textrm{-height}$ representations of $G_{R}$.
	Then we have that $\mbfn(V \oplus V\prm) = \mbfn(V) \oplus \mbfn(V\prm)$ and $\mbfn(V \otimes V\prm) = \mbfn(V) \otimes \mbfn(V\prm)$.
	Similar statements hold for $\mbfn(T)$ and $\mbfn(T\prm)$.
\end{prop}
\begin{proof}
	We note similar to previous lemma that it is enough to show the statement for $V$ and $V\prm$ such that both representations are positive.
	By uniqueness of Wach modules proved in Proposition \ref{prop:wach_module_uniqueness}, it is enough to show that direct sum and tensor product of finite $q\textrm{-height}$ representations are again of finite $q\textrm{-height}$.

	First, it is straightforward to see that $\mbfn(T) \oplus \mbfn(T\prm) \subset \mbfd^+(T \oplus T\prm)$ is a projective $\mbfa_R^+\textrm{-module}$ of rank $\textup{rk}_{\ZZ_p} (T \oplus T\prm)$ such that $\mbfa_{R} \otimes_{\mbfa_R^+} (\mbfn(T) \oplus \mbfn(T\prm)) \isomorphic \mbfd(T) \oplus \mbfd(T\prm)$.
	Similarly, we have that $\mbfn(T) \otimes \mbfn(T\prm) \subset \mbfd^+(T \otimes T\prm)$ is a projective $\mbfa_R^+\textrm{-module}$ of rank $\textup{rk}_{\ZZ_p} (T \otimes T\prm)$ such that $\mbfa_{R} \otimes_{\mbfa_R^+} (\mbfn(T) \otimes \mbfn(T\prm)) \isomorphic \mbfd(T) \otimes \mbfd(T\prm)$.
	
	Next, let $s$ and $s\prm$ denote the height of representations $V$ and $V\prm$ respectively and let $i := \max(s, s\prm)$.
	Then we see that $(\mbfn(T) \oplus \mbfn(T\prm)) / \varphi^{\ast}(\mbfn(T) \oplus \mbfn(T\prm))$ is killed by $q^i$ and $(\mbfn(T) \otimes \mbfn(T\prm)) / \varphi^{\ast}(\mbfn(T) \otimes \mbfn(T\prm))$ is killed by $q^{s+s\prm}$.
	Further, $\Gamma_{R}$ acts trivially modulo $\pi$ on $\mbfn(T) \oplus \mbfn(T\prm)$ and $\mbfn(T) \otimes \mbfn(T\prm)$.
	This verifies conditions (i), (ii) and (iii) for these modules.
For condition (iv), note that given any two finite étale extensions $R\prm$ and $R\dprm$ of $R$, there exists a finite étale extension $S$ over $R$ such that $S$ is finite étale over $R\prm$ as well as $R\dprm$.
	Hence, we get the claim.
\end{proof}

\begin{cor}\label{cor:wach_module_sym_wedge}
	Let $V$ be a finite $q\textrm{-height}$ representation of $G_{R}$.
	Then, for $k \in \NN$ the representations $\Sym^k(V)$ and $\wedge^k V$ are of finite $q\textrm{-height}$.
\end{cor}
\begin{proof}
	Note that the compatibility with tensor products in Proposition \ref{prop:wach_module_sum_tensor} is enough to establish the compatibility with symmetric powers and exterior powers because then we can set
	\begin{equation*}
		\mbfn\big(\Sym^k(T)\big) := \Sym^k(\mbfn(T)), \hspace{2mm} \textrm{and} \hspace{2mm} \mbfn\big(\wedge^k T\big) := \wedge^k\mbfn(T).
	\end{equation*}
	We have $\mbfn\big(\Sym^k(T)\big) \subset \Sym^k(\mbfd^+(T)) \subset \mbfd^+\big(\Sym^k(T)\big)$, since $\mbfa^+ \otimes_{\mbfa_R^+} \Sym^k(\mbfd^+(T)) \subset \mbfa^+ \otimes_{\mbfa_R^+} \mbfd^+\big(\Sym^k(T)\big)$.
	Similarly, $\mbfn\big(\wedge^k T\big) \subset \mbfd^+\big(\wedge^k T\big)$.
	Rest of the assumptions of Definition \ref{defi:wach_reps} follows in a same manner as in the proof of Proposition \ref{prop:wach_module_sum_tensor}.
	This establshes that $\Sym^k(V)$ and $\wedge^k V$ are finite $q\textrm{-height}$ representations and gives us the corresponding Wach modules.
\end{proof}

\subsubsection{Filtration on Wach modules}

There is a natural filtration on Wach modules associated to finite $q\textrm{-height}$ representations.
We will introduce this filtration next and prove a lemma concerning this filtration. 
\begin{defi}\label{defi:wach_mod_fil}
	Let $V$ be a positive finite $q\textrm{-height}$ represenation of $G_{R}$ and $r \in \NN$.
	Then there is a natural filtration on the associated Wach modules given as
	\begin{equation*}
		\Fil^k \mbfn(V(r)) := \{x \in \mbfn(V(r)), \hspace{1mm} \textrm{such that} \hspace{1mm} \varphi(x) \in q^k \mbfn(V(r))\} \hspace{2mm} \textrm{for} \hspace{1mm} k \in \ZZ,
	\end{equation*}
	and we set $\Fil^k \mbfn(T(r)) := \Fil^k \mbfn(V(r)) \cap \mbfn(T(r))$, where the intersection is taken inside $\mbfn(V(r))$.
\end{defi}

\begin{lem}\label{lem:wach_mod_twist_fil}
	With notations as above, we have
	\begin{enumromanup}
	\item $\Fil^k \mbfn(T(r)) = \{x \in \mbfn(T(r)), \hspace{1mm} \textrm{such that} \hspace{1mm} \varphi(x) \in q^k \mbfn(T(r))\}$.

	\item $\Fil^k \mbfn(V(r)) = \Fil^k \pi^{-r} \mbfn(V)(r) = \pi^{-r} \Fil^{k+r} \mbfn(V)(r)$ and similarly for $\Fil^k \mbfn(T(r))$.
	\end{enumromanup}
\end{lem}
\begin{proof}
	\begin{enumromanup}
	\item For $k \leq 0$, the claim is obvious, so we assume that $k > 0$.
		Then we are reduced to showing that $q^k \mbfn(V(r)) \cap \mbfn(T(r)) = q^k \mbfn(T(r))$.

		To prove the latter claim, note that it is enough to work under the assumption that $\mbfn(T(r))$ is free.
		Indeed, for any finite $q\textrm{-height}$ representation $V(r)$, there exists a finite étale $R\textrm{-algebra}$ $R\prm$ such that $\mbfa_{R\prm}^+ \otimes_{\mbfa_R^+} \mbfn(T(r))$ is free.
		Since $\mbfa_{R\prm}^+$ is faithfully flat over $\mbfa_R^+$, the claim is equivalent to showing that $\mbfa_{R\prm}^+ \otimes_{\mbfa_R^+} (q^k \mbfn(V) \cap \mbfn(T)) = q^k \mbfa_{R\prm}^+ \otimes_{\mbfa_R^+} \mbfn(T)$.
		But one can easily obtain that $\mbfa_{R\prm}^+ \otimes_{\mbfa_R^+} (q^k \mbfn(V) \cap \mbfn(T)) = (q^k \mbfa_{R\prm}^+ \otimes_{\mbfa_R^+} \mbfn(V)) \cap (\mbfa_{R\prm}^+ \otimes_{\mbfa_R^+} \mbfn(T))$ (or see \cite[Theorem 7.4 (i)]{matsumura-reid-commutative}) as submodules of $\mbfa_{R\prm}^+ \otimes_{\mbfa_R^+} \mbfn(V)$.
		So below we will assume that $\mbfn(T(r))$ is free over $\mbfa_R^+$ with a basis $\{f_1, \ldots, f_h\}$, where $h = \dim_{\QQ_p} V(r)$.

		Let $x = \sum_{i=1}^h x_i f_i \in q^k \mbfn(V(r)) \cap \mbfn(T(r))$ with $x_i \in \mbfa_R^+$.
		Since $\{f_1, \ldots, f_h\}$ is also a $\mbfb_{R}^+\textrm{-basis}$ of $\mbfn(V(r))$, we can write $x = q^k \sum_{i=1}^h y_i f_i$ with $y_i \in \mbfb_{R}^+$.
		Comparing the two expressions for $x$ we obtain that $q^k y_i = x_i \in \mbfa_R^+$, i.e. $y_i \in \mbfa_{R}$ for $1 \leq i \leq h$.
		But this just means that $y_i \in \mbfb_{R}^+ \cap \mbfa_{R} = \mbfa_R^+$, therefore $x_i = q^k y_i \in q^k \mbfa_R^+$ for $1 \leq i \leq h$.
		Hence, $x \in q^k \mbfn(T(r))$ as desired.
		The other inclusion is obvious.

	\item Note that the inclusion $\pi^{-r} \Fil^{k+r} \mbfn(V)(r) \subset \Fil^k \pi^{-r} \mbfn(V)(r)$ is obvious.
		To show the converse let $\pi^{-r} x \otimes \epsilon^{\otimes r} \in \Fil^k \pi^{-r} \mbfn(V)(r)$, with $x \in \mbfn(V)$ and $\epsilon^{\otimes r}$ being a basis of $\QQ_p(r)$.
		Then we have that $\varphi(\pi^{-r} x \otimes \epsilon^{\otimes r}) = q^{-r}\pi^{-r} \varphi(x) \otimes \epsilon^{\otimes r} \in q^k \pi^{-r} \mbfn(V)(r)$.
		Therefore, we obtain that $\varphi(x) \in q^{k+r} \mbfn(V)$, i.e. $x \in \Fil^{k+r} \mbfn(V)$.
	\end{enumromanup}
\end{proof}

\begin{rem}\label{rem:fil_coincide_ar0plus}
	For $V = \QQ_p$ the filtration in Definition \ref{defi:wach_mod_fil} coincides with the filtration in Lemma \ref{lem:fil_ar0plus}
\end{rem}
\begin{proof}
	We have $T = \ZZ_p$ and $\mbfn(T) = \mbfa_R^+$ and let $\varpi = \zeta_p - 1$ (let $\varpi = \zeta_{p^2}-1$ if $p=2$) in this proof.
	Since $\pi^k \mbfa_R^+ \subset \Fil^k \mbfn(T)$ (where the term on right is the filtration in Definiton \ref{defi:wach_mod_fil}), we only need to show that $\Fil^k \mbfn(T) \subset \pi^k \mbfa_R^+ = \mbfa_R^+ \cap \xi^k \mbfa_{R, \varpi}^+$.
	Let $x \in \mbfa_R^+$ such that $\varphi(x) = q^k y$ for some $y \in \mbfa_R^+$.
	As we have $\mbfa_R^+ \subset \mbfa_{R, \varpi}$, we can also write $\varphi(x) = \varphi(\xi^k) y \in \varphi(\mbfa_{R, \varpi}) \subset \mbfa_{R}$, i.e. $y \in \varphi(\mbfa_{R, \varpi}) \cap \mbfa_R^+ = \varphi(\mbfa_{R, \varpi}^+)$ (where the intersection is taken inside $\mbfa_{R}$).
	Therefore, we obtain that $y = \varphi(z)$ for some $z \in \mbfa_{R, \varpi}^+$.
	Since $\varphi : \mbfa_{R, \varpi}^+ \rightarrow  \mbfa_{R, \varpi}^+$ is injective, we must have $x = \xi^k z \in \mbfa_R^+ \cap \xi^k \mbfa_{R, \varpi}^+$, as desired.
\end{proof}

\subsection{Statement of the main result}\label{subsec:main_result}

In this section, we will relate the notion of crystalline and finite $q\textrm{-height}$ representations.
As we will see, we can recover the $R\big[\frac{1}{p}\big]\textrm{-module}$ $\pazo \mbfd_{\crys}(V)$ from the $\mbfa_R^+\textrm{-module}$ $\mbfn(T)$ after passing to a larger period ring and inverting $p$.
We begin by introducing this ring below.

Recall from \S \ref{subsec:setup_nota} that we have $F$ as a finite unramified extenion of $\QQ_p$ with ring of integers $O_F$ and we take $K = F(\zeta_{p^m})$ for a fixed $m \in \NN_{\geq 1}$ (fix $m \in \NN_{\geq 2}$ if $p=2$).
Note that the formulation of the results and proofs depend on $m$ and it is necessary to have $m \geq 1$ ($m \geq 2$ if $p=2$) for the discussion below to make sense.

\subsubsection{The ring \texorpdfstring{$\pazo \mbfa_{R, \varpi}^{\textpd}$}{-}}\label{subsubsec:oarpd}

In this section, we will work with the ring $\mbfa_{R, \varpi}^+$ defined in \S \ref{subsec:cyclotomic_embeddings}, equipped with an action of the Frobenius $\varphi$ and a continuous action of $\Gamma_{R}$.
Since we have a natural injection $\mbfa_{R, \varpi}^+ \rightarrowtail \mbfa_{\inf}(\overline{R})$, we obtain a $G_{R}\textrm{-equivariant}$ commutative diagram
\begin{center}
	\begin{tikzcd}[row sep=large, column sep=large]
		\mbfa_{R, \varpi}^+ \arrow[d, rightarrowtail] \arrow[r, twoheadrightarrow, "\theta"] & R[\varpi] \arrow[d, rightarrowtail] \\
		\mbfa_{\inf}(\overline{R}) \arrow[r, twoheadrightarrow, "\theta"] & \CC^+(\overline{R}).
	\end{tikzcd}
\end{center}
By $R\textrm{-linearity}$, extending scalars for the map $\theta$ above, we obtain a ring homomorphism
\begin{equation*}
	\theta_{R} : R \otimes_{\ZZ} \mbfa_{R, \varpi}^+ \longrightarrow R[\varpi],
\end{equation*}
sending $X_i \otimes 1 \mapsto X_i$, $1 \otimes [X_i^{\flat}] \mapsto X_i$ for $1 \leq i \leq d$ and $1 \otimes \pi_m \mapsto \zeta_{p^m}-1$.
Note that we have inclusion of ideals $\big(\xi, X_i \otimes 1 - 1 \otimes [X_i^{\flat}], \hspace{1mm} \textrm{for} \hspace{1mm} 1 \leq i \leq d\big) \subset \kert \theta_{R} \subset R \otimes_{\ZZ} \mbfa_{R, \varpi}^+$, where $\xi = \frac{\pi}{\pi_1}$.
We have $\mbfa_{R, \varpi}^+ \subset \mbfa_{\inf}(\overline{R})$ and $\theta_{R}$ above is the restriction of $\theta_{R} : R \otimes_{\ZZ} \mbfa_{\inf}(\overline{R}) \twoheadrightarrow \CC^+(\overline{R})$ (see \S \ref{subsubsec:crystalline_defi}).
So similar to $\pazo \mbfa_{\inf}(\overline{R})$ in \S \ref{subsubsec:deRham_prop} and $\pazo \mbfa_{\crys}(\overline{R})$ in \S \ref{subsubsec:crystalline_prop} we define the following rings:

\begin{defi}\phantomsection\label{defi:oarpd}
	\begin{enumromanup}
	\item Define $\pazo \mbfa_{R, \varpi}^+$ to be $\theta_R^{-1}(pR[\varpi])\textrm{-adic}$ completion of $R \otimes_{\ZZ} \mbfa_{R, \varpi}^+$.

	\item Let $x^{[n]} := x^n/n!$ for $x \in \kert \theta_{R}$.
		Define $\pazo \mbfa_{R,\varpi}^{\textpd}$ to be the $\padic$ completion of the divided power envelope of $R \otimes_{\ZZ} \mbfa_{R, \varpi}^+$ with respect to $\kert \theta_{R}$.
	\end{enumromanup}
	Note that we have $\pazo \mbfa_{R, \varpi}^+ = \pazo \mbfa_{\inf}(\overline{R}) \cap \pazo \mbfa_{R, \varpi}^{\textpd} \subset \pazo \mbfa_{\crys}(\overline{R})$.
\end{defi}

Next, taking the divided power envelope of $\theta_{R} / p^n$, we notice that $\pazo \mbfa_{R,\varpi}^{\textpd} / p^n \rightarrowtail \pazo \mbfa_{\crys}(\overline{R}) / p^n$.
Since $\pazo \mbfa_{R,\varpi}^{\textpd} = \lim_n \pazo \mbfa_{R,\varpi}^{\textpd} / p^n$ and $\pazo \mbfa_{\crys}(\overline{R}) = \lim_n \pazo \mbfa_{\crys}(\overline{R}) / p^n$, and (projective) limit is left exact, it follows that for the $\padic$ completion of divided power envelope of $\theta_{R}$, we have $\pazo \mbfa_{R,\varpi}^{\textpd} \subset \pazo \mbfa_{\crys}(\overline{R})$.
Now, over the ring $\pazo \mbfa_{R, \varpi}^{\textpd}$ we can consider the induced action of $\Gamma_{R}$ under which it is stable, and it admits a Frobenius endomorphism arising from the Frobenius on each component of the tensor product.
In particular, from the diagram above we obtain a $G_{R}\textrm{-equivariant}$ commutative diagram
\begin{center}
	\begin{tikzcd}[row sep=large, column sep=large]
		\pazo \mbfa_{R,\varpi}^{\textpd} \arrow[d, rightarrowtail] \arrow[r, twoheadrightarrow, "\theta_{R}"] & R[\varpi] \arrow[d, rightarrowtail] \\
		\pazo \mbfa_{\crys}(\overline{R}) \arrow[r, twoheadrightarrow, "\theta_{R}"] & \CC^+(\overline{R}).
	\end{tikzcd}
\end{center}
Note that the left vertical arrow is Frobenius-equivariant.

Next, we will give an alternative description of the ring $\pazo \mbfa_{R,\varpi}^{\textpd}$.
Let $T = (T_1, \ldots, T_d)$ denote a set of indeterminates and let $\mbfa_{\crys}(\overline{R})\langle T \rangle^{\wedge}$ denote the $\padic$ completion of the divided power polynomial algebra $\mbfa_{\crys}(\overline{R})\langle T \rangle = \mbfa_{\crys}(\overline{R})[T_i^{[n]}, \hspace{1mm} n \in \NN, \hspace{1mm} 1 \leq i \leq d]$.
Recall from \S \ref{subsubsec:crystalline_prop} that we have an isomorphism of rings
\begin{align*}
	f_{\crys} : \mbfa_{\crys}(\overline{R})\langle T \rangle^{\wedge} &\isomorphic \pazo \mbfa_{\crys}(\overline{R})\\
		T_i &\longmapsto X_i \otimes 1 - 1 \otimes [X_i^{\flat}], \hspace{2mm} \textrm{for} \hspace{1mm} 1 \leq i \leq d.
\end{align*}
Now recall that $\mbfa_{R,\varpi}^{\textpd}$ is the $\padic$ completion of the divided power envelope of the surjective map $\theta : \mbfa_{R, \varpi}^+ \twoheadrightarrow R[\varpi]$ with respect to its kernel (see \S \ref{subsec:pd_envelope}).
Next, let $\mbfa_{R,\varpi}^{\textpd}\langle T \rangle^{\wedge}$ denote the $\padic$ completion of the divided power polynomial algebra $\mbfa_{R,\varpi}^{\textpd}\langle T \rangle = \mbfa_{R,\varpi}^{\textpd}[T_i^{[n]}, \hspace{1mm} n \in \NN, \hspace{1mm} 1 \leq i \leq d]$.
Then via the isomorphism $f^{\textpd}$ (see Lemma \ref{lem:fpd_iso} below), we will show that the preimage of $\pazo \mbfa_{R,\varpi}^{\textpd}$, under $f_{\crys}$ is exactly $\mbfa_{R,\varpi}^{\textpd}\langle T \rangle^{\wedge}$.
In other words,
\begin{lem}\label{lem:fpd_iso}
	The morphism of rings
	\begin{align*}
		f^{\textpd} : \mbfa_{R,\varpi}^{\textpd} \langle T \rangle^{\wedge} &\longrightarrow \pazo \mbfa_{R,\varpi}^{\textpd}\\
		T_i &\longmapsto X_i \otimes 1 - 1 \otimes [X_i^{\flat}], \hspace{2mm} \textrm{for} \hspace{1mm} 1 \leq i \leq d,
	\end{align*}
	is an isomorphism.
\end{lem}
\begin{proof}
	The proof follows \cite[Proposition 6.1.5]{brinon-padicrep-relatif} closely.

	Recall that we have a surjective ring homomorphism $\theta : \mbfa_{R,\varpi}^{\textpd} \twoheadrightarrow R[\varpi]$, which is the restriction of the map $\theta : \mbfa_{\crys}(\overline{R}) \twoheadrightarrow \CC^+(\overline{R})$ defined in \S \ref{subsec:relative_crystalline_period_rings}.
	This can be extended in a unique manner into the homomorphism $\theta : \mbfa_{\crys}(\overline{R})\langle T \rangle^{\wedge} \twoheadrightarrow \CC^+(\overline{R})$.
	Restriction of the latter map gives us $\theta : \mbfa_{R,\varpi}^{\textpd}\langle T \rangle^{\wedge} \twoheadrightarrow R[\varpi]$ such that $\theta(T_i^{[n]}) = 0$ for $1 \leq i \leq d$ and $n \geq 1$.

	First, we will show that the $O_F\{X^{\pm 1}\}\textrm{-algebra}$ structure on $\mbfa_{R,\varpi}^{\textpd}\langle T \rangle^{\wedge}$ given by $X_i \mapsto [X_i^{\flat}] + T_i$, extends uniquely to an $R\textrm{-algebra}$ structure.	
	Let $\paza := (\mbfe_{R,\varpi}^+ / \overline{\pi}^{p-1} \mbfe_{R,\varpi}^+)[T_1, \ldots, T_d] / (T_1^p, \ldots, T_d^p)$.
	We have a surjective map $\theta : \mbfa_{R, \varpi}^+ \twoheadrightarrow R[\varpi]$ and its reduction modulo $p$ is given as $\overline{\theta} : \mbfe_{R,\varpi}^+ \twoheadrightarrow R[\varpi] / pR[\varpi]$.
	Since $\xi^p \equiv \overline{\pi}^{p-1} \mod p$, where $\xi = \frac{\pi}{\pi_1}$ is a generator of $\kert \theta \subset \mbfa_{R, \varpi}^+$, we obtain that $\overline{\theta}$ factors as $\overline{\theta} : \mbfe_{R,\varpi}^+ / \overline{\pi}^{p-1} \mbfe_{R,\varpi}^+ \twoheadrightarrow R[\varpi] / pR[\varpi]$.
	This can be extended to a map $\overline{\theta} : \paza \twoheadrightarrow R[\varpi] / pR[\varpi]$ by setting $\overline{\theta}(T_i) = 0$ for $1 \leq i \leq d$.
	The kernel $\pazi = \kert \overline{\theta} \subset \paza$ is generated by $\xi \equiv \overline{\pi}_1^{p-1} \mod p$ and $\{T_i\}_{1 \leq i \leq d}$.
	Now from the natural inclusion $R/pR \rightarrowtail R[\varpi] / pR[\varpi]$ and the isomorphism $\paza/\pazi \isomorphic R[\varpi] / pR[\varpi]$ via $\overline{\theta}$, we obtain a map $\overline{g} : R/pR \rightarrow \paza/\pazi$ such that $\overline{g}(X_i) = X_i$, which is the image of $X_i^{\flat} \in \paza$ under the map $\overline{\theta}$.
	So we obtain a commutative diagram
	\begin{center}
		\begin{tikzcd}[row sep=large, column sep=large]
			\kappa[X^{\pm 1}] \arrow[d] \arrow[r] & \paza \arrow[d, twoheadrightarrow]\\
			R/pR \arrow[r, rightarrowtail] \arrow[ru, dashed, "\overline{g}"] & \paza/\pazi
		\end{tikzcd}
	\end{center}
	where the top horizontal arrow is the map $X_i \mapsto X_i^{\flat} + T_i$.
	Note that $\pazi^{(d+1)p} = 0$.
	Since $R/pR$ is \'etale over $\kappa[X^{\pm 1}]$, there exists a unique lift of $\overline{g} : R/pR \rightarrow \paza/\pazi$ to a homomorphism $\overline{g} : R/pR \rightarrow \paza$ (which we again denote by $\overline{g}$ by slight abuse of notations).

	Further, by the description of divided power envelope in \cite[Proposition 6.1.1]{brinon-padicrep-relatif} we have that
	\begin{align*}\label{eq:arpd_alt_desc}
		\begin{split}
			\mbfa_{R, \varpi}^+[Y_0, Y_1, \ldots] / (pY_0 - \xi^p, pY_{n+1} - Y_n^p)_{n \geq 1} &\isomorphic \mbfa_{R,\varpi}^{\textpd}\\
					Y_n &\longmapsto \tfrac{\xi^{p^{n+1}}}{p^{n+1}}.
		\end{split}
	\end{align*}
	Therefore,
	\begin{equation*}
		(\mbfe_{R,\varpi}^+ / \overline{\pi}^{p-1} \mbfe_{R,\varpi}^+)[Y_0, Y_1, \ldots] / (Y_n^p)_{n \geq 1} \isomorphic \mbfa_{R,\varpi}^{\textpd} / p \mbfa_{R,\varpi}^{\textpd},
	\end{equation*}
	Similarly, we have 
	\begin{equation*}\label{eq:arpd_al_desc}
		(\mbfa_{R,\varpi}^{\textpd}[T_1, \ldots, T_d])[T_{i,0}, T_{i,1}, \ldots] / (pT_{i,0} - T_i^p, pT_{i, n+1} - T_{i, n}^p)_{1 \leq i \leq d, \hspace{0.5mm} n \in \NN} \isomorphic \mbfa_{R,\varpi}^{\textpd}\langle T \rangle.
	\end{equation*}
	Therefore,
	\begin{equation*}
		(\mbfa_{R,\varpi}^{\textpd} / p\mbfa_{R,\varpi}^{\textpd})[T_1, \ldots, T_d][T_{i, 0}, T_{i, 1}, \ldots] / (T_i^p, T_{i, n}^p)_{1 \leq i \leq d, \hspace{0.5mm} n \in \NN} \isomorphic \mbfa_{R,\varpi}^{\textpd}\langle T \rangle / p\mbfa_{R,\varpi}^{\textpd}\langle T \rangle.
	\end{equation*}
	In conclusion, we have
	\begin{equation*}
		\paza [Y_0, Y_1, \ldots, T_{i, 0}, T_{i, 1}, \ldots] / (Y_n^p, T_{i, n}^p)_{1 \leq i \leq d, \hspace{0.5mm} n \in \NN} \isomorphic \mbfa_{R,\varpi}^{\textpd}\langle T \rangle / p\mbfa_{R,\varpi}^{\textpd}\langle T \rangle.
	\end{equation*}
	Therefore, from the discussion above we obtain a natural map of $\kappa[X^{\pm 1}]\textrm{-algebras}$ by composition $\overline{g}_1 : R/pR \rightarrow \paza \rightarrow \mbfa_{R,\varpi}^{\textpd}\langle T \rangle / p \mbfa_{R,\varpi}^{\textpd}\langle T \rangle$.

	Now let $n \in \NN$, then modulo $p^n$ we have the natural map $O_F\{X^{\pm 1}\} / p^n O_F\{X^{\pm 1}\} \rightarrow \mbfa_{R,\varpi}^{\textpd}\langle T \rangle / p^n \mbfa_{R,\varpi}^{\textpd}\langle T \rangle$.
	Again, since $R / p^n R$ is \'etale over $O_F\{X^{\pm 1}\} / p^n O_F\{X^{\pm 1}\}$, we have a unique lift of $\overline{g}_n : R/p^nR \rightarrow \mbfa_{R,\varpi}^{\textpd}\langle T \rangle / p^n \mbfa_{R,\varpi}^{\textpd}\langle T \rangle$ in the commutative diagram
	\begin{center}
		\begin{tikzcd}[row sep=large, column sep=large]
			O_F\{X^{\pm 1}\} / p^n O_F\{X^{\pm 1}\} \arrow[d] \arrow[r] & \mbfa_{R,\varpi}^{\textpd}\langle T \rangle / p^n \mbfa_{R,\varpi}^{\textpd}\langle T \rangle \arrow[d, twoheadrightarrow]\\
			R / p^n R \arrow[r] \arrow[ru, dashed, "\overline{g}_n"] & \mbfa_{R,\varpi}^{\textpd}\langle T \rangle / p \mbfa_{R,\varpi}^{\textpd}\langle T \rangle.
		\end{tikzcd}
	\end{center}
	Via this lifting, the following diagram commutes
	\begin{center}
		\begin{tikzcd}[row sep=large, column sep=large]
			R/p^{n+1}R \arrow[d] \arrow[r] & \mbfa_{R,\varpi}^{\textpd}\langle T \rangle / p^{n+1} \mbfa_{R,\varpi}^{\textpd}\langle T \rangle \arrow[d]\\
			R/p^nR \arrow[r] & \mbfa_{R,\varpi}^{\textpd}\langle T \rangle / p^n \mbfa_{R,\varpi}^{\textpd}\langle T \rangle,
		\end{tikzcd}
	\end{center}
	where the vertical arrows are natural projection maps.
	From the universal property of inverse limit of the right side of the diagram, we obtain a natural map of $O_F\{X^{\pm 1}\}\textrm{-algebras}$
	\begin{equation*}
		g : R \longrightarrow \lim_n \mbfa_{R,\varpi}^{\textpd}\langle T \rangle / p^n \mbfa_{R,\varpi}^{\textpd}\langle T \rangle = \mbfa_{R,\varpi}^{\textpd} \langle T \rangle^{\wedge}.
	\end{equation*}
	
	Now, let $\overline{\theta} : \mbfa_{R,\varpi}^{\textpd}\langle T \rangle / p\mbfa_{R,\varpi}^{\textpd}\langle T \rangle \rightarrow R[\varpi] / pR[\varpi]$ denote the reduction of $\theta$ modulo $p$.
	Recall that by construction, $\overline{\theta} \circ \overline{g}$ is the inclusion of $R/pR$ in $R[\varpi] / pR[\varpi]$.
	Therefore, the reduction modulo $p$ of $\theta \circ g$ and the natural inclusion $R \rightarrowtail R[\varpi]$ coincide.
	Since $R$ is $p\textrm{-torsion}$ free, arguing as above we obtain that for each $n \in \NN$, the natural inclusion and $\theta \circ g$ coincide modulo $p^n$.

	Next, by $\mbfa_{R, \varpi}^+\textrm{-linearity}$, $g$ can be extended to a map $g : R \otimes_{O_F} \mbfa_{R, \varpi}^+ \rightarrow \mbfa_{R,\varpi}^{\textpd} \langle T \rangle^{\wedge}$.
	From the discussion above and the definition of $\theta_{R}$, we have that $\theta_R$ coincides with the homomorphism $\theta \circ g : R \otimes_{O_F} \mbfa_{R, \varpi}^+ \rightarrow R[\varpi]$.
	In particular, $g(\kert \theta_{R}) \subset \kert \theta \subset \mbfa_{R,\varpi}^{\textpd} \langle T \rangle ^{\wedge}$.
	Since $\kert \theta$ contains divided powers, the map $g$ extends to a map 
	\begin{equation*}
		g : (R \otimes_{O_F} \mbfa_{R, \varpi}^+)[x^{[n]}, x \in \kert \theta_{R}, n \in \NN] \longrightarrow \mbfa_{R,\varpi}^{\textpd}\langle T \rangle^{\wedge}.
	\end{equation*}
	Finally, since $\mbfa_{R,\varpi}^{\textpd}\langle T \rangle^{\wedge}$ is $\padic$ally complete, $g$ extends to a map $g : \pazo \mbfa_{R,\varpi}^{\textpd} \rightarrow \mbfa_{R,\varpi}^{\textpd}\langle T \rangle^{\wedge}$.

	Now by uniqueness of $g : R \rightarrow \mbfa_{R,\varpi}^{\textpd}\langle T \rangle^{\wedge}$, the composition
	\begin{equation*}
		\pazo \mbfa_{R,\varpi}^{\textpd} \xrightarrow{\hspace{1mm} g \hspace{1mm}} \mbfa_{R,\varpi}^{\textpd}\langle T \rangle^{\wedge} \xrightarrow{\hspace{1mm} f^{\textpd} \hspace{1mm}} \pazo \mbfa_{R,\varpi}^{\textpd},
	\end{equation*}
	coincides with the identity over $R \subset \pazo \mbfa_{R,\varpi}^{\textpd}$.
	Since it also coincides with identity on the image of $\mbfa_{R, \varpi}^+$ (by $\mbfa_{R, \varpi}^+\textrm{-linearity}$), we obtain that $f^{\textpd} \circ g = \textup{id}$ over $\pazo \mbfa_{R,\varpi}^{\textpd}$.
	Similarly, the homomorphism $g \circ f^{\textpd}$ coincides with identity over $\mbfa_{R, \varpi}^+$ as well as over $O_F\{X^{\pm 1}\}$ (since $g$ lifts the map $O_F\{X^{\pm 1}\} \rightarrow \mbfa_{R,\varpi}^{\textpd}\langle T \rangle^{\wedge}$), therefore it is identity over $\mbfa_{R,\varpi}^{\textpd}\langle T \rangle^{\wedge}$.
	This establishes that $f^{\textpd}$ is an isomorphism of rings.
\end{proof}


\begin{rem}\label{rem:fat_pd_ring_struct}
	We can give an alternative construction of the ring $\pazo \mbfa_{R,\varpi}^{\textpd}$.
	Note that we have a ring homomorphism $\iota : R \rightarrow \mbfa_{R,\varpi}^{\textpd}$, where $X_i \mapsto [X_i^{\flat}]$ for $1 \leq i \leq d$.
	As in Definition \ref{defi:fat_ring_const}, we define a map $g : R \otimes_{\ZZ} \mbfa_{R,\varpi}^{\textpd} \rightarrow \mbfa_{R,\varpi}^{\textpd}$, where $x \otimes y \mapsto \iota(x)y$.
	We obtain that $\kert g = \big(X_i \otimes 1 - 1 \otimes [X_i^{\flat}], \hspace{1mm} \textrm{for} \hspace{1mm} 1 \leq i \leq d\big) \subset \kert \theta_{R} \subset \pazo \mbfa_{\crys}(\overline{R})$.
	Since $R \otimes_{\ZZ} \mbfa_{R,\varpi}^{\textpd}$ already contains divided powers of $\xi$, from Definition \ref{defi:oarpd} we obtain that the $\padic$ completion of the divided power envelope of $R \otimes_{\ZZ} \mbfa_{R,\varpi}^{\textpd}$ with respect to $\kert g$ is the same as $\pazo \mbfa_{R,\varpi}^{\textpd}$.
\end{rem}

There is a natural filtration over the ring $\pazo \mbfa_{R,\varpi}^{\textpd}$ by $\Gamma_{R}\textrm{-stable}$ submodules:
\begin{defi}\label{defi:fil_oarpd}
	Let $U_i := \frac{1 \otimes [X_i^{\flat}]}{X_i \otimes 1}$ for $1 \leq i \leq d$ and $r \in \ZZ$, define the filtration over $\pazo \mbfa_{R,\varpi}^{\textpd}$ as
	\begin{equation*}
		\Fil^r \pazo \mbfa_{R,\varpi}^{\textpd} := \Big\langle(a \otimes b) \prod_{i=1}^d(U_i-1)^{[k_i]} \in \pazo \mbfa_{R,\varpi}^{\textpd}, \hspace{1mm} \textrm{such that} \hspace{1mm} a \in R, b \in \Fil^j \mbfa_{R,\varpi}^{\textpd}, \hspace{1mm} \textrm{and} \hspace{1mm} j + \sum_i k_i \geq r\Big\rangle.
	\end{equation*}
\end{defi}

\begin{rem}\label{rem:oarpd_oacris_fil_comp}
	The filtration over $\mbfa_{R,\varpi}^{\textpd}$ (via its identification with $R_{\varpi}^{\textpd}$, see \S \ref{subsec:cyclotomic_embeddings} and Definition \ref{defi:filtration_vanishing_varpi}) coincides with the filtration induced from its embedding in $\mbfa_{\crys}(\overline{R})$.
	Indeed, in both cases we have $\Fil^r \mbfa_{R,\varpi}^{\textpd} = \big(\xi^{[k]}, \hspace{1mm} k \leq r\big) \subset \mbfa_{R,\varpi}^{\textpd}$ for $r \geq 0$, whereas $\Fil^r \mbfa_{R,\varpi}^{\textpd} = \mbfa_{R,\varpi}^{\textpd}$ for $r < 0$.
	Next, the filtration on $\pazo \mbfa_{\crys}(\overline{R})$ is defined as the induced filtration from its embedding inside $\pazo \mbfb_{\dR}^+(\overline{R})$ and the filtration on the latter ring is given by powers of $\kert \theta_R$ (see \S \ref{subsec:rel_deRham_ring} \& \ref{subsec:relative_crystalline_period_rings} for definition and notation).
	The induced filtration over $\pazo \mbfa_{\crys}(\overline{R})$ is therefore given by divided powers of the ideal $\kert \theta_{R} \subset \pazo \mbfa_{\crys}(\overline{R})$.
	Since the filtration over $\pazo \mbfa_{R,\varpi}^{\textpd}$ in Definition \ref{defi:fil_oarpd} is again given by divided powers of the ideal $\kert \theta_{R} \subset \pazo \mbfa_{R,\varpi}^{\textpd}$, we infer that this filtration coincides with the one induced by its embedding into $\pazo \mbfa_{\crys}(\overline{R})$.
\end{rem}

\begin{lem}\phantomsection\label{lem:gamma_inv_oarpd}
	\begin{enumromanup}
	\item The action of $\Gamma_{R, \varpi}$ is trivial on $\pazo \mbfa_{R,\varpi}^{\textpd} / \pi$, whereas $\Gamma_R / \Gamma_{R, \varpi}$ acts trivially over $\pazo \mbfa_{R,\varpi}^{\textpd} / \pi_m$.

	\item We have $(\pazo \mbfa_{R,\varpi}^{\textpd})^{\Gamma_{R}} = R$ and $(\Fil^1 \pazo \mbfa_{R}^{\textpd})^{\Gamma_{R}} = 0$.
	\end{enumromanup}
\end{lem}
\begin{proof}
	\begin{enumromanup}
	\item The first part follows from the definition of $\pazo \mbfa_{R,\varpi}^{\textpd}$ and the action of $\Gamma_{R, \varpi}$ on $\mbfa_{R,\varpi}^{\textpd}$ (see Lemma \ref{lem:gamma_minus_1_pd}).
		The second part follows from observing that $\Gamma_R / \Gamma_{R, \varpi} = \Gamma_F / \Gamma_K$ is a finite cyclic group of order $[K:F] = p^{m-1}(p-1)$, and a lift $g \in \Gamma_{R}$ of a generator of $\Gamma_R / \Gamma_{R, \varpi}$ acts as $g(\pi_m) = (1+\pi_m)^{\chi(g)} - 1$.

	\item This is straightforward, since $R \subset \big(\pazo \mbfa_{R,\varpi}^{\textpd}\big)^{\Gamma_{R}} \subset \big(\pazo \mbfa_{\crys}(\overline{R})\big)^{G_{R}} = R$ and $(\Fil^1 \pazo \mbfa_{R}^{\textpd})^{\Gamma_{R}} \subset (\Fil^1 \pazo \mbfb_{\crys}(\overline{R}))^{G_{R}} \subset (\Fil^1 \pazo \mbfb_{\dR}(\overline{R}))^{G_{R}} = 0$ (for last equality see the proof of \cite[Proposition 5.2.12]{brinon-padicrep-relatif}).
	\end{enumromanup}
\end{proof}

Next we consider a connection over $\pazo \mbfa_{R,\varpi}^{\textpd}$ induced by the connection on $\pazo \mbfa_{\crys}(\overline{R})$,
\begin{equation*}
	\partial : \pazo \mbfa_{R,\varpi}^{\textpd} \longrightarrow \pazo \mbfa_{R,\varpi}^{\textpd} \otimes \Omega^1_{R},
\end{equation*}
where we have $\partial\big(X_i \otimes 1 - 1 \otimes [X_i^{\flat}]\big)^{[n]} = \big(X_i \otimes 1 - 1 \otimes [X_i^{\flat}]\big)^{[n-1]} \hspace{1mm} dX_i$.
This connection over $\pazo \mbfa_{R,\varpi}^{\textpd}$ satisfies Griffiths transversality with respect to the filtration since it does so over $\pazo \mbfa_{\crys}(\overline{R})$.

\subsubsection{Main result}

\begin{thm}\label{thm:crys_wach_comparison}
	Let $V$ be a positive finite $q\textrm{-height}$ representation of $G_{R}$, then
	\begin{enumromanup}
	\item $V$ is a positive crystalline representation.

	\item Let $M := \big(\pazo\mbfa_{R,\varpi}^{\textpd} \otimes_{\mbfa_R^+} \mbfn(T)\big)^{\Gamma_{R}}$, then after extending scalars to $\pazo \mbfa_{R,\varpi}^{\textpd}$ and inverting $p$, we obtain a natural isomorphism
		\begin{equation*}
			\pazo \mbfa_{R,\varpi}^{\textpd} \otimes_{R} M\big[\tfrac{1}{p}\big] \isomorphic \pazo \mbfa_{R,\varpi}^{\textpd} \otimes_{\mbfa_R^+} \mbfn(V),
		\end{equation*}
		compatible with Frobenius, filtration, connection and the action of $\Gamma_{R}$ on each side.

	\item We have an isomorphism of $R\big[\frac{1}{p}\big]\textrm{-modules}$ 
		\begin{equation*}
			\pazo \mbfd_{\crys}(V) \lisomorphic \big(\pazo \mbfa_{R,\varpi}^{\textpd} \otimes_{\mbfa_R^+} \mbfn(T)\big)^{\Gamma_R}\big[\tfrac{1}{p}\big],
		\end{equation*}
		compatible with Frobenius, filtration, and connection on each side.
		Therefore, we obtain a comparison isomorphism
		\begin{equation*}
			\pazo \mbfa_{R,\varpi}^{\textpd} \otimes_{\mbfa_R^+} \mbfn(V) \isomorphic \pazo \mbfa_{R,\varpi}^{\textpd} \otimes_{R} \pazo \mbfd_{\crys}(V),
		\end{equation*}
		compatible with Frobenius, filtration, connection and the action of $\Gamma_{R}$ on each side.
	\end{enumromanup}
\end{thm}

\begin{rem}
	The statement of Theorem \ref{thm:crys_wach_comparison} can be seen an analogue of the result of Berger \cite[Proposition II.2.1]{berger-limites-cristallines} (see the discussion after Proposition \ref{prop:wach_modules_equivalent_categories}).
\end{rem}

Recall that from Definition \ref{defi:wach_reps} any finite $q\textrm{-height}$ representation is a twist of a positive finite $q\textrm{-height}$ representation by $\QQ_p(r)$, for $r \in \NN$.
Since twist by $\QQ_p(r)$ of crystalline representations are again crystalline, we obtain that:
\begin{cor}
	All finite $q\textrm{-height}$ representations of $G_{R}$ are crystalline.
\end{cor}

The proof of Theorem \ref{thm:crys_wach_comparison} will proceed in two steps: 
First, we will describe a process by which we can recover a submodule of $\pazo \mbfd_{\crys}(V)$ starting from the Wach module (see Proposition \ref{prop:crys_from_wach_mod}), here we establish the comparison displayed in (ii).
Next, the remaining claims made in the theorem are shown by exploiting some properties of Wach modules and the comparison obtained in the first step.

In \S \ref{subsec:onedim_reps}, we will explicitly state the structure of Wach module attached to a one-dimensional finite $q\textrm{-height}$ representation and we will also show that all one-dimensional crystalline representations are of finite $q\textrm{-height}$ and one can recover $\pazo \mbfd_{\crys}(V)$ starting with the Wach module $\mbfn(V)$.
Combining this with the theorem above, we will obtain that the notion of crystalline representations and finite $q\textrm{-height}$ representations coincide in dimension 1.

\subsection{From \texorpdfstring{$(\varphi, \Gamma)$}{-}-modules to \texorpdfstring{$(\varphi, \partial)$}{-}-modules}

The objective of this section is to prove the following:
\begin{prop}\label{prop:crys_from_wach_mod}
	Let $V$ be an $h\textrm{-dimensional}$ positive finite $q\textrm{-height}$ representation of $G_{R}$, $T \subset V$ a $\ZZ_p\textrm{-lattice}$ of rank $h$ stable under the action of $G_{R}$ and $\mbfn(T)$ the associated Wach module.
	Then 
	\begin{enumromanup}
	\item $M := \big(\pazo \mbfa_{R,\varpi}^{\textpd} \otimes_{\mbfa_R^+} \mbfn(T)\big)^{\Gamma_{R}}$ is a finitely generated $R\textrm{-module}$ contained in $\pazo \mbfd_{\crys}(V)$.

	\item $M\big[\frac{1}{p}\big]$ is a finitely generated projective $R\big[\frac{1}{p}\big]\textrm{-module}$ of rank $h$ and the natural inclusion
		\begin{equation*}
			\pazo \mbfa_{R,\varpi}^{\textpd} \otimes_{R} M\big[\tfrac{1}{p}\big] \longrightarrow \pazo \mbfa_{R,\varpi}^{\textpd} \otimes_{\mbfa_R^+} \mbfn(V),
		\end{equation*}
		is an isomorphism compatible with Frobenius, filtration, connection and the action of $\Gamma_{R}$.

	\item If $\mbfn(T)$ is free over $\mbfa_R^+$ then there exists a free $R\textrm{-module}$ $M_0 \subset M$ such that $M_0\big[\frac{1}{p}\big] = M\big[\frac{1}{p}\big]$ are free modules of rank $h$ over $R\big[\frac{1}{p}\big]$.
	\end{enumromanup}
\end{prop}
\begin{proof}
	We will use the notation of Definition \ref{defi:wach_reps} without repeating them.
	The first claim is easy to establish.
	Since we have $H_{R} = \Gal\big(\overline{R}\big[\frac{1}{p}\big] / R_{\infty}\big[\frac{1}{p}\big]\big)$, therefore we can write 
	\begin{align}\label{eq:wach_invar_in_crys}
		\begin{split}
			M &= \big(\pazo \mbfa_{R,\varpi}^{\textpd} \otimes_{\mbfa_R^+} \mbfn(T)\big)^{\Gamma_{R}} \subset \big(\pazo \mbfa_{R,\varpi}^{\textpd} \otimes_{\mbfa_R^+} \mbfd^+(T)\big)^{\Gamma_{R}} \subset \big(\pazo \mbfa_{\crys}(\overline{R})^{H_{R}} \otimes_{\mbfa_R^+} \mbfd^+(T) \big)^{\Gamma_{R}} \\
			&\subset \big(\pazo\mbfa_{\crys}(\overline{R})^{H_{R}} \otimes_{\mbfa_R^+} \big(\mbfa^+ \otimes_{\ZZ_p} T\big)^{H_{R}}\big)^{\Gamma_{R}} \subset \big(\pazo\mbfa_{\crys}(\overline{R}) \otimes_{\ZZ_p} T\big)^{G_{R}} \subset \pazo \mbfd_{\crys}(V).
		\end{split}
	\end{align}
	The module $\big(\pazo \mbfa_{\crys}(\overline{R}) \otimes_{\ZZ_p} T\big)^{G_{R}}$ is finitely generated over $R$.
	Since $R$ is Noetherian, $M$ is finitely generated.

	Independently, we have that $R\big[\frac{1}{p}\big]$ is Noetherian and $\pazo \mbfd_{\crys}(V)$ is a finitely generated $R\big[\frac{1}{p}\big]\textrm{-module}$, therefore $M\big[\frac{1}{p}\big] \subset \pazo \mbfd_{\crys}(V)$ is finitely generated over $R\big[\frac{1}{p}\big]$.
	Moreover, the module $\pazo \mbfa_{R,\varpi}^{\textpd} \otimes_{\mbfa_R^+} \mbfn(T)$ is equipped with an $\mbfa_{R,\varpi}^{\textpd}\textrm{-linear}$ and integrable connection $\partial_N = \partial \otimes 1$, where $\partial$ is the connection on $\pazo \mbfa_{R,\varpi}^{\textpd}$ described after Lemma \ref{lem:gamma_inv_oarpd}.
	Therefore, we can consider the induced connection on $M\big[\frac{1}{p}\big]$, which is integrable since it is integrable over $\pazo \mbfa_{R,\varpi}^{\textpd} \otimes_{\mbfa_R^+} \mbfn(T)$.
	This connection is compatible with the one on $\pazo \mbfd_{\crys}(V)$ since the connection over $\pazo \mbfa_{R,\varpi}^{\textpd}$ is induced from the connection over $\pazo \mbfa_{\crys}(\overline{R})$.
	So by \cite[Proposition 7.1.2]{brinon-padicrep-relatif} we obtain that $M\big[\frac{1}{p}\big]$ must be projective of rank $\leq h$.
	Furthermore, the inclusion $M\big[\frac{1}{p}\big] \subset \pazo \mbfd_{\crys}(V)$ is compatible with natural Frobenius on each module since all the inclusions in \eqref{eq:wach_invar_in_crys} are compatible with Frobenius.

	Next, we will show that the rank of $M\big[\frac{1}{p}\big]$ as a projective $R\big[\frac{1}{p}\big]\textrm{-module}$ is exactly $h$.
	But first let us prove that it is enough to show that the rank is $h$ after a finite \'etale extension of $R$.
	Let us consider $R\prm$ to be a finite \'etale extension of $R$ such that the corresponding scalar extension $\mbfa_{R\prm}^+ \otimes_{\mbfa_R^+} \mbfn(T)$ is a free module of rank $h$ (see Definition \ref{defi:wach_reps}) and $R\prm\big[\frac{1}{p}\big] / R\big[\frac{1}{p}\big]$ is Galois.
	The discussion of previous chapters hold for $R\prm$ (see \cite[Chapitre 2]{brinon-padicrep-relatif} and \cite[\S 2]{andreatta-iovita-relative-phiGamma} for more on this).
	In particular, for $R\prm[\varpi]$ we have rings $\mbfa_{R\prm}^+$, $\mbfa_{R\prm, \varpi}^+$, $\mbfa_{R\prm,\varpi}^{\textpd}$ and $\pazo \mbfa_{R\prm,\varpi}^{\textpd}$.
	Let $R_{\infty}\prm\big[\frac{1}{p}\big]$ denote the cyclotomic tower over $R\prm\big[\frac{1}{p}\big]$ and
	\begin{equation*}
		\Gamma_{R\prm} = \Gal\big(R_{\infty}\prm\big[\tfrac{1}{p}\big] / R\prm\big[\tfrac{1}{p}\big]\big) \hspace{2mm} \textrm{and} \hspace{2mm} H_{R\prm} = \kert(G_{R\prm} \rightarrow \Gamma_{R\prm}).
	\end{equation*}
	Similarly, we have Galois groups $\Gamma_{R\prm}$ and $H_{R\prm}$.
	Let 
	\begin{equation*}
		G\prm := \Gal\big(R_{\infty}\prm\big[\tfrac{1}{p}\big] / R_{\infty}\big[\tfrac{1}{p}\big]\big) = \Gal\big(R\prm[\varpi]\big[\tfrac{1}{p}\big] / R[\varpi]\big[\tfrac{1}{p}\big]\big) = \Gal\big(R\prm\big[\tfrac{1}{p}\big] / R\big[\tfrac{1}{p}\big]\big),
	\end{equation*}
	then we have that $H_{R, \varpi} / H_{R\prm, \varpi} = H_{R} / H_{R\prm} = G\prm$.
	So we obtain that 
	\begin{equation*}
		\mbfa_R^+ = (\mbfa^+)^{H_{R}} = \big((\mbfa^+)^{H_{R\prm}}\big)^{H_{R} / H_{R\prm}} = \big(\mbfa_{R\prm}^+\big)^{G\prm}.
	\end{equation*}
	Moreover, for the base ring $R[\varpi]$ (instead of $R$) one can consider the ring $\mbfa_{\varpi}^+$ as in Remark \ref{rem:a_varpi}.
	Then we have
	\begin{equation*}
		\mbfa_{R, \varpi}^+ = (\mbfa_{\varpi}^+)^{H_{R, \varpi}} = \big((\mbfa_{\varpi}^+)^{H_{R\prm, \varpi}}\big)^{H_{R, \varpi} / H_{R\prm, \varpi}} = \big(\mbfa_{R\prm,\varpi}^+\big)^{G\prm}.
	\end{equation*}
	From these equalities and the description of the action of $\Gamma_{R}$ on $\xi = \frac{\pi}{\pi_1}$, it is clear that
	\begin{equation*}
		\mbfa_{R,\varpi}^{\textpd} = \big(\mbfa_{R\prm,\varpi}^{\textpd}\big)^{G\prm}, \hspace{2mm} \textrm{and therefore} \hspace{2mm} \pazo \mbfa_{R,\varpi}^{\textpd} = \big(\pazo \mbfa_{R\prm,\varpi}^{\textpd}\big)^{G\prm}.
	\end{equation*}
	Now, since $\mbfn(T)$ is projective and $G\prm$ acts trivially on it, we obtain that
	\begin{align*}
		\big(\pazo \mbfa_{R\prm,\varpi}^{\textpd} \otimes_{\mbfa_{R\prm}^+} \big(\mbfa_{R\prm}^+ \otimes_{\mbfa_R^+} \mbfn(T)\big)\big)^{G\prm} &= \pazo \mbfa_{R,\varpi}^{\textpd} \otimes_{\mbfa_R^+} \mbfn(T)\\
		\big(\pazo \mbfa_{R\prm,\varpi}^{\textpd} \otimes_{R\prm} \big(R\prm \otimes_{R} M\big[\tfrac{1}{p}\big]\big)\big)^{G\prm} &= \pazo \mbfa_{R,\varpi}^{\textpd} \otimes_{R} M\big[\tfrac{1}{p}\big].
	\end{align*}
	In particular, base changing to $\mbfa_{R\prm}^+$ to obtain $\mbfn(T)$ as a free module is harmless.
	For the convenience in notation, below we will replace $R\prm$ obtained in this manner by $R$ and assume $\mbfn(T)$ to be free over $\mbfa_R^+$.

	In order to show that the rank of $M\big[\frac{1}{p}\big]$ is at least $h$, we will find $\Gamma_R\textrm{-fixed}$ elements of $\pazo \mbfa_{R,\varpi}^{\textpd} \otimes_{\mbfa_R^+} \mbfn(T)$ corresponding to a basis of $\mbfn(T)$, which are linearly independent elements of $M\big[\frac{1}{p}\big]$.
	To carry this out, first we will define several new rings following \cite[\S B.1]{wach-pot-crys} and examine their relation with $\pazo \mbfa_{R,\varpi}^{\textpd}$.
	After extending scalars of $\mbfn(T)$, we will define differential operators on the obtained module, corresponding to the topological generators of $\Gamma_R$.
	Next, for any element of $\mbfn(T)$, we will write down a corresponding element killed by the differential operators which we will show is fixed by $\Gamma_R$.

\begin{rem}
	Note that the $\Gamma_R\textrm{-fixed}$ elements of $\pazo \mbfa_{R,\varpi}^{\textpd} \otimes_{\mbfa_R^+} \mbfn(T)$ can be obtained by successive approximation as well.
	This computation was carried out by the author in his thesis (see \cite[\S 3.2.3]{abhinandan-thesis}).
\end{rem}

\subsubsection{Auxiliary rings and modules}

	For $n \in \NN$, let us define a $p\textrm{-adically}$ complete ring
	\begin{equation*}
		S_n^{\textpd} := \mbfa_R^+\big\{\tfrac{\pi}{p^n}, \tfrac{\pi^2}{2!p^{2n}}, \ldots, \tfrac{\pi^k}{k!p^{kn}}, \ldots\big\}.
	\end{equation*}
	Let $I_n^{[i]}$ denote the ideal of $S_n^{\textpd}$ generated by $\frac{\pi^k}{k!p^{kn}}$ for $k \geq i$ and we set
	\begin{equation}\label{eq:snpd}
		\widehat{S}_n^{\textpd} := \lim_i S_n^{\textpd}\big/I_n^{[i]}.
	\end{equation}
	Note that $\widehat{S}_n^{\textpd}$ is $p\textrm{-adically}$ complete as well.
	Further, note that we can write $\varphi(\pi) = (1+\pi)^p - 1 = \pi^p + p \pi x$ for some $x \in \mbfa_F^+$, therefore
	\begin{align*}
		\frac{\varphi(\pi^k)}{k!p^{kn}} &= \frac{(\pi^p + p \pi x)^k}{k!p^{kn}} = \frac{\sum_{i=0}^{k} \binom{k}{i} \pi^{pi} \cdot (p \pi x)^{k-i}}{k!p^{kn}}\\
		&= \sum_{i=0}^{k} \frac{(k+(p-1)i)! p^{i(n(p-1)-p)}}{i!(k-i)!} \cdot \frac{\pi^{k+(p-1)i}x^{k-i}}{(k+(p-1)i)!p^{(k+(p-1)i)(n-1)}} \in \widehat{S}_{n-1}^{\textpd}
	\end{align*}
	Using this, the Frobenius operator on $S$ can be extended to a map $\varphi : \widehat{S}_n^{\textpd} \rightarrow \widehat{S}_{n-1}^{\textpd}$, which we will again call Frobenius.
	The ring $\widehat{S}_n^{\textpd}$ readily admits a continuous action of $\Gamma_{R}$ which commutes with the Frobenius.

\begin{lem}
	The ring $\widehat{S}_0^{\textpd}$ is a subring of $\mbfa_{R,\varpi}^{\textpd}$, and therefore $\varphi^n\big(\widehat{S}_n^{\textpd}\big) \subset \mbfa_{R,\varpi}^{\textpd}$.
\end{lem}
\begin{proof}
	The first claim is true because we have
	\begin{equation*}
		\pi_1^p \equiv \pi \hspace{1mm} \textrm{mod} \hspace{1mm} p \mbfa_{F, \varpi}^+, \hspace{2mm} \textrm{which gives} \hspace{2mm} \pi_1^{p^{i}} \equiv \pi^{p^{i-1}} \hspace{1mm} \textrm{mod} \hspace{1mm} p^i \mbfa_{F, \varpi}^+.
	\end{equation*}
	So for $k \geq p^i$ we can write
	\begin{equation*}
		\frac{\pi^k}{k!} = \frac{\xi^k \pi_1^k}{k!} = \frac{\xi^k}{k!} \pi_1^{k-p^i}\big(\pi^{p^{i-1}} + p^ia\big) = p^i a \pi_1^{k-p^i} \frac{\xi^k}{k!} + p^{i-1} \pi_1^{p^{i-1}} \frac{(k+p^{i-1})!}{k!p^{i-1}} \frac{\xi^{k+p^{i-1}}}{(k+p^{i-1})!} \in p^{i-1} \mbfa_{F, \varpi}^{\textpd},
	\end{equation*}
	for some $a \in \mbfa_{F, \varpi}^+$.
	Therefore, we get that $I_0^{[p^i]} \subset p^{i-1} \mbfa_{R,\varpi}^{\textpd}$ and hence $\widehat{S}_0^{\textpd} \subset \mbfa_{R,\varpi}^{\textpd}$.
	The second claim is obvious.
\end{proof}
	
	In the relative setting, we need slightly larger rings.
	Let us consider the $O_F\textrm{-linear}$ homomorphism of rings
	\begin{align*}
		\iota : R &\longrightarrow \widehat{S}_n^{\textpd} \\
		X_j &\longmapsto [X_j^{\flat}] \hspace{2mm} \textrm{for} \hspace{1mm} 1 \leq j \leq d.
	\end{align*}
	Using $\iota$ we can define an $O_F\textrm{-linear}$ morphism of rings
	\begin{align*}
		f : R \otimes_{O_F} \widehat{S}_n^{\textpd} &\longrightarrow \widehat{S}_n^{\textpd}\\
			a \otimes b &\longmapsto \iota(a) b.
	\end{align*}
	Let $\pazo \widehat{S}_n^{\textpd}$ denote the $\padic$ completion of the divided power envelope of $R \otimes_{O_F} \widehat{S}_n^{\textpd}$ with respect to $\kert f$.
	Further, the morphism $f$ extends uniquely to a continuous morphism $f : \pazo \widehat{S}_n^{\textpd} \rightarrow \widehat{S}_n^{\textpd}$.
	Now, it easily follows from the discussion in \S \ref{subsec:fat_period_rings} that the kernel of the morphism $f$ is generated by divided powers of the ideal generated by $(1-V_1, \ldots, 1-V_d)$, where $V_j = \frac{X_j \otimes 1}{1 \otimes [X_j^{\flat}]}$ for $1 \leq j \leq d$.
	The Frobenius operator extends to $\pazo \widehat{S}_n^{\textpd}$ as well as the continuous action of $\Gamma_{R}$.
	From the discussion above we have $\varphi^n(\widehat{S}_n^{\textpd}) \subset \widehat{S}_0^{\textpd} \subset \mbfa_{R,\varpi}^{\textpd}$, and following the description of $\pazo \widehat{S}_0^{\textpd}$ in \S \ref{subsec:fat_period_rings} and of $\pazo \mbfa_{R,\varpi}^{\textpd}$ from Remark \ref{rem:fat_pd_ring_struct}, we obtain that
	\begin{equation*}
		\pazo \widehat{S}_0^{\textpd} \subset \pazo \mbfa_{R,\varpi}^{\textpd} \hspace{2mm} \textrm{and} \hspace{2mm} \varphi^n\big(\pazo \widehat{S}_n^{\textpd}\big) \subset \pazo \mbfa_{R,\varpi}^{\textpd}.
	\end{equation*}
	Moreover, we have a canonical inclusion of $\widehat{S}_n^{\textpd} \subset \pazo \widehat{S}_n^{\textpd}$ compatible with all the structures.
	
	Now let us take $n \in \NN_{\geq 1}$ and consider the ring $\pazo \widehat{S}_n^{\textpd}$ below.
	We set the ideal
	\begin{equation*}
		J := \big(\tfrac{\pi}{p^n}, 1-V_1, \ldots, 1-V_d\big) \subset \pazo \widehat{S}_n^{\textpd},
	\end{equation*}
	and its divided power
	\begin{equation*}
		J^{[i]} := \Big\langle \tfrac{\pi^{[k_0]}}{p^{nk_0}} \prod_{j=1}^d (1-V_j)^{[k_j]}, \hspace{1mm} \smbfk = (k_0, k_1, \ldots, k_d) \in \NN^{d+1} \hspace{1mm} \textrm{such that} \hspace{1mm} \sum_{j=0}^d k_j \geq i\Big\rangle \subset \pazo \widehat{S}_n^{\textpd}.
	\end{equation*}
	By construction of $\pazo \widehat{S}_n^{\textpd}$, it is clear that a summation $\sum_{i \in \NN} x_i a_i$ with $a_i \in J^{[i]}$ and $x_i \in \widehat{S}_n^{\textpd}$ goes to $0$ as $i \rightarrow +\infty$, converges in $\pazo \widehat{S}_n^{\textpd}$.
	Moreover, every $x \in \pazo \widehat{S}_n^{\textpd}$ has a presentation as $x = \sum_{\smbfk \in \NN^{d+1}} x_{\smbfk} \frac{\pi^{[k_0]}}{p^{nk_0}} \prod_{j=1}^d (1-V_j)^{[k_j]}$, where $x_{\smbfk} \in \mbfa_R^+$ goes to $0$ as $|\smbfk| = \sum_j k_j \rightarrow +\infty$.

	Next, we set 
	\begin{equation*}
		\pazo N_n^{\textpd} := \pazo \widehat{S}_n^{\textpd} \otimes_{\mbfa_R^+} \mbfn(T).
	\end{equation*}
	Again, $\pazo N_n^{\textpd}$ is $p\textrm{-adically}$ complete and it is equipped with a Frobenius-semilinear operator $\varphi : \pazo \widehat{S}_n^{\textpd} \otimes_{\mbfa_R^+} \mbfn(T) \rightarrow \pazo \widehat{S}_{n-1}^{\textpd} \otimes_{\mbfa_R^+} \mbfn(T)$ and a continuous and semilinear action of $\Gamma_{R}$.
	Now recall that we fixed $m \in \NN_{\geq 1}$ (fix $m \in \NN_{\geq 2}$ if $p=2$) such that $K = F(\zeta_{p^m})$.
	So we take
	\begin{equation*}
		M\prm := \big(\pazo N_m^{\textpd}\big)^{\Gamma_{R}\prm} \hspace{2mm} \textrm{and} \hspace{2mm} M\dprm := (M\prm)^{\Gamma_F} = \big(\pazo N_m^{\textpd}\big)^{\Gamma_{R}}.
	\end{equation*}
	Since we assumed $\mbfn(T)$ to be free, we have that $\pazo N_m^{\textpd}$ is a free $\pazo \widehat{S}_m^{\textpd}\textrm{-module}$ of rank $h$.
	As we have $\varphi^m\big(\pazo \widehat{S}_m^{\textpd}\big) \subset \pazo \mbfa_{R,\varpi}^{\textpd}$ so we get that $\varphi^m(M\dprm) \subset \big(\pazo \mbfa_{R,\varpi}^{\textpd} \otimes_{\mbfa_R^+} \mbfn(T)\big)^{\Gamma_{R}} = M$.
	Therefore, to show that the $R\big[\frac{1}{p}\big]\textrm{-rank}$ of $M\big[\frac{1}{p}\big]$ is at least $h$, it is enough to show that for each $x \in \mbfn(T)$ there exists unique $x\dprm \in M\dprm \subset \pazo N_m^{\textpd}$ fixed by $\Gamma_R$ and $x \equiv x\dprm \mod J^{[1]} \pazo N_m^{\textpd}$ (see Lemma \ref{lem:unique_lift}).

\subsubsection{Infinitesimal action of \texorpdfstring{$\Gamma_{R}$}{-}}

From \S \ref{subsec:relative_phi_gamma_mod} recall that we have $\{\gamma, \gamma_1, \ldots, \gamma_d\}$ as a set of topological generators of $\Gamma_{R}$ such that $\{\gamma_1, \ldots, \gamma_d\}$ generate $\Gamma_{R}\prm$ topologically, and $\gamma$ is a lift of a topological generator of $\Gamma_F$ where $\gamma^{e} = \gamma_0$ is a lift of a topological generator of $\Gamma_K$, $e = [K:F]$ and $\chi(\gamma_0) = \exp(p^m)$ where we fixed $m \in \NN_{\geq 1}$ (fix $m \in \NN_{\geq 2}$ if $p=2$).
Further, we have the identity $\gamma_0 \gamma_i = \gamma_i^{\chi(\gamma_0)} \gamma_0$ for $1 \leq i \leq d$.
In this section we will study the infinitesimal action of $\Gamma_R$ on the rings and modules constructed in previous section.

\begin{lem}\label{lem:gamma_minus_1_smpd}
	Let $k \in \NN$, $n \geq m$ and $i \in \{0, 1, \ldots, d\}$.
	Then $(\gamma_i-1) (p^m, \pi)^k \widehat{S}_n^{\textpd} \subset (p^m, \pi)^{k+1} \widehat{S}_n^{\textpd}$.
\end{lem}
\begin{proof}
	First, let $i = 0$.
	Recall that we have $\chi(\gamma_0) = \exp(p^m)  = 1 + p^m a\in 1+ p^m \ZZ_p$.
	So we can write
	\begin{align*}
		(\gamma_0-1)\pi &= (1 + \pi)^{\chi(\gamma_0)} - (1 + \pi)\\
				&= \big(\chi(\gamma_0)\pi + \tfrac{\chi(\gamma_0)(\chi(\gamma_0)-1)}{2!}\pi^2 + \tfrac{\chi(\gamma_0)(\chi(\gamma_0)-1)(\chi(\gamma_0)-2)}{3!} \pi^3 + \cdots\big) - \pi \\
				&= (\chi(\gamma_0)u-1) \pi,
	\end{align*}
	for some $u = 1 + \pi x \in 1 + \pi \mbfa_R^+$.
	Therefore, $\chi(\gamma_0)u - 1 = p^m a + \pi x + p^m a \pi x \in (p^m, \pi) \mbfa_R^+$ which gives us that $(\gamma_0-1)\pi \in (p^m, \pi) \pi \mbfa_R^+$.
	Now we have $(\gamma_0-1) \mbfa_R^+ \subset \pi \mbfa_R^+ \subset (p^m, \pi) \mbfa_R^+$, so proceeding by induction on $k \geq 1$ and using the fact that $\gamma_0-1$ acts as a twisted derivation (i.e. $(\gamma_0-1)xy = (\gamma_0-1)x \cdot y + \gamma_0(x)(\gamma_0-1)y$ for $x, y \in \mbfa_R^+$), we conclude that
	\begin{equation*}
		(\gamma_0 - 1) (p^m, \pi\big)^k \mbfa_R^+ \subset (p^m, \pi)^{k+1} \mbfa_R^+.
	\end{equation*}

	Next, any $f \in \widehat{S}_n^{\textpd}$ can be written as $f = \sum_{s \in \NN} f_s \frac{\pi^s}{s!p^{ns}}$ such that $f_s \in \mbfa_R^+$ goes to $0$ as $s \rightarrow +\infty$.
	Clearly we have
	\begin{equation*}
		(\gamma_0-1)\frac{\pi^s}{s!p^{ns}} = \frac{(\chi(\gamma_0)^s u^s - 1) \pi^s}{s!p^{ns}} \in (p^m, \pi) \frac{\pi^s}{s! p^{ns}} \widehat{S}_n^{\textpd}.
	\end{equation*}
	Combining the discussion for $\mbfa_R^+$ and $\frac{\pi^s}{s!p^{ns}}$, using induction on $k \geq 1$ and using the fact that $\gamma_0-1$ acts as a twisted derivation, we conclude that
	\begin{equation*}
		(\gamma_0-1) (p^m, \pi)^k \widehat{S}_n^{\textpd} \subset (p^m, \pi)^{k+1} \widehat{S}_n^{\textpd}.
	\end{equation*}

	Finally, for $i \in \{1, \ldots, d\}$ we have $(\gamma_i-1)[X_i^{\flat}] = \pi [X_i^{\flat}] \in (p^m, \pi) \mbfa_R^+$ and $(\gamma_i-1)\big([X_i^{\flat}]^{-1}\big) = -\pi(1+\pi)^{-1}[X_i^{\flat}]^{-1} \in (p^m, \pi) \mbfa_R^+$.
	Again by induction on $k \geq 1$ and using the fact that $\gamma_i-1$ acts as a twisted derivation, we get that
	\begin{equation*}
		(\gamma_i-1) (p^m, \pi)^k \mbfa_R^+ \subset (p^m, \pi)^{k+1} \mbfa_R^+.
	\end{equation*}
	Now any $f \in \widehat{S}_n^{\textpd}$ can be written as $f = \sum_{s \in \NN} f_s \frac{\pi^s}{s!p^{ns}}$ such that $f_s \in \mbfa_{R}^+$ goes to $0$ as $s \rightarrow +\infty$, and $\gamma_i$ acts trivially on $\pi$ for $1 \leq i \leq d$, so we conclude that
	\begin{equation*}
		(\gamma_i-1) (p^m, \pi)^k \widehat{S}_n^{\textpd} \subset (p^m, \pi)^{k+1} \widehat{S}_n^{\textpd}.
	\end{equation*}
\end{proof}

\begin{lem}\label{lem:log_gamma_converges}
	For $n \geq m$ and $i \in \{0, 1, \ldots, d\}$ the operators
	\begin{equation*}
		\nabla_i := \log \gamma_i = \sum_{k \in \NN}(-1)^k \tfrac{(\gamma_i-1)^{k+1}}{k+1},
	\end{equation*}
	converge as series of operators on $\widehat{S}_n^{\textpd}$.
\end{lem}
\begin{proof}
	From Lemma \ref{lem:gamma_minus_1_smpd}, we have that for $k \in \NN$
	\begin{equation*}
		(\gamma_i - 1) (p^m, \pi)^k \widehat{S}_n^{\textpd} \subset (p^m, \pi)^{k+1} \widehat{S}_n^{\textpd}.
	\end{equation*}
	Therefore, using the fact that $\gamma_i-1$ acts as a twisted derivation (i.e. $(\gamma_i-1)xy = (\gamma_i-1)x \cdot y + \gamma_i(x)(\gamma_i-1)y$ for $x, y \in \widehat{S}_n^{\textpd}$), we obtain that for $x \in \widehat{S}_n^{\textpd}$
	\begin{equation}\label{eq:gamma0_minus_1_action}
		(\gamma_i-1)^{k+1}(x) \subset (p^m, \pi)^{k+1} \widehat{S}_n^{\textpd}.
	\end{equation}
	Therefore, the following series converges in $\widehat{S}_n^{\textpd}$
	\begin{equation*}
		\nabla_i(x) = \sum_{k \in \NN}(-1)^k \tfrac{(\gamma_i-1)^{k+1}(x)}{k+1}.
	\end{equation*}
	This allows us to conlcude.
\end{proof}

\begin{rem}\label{rem:gamma_mod_pi_sm}
	Note that $\Gamma_{R}$ acts trivially modulo $\pi$ on $\mbfa_R^+$.
	Therefore, we also get that it acts trivially modulo $\pi$ over $\widehat{S}_n^{\textpd}$.
	Hence, for $0 \leq i \leq d$ we have $\nabla_i(\widehat{S}_n^{\textpd}) \subset \pi \widehat{S}_n^{\textpd} = t \widehat{S}_n^{\textpd}$, where the last equality follow from the fact that $\frac{t}{\pi}$ is a unit in $\widehat{S}_n^{\textpd}$ (see Lemma \ref{lem:t_over_pi_sm} below).
\end{rem}

\begin{rem}
	The operators $\nabla_i$ for $0 \leq i \leq d$, defined in Lemma \ref{lem:log_gamma_converges}, describe the action of the Lie algebra $\Lie \Gamma_R$ on $\widehat{S}_n^{\textpd}$, i.e. $\nabla_i$ acts as a differential operator on $\widehat{S}_n^{\textpd}$.
\end{rem}

\begin{lem}\label{lem:t_over_pi_sm}
	$\frac{t}{\pi}$ is a unit in $\widehat{S}_n^{\textpd}$ for $n \geq m$.
\end{lem}
\begin{proof}
	We can write the fraction
	\begin{equation*}
		\frac{t}{\pi} = \frac{\log (1+\pi)}{\pi} = \sum_{k\geq 0} (-1)^k \tfrac{\pi^k}{k+1}.
	\end{equation*}
	Formally, we can write
	\begin{equation*}
		\frac{\pi}{t} = \frac{\pi}{\log(1+\pi)} = b_0 + b_1 \pi + b_2 \pi^2 + b_3 \pi^3 + \cdots,
	\end{equation*}
	where $b_0 = 1$ and $\upsilon_p(b_k) \geq -\frac{k}{p-1}$ for all $k \geq 1$.
	But rewriting the series as a power series in $\frac{\pi^k}{k!p^{nk}}$, we get that
	\begin{equation*}
		\frac{\pi}{t} = \sum_{k \in \NN} b_k k!p^{nk} \tfrac{\pi^k}{k!p^{nk}}.
	\end{equation*}
	The $\padic$ valuation of coefficients in the series above is given as
	\begin{equation*}
		\upsilon_p(b_k k! p^{nk}) \geq \tfrac{-k}{p-1} + nk + \upsilon_p(k!) = \tfrac{p-2}{p-1}nk + \upsilon_p(k!),
	\end{equation*}
	which clearly goes to $+\infty$ as $k \rightarrow +\infty$.
	Hence, $\frac{\pi}{t}$ converges in $\widehat{S}_n^{\textpd}$ and is an inverse to $\frac{t}{\pi}$.
\end{proof}

Now let us consider the ring $\pazo \widehat{S}_n^{\textpd}$ and divided power ideals
\begin{equation*}
	J^{[i]} := \Big\langle \tfrac{\pi^{[k_0]}}{p^{nk_0}} \prod_{j=1}^d (1-V_j)^{[k_j]}, \hspace{1mm} \smbfk = (k_0, k_1, \ldots, k_d) \in \NN^{d+1} \hspace{1mm} \textrm{such that} \hspace{1mm} \sum_{j=0}^d k_j \geq i\Big\rangle \subset \pazo \widehat{S}_n^{\textpd}.
\end{equation*}

Arguments similar to Lemmas \ref{lem:gamma_minus_1_smpd} and \ref{lem:log_gamma_converges} show that for $0 \leq i \leq d$ the series of operators $\nabla_i = \log \gamma_i = \sum_{k \in \NN} (-1)^k \frac{\pi^{k+1}}{k+1}$ converge over $\pazo \widehat{S}_n^{\textpd}$.
Moreover, from Remark \ref{rem:gamma_mod_pi_sm} we obtain that for $0 \leq i \leq d$, we have $\nabla_i(\pazo \widehat{S}_n^{\textpd}) \subset t \pazo \widehat{S}_n^{\textpd}$.
Also, it is easy to observe that we have $\nabla_0(t) = \log(\chi(\gamma_0)) t = p^m t$ and $\nabla_i(V_i) = tV_i$ for $1 \leq i \leq d$.
Finally, recall that $\gamma_i \gamma_j = \gamma_j \gamma_i$ for $1 \leq i, j \leq d$ and $\gamma_0 \gamma_i = \gamma_i^{\chi(\gamma_0)} \gamma_0$, therefore we conclude that
\begin{align*}
	[\nabla_i, \nabla_j] &= 0, \\
	[\nabla_i, \nabla_0] &= \log(\chi(\gamma_0)) \nabla_i = p^m \nabla_i.
\end{align*}

Now we will adapt the discussion above to scalar extension of Wach module $\mbfn(T)$ to $\pazo \widehat{S}_n^{\textpd}$, i.e. for $\pazo N_n^{\textpd} := \pazo \widehat{S}_n^{\textpd} \otimes_{\mbfa_R^+} \mbfn(T)$.

\begin{lem}\label{lem:nabla_converges_relative}
	For $n \geq m$ and $i \in \{0, 1, \ldots, d\}$ the operators
	\begin{equation*}
		\nabla_i = \log \gamma_i = \sum_{k \in \NN} (-1)^{k+1} \tfrac{(\gamma_i-1)^{k+1}}{k+1}
	\end{equation*}
	converge as series of operators on $\pazo N_n^{\textpd}$.
\end{lem}
\begin{proof}
	For $0 \leq i \leq d$, observe that $\gamma_i-1$ acts as a twisted derivation, i.e. for $a \in \pazo \widehat{S}_n^{\textpd}$ and $x \in \mbfn(T)$, we have
	\begin{equation*}
		(\gamma_i-1)(ax) = (\gamma_i-1)a \cdot x + \gamma_i(a) (\gamma_i-1)x.
	\end{equation*}
	The action of $\Gamma_R$ is trivial on $\mbfn(T)/\pi\mbfn(T)$, so we can write $(\gamma_i-1)x = \pi y$, for some $y \in \mbfn(T)$, i.e. $(\gamma_i-1) \pazo N_n^{\textpd} \subset (p^m, \pi) \pazo N_n^{\textpd}$.
	From the proof of Lemma \ref{lem:log_gamma_converges} and \eqref{eq:gamma0_minus_1_action} and induction over $k \geq 1$, it follows that
	\begin{equation*}
		(\gamma_i-1) (p^m, \pi)^k \pazo N_n^{\textpd} \subset (p^m, \pi)^{k+1} \pazo N_n^{\textpd}.
	\end{equation*}
	Next, using the fact that $\gamma_i-1$ acts as a twisted derivation, we obtain that 
	\begin{equation*}
		(\gamma_i-1)^{k+1}(ax) \subset (p^m, \pi)^{k+1} \pazo N_n^{\textpd}.
	\end{equation*}
	Therefore, the following series converges in $\pazo N_n^{\textpd}$
	\begin{equation*}
		\nabla_i(ax) = \sum_{k \in \NN}(-1)^k \tfrac{(\gamma_i-1)^{k+1}(ax)}{k+1}.
	\end{equation*}
	This allows us to conlcude.
\end{proof}

\begin{rem}\label{rem:gamma_mod_pi_on}
	Note that $\Gamma_{R}$ acts trivially modulo $\pi$ on $\pazo \widehat{S}_n^{\textpd}$ and $\mbfn(T)$.
	Therefore, we also get that it acts trivially modulo $\pi$ over $\pazo N_n^{\textpd}$.
	Hence, for $0 \leq i \leq d$ we have $\nabla_i(\pazo N_n^{\textpd}) \subset \pi \pazo N_n^{\textpd} = t \pazo N_n^{\textpd}$, where the last equality follows from the fact that $\frac{t}{\pi}$ is a unit in $\pazo \widehat{S}_n^{\textpd}$ (see Lemma \ref{lem:t_over_pi_sm}).
\end{rem}

Again, over $\pazo N_n^{\textpd}$ we have
\begin{align*}
	[\nabla_i, \nabla_j] &= 0, \\
	[\nabla_i, \nabla_0] &= \log(\chi(\gamma_0)) \nabla_i = p^m \nabla_i,
\end{align*}
which enables us to define differential operators $\partial_i$ over $\pazo N_n^{\textpd}$ using the formula
\begin{equation*}
	\partial_i = \left\{
		\begin{array}{ll}
			-t^{-1}\nabla_0 & \textrm{for} \hspace{1mm} i = 0,\\
			t^{-1}V_i^{-1}\nabla_i & \textrm{for} \hspace{1mm} 1 \leq i \leq d,
		\end{array}
	\right.
\end{equation*}
where $V_i = \frac{X_i \otimes 1}{1 \otimes [X_i^{\flat}]}$ for $1 \leq i \leq d$.
Note that $\partial_i$ is well defined since $\nabla_i(\pazo N_n^{\textpd}) \subset t \pazo N_n^{\textpd}$ (see Remark \ref{rem:gamma_mod_pi_on}).

\begin{lem}\label{lem:diff_op_commute}
	For $n \geq m$, the differential operators defined on $\pazo N_n^{\textpd}$ commute, i.e. $\partial_i \circ \partial_j = \partial_j \circ \partial_i$ for $0 \leq i, j \leq d$.
\end{lem}
\begin{proof}
	From above we have $[\nabla_i, \nabla_j] = 0$ for $1 \leq i, j \leq d$, whereas $[\nabla_0, \nabla_i] = p^m \nabla_i$, for $1 \leq i \leq d$.
	So it follows that over $\pazo N_n^{\textpd}$ we have the composition of operators
	\begin{equation*}
		t^2V_iV_j(\partial_i \circ \partial_j - \partial_j \circ \partial_i) = t V_i \partial_i \circ t V_j\partial_j - t V_j\partial_j \circ t V_i\partial_i = \nabla_i \circ \nabla_j - \nabla_j \circ \nabla_i = 0, \hspace{2mm} \textrm{for} \hspace{1mm} 1 \leq i, j \leq d.
	\end{equation*}
	Next, for $1 \leq i \leq d$, we have
	\begin{align*}
		\nabla_0 \circ \nabla_i - \nabla_i \circ \nabla_0 &= -t\partial_0 \circ (t V_i \partial_i) + tV_i\partial_i \circ (t\partial_0)\\
						&= -p^m t V_i \partial_i - t^2 V_i \partial_0 \circ \partial_i + t^2V_i\partial_i \circ \partial_0 = p^m \nabla_i - t^2 V_i (\partial_0 \circ \partial_i - \partial_i \circ \partial_0).
	\end{align*}
	In particular, $\partial_i \circ \partial_j - \partial_j \circ \partial_i = 0$ for $0 \leq i, j \leq d$ since $\pazo N_n^{\textpd}$ is $t\textrm{-torsion}$ free.
\end{proof}

For the rest of the section, let us now assume $n = m$.
\begin{lem}\label{lem:quasi_nilpotence}
	Let $1 \leq i \leq d$ and $x \in \mbfn(T)$, then we have that $\partial_i^k(x) \rightarrow 0$ in $\pazo N_m^{\textpd}$ as $k \rightarrow +\infty$.
\end{lem}
\begin{proof}
	First, let us note that since $\partial_i(V_i) = 1$, $\partial_i(V_j) = 0$ for $j \neq i$ and $\partial_i(\pi) = 0$, so we have that $\partial_i^p(\pazo \widehat{S}_m^{\textpd}) \subset p \pazo \widehat{S}_m^{\textpd}$.
	Moreover, an easy computation shows that for $x \in \mbfn(T)$ we have 
	\begin{equation*}
		\partial_i(\varphi(x)) = \tfrac{\nabla_i(\varphi(x))}{tV_i} = \tfrac{\varphi(\nabla_i(x))}{tV_i} = p V_i^{p-1} \varphi(\partial_i(x)) \in \pazo \widehat{S}_m^{\textpd} \otimes_{\varphi(\mbfa_R^+)} \varphi(\mbfn(T)),
	\end{equation*}
	where note that we have $\varphi(\partial_i(x)) \in \varphi(\pazo \widehat{S}_{m+1}^{\textpd} \otimes_{\mbfa_R^+} \mbfn(T)) \subset \pazo \widehat{S}_m^{\textpd} \otimes_{\varphi(\mbfa_R^+)} \varphi(\mbfn(T))$ since $\partial_i(x)$ converges over $\pazo \widehat{S}_{m+1}^{\textpd} \otimes_{\mbfa_R^+} \mbfn(T)$ by Lemma \ref{lem:nabla_converges_relative}.

	Next, from Definiton \ref{defi:wach_reps} recall that we have $q^s \mbfn(T) \subset \varphi^{\ast}(\mbfn(T))$.
	Let us write $q^s x = \sum_{j=1}^h a_j \varphi(e_j)$ for $a_j \in \mbfa_R^+$ and $\{e_1, \ldots, e_h\}$ an $\mbfa_R^+\textrm{-basis}$ of $\mbfn(T)$.
	Then from the discussion above it follows that $\partial_i^p(q^s x) \in p q^s \big(\pazo \widehat{S}_m^{\textpd} \otimes_{\varphi(\mbfa_R^+)} \varphi(\mbfn(T))\big)$, therefore $\partial_i^p(x) \in p \big(\pazo \widehat{S}_m^{\textpd} \otimes_{\varphi(\mbfa_R^+)} \varphi(\mbfn(T))\big)$.
	By induction on $k$ we see that $\partial_i^{pk}(x) \in p^k \big(\pazo \widehat{S}_m^{\textpd} \otimes_{\varphi(\mbfa_R^+)} \varphi(\mbfn(T))\big) \subset p^k \pazo N_m^{\textpd}$.
	Hence, the claim follows.
\end{proof}

\begin{rem}\label{lem:recover_gamma_exp}
	Note that one can recover the action of $\gamma_i$ using the differential operator $\partial_i$.
	For $i \in \{1, \ldots, d\}$ we have $\gamma_i = \exp(tV_i\partial_i)$, whereas for $i=0$ we have $\gamma_0 = \exp(-t\partial_0)$.
\end{rem}

From the remark above it is clear that for $0 \leq i \leq d$ and $x \in \pazo N_m^{\textpd}$ we have $\gamma_i(x) = x$ if and only if $\partial_i(x) = 0$.
\begin{lem}\label{lem:horizontal_elems}
	For any $x \in \mbfn(T)$ there exists a unique $x\dprm \in \pazo N_m^{\textpd}$ such that
	\begin{align*}
		x\dprm &\equiv x \hspace{2mm} \mod J^{[1]} \pazo N_m^{\textpd},\\
		\gamma_i(x\dprm) &= x\dprm \hspace{4mm} \textrm{for} \hspace{1mm} 0 \leq i \leq d.
	\end{align*}
	In particular, $x\dprm \in M\dprm = \big(\pazo N_m^{\textpd}\big)^{\Gamma_{R}}$.
\end{lem}
\begin{proof}
	For $x \in \mbfn(T)$, we set
	\begin{equation*}
		x' = \sum_{\smbfk \in \NN^{d}} \partial_1^{k_1} \circ \cdots \circ \partial_d^{k_d} (x) (1-V_1)^{[k_1]} \cdots (1-V_d)^{[k_d]} \in \pazo N_m^{\textpd}
	\end{equation*}
	The summation converges since for $1 \leq i \leq d$ we have that $\partial_0^{k_0} \circ \partial_1^{k_1} \circ \cdots \circ \partial_d^{k_d} (x) \rightarrow 0$ as $|\smbfk| = \sum_{i=1}^d k_i \rightarrow +\infty$ from Lemma \ref{lem:quasi_nilpotence}.
	Note that we have an isomorphism of rings $\widehat{S}_m^{\textpd} \isomorphic (\pazo \widehat{S}_m^{\textpd})^{\Gamma_{R'}}$ compatible with $\Gamma_R/\Gamma_{R'} = \Gamma_F\textrm{-action}$.
	Therefore, by the description of $\widehat{S}_m^{\textpd}$ in \eqref{eq:snpd} and since $x' \in (\pazo N_m^{\textpd})^{\Gamma_R'}$ we see that the following sum converges
	\begin{equation*}
		x\dprm = \sum_{k_0 \in \NN} \partial_0^{k_0} (x') \tfrac{t^{[k_0]}}{p^{mk_0}} \in \pazo N_m^{\textpd}.
	\end{equation*}
	Since the differential operators on $\pazo N_m^{\textpd}$ commute by Lemma \ref{lem:diff_op_commute}, we get that 
	\begin{equation}\label{eq:horizontal_elems}
		x\dprm = \sum_{\smbfk \in \NN^{d+1}} \partial_0^{k_0} \circ \partial_1^{k_1} \circ \cdots \circ \partial_d^{k_d} (x) \tfrac{t^{[k_0]}}{p^{mk_0}} (1-V_1)^{[k_1]} \cdots (1-V_d)^{[k_d]} \in \pazo N_m^{\textpd}
	\end{equation}
	By the definition of $x\dprm$ it is clear  that $x\dprm \equiv x \mod J^{[1]} \pazo N_m^{\textpd}$.
	Next, using the fact that $\partial_i \circ \partial_j = \partial_j \circ \partial_i$ for $0 \leq i, j \leq d$ (see Lemma \ref{lem:diff_op_commute}) as well as $\partial_0(t) = -p^m$ and $\partial_i(V_i) = 1$ for $1 \leq i \leq d$, it is easy to deduce that $\partial_i(x\dprm) = 0$ for $0 \leq i \leq d$.
	So by Remark \ref{lem:recover_gamma_exp}, we get that $\gamma_i(x\dprm) = x\dprm$ for $0 \leq i \leq d$.

	Uniqueness of $x\dprm$ follows from Lemma \ref{lem:unique_lift}.
	Finally, let $g \in \Gamma_{F}$ be a lift of a generator of the cyclic group $\Gamma_F / \Gamma_K$.
	Then we have that $g(x\dprm) \in \pazo N_m^{\textpd}$ satisfies the conditions of the claim (since $(g-1)x \in \pi \mbfn(T) \subset J^{[1]} \pazo N_m^{\textpd}$).
	But by uniqueness, we obtain that $g(x\dprm) = x\dprm$, i.e. $x\dprm \in \big(\pazo N_m^{\textpd}\big)^{\Gamma_{R}} = M\dprm$.
\end{proof}

\begin{rem}
	Note that the lemma above can also be obtained by a ``successive approximation'' argument (see \cite[Lemmas 3.33 \& 3.37]{abhinandan-thesis}).
\end{rem}

Following claim was used above:
\begin{lem}\label{lem:unique_lift}
	For any $x \in \mbfn(T)$ suppose there exists $x\dprm \in \pazo N_m^{\textpd}$ such that
	\begin{align*}
		x\dprm &\equiv x \hspace{2mm} \mod J^{[1]} \pazo N_m^{\textpd},\\
		\gamma_i(x\dprm) &= x\dprm \hspace{4mm} \textrm{for} \hspace{1mm} 0 \leq i \leq d.
	\end{align*}
	Then $x\dprm$ is unique.
\end{lem}
\begin{proof}
	Let $\{f_1, \ldots, f_h\}$ denote an $\mbfa_R^+\textrm{-basis}$ of $\mbfn(T)$.
	Then $\{f_1, \ldots, f_h\}$ is also an $\pazo \widehat{S}_m^{\textpd}\textrm{-basis}$ of $\pazo N_m^{\textpd}$.
	Now using the formula in \eqref{eq:horizontal_elems}, for all $1 \leq i \leq h$ let
	\begin{equation*}
		f_i\dprm = \sum_{\smbfk \in \NN^{d+1}} \partial_0^{k_0} \circ \partial_1^{k_1} \circ \cdots \circ \partial_d^{k_d} (f_i) \tfrac{t^{[k_0]}}{p^{mk_0}} (1-V_1)^{[k_1]} \cdots (1-V_d)^{[k_d]} \in \pazo N_m^{\textpd}.
	\end{equation*}
	We want to show that $\{f_1\dprm, \ldots, f_h\dprm\}$ also form an $\pazo \widehat{S}_m^{\textpd}\textrm{-basis}$ of $\pazo N_m^{\textpd}$
	Let us write $f_i\dprm = f_i + \sum_{j=1}^h a_{ij} f_j$ with $a_{ij} \in J^{[1]} \pazo \widehat{S}_m^{\textpd}$ and let $A = id_h + (a_{ij}) \in \Mat(h, \pazo \widehat{S}_m^{\textpd})$ denote the $h \times h$ matrix thus obtained.
	We have that $\det A = 1 + x$ with $x \in J^{[1]} \pazo \widehat{S}_m^{\textpd}$ and $1 - x + x^2 - x^3 + \cdots = \sum_{n \in \NN} (-1)^n n! x^{[n]}$ converges in $\pazo \widehat{S}_m^{\textpd}$ as an inverse of $1+x$, i.e. $\det A$ is invertible in $\pazo \widehat{S}_m^{\textpd}$.
	Therefore, $\{f_1\dprm, \ldots, f_h\dprm\}$ form a basis of $\pazo N_m^{\textpd}$.
	
	Now for any $x \in \mbfn(T)$, writing $x = \sum_{i=1}^h x_i f_i\dprm$ and plugging into the formula \eqref{eq:horizontal_elems} we obtain $x\dprm \in \pazo N_m^{\textpd}$ such that $x\dprm \equiv x \mod J^{[1]} \pazo N_m^{\textpd}$ and $\gamma_j(x\dprm) = x\dprm$ for all $0 \leq j \leq d$.
	By linear independence of $\{f_1\dprm, \ldots, f_h\dprm\}$ over $\pazo \widehat{S}_m^{\textpd}$ we obtain that $x\dprm$ is unique.
\end{proof}

\begin{rem}
	The uniquess claim can also be established by a ``successive approximation'' argument (see \cite[p.63-p.65]{abhinandan-thesis}).
\end{rem}

\begin{lem}\label{lem:osmpd_comp}
	We have $\pazo \widehat{S}_m^{\textpd} \otimes_{R} M\dprm \isomorphic \pazo \widehat{S}_m^{\textpd} \otimes_{\mbfa_R^+} \mbfn(T)$.
\end{lem}
\begin{proof}
	Let $\{f_1, \ldots, f_h\}$ denote an $\mbfa_R^+\textrm{-basis}$ of $\mbfn(T)$.
	Then $\{f_1, \ldots, f_h\}$ is also an $\pazo \widehat{S}_m^{\textpd}\textrm{-basis}$ of $\pazo N_m^{\textpd}$.
	From the proof of Lemmas \ref{lem:horizontal_elems} \& \ref{lem:unique_lift} we have $f_i\dprm \in M\dprm$ for all $1 \leq i \leq h$, such that $\{f_1\dprm, \ldots, f_d\dprm\}$ also form an $\pazo \widehat{S}_m^{\textpd}\textrm{-basis}$ of $\pazo N_m^{\textpd}$.
	Therefore, $\pazo \widehat{S}_m^{\textpd} \otimes_{R} M\dprm \isomorphic \pazo N_m^{\textpd}$.
\end{proof}

\subsubsection{Finishing the proof of Proposition \ref{prop:crys_from_wach_mod}}

	Recall that at the beginning of the proof we assumed $\mbfn(T)$ to be free of rank $h$ (after extension of scalars to $\mbfa_{R\prm}^+$ which we again wrote as $\mbfa_R^+$ by abusing notations), therefore $\pazo N_m^{\textpd}$ is free of rank $h$.
	Further, we have $M = \big(\pazo \mbfa_{R,\varpi}^{\textpd} \otimes_{\mbfa_R^+} \mbfn(T)\big)^{\Gamma_{R}}$ and since $M\big[\frac{1}{p}\big]$ is equipped with an integrable connection, it is projective of rank $\leq h$ (see the beginning of the proof).
	So applying Lemma \ref{lem:horizontal_elems} to a basis of $\mbfn(T)$, we obtain that the rank of $M\big[\frac{1}{p}\big]$ as an $R\big[\frac{1}{p}\big]\textrm{-module}$ is exactly $h$.

	Next, we want to show that the natural inclusion $\pazo \mbfa_{R,\varpi}^{\textpd} \otimes_{R} M\big[\frac{1}{p}\big] \rightarrowtail \pazo \mbfa_{R,\varpi}^{\textpd} \otimes_{\mbfa_R^+} \mbfn(V)$ is bijective.
	To show this claim, we require the following lemma:
\begin{lem}
	We have $\varphi^{\ast}\big(\pazo \mbfa_{R,\varpi}^{\textpd} \otimes_{\mbfa_R^+} \mbfn(V)\big) \isomorphic \big(\pazo \mbfa_{R,\varpi}^{\textpd} \otimes_{\mbfa_R^+} \mbfn(V)\big)$.
\end{lem}
\begin{proof}
	Recall that we are working under the assumption that $\mbfn(V)$ is free and by definition of a positive finite $q\textrm{-height}$ representation we have that the cokernel of the inclusion $\varphi^{\ast}(\mbfn(V)) \rightarrow \mbfn(V)$ is killed by $q^s$ where $s \in \NN$ is the height of the representation $V$.
	Extending scalars to $\pazo\mbfa_{R,\varpi}^{\textpd}$, we obtain that the cokernel of the inclusion $\varphi^{\ast}\big(\pazo \mbfa_{R,\varpi}^{\textpd} \otimes_{\mbfa_R^+} \mbfn(V)\big) \rightarrow \big(\pazo \mbfa_{R,\varpi}^{\textpd} \otimes_{\mbfa_R^+} \mbfn(V)\big)$ is killed by $q^s$.
	Now note that we have $q = \frac{\varphi(\pi)}{\pi} = p \varphi\big(\frac{\pi}{t}\big)\frac{t}{\pi}$ where $\frac{t}{\pi}$ is a unit in $\mbfa_{R,\varpi}^{\textpd}$ (see Lemma \ref{lem:t_over_pi_unit}), i.e. $p$ and $q$ are associates in $\mbfa_{R,\varpi}^{\textpd}$.
	Therefore, the cokernel of the inclusion in the claim is killed by $p^s$.
	But, $p$ is invertible in $\pazo \mbfa_{R,\varpi}^{\textpd}\big[\frac{1}{p}\big]$.
	Hence, we obtain that $\varphi^{\ast}\big(\pazo \mbfa_{R,\varpi}^{\textpd} \otimes_{\mbfa_R^+} \mbfn(V)\big) \isomorphic \big(\pazo \mbfa_{R,\varpi}^{\textpd} \otimes_{\mbfa_R^+} \mbfn(V)\big)$.
\end{proof}

	Since we assumed $\mbfn(T)$ to be a free module, let $\{f_1, \ldots, f_h\}$ be its $\mbfa_R^+\textrm{-basis}$.
	Let $P \in \textup{Mat}(h, \mbfa_R^+)$ denote the matrix for the action of Frobenius on $\mbfn(T)$ in the chosen basis.
	Using the lemma above, we have also obtained that $\det P$ is invertible in $\pazo \mbfa_{R,\varpi}^{\textpd}\big[\frac{1}{p}\big]$.

	Now, recall that $\pazo N_m^{\textpd} = \pazo \widehat{S}_m^{\textpd} \otimes_{\mbfa_R^+} \mbfn(T)$ and $M\dprm = \big(\pazo N_m^{\textpd}\big)^{\Gamma_{R}}$.
	So we consider the following commutative diagram
	\begin{center}
		\begin{tikzcd}[row sep=large]
			\pazo \widehat{S}_m^{\textpd} \otimes_{R} M\dprm \arrow[r, "\sim"] \arrow[d, "\varphi^m \otimes \varphi^m"'] & \pazo N_m^{\textpd} \arrow[d, "\varphi^m"]\\
			\pazo \mbfa_{R,\varpi}^{\textpd} \otimes_{R} M \arrow[r] & \pazo \mbfa_{R,\varpi}^{\textpd} \otimes_{\mbfa_R^+} \mbfn(T),
		\end{tikzcd}
	\end{center}
	where the top horizontal arrow is bijective (see Lemma \ref{lem:osmpd_comp}) and all other arrows are injective.
	We also have that $\{f_1, \ldots, f_h\}$ is an $\pazo \mbfa_{R,\varpi}^{\textpd}\textrm{-basis}$ of $\pazo \mbfa_{R,\varpi}^{\textpd} \otimes_{\mbfa_R^+} \mbfn(T)$ as well as an $\pazo \widehat{S}_m^{\textpd}\textrm{-basis}$ of $\pazo N_m^{\textpd}$.
	From Lemmas \ref{lem:horizontal_elems} \& \ref{lem:osmpd_comp} and the discussion above, for $1 \leq i \leq h$ we have $f_i\dprm \in M\dprm$ such that $f_i\dprm = f_i + \sum_{i=1}^h a_{ij} f_j$ for $a_{ij} \in J^{[1]} \pazo \widehat{S}_m^{\textpd}$ and let $A := id_h + (a_{ij}) \in \textup{Mat}(h, \pazo\widehat{S}_m^{\textpd})$ denote the $h \times h$ matrix obtained in this manner.
	From the proof of Lemma \ref{lem:unique_lift} we have that $\det A$ is invertible in $\pazo \widehat{S}_m^{\textpd}$.

	Now let $v_i = (\varphi^m \otimes \varphi^m)f_i\dprm = \varphi^m(f_i) + \sum_{j=1}^h \varphi^m(a_{ij}) \varphi^m(f_j) \in M$ and let $M_0$ be the free $R\textrm{-submodule}$ of $M$ generated by $\{v_1, \ldots, v_h\}$.
	From the expression of $\{v_1, \ldots, v_h\}$ in the basis of $\pazo \mbfa_{R,\varpi}^{\textpd} \otimes_{\mbfa_R^+} \mbfn(T)$, we get that the determinant of the inclusion $\pazo \mbfa_{R,\varpi}^{\textpd} \otimes_{R} M_0 \rightarrowtail \pazo \mbfa_{R,\varpi}^{\textpd} \otimes_{\mbfa_R^+} \mbfn(T)$ is given by $\varphi^m(\det A)\varphi^{m-1}(\det P)\varphi^{m-2}(\det P) \cdots \varphi(\det P) (\det P)$.
	Since $\det A$ is invertible in $\pazo \widehat{S}_m^{\textpd}$, we have that $\varphi^m(\det A)$ is invertible in $\pazo \mbfa_{R,\varpi}^{\textpd}$ and from above we already have that $\det P$ is invertible in $\pazo \mbfa_{R,\varpi}^{\textpd}\big[\frac{1}{p}\big]$.
	Therefore, the natural inclusions
	\begin{equation*}
		\pazo \mbfa_{R,\varpi}^{\textpd} \otimes_{R} M_0\big[\tfrac{1}{p}\big] \longrightarrow \pazo \mbfa_{R,\varpi}^{\textpd} \otimes_{R} M\big[\tfrac{1}{p}\big] \longrightarrow \pazo \mbfa_{R,\varpi}^{\textpd} \otimes_{\mbfa_R^+} \mbfn(V),
	\end{equation*}
	are bijective.
	These inclusions are compatible with Frobenius, filtration, connection and the action of $\Gamma_{R}$ on each side, which shows the second claim of Proposition \ref{prop:crys_from_wach_mod}.

	Finally, note that above we assumed $\mbfn(T)$ to be free of rank $h$, therefore we obtain a free $R\textrm{-submodule}$ $M_0 \subset M$ such that
	\begin{equation*}
		M_0\big[\tfrac{1}{p}\big] = \big(\pazo \mbfa_{R,\varpi}^{\textpd} \otimes_{R} M_0\big[\tfrac{1}{p}\big]\big)^{\Gamma_{R}} \isomorphic \big(\pazo \mbfa_{R,\varpi}^{\textpd} \otimes_{R} M\big[\tfrac{1}{p}\big]\big)^{\Gamma_{R}} = M\big[\tfrac{1}{p}\big],
	\end{equation*}
	which are free of rank $h$ over $R\big[\frac{1}{p}\big]$.
	This shows the last claim of Propostion \ref{prop:crys_from_wach_mod}.
	In general, when $\mbfn(T)$ is projective of rank $h$, we obtain that $M\big[\frac{1}{p}\big]$ is projective of rank $h$.
	This sums up our proof.
\end{proof}

\subsection{Proof of Theorem \ref{thm:crys_wach_comparison}}

Let $M = \big(\pazo \mbfa_{R,\varpi}^{\textpd} \otimes_{\mbfa_R^+} \mbfn(T)\big)^{\Gamma_{R}}$.
From Proposition \ref{prop:crys_from_wach_mod} we already have the isomorphism of $\pazo \mbfa_{R,\varpi}^{\textpd}\big[\frac{1}{p}\big]\textrm{-modules}$
\begin{equation*}
	\pazo \mbfa_{R,\varpi}^{\textpd} \otimes_{R} M\big[\tfrac{1}{p}\big] \isomorphic \pazo \mbfa_{R,\varpi}^{\textpd} \otimes_{\mbfa_R^+} \mbfn(V),
\end{equation*}
compatible with Frobenius, filtration, connection and the action of $\Gamma_{R}$ on each side.
This proves the second claim and we are left to show that $V$ is crystalline and $M\big[\tfrac{1}{p}\big] \isomorphic \pazo \mbfd_{\crys}(V)$ compatible with supplementary structures.
Also note from Proposition \ref{prop:crys_from_wach_mod} that we already have the inclusion of projective $R\big[\frac{1}{p}\big]\textrm{-modules}$ of rank $h = \dim_{\QQ_p} V$, $M\big[\tfrac{1}{p}\big] \subset \pazo \mbfd_{\crys}(V)$.
So we are left to show that this inclusion is bijective and compatible with supplementary structures.

First, we will show that $V$ is crystalline and the inclusion described above is in fact bijective.
Extending scalars along $\pazo \mbfa_{R,\varpi}^{\textpd}\big[\frac{1}{p}\big] \rightarrowtail \pazo \mbfb_{\crys}(\overline{R})$ for the isomorphism $\pazo \mbfa_{R,\varpi}^{\textpd} \otimes_{R} M\big[\tfrac{1}{p}\big] \isomorphic \pazo \mbfa_{R,\varpi}^{\textpd} \otimes_{\mbfa_R^+} \mbfn(V)$, we obtain an isomorphism of $\pazo \mbfb_{\crys}(\overline{R})\textrm{-modules}$
\begin{equation*}
	\pazo \mbfb_{\crys}(\overline{R}) \otimes_{R[\frac{1}{p}]} M\big[\tfrac{1}{p}\big] \isomorphic \pazo \mbfb_{\crys}(\overline{R}) \otimes_{\mbfb_{R}^+} \mbfn(V),
\end{equation*}
compatible with Frobenius, filtration, connection and $G_R\textrm{-action}$.
Now, recall that from the definitions we have a natural inclusion of free $\mbfa^+\textrm{-modules}$ $\mbfa^+ \otimes_{\mbfa_R^+} \mbfn(V) \rightarrowtail \mbfa^+ \otimes_{\mbfa_R^+} V$ compatible with supplementary structures and the cokernel of this inclusion is killed by $\pi^s$ (see Proposition \ref{prop:wach_approx_aplus_admis}).
Since $\pi$ is invertible in $\pazo \mbfb_{\crys}(\overline{R})$, extending scalars along $\mbfa^+ \rightarrowtail \pazo \mbfb_{\crys}(\overline{R})$, we obtain an isomorphism of $\pazo \mbfb_{\crys}(\overline{R})\textrm{-modules}$
\begin{equation*}
	\pazo \mbfb_{\crys}(\overline{R}) \otimes_{\mbfb_{R}^+} \mbfn(V) \isomorphic \pazo \mbfb_{\crys}(\overline{R}) \otimes_{\QQ_p} V,
\end{equation*}
compatible with Frobenius, connection and $G_R\textrm{-action}$.
Finally, since $R\big[\frac{1}{p}\big] \rightarrow \pazo \mbfb_{\crys}(\overline{R})$ is faithfully flat (see \cite[Th\'eor\`eme 6.3.8]{brinon-padicrep-relatif}), we obtain an inclusion of $\pazo \mbfb_{\crys}(\overline{R})\textrm{-modules}$
\begin{equation*}
	\pazo \mbfb_{\crys}(\overline{R}) \otimes_{R[\frac{1}{p}]} M\big[\tfrac{1}{p}\big] \subset \pazo \mbfb_{\crys}(\overline{R}) \otimes_{R[\frac{1}{p}]} \pazo \mbfd_{\crys}(V),
\end{equation*}
compatible with Frobenius, connection and the action of $G_{R}$.
In particular, we have a commutative diagram
\begin{center}
	\begin{tikzcd}[row sep=large, column sep=large]
		\pazo \mbfb_{\crys}(\overline{R}) \otimes_{R[\frac{1}{p}]} M\big[\tfrac{1}{p}\big] \arrow[r, "\sim"] \arrow[d, rightarrowtail] & \pazo \mbfb_{\crys}(\overline{R}) \otimes_{\mbfb_{R}^+} \mbfn(V) \arrow[d, "\sim"]\\
		\pazo \mbfb_{\crys}(\overline{R}) \otimes_{R[\frac{1}{p}]} \pazo \mbfd_{\crys}(V) \arrow[r, rightarrowtail] & \pazo \mbfb_{\crys}(\overline{R}) \otimes_{\QQ_p} V,
	\end{tikzcd}
\end{center}
compatible with Frobenius, connection and $G_R\textrm{-action}$.
As the top horizontal arrow and right vertical arrow are bijections, it is immediately clear from the diagram that the left vertical arrow and bottom horizontal arrow must be bijective as well.
The bijection of bottom horizontal arrow implies that $V$ is a crystalline representation of $G_{R}$.
Moreover, since $R\big[\frac{1}{p}\big] \rightarrow \pazo \mbfb_{\crys}(\overline{R})$ is faithfully flat (see \cite[Th\'eor\`eme 6.3.8]{brinon-padicrep-relatif}), we obtain an isomorphism of $R\big[\frac{1}{p}\big]\textrm{-modules}$ $M\big[\tfrac{1}{p}\big] \isomorphic \pazo \mbfd_{\crys}(V)$.

Finally, we note that the isomorphism $M\big[\tfrac{1}{p}\big] \isomorphic \pazo \mbfd_{\crys}(V)$ is compatible with supplementary structures.
From Proposition \ref{prop:crys_from_wach_mod} it is clear that this isomorphism is compatible with Frobenius and connection.
Combining Proposition \ref{prop:fil_compatible} with observations made before, we obtain that the isomorphism of $R\big[\frac{1}{p}\big]\textrm{-modules}$ $M\big[\frac{1}{p}\big] \isomorphic \pazo \mbfd_{\crys}(V)$ is compatible with Frobenius, filtration and connection on each side.

Finally, we can compose these natural maps as
\begin{equation*}
	\pazo \mbfa_{R,\varpi}^{\textpd} \otimes_{R} \pazo \mbfd_{\crys}(V) \lisomorphic \pazo \mbfa_{R,\varpi}^{\textpd} \otimes_{R} \big(\pazo \mbfa_{R,\varpi}^{\textpd} \otimes_{\mbfa_R^+} \mbfn(V)\big)^{\Gamma_{R}} \isomorphic \pazo \mbfa_{R,\varpi}^{\textpd} \otimes_{\mbfa_R^+} \mbfn(V),
\end{equation*}
where the second map is compatible with the Frobenius, filtration, connection and the action of $\Gamma_{R}$ on each side (see Proposition \ref{prop:crys_from_wach_mod}).
This proves the theorem.
\hfill \qedsymbol

\begin{rem}\label{rem:crys_wach_comparison_int}
	In the case when $\mbfn(T)$ is a free $\mbfa_R^+\textrm{-}$ module of rank $h$, from Proposition \ref{prop:crys_from_wach_mod} we obtain that $M\big[\frac{1}{p}\big] \isomorphic \pazo \mbfd_{\crys}(V)$ is a free $R\big[\frac{1}{p}\big]\textrm{-module}$ of rank $h$.
	In particular, for finite $q\textrm{-height}$ representations there exists a finite \'etale extension $R\prm$ over $R$ such that $R\prm\big[\frac{1}{p}\big] \otimes_{R[\frac{1}{p}]} \pazo \mbfd_{\crys}(V)$ is free of rank $h$.
\end{rem}

\begin{rem}
	For $0 \leq i \leq d$, one can define $[\varepsilon]\textrm{-derivatives}$ by the formula $\frac{\gamma_i-1}{\pi} : \mbfn(T) \rightarrow \mbfn(T)$.
	Considering the reduction modulo $\pi$ of Frobenius, filtration and $[\varepsilon]\textrm{-connection}$ on $\mbfn(T)$ defined above, we conjecture that we have $(\mbfn(T) / \pi \mbfn(T))\big[\frac{1}{p}\big] \isomorphic \pazo \mbfd_{\crys}(V)$ as filtered $(\varphi, \partial)\textrm{-modules}$ over $R\big[\frac{1}{p}\big]$.
	Details on this line of thought and its connection with \cite{bhatt-scholze-prismatic} and \cite{gros-lestum-quiros-qcrystals} will appear elsewhere.
\end{rem}

\subsubsection{Compatibility between filtrations}

Recall that using Definition \ref{defi:wach_mod_fil} and Remark \ref{rem:fat_pd_ring_struct} (ii), the filtration on $M\big[\frac{1}{p}\big]$ is given as
\begin{equation*}
	\Fil^k M\big[\tfrac{1}{p}\big] = \Big(\sum_{i \in \NN} \Fil^i \pazo \mbfa_{R,\varpi}^{\textpd} \widehat{\otimes}_{\mbfa_R^+} \Fil^{k-i} \mbfn(V)\Big)^{\Gamma_{R}}.
\end{equation*}
\begin{prop}\label{prop:fil_compatible}
	In the notations already described, we have $\Fil^k M\big[\frac{1}{p}\big] = \Fil^k \pazo \mbfd_{\crys}(V)$ for $k \in \ZZ$
\end{prop}
\begin{proof}
	We only need to show the claim for $k \geq 1$.
	Note that from \eqref{eq:wach_invar_in_crys} we have
	\begin{equation*}
		\Fil^k M\big[\tfrac{1}{p}\big] = \big(\Fil^k(\pazo \mbfa_{R,\varpi}^{\textpd} \otimes_{\mbfa_R^+} \mbfn(V))\big)^{\Gamma_{R}} \subset \big(\Fil^k(\pazo \mbfb_{\crys}(\overline{R}) \otimes_{\QQ_p} V)\big)^{G_{R}} = \Fil^k \pazo \mbfd_{\crys}(V).
	\end{equation*}

	Conversely, let $\{e_1, \ldots, e_h\}$ denote a $\QQ_p\textrm{-basis}$ of $V$ and let $x \in \Fil^k \pazo \mbfd_{\crys}(V) \setminus \Fil^{k+1} \pazo \mbfd_{\crys}(V)$.
	Since $x \neq 0$, we can write $x = \sum_{i=1}^h b_i e_i$ where either $b_i = 0$ or $b_i \in \Fil^k \pazo \mbfb_{\crys}(\overline{R}) \setminus \Fil^{k+1} \pazo \mbfb_{\crys}(\overline{R})$ for each $1 \leq i \leq h$ and at least one $b_i \neq 0$.
	Moreover, we have $M\big[\frac{1}{p}\big] \isomorphic \pazo \mbfd_{\crys}(V)$ as $R\big[\frac{1}{p}\big]\textrm{-modules}$, so we take $r \leq k$ to be the largest integer such that $x \in \Fil^r M\big[\frac{1}{p}\big]$, in particular $x \not\in \Fil^{r+1} M\big[\frac{1}{p}\big]$.
	Let us write $x = \sum_{j \in \NN} c_j \otimes f_{r-j}$ with $c_j \in \Fil^j \pazo \mbfa_{R,\varpi}^{\textpd}$ and $f_{r-j} \in \Fil^{r-j} \mbfn(V)$ for all $j \in \NN$.
	By assumption on $x$ there exists $\emptyset \neq I \subset \NN$ such that for each $j \in I$ we have $c_{j} \in \Fil^j \pazo \mbfa_{R,\varpi}^{\textpd} \setminus \Fil^{j+1} \pazo \mbfa_{R,\varpi}^{\textpd}$, $f_{r-j} \in \Fil^{r-j} \mbfn(V) \setminus \Fil^{r-j+1} \mbfn(V)$ with 
	\begin{align*}
		\sum_{j \in I} c_j \otimes f_{r-j} &\in \Fil^r (\pazo \mbfa_{R,\varpi}^{\textpd} \otimes_{\mbfa_R^+} \mbfn(V)) \setminus \Fil^{r+1} (\pazo \mbfa_{R,\varpi}^{\textpd} \otimes_{\mbfa_R^+} \mbfn(V)) \hspace{2mm} \textrm{and}\\
		\sum_{j \in \NN \setminus I} c_j \otimes f_{r-j} &\in \Fil^{r+1} (\pazo \mbfa_{R,\varpi}^{\textpd} \otimes_{\mbfa_R^+} \mbfn(V)).
	\end{align*}

	Next, we equip $\mbfb^+$ with the induced filtration $\Fil^n \mbfb^+ := \mbfb^+ \cap \Fil^n \mbfb_{\crys}(\overline{R}) := \mbfb^+ \cap \Fil^n \big(\mbfa_{\inf}(\overline{R})\big[\frac{1}{p}\big]\big)$ for $n \in \NN$.
	Using the definition of filtration on $\mbfn(V)$ (see Definition \ref{defi:wach_mod_fil}) and Lemma \ref{lem:filb_wach}, we have that $\Fil^{r-j} \mbfn(V) = (\Fil^{r-j} \mbfb^+ \otimes_{\QQ_p} V) \cap \mbfn(V)$ for all $j \in \NN$.
	Therefore, in the expression $\sum_{j \in I} c_j \otimes f_{r-j}$ we must have $f_{r-j} \in (\Fil^{r-j}\mbfb^+ \otimes_{\QQ_p} V) \setminus (\Fil^{r-j+1}\mbfb^+ \otimes_{\QQ_p} V)$ for all $j \in I$.
	This implies that in the basis of $V$ we can write $f_{r-j} = \sum_{i=1}^h f_{r-j}^{(i)} e_i$ with $f_{r-j}^{(i)} \in \Fil^{r-j} \mbfb^+ \setminus \Fil^{r-j+1} \mbfb^+$ for all $j \in I$ and all $1 \leq i \leq h$.
	In conclusion, we obtain 
	\begin{equation}\label{eq:xcjfrj}
		x - \sum_{j \in \NN \setminus I} c_j \otimes f_{r-j} = \sum_{j \in I} c_j \otimes \big(\sum_{i=1}^h f_{r-j}^{(i)} e_i\big) = \sum_{i=1}^h \big(\sum_{j \in I} c_j \otimes f_{r-j}^{(i)}\big) e_i,
	\end{equation}
	with $c_{j} \in \Fil^{j} \pazo \mbfa_{R,\varpi}^{\textpd} \setminus \Fil^{j+1} \pazo \mbfa_{R,\varpi}^{\textpd}$ and $f_{r-j}^{(i)} \in \Fil^{r-j} \mbfb^+ \setminus \Fil^{r-j+1} \mbfb^+$ for all $1 \leq i \leq h$ and $j \in I$.

	Let us set $g_i = \sum_{j \in I} c_j \otimes f_{r-j}^{(i)}$ for $1 \leq i \leq h$.
	Then by the the discussion above we have that $g_i \in \Fil^r(\pazo \mbfa_{R,\varpi}^{\textpd} \otimes_{\mbfa_R^+} \mbfb^+)$ for $1 \leq i \leq h$, where $\pazo \mbfa_{R,\varpi}^{\textpd} \otimes_{\mbfa_R^+} \mbfb^+$ is equipped with the tensor product filtration.
	Note that $x \in \Fil^r M\big[\frac{1}{p}\big] \setminus \Fil^{r+1} M\big[\frac{1}{p}\big]$ and $\sum_{j \in \NN \setminus I} c_j \otimes f_{r-j} \in \Fil^{r+1} (\pazo \mbfa_{R,\varpi}^{\textpd} \otimes_{\mbfa_R^+} \mbfn(V))$.
	Moreover, from Lemma \ref{lem:fil_oarpdn} we deduce that for $n \in \NN$ we have
	\begin{equation*}
		\Fil^n\big(\pazo \mbfa_{R,\varpi}^{\textpd} \otimes_{\mbfa_R^+} \mbfn(V)\big) = \big(\Fil^n \big(\pazo \mbfa_{R,\varpi}^{\textpd} \otimes_{\mbfa_R^+} \mbfb^+\big) \otimes_{\ZZ_p} V\big) \cap \big(\pazo \mbfa_{R,\varpi}^{\textpd} \otimes_{\mbfa_R^+} \mbfn(V)\big).
	\end{equation*}
	Therefore, we conclude that we must have at least one $i = i_0$ such that $g_{i_0} \in \Fil^r (\pazo \mbfa_{R,\varpi}^{\textpd} \otimes_{\mbfa_R^+} \mbfb^+) \setminus \Fil^{r+1} (\pazo \mbfa_{R,\varpi}^{\textpd} \otimes_{\mbfa_R^+} \mbfb^+)$.
	Now using Lemma \ref{lem:bpipd_fil} we further note that for $n \in \NN$
	\begin{equation*}
		\Fil^n \big(\pazo \mbfa_{R,\varpi}^{\textpd} \otimes_{\mbfa_R^+} \mbfb^+) = \big(\pazo \mbfa_{R,\varpi}^{\textpd} \otimes_{\mbfa_R^+} \mbfb^+\big) \cap \Fil^n \pazo \mbfb_{\crys}(\overline{R}) \subset \pazo \mbfb_{\crys}(\overline{R}).
	\end{equation*}
	Therefore, we get that $g_i \in \Fil^r \pazo \mbfb_{\crys}(\overline{R})$ for all $1 \leq i \leq h$ and $g_{i_0} \in \Fil^r \pazo \mbfb_{\crys}(\overline{R}) \setminus \Fil^{r+1} \pazo \mbfb_{\crys}(\overline{R})$.
	For convenience, let us write $\sum_{j \in \NN \setminus I} c_j \otimes f_{r-j} = \sum_{i=1}^h d_i e_i$ with $d_i \in \Fil^{r+1} \pazo \mbfb_{\crys}(\overline{R})$ for all $1 \leq i \leq h$.
	In particular, comparing \eqref{eq:xcjfrj} with the expression $x = \sum_{i=1}^h b_i e_i$ at the start of the proof, we get $b_{i_0} = g_{i_0} + d_{i_0} $.

	Finally, since $r \leq k$, consider the following commutative diagram with exact rows
	\begin{center}
		\begin{tikzcd}
			0 \arrow[r] & \Fil^{k+1} \pazo \mbfb_{\crys}(\overline{R}) \arrow[r] \arrow[d] & \Fil^{k} \pazo \mbfb_{\crys}(\overline{R}) \arrow[r] \arrow[d] & \gr^{k} \pazo \mbfb_{\crys}(\overline{R}) \arrow[r] \arrow[d] & 0\\
			0 \arrow[r] & \Fil^{r+1} \pazo \mbfb_{\crys}(\overline{R}) \arrow[r] & \Fil^{r} \pazo \mbfb_{\crys}(\overline{R}) \arrow[r] & \gr^{r} \pazo \mbfb_{\crys}(\overline{R}) \arrow[r] & 0,
		\end{tikzcd}
	\end{center}
	where the left and middle vertical arrows are injective and the right vertical arrow is non-trivial if and only if $r = k$.
	From the fact that $g_{i_0} \in \Fil^r \pazo \mbfb_{\crys}(\overline{R}) \setminus \Fil^{r+1} \pazo \mbfb_{\crys}(\overline{R})$, we see that the image of $b_{i_0}$ is non-zero in $\gr^{r} \pazo \mbfb_{\crys}(\overline{R})$.
	But we already have that image of $b_{i_0}$ is non-zero in $\gr^{k} \pazo \mbfb_{\crys}(\overline{R})$.
	Therefore, the left vertical arrow must be non-trivial, i.e.\ $r=k$.
	Hence $x \in \Fil^k M\big[\frac{1}{p}\big]$.
	This proves the claim.
\end{proof}

\begin{lem}\label{lem:fil_oarpdn}
	For $k \in \NN$ we have
	\begin{equation*}
		\Fil^k\big(\pazo \mbfa_{R,\varpi}^{\textpd} \otimes_{\mbfa_R^+} \mbfn(T)\big) = \big(\Fil^k \big(\pazo \mbfa_{R,\varpi}^{\textpd} \otimes_{\mbfa_R^+} \mbfa^+\big) \otimes_{\ZZ_p} T\big) \cap \big(\pazo \mbfa_{R,\varpi}^{\textpd} \otimes_{\mbfa_R^+} \mbfn(T)\big).
	\end{equation*}
\end{lem}
\begin{proof}
	From \S \ref{subsec:relative_phi_gamma_mod} we have rings $\mbfa^+ \subset \mbfa_{\varpi}^+ \subset \mbfa_{\inf}(\overline{R})$ equipped with an induced filtration from $\mbfa_{\crys}(\overline{R})$ and from Remark \ref{rem:a_varpi_iso} we have an isomorphism $\mbfa_{R, \varpi}^+ \otimes_{\mbfa_R^+} \mbfa^+ \isomorphic \mbfa_{\varpi}^+$ compatible with Frobenius, filtration and $G_R\textrm{-action}$.
	So using Lemma \ref{lem:filb_wach}, the fact that $\mbfa_R^+ \rightarrow \mbfa_{R, \varpi}^+$ is flat and $\Fil^i \mbfa_{R, \varpi}^+ = \xi^i \mbfa_{R, \varpi}^+$ we note that
	\begin{align}\label{eq:arpiplus_fil}
		\begin{split}
			\Fil^k \big(\mbfa_{R, \varpi}^+ \otimes_{\mbfa_R^+} \mbfn(T)\big) &= \sum_{i+j=k} \Fil^i \mbfa_{R, \varpi}^+ \otimes_{\mbfa_R^+} ((\Fil^j\mbfa^+ \otimes_{\ZZ_p} T) \cap \mbfn(T))\\
			&= \Big(\sum_{i+j=k} \Fil^i \mbfa_{R, \varpi}^+ \otimes_{\mbfa_R^+} \Fil^j\mbfa^+ \otimes_{\ZZ_p} T \Big) \cap \big(\mbfa_{R, \varpi}^+ \otimes_{\mbfa_R^+} \mbfn(T)\big)\\
			&= \big(\Fil^k \mbfa_{\varpi}^+ \otimes_{\ZZ_p} T\big) \cap \big(\mbfa_{R, \varpi}^+ \otimes_{\mbfa_R^+} \mbfn(T)\big).
		\end{split}
	\end{align}
	Next, from Definition \ref{defi:oarpd} we have the ring $\pazo \mbfa_{R, \varpi}^+$ flat over $\mbfa_{R, \varpi}^+$ since
	\begin{equation*}
		\pazo \mbfa_{R, \varpi}^+ = \oplus_{\smbfk \in \NN^d} \mbfa_{R, \varpi}^+ (X_1-[X_1^{\flat}])^{k_1} \cdots (X_d-[X_d^{\flat}])^{k_d},
	\end{equation*}
	with the structure map $\mbfa_{R, \varpi}^+ \rightarrow \pazo \mbfa_{R, \varpi}^+$ being injective and its image is identified with term at index $\smbfk = (0, \ldots, 0)$.
	Let us set $\pazo \mbfa_{\varpi}^+ := \pazo \mbfa_{R, \varpi}^+ \otimes_{\mbfa_{R}^+} \mbfa^+ \isomorphic \pazo \mbfa_{R, \varpi}^+ \otimes_{\mbfa_{R, \varpi}^+} \mbfa_{\varpi}^+$ equipped with natural filtration, Frobenius and $G_R\textrm{-action}$.
	Let $J = (X_1-[X_1^{\flat}], \ldots, X_d-[X_d^{\flat}])\pazo \mbfa_{R, \varpi}^+$ then the filtration on $\pazo \mbfa_{\varpi}^+$ can also be given as $\Fil^k \pazo \mbfa_{\varpi}^+ = \sum_{i+j=k} J^i \pazo \mbfa_{R, \varpi}^+ \otimes_{\mbfa_{R, \varpi}^+} \xi^j \mbfa_{\varpi}^+$.
	Let us set $N_{R, \varpi}^+ = \mbfa_{R, \varpi}^+ \otimes_{\mbfa_R^+} \mbfn(T)$ equipped with tensor product filtration.
	Then since $J$ is flat as an $\mbfa_{R, \varpi}^+\textrm{-module}$ an argument similar to \eqref{eq:arpiplus_fil} gives us that
	\begin{align}\label{eq:oarpiplus_fil}
		\begin{split}
			\Fil^k \big(\pazo \mbfa_{R, \varpi}^+ \otimes_{\mbfa_R^+} \mbfn(T)\big) &= \Fil^k \big(\pazo \mbfa_{R, \varpi}^+ \otimes_{\mbfa_{R, \varpi}^+} N_{R, \varpi}^+\big)\\
			&= \sum_{i+j=k} J^i \pazo \mbfa_{R, \varpi}^+ \otimes_{\mbfa_{R, \varpi}^+} ((\Fil^j \mbfa_{\varpi}^+ \otimes_{\ZZ_p} T) \cap N_{R, \varpi}^+)\\
			&= \Big(\sum_{i+j=k} J^i \pazo \mbfa_{R, \varpi}^+ \otimes_{\mbfa_{R, \varpi}^+} \xi^j \mbfa_{\varpi}^+ \otimes_{\ZZ_p} T \Big) \cap \big(\pazo \mbfa_{R, \varpi}^+ \otimes_{\mbfa_{R, \varpi}^+} N_{R, \varpi}^+\big)\\
			&= \big(\Fil^k \pazo \mbfa_{\varpi}^+ \otimes_{\ZZ_p} T\big) \cap \big(\pazo \mbfa_{R, \varpi}^+ \otimes_{\mbfa_R^+} \mbfn(T)\big).
		\end{split}
	\end{align}
	Furthermore, let us set $\pazo \mbfa_{\varpi}^{\textpd} := \pazo \mbfa_{R,\varpi}^{\textpd} \otimes_{\mbfa_{R, \varpi}^+} \mbfa_{\varpi}^+ \isomorphic \pazo \mbfa_{R,\varpi}^{\textpd} \otimes_{\mbfa_R^+} \mbfa^+$ where the isomorphism is compatible with Frobenius, filtration, connection and $G_R\textrm{-action}$.

	Now we will show our claim
	\begin{equation*}
		\Fil^k\big(\pazo \mbfa_{R,\varpi}^{\textpd} \otimes_{\mbfa_R^+} \mbfn(T)\big) = \big(\Fil^k \pazo \mbfa_{\varpi}^{\textpd} \otimes_{\ZZ_p} T\big) \cap \big(\pazo \mbfa_{R,\varpi}^{\textpd} \otimes_{\mbfa_R^+} \mbfn(T)\big).
	\end{equation*}
	Let $f \in \{\xi, X_1 - [X_1^{\flat}], \ldots, X_d - [X_d^{\flat}]\}$ be one of the generators of the ideal $(\xi, X_1 - [X_1^{\flat}], \ldots, X_d - [X_d^{\flat}]) \pazo \mbfa_{\varpi}^{\textpd}$.
	Then to obtain our claim, it is enough to show that if $f^{[k]} x \in \big(f^{[k]} \pazo \mbfa_{\varpi}^{\textpd} \otimes_{\ZZ_p} T\big) \setminus \big(f^{[k+1]} \pazo \mbfa_{\varpi}^{\textpd} \otimes_{\ZZ_p} T\big)$ such that $f^{[k]} x \in \big(\pazo \mbfa_{R, \varpi}^{\textpd} \otimes_{\mbfa_R^+} \mbfn(T)\big)$ then $f^{[k]} x \in \Fil^k \big(\pazo \mbfa_{R,\varpi}^{\textpd} \otimes_{\mbfa_R^+} \mbfn(T)\big)$.

	Note that the claim is true for $k = 0$.
	So let $k \geq 1$ and $f$ as above.
	Let $f^{[k]} x \in \big(f^{[k]} \pazo \mbfa_{\varpi}^{\textpd} \otimes_{\ZZ_p} T\big) \setminus \big(f^{[k+1]} \pazo \mbfa_{\varpi}^{\textpd} \otimes_{\ZZ_p} T\big)$ such that $f^{[k]} x \in \big(\pazo \mbfa_{R,\varpi}^{\textpd} \otimes_{\mbfa_R^+} \mbfn(T)\big)$.
	Since $x \neq 0$, by induction on $k$ we may assume that $x = \sum_{i=1}^h x_i e_i \in \pazo \mbfa_{\varpi}^+ \otimes_{\ZZ_p} T$ with either $x_i=0$ or $x_i \in f^k \pazo \mbfa_{\varpi}^+ \setminus f^{k+1} \pazo \mbfa_{\varpi}^+$ for each $1 \leq i \leq h$ and at least one $x_i \neq 0$.
	Recall that we have $\pi^s \mbfa^+ \otimes_{\ZZ_p} T \subset \mbfa^+ \otimes_{\mbfa_R^+} \mbfn(T)$, therefore $\pi^s x \in \pazo \mbfa_{\varpi}^+ \otimes_{\mbfa_R^+} \mbfn(T)$.
	But then inside $\pazo \mbfb_{\crys}(\overline{R}) \otimes_{\mbfa_R^+} \mbfn(T)$ we must have
	\begin{equation*}
		f^{k} x = k!f^{[k]} x \in \big(\pazo \mbfa_{R,\varpi}^{\textpd} \otimes_{\mbfa_R^+} \mbfn(T)\big) \cap \tfrac{1}{\pi^s} \big(\pazo \mbfa_{\varpi}^+ \otimes_{\mbfa_R^+} \mbfn(T)\big) = \pazo \mbfa_{R,\varpi}^+ \otimes_{\mbfa_R^+} \mbfn(T).
	\end{equation*}
	Therefore, $f^k x \in \big(f^k \pazo \mbfa_{\varpi}^+ \otimes_{\ZZ_p} T) \cap \big(\pazo \mbfa_{R,\varpi}^+ \otimes_{\mbfa_R^+} \mbfn(T)\big) = \Fil^k\big(\pazo \mbfa_{R,\varpi}^+ \otimes_{\mbfa_R^+} \mbfn(T)\big)$ where the last equality follows from \eqref{eq:oarpiplus_fil}.
	Hence, inside $\pazo \mbfb_{\crys}(\overline{R}) \otimes_{\mbfa_R^+} \mbfn(T)$ we have $f^{[k]} x \in \frac{1}{k!} \Fil^k \big(\pazo \mbfa_{R,\varpi}^+ \otimes_{\mbfa_R^+} \mbfn(T)\big) \cap \big(\pazo \mbfa_{R,\varpi}^{\textpd} \otimes_{\mbfa_R^+} \mbfn(T)\big) = \Fil^k \big(\pazo \mbfa_{R,\varpi}^{\textpd} \otimes_{\mbfa_R^+} \mbfn(T)\big)$.

\end{proof}

\begin{lem}\label{lem:bpipd_fil}
	For $k \in \NN$ we have
	\begin{equation*}
		\Fil^k \big(\pazo \mbfa_{R,\varpi}^{\textpd} \otimes_{\mbfa_R^+} \mbfa^+) = \big(\pazo \mbfa_{R,\varpi}^{\textpd} \otimes_{\mbfa_R^+} \mbfa^+\big) \cap \Fil^k \pazo \mbfa_{\crys}(\overline{R}) \subset \pazo \mbfa_{\crys}(\overline{R}).
	\end{equation*}
\end{lem}
\begin{proof}
	Recall that filtrations on $\pazo \mbfa_{R,\varpi}^{\textpd}$ and $\pazo \mbfa_{\crys}(\overline{R})$ are compatible (see Remark \ref{rem:oarpd_oacris_fil_comp}).
	Moreover, from \S \ref{subsec:relative_phi_gamma_mod} the inclusion of rings $\mbfa^+ \subset \mbfa_{\varpi}^+ \subset \mbfa_{\inf}(\overline{R})$ is compatible with induced filtration from $\mbfa_{\crys}(\overline{R})$.
	From the discussion in Lemma \ref{lem:fil_oarpdn} we have an isomorphism of rings $\pazo \mbfa_{\varpi}^{\textpd} = \pazo \mbfa_{R,\varpi}^{\textpd} \otimes_{\mbfa_R^+} \mbfa^+ \isomorphic \pazo \mbfa_{R,\varpi}^{\textpd} \otimes_{\mbfa_{R, \varpi}^+} \mbfa_{\varpi}^+$ compatible with tensor product filtrations.
	Now by the description of filtration on the rightmost term we get that $\pazo \mbfa_{\varpi}^{\textpd}$ is equipped with filtration by divided powers of the ideal $(\xi, X_1-[X_1^{\flat}], \ldots, X_d-[X_d^{\flat}])\pazo \mbfa_{\varpi}^{\textpd}$.
	Finally, the natural multiplication map $\pazo \mbfa_{R,\varpi}^{\textpd} \otimes_{\mbfa_{R, \varpi}^+} \mbfa_{\varpi}^+ \rightarrow \pazo \mbfa_{\crys}(\overline{R})$ is injective.
	Hence, it follows that for $k \in \NN$
	\begin{equation*}
		\Fil^k \big(\pazo \mbfa_{R,\varpi}^{\textpd} \otimes_{\mbfa_R^+} \mbfa^+) = \big(\pazo \mbfa_{R,\varpi}^{\textpd} \otimes_{\mbfa_R^+} \mbfa^+\big) \cap \Fil^k \pazo \mbfa_{\crys}(\overline{R}) \subset \pazo \mbfa_{\crys}(\overline{R}).
	\end{equation*}
\end{proof}

\begin{lem}\label{lem:filb_wach}
	For $k \in \NN$ we have $(\Fil^k\mbfa^+ \otimes_{\ZZ_p} T) \cap \mbfn(T) = \Fil^k \mbfn(T)$.
\end{lem}
\begin{proof}
	It is enough to show that $(\Fil^k\mbfb^+ \otimes_{\QQ_p} V) \cap \mbfn(V) = \Fil^k \mbfn(V)$.
	Indeed, from Definition \ref{defi:wach_mod_fil} we have $\Fil^k \mbfn(T) = \Fil^k \mbfn(V) \cap \mbfn(T) = (\Fil^k\mbfb^+ \otimes_{\QQ_p} V) \cap \mbfn(V) \cap \mbfn(T) = (\Fil^k\mbfa^+ \otimes_{\QQ_p} T) \cap \mbfn(T)$ since $\Fil^k \mbfb^+ \cap \mbfa^+ = \Fil^k \mbfa^+$.

	Now let us show the modified claim.
	The inclusion $\Fil^k \mbfn(V) \subset (\Fil^k \mbfb^+ \otimes_{\QQ_p} V)$ is obvious.
	For the converse, we claim that it is enough to show that $(q^k\mbfb^+ \otimes_{\QQ_p} V) \cap \mbfn(V) = q^k \mbfn(V)$.
	Indeed, if we have $x \in (\Fil^k\mbfb^+ \otimes_{\QQ_p} V) \cap \mbfn(V)$ then $\varphi(x) \in (q^k\mbfb^+ \otimes_{\QQ_p} V) \cap \mbfn(V) = q^k \mbfn(V)$, i.e. $x \in \Fil^k \mbfn(V)$.

	The inclusion $q^k \mbfn(V) \subset (q^k\mbfb^+ \otimes_{\QQ_p} V) \cap \mbfn(V)$ is obvious.
	To show the converse, first let us assume that $\mbfn(V)$ is free with $\{f_1, f_2, \ldots, f_h\}$ as a $\mbfb_{R}^+\textrm{-basis}$, and let $\{e_1, \ldots, e_h\}$ be a $\QQ_p\textrm{-basis}$ of $V$.
	Now let $q^k x \in (q^k \mbfb^+ \otimes_{\QQ_p} V) \cap \mbfn(V)$ for $x = \sum_{i=1}^h x_i e_i \in \mbfb^+ \otimes_{\QQ_p} V$.
	We can also write $q^k x = \sum_{i=1}^h y_i f_i \in \mbfn(V)$ with $y_i \in \mbfb_{R}^+$.
	Next, from Proposition \ref{prop:wach_approx_aplus_admis} we have $\pi^s \mbfb^+ \otimes_{\QQ_p} V \subset \mbfb^+ \otimes_{\mbfb_{R}^+} \mbfn(V)$, so we can write 
	\begin{equation*}
		q^k x = \pi^{-s}q^k \sum_{i=1}^h x_i \pi^s e_i = \pi^{-s}q^k\sum_{i=1}^h x_i \sum_{j=1}^h z_{ij} f_j = \pi^{-s} q^k \sum_{i=1}^h (\sum_{j=1}^h x_j z_{ji}) f_i,
	\end{equation*}
	with $z_{ij} \in \mbfb^+$.
	But then we must have $\pi^{-s}q^k\sum_{j=1}^h x_j z_{ji} = y_i$ for all $1 \leq i \leq h$.
	Since $H_{R}$ acts trivially on $\pi$, $q$ and $y_i$, we get that $w_i := \sum_{j=1}^h x_j z_{ji} \in \mbfb_{R}^+$.
	But $y_i \in \mbfb_{R}^+$ and $\pi$ and $q$ are coprime in $\mbfb_{R}^+$ (since $q \equiv p \mod \pi \mbfb_R^+$), therefore we obtain that $w_i \in \pi^s\mbfb_{R}^+$.
	In particular, $y_i \in q^k \mbfb_{R}^+$, therefore $q^k x = \sum_{i=1}^h y_i f_i \in q^k \mbfn(V)$.
	Hence, $(q^k\mbfb^+ \otimes_{\QQ_p} V) \cap \mbfn(V) = q^k \mbfn(V)$.

	Next, if $\mbfn(V)$ is projective (and not free) over $\mbfb_{R}^+$, let $R\prm$ be the $\padic$ completion of a finite \'etale algebra over $R$ such that the scalar extension $\mbfb_{R\prm}^+ \otimes_{\mbfb_{R}^+} \mbfn(V)$ is a free module over $\mbfb_{R\prm}^+$ and $R\prm\big[\frac{1}{p}\big] / R\big[\frac{1}{p}\big]$ is Galois (see Definition \ref{defi:wach_reps}).
	Then we can argue as above and conclude by taking $\Gal\big(R\prm\big[\frac{1}{p}\big] / R\big[\frac{1}{p}\big]\big)\textrm{-invariants}$ of $q^k \mbfb_{R\prm}^+ \otimes_{\mbfb_{R}^+} \mbfn(V)$.
\end{proof}

\subsection{One-dimensional representations}\label{subsec:onedim_reps}

In this section we will show that all one-dimensional crystalline representations are of finite $q\textrm{-height}$ by writing down the corresponding Wach modules precisely.

\begin{prop}\label{prop:one_dim_crys_wach}
	All one-dimensional crystalline representations of $G_{R}$ are of finite $q\textrm{-height}$.
	Furthermore, for a one-dimensional crystalline representation $V$ we have an isomorphism of $R\big[\frac{1}{p}\big]\textrm{-modules}$
	\begin{equation*}
		\big(\pazo\mbfa_{R,\varpi}^{\textpd} \otimes_{\mbfa_R^+} \mbfn(V)\big)^{\Gamma_{R}} \isomorphic \pazo \mbfd_{\crys}(V).
	\end{equation*}
	Therefore, there exists natural isomorphisms
	\begin{equation*}
		\pazo \mbfa_{R,\varpi}^{\textpd} \otimes_{R} \pazo \mbfd_{\crys}(V) \lisomorphic \pazo \mbfa_{R,\varpi}^{\textpd} \otimes_{R} \big(\pazo\mbfa_{R,\varpi}^{\textpd} \otimes_{\mbfa_R^+} \mbfn(V)\big)^{\Gamma_{R}} \isomorphic \pazo \mbfa_{R,\varpi}^{\textpd} \otimes_{\mbfa_R^+} \mbfn(V),
	\end{equation*}
	compatible with Frobenius, filtration and the action of $\Gamma_{R}$.
\end{prop}
\begin{proof}
	The structure of one-dimensional crystalline representations of $G_{R}$ is well-known (see \cite[\S 8.6]{brinon-padicrep-relatif}).
	From Proposition \ref{prop:onedim_unramrep_struct} we have that for $\eta : G_{R} \rightarrow \ZZ_p^{\times}$, a continuous character, $V = \QQ_p(\eta)$ is crystalline if and only if we can write $\eta = \eta_{\fini}\eta_{\unrami} \chi^n$ with $n \in \ZZ$, and where $\eta_{\fini}$ is a finite unramified character, $\eta_{\unrami}$ is an unramified character taking values in $1 + p\ZZ_p$ and trivialized by an element $\alpha \in 1 + p\widehat{R^{\unrami}}$, and $\chi$ is the $\padic$ cyclotomic character.
	Recall that a $\padic$ representation of $G_{R}$ is unramified if the action of $G_{R}$ factorizes through the quotient $G_{R}^{\unrami}$ (see \S \ref{subsec:relative_padic_reps}).
	Moreover, if $\eta_{\fini}$ is trivial then $\pazo\mbfd_{\crys}(V)$ is a free $R\big[\frac{1}{p}\big]\textrm{-module}$ of rank 1.

	In Lemma \ref{lem:onedim_crysrep_wach_module} below, we show that crystalline representations $V_1 := \QQ_p(\eta_{\fini}\eta_{\unrami})$ and $V_2 := \QQ_p(\chi^n)$ are of finite $q\textrm{-height}$.
	For a one-dimensional crystalline representation $V := \QQ_p(\eta) = \QQ_p(\eta_{\fini}\eta_{\unrami}) \otimes_{\QQ_p} \QQ_p(\chi^n) = V_1 \otimes_{\QQ_p} V_2$ as above, by compatibility of tensor products in Propositions \ref{prop:wach_module_sum_tensor} we get that $V$ is a finite $q\textrm{-height}$ representation as well with $\mbfn(V) = \mbfn(V_1) \otimes_{\mbfb_{R}^+} \mbfn(V_2)$.

	Now, from the isomorphisms of $\pazo \mbfa_{R,\varpi}^{\textpd}\textrm{-modules}$ in Lemma \ref{lem:onedim_crysrep_wach_module} and compatibility of tensor product of Wach modules in Proposition \ref{prop:wach_module_sum_tensor} and compatibility of the functor $\pazo \mbfd_{\crys}$ with tensor products in \S \ref{subsec:relative_padic_reps} (see also \cite[Th\'eor\`eme 8.4.2]{brinon-padicrep-relatif}), we get a string of isomorphisms of $\pazo\mbfb_{R, \varpi}^{\textpd} := \pazo \mbfa_{R,\varpi}^{\textpd}\big[\frac{1}{p}\big]\textrm{-modules}$ compatible with Frobenius, filtration and the action of $\Gamma_{R}$,
	\begin{align*}
		\pazo \mbfa_{R,\varpi}^{\textpd} \otimes_{R} \pazo \mbfd_{\crys}(V) &\isomorphic \big(\pazo \mbfa_{R,\varpi}^{\textpd} \otimes_{R} \pazo \mbfd_{\crys}(V_1)\big) \otimes_{\pazo \mbfb_{R, \varpi}^{\textpd}} \big(\pazo \mbfa_{R,\varpi}^{\textpd} \otimes_{R} \pazo \mbfd_{\crys}(V_2)\big)\\
		&\lisomorphic \big(\pazo \mbfa_{R,\varpi}^{\textpd} \otimes_{\mbfa_R^+} \mbfn(V_1)\big) \otimes_{\pazo \mbfb_R^{\textpd}} \big(\pazo \mbfa_{R,\varpi}^{\textpd} \otimes_{\mbfa_R^+} \mbfn(V_2)\big)\\
		&\isomorphic \pazo \mbfa_{R,\varpi}^{\textpd} \otimes_{\mbfa_R^+} \mbfn(V_1) \otimes_{\mbfb_{R}^+} \mbfn(V_2)\\
		&\isomorphic \pazo \mbfa_{R,\varpi}^{\textpd} \otimes_{\mbfa_R^+} \mbfn(V_1 \otimes_{\QQ_p} V_2) \isomorphic \pazo \mbfa_{R,\varpi}^{\textpd} \otimes_{\mbfa_R^+} \mbfn(V).
	\end{align*}
	Taking $\Gamma_{R}\textrm{-invariants}$ of the first and the last term gives us that $\pazo \mbfd_{\crys}(V) \isomorphic \big(\pazo \mbfa_{R,\varpi}^{\textpd} \otimes_{\mbfa_R^+} \mbfn(V)\big)^{\Gamma_{R}}$, compatible with Frobenius and filtration.
\end{proof}

Following claim was used above:
\begin{lem}\phantomsection\label{lem:onedim_crysrep_wach_module}
	\begin{enumromanup}
	\item Let $\eta : G_{R} \rightarrow \ZZ_p^{\times}$ be a continuous unramified character.
		Then the $\padic$ representation $\QQ_p(\eta)$ is a finite $q\textrm{-height}$ representation.

	\item Let $\chi$ be the $\padic$ cyclotomic character then for $n \in \ZZ$, the $\padic$ representation $\QQ_p(n)$ is a finite $q\textrm{-height}$ representation.
	\end{enumromanup}

	Further, for $V = \QQ_p(\eta), \QQ_p(n)$ we have an isomorphism of $R\big[\frac{1}{p}\big]\textrm{-modules}$
	\begin{equation*}
		\big(\pazo\mbfa_{R,\varpi}^{\textpd} \otimes_{\mbfa_R^+} \mbfn(V)\big)^{\Gamma_{R}} \isomorphic \pazo \mbfd_{\crys}(V).
	\end{equation*}
	Therefore, there exists natural isomorphisms
	\begin{equation*}
		\pazo \mbfa_{R,\varpi}^{\textpd} \otimes_{R} \pazo \mbfd_{\crys}(V) \lisomorphic \pazo \mbfa_{R,\varpi}^{\textpd} \otimes_{R} \big(\pazo\mbfa_{R,\varpi}^{\textpd} \otimes_{\mbfa_R^+} \mbfn(V)\big)^{\Gamma_{R}} \isomorphic \pazo \mbfa_{R,\varpi}^{\textpd} \otimes_{\mbfa_R^+} \mbfn(V),
	\end{equation*}
	compatible with Frobenius, filtration and the action of $\Gamma_{R}$.
\end{lem}
\begin{proof}
	Let $\eta = \eta_{\fini} \eta_{\unrami}$, where $\eta_{\fini}$ is an unramified character of finite order and $\eta_{\unrami}$ is an unramified character taking values in $1+p\ZZ_p$ and trivialised by an element $\alpha \in 1+p\widehat{R^{\unrami}}$ (see Proposition \ref{prop:onedim_unramrep_struct}).

	First, let us consider the finite unramified character $\eta_{\fini}$.
	Set $T = \ZZ_p(\eta_{\fini}) = \ZZ_p e$, such that $g(e) = \eta_{\fini}(g)e$.
	We have
	\begin{equation*}
		\mbfd^+\big(\ZZ_p(\eta_{\fini})\big) = \big(\mbfa^+ \otimes_{\ZZ_p} \ZZ_p(\eta_{\fini})\big)^{H_{R}} \isomorphic \big\{a \otimes e, \hspace{0.5mm} \textrm{with} \hspace{1mm} a \in \mbfa^+ \hspace{1mm} \textrm{such that} \hspace{1mm} g(a) = \eta_{\fini}^{-1}(g)a, \hspace{0.5mm} \textrm{for} \hspace{1mm} g \in H_{R}\big\}.
	\end{equation*}
	Since $\eta_{\fini}$ is a finite unramified character, it trivializes over a finite Galois extension $S$ over $R$ (see \cite[Proposition 8.6.1]{brinon-padicrep-relatif}), and we have that $\Gal\big(S\big[\frac{1}{p}\big] / R\big[\frac{1}{p}\big]\big) = G_{R} / G_{S} = H_{R} / H_{S} = \Gamma_{R} / \Gamma_{S}$.
	As $S$ is finite \'etale over $R$ the construction of previous chapters apply and we obtain that the $\mbfa_{S}^+\textrm{-module}$ $\mbfd_{S}^+\big(\ZZ_p(\eta_{\fini})\big) = \big(\mbfa^+ \otimes_{\ZZ_p} \ZZ_p(\eta_{\fini})\big)^{H_{S}} = \mbfa_{S}^+ (\eta_{\fini}) = \mbfa_{S}^+ e$ is free of rank $1$.
	Further, we know that $\mbfd^+\big(\ZZ_p(\eta_{\fini})\big) = \mbfd_{S}^+\big(\ZZ_p(\eta_{\fini})\big)^{H_{R} / H_{S}}$, which implies that the natural inclusion
	\begin{equation*}
		\mbfa_{S}^+ \otimes_{\mbfa_R^+} \mbfd^+\big(\ZZ_p(\eta_{\fini})\big) \longrightarrow \mbfd_{S}^+\big(\ZZ_p(\eta_{\fini})\big),
	\end{equation*}
	is bijective.
	Now, since $\mbfa_R^+ \rightarrow \mbfa_{S}^+$ is faithfully flat, we obtain that $\mbfd^+\big(\ZZ_p(\eta_{\fini})\big)$ is projective of rank $1$.
	Moreover, $\mbfd^+\big(\ZZ_p(\eta_{\fini})\big)$ admits a Frobenius-semilinear endomorphism $\varphi$ such that $\mbfd^+\big(\ZZ_p(\eta_{\fini})\big) \isomorphic \varphi^{\ast}\big(\mbfd^+\big(\ZZ_p(\eta_{\fini})\big)\big)$ (one can obtain this after faithfully flat scalar extension $\mbfa_R^+ \rightarrow \mbfa_{S}^+$ and applying descent as above, since $\varphi$ commutes with $G_{R}\textrm{-action}$).
	The action of $\Gamma_{R}$ is trivial on $\mbfd^+\big(\ZZ_p(\eta_{\fini})\big)$.
	Now, we can take $\mbfn\big(\ZZ_p(\eta_{\fini})\big) = \mbfd^+\big(\ZZ_p(\eta_{\fini})\big)$.
	From the discussion above, $\mbfn\big(\ZZ_p(\eta_{\fini})\big)$ clearly satisfies the conditions of Definition \ref{defi:wach_reps}.
	Also, we have that $\mbfn(\QQ_p(\eta_{\fini})) = \mbfd^+(\QQ_p(\eta_{\fini}))$.
	On the other hand, we have
	\begin{equation*}
		\pazo \mbfd_{\crys}\big(\QQ_p(\eta_{\fini})\big) = \big(\pazo \mbfb_{\crys}(\overline{R}) \otimes_{\QQ_p} \QQ_p(\eta_{\fini})\big)^{G_{R}} = \big\{b \otimes e, \hspace{1mm} \textrm{with} \hspace{1mm} b \in \pazo\mbfb_{\crys}(\overline{R}) \hspace{1mm} \textrm{such that} \hspace{1mm} g(b) = \eta_{\fini}(g)b\big\}.
	\end{equation*}
	Since $\eta_{\fini}$ trivializes over the finite Galois extension $S$ over $R$, we have
	\begin{equation*}
		\big(\pazo \mbfa_{S, \varpi}^{\textpd} \otimes_{\mbfa_R^+} \mbfn\big(\QQ_p(\eta_{\fini})\big)\big)^{\Gamma_{S}} = S_0\big[\tfrac{1}{p}\big]e = \big(\pazo \mbfb_{\crys}(\overline{S}) \otimes_{\QQ_p} \QQ_p(\eta_{\fini})\big)^{G_{S}},
	\end{equation*}
	where the rings $\pazo \mbfa_{S, \varpi}^{\textpd}$ and $\pazo \mbfb_{\crys}(\overline{S})$ are defined for $S$ over which all the construction of previous sections apply (since $S$ is finite \'etale over $R$).
	Now taking invariants under the finite Galois group $\Gal\big(S\big[\frac{1}{p}\big] / R\big[\frac{1}{p}\big]\big) = G_{R} / G_{S}$, gives us
	\begin{equation*}
		\big(\pazo \mbfa_{R,\varpi}^{\textpd} \otimes_{\mbfa_R^+} \mbfn\big(\QQ_p\big(\eta_{\fini}\big)\big)\big)^{\Gamma_{R}} = \pazo \mbfd_{\crys}\big(\QQ_p(\eta_{\fini})\big).
	\end{equation*}
	Clearly, the natural maps
	\begin{equation*}
		\pazo \mbfa_{R,\varpi}^{\textpd} \otimes_{R} \pazo \mbfd_{\crys}\big(\QQ_p(\eta_{\fini})\big) \lisomorphic \pazo \mbfa_{R,\varpi}^{\textpd} \otimes_{R} \big(\pazo \mbfa_{R,\varpi}^{\textpd} \otimes_{\mbfa_R^+} \mbfn\big(\QQ_p(\eta_{\fini})\big)\big)^{\Gamma_{R}} \isomorphic \pazo \mbfa_{R,\varpi}^{\textpd} \otimes_{\mbfa_R^+} \mbfn\big(\QQ_p(\eta_{\fini})\big),
	\end{equation*}
	are isomorphisms compatible with Frobenius, filtration and the action of $\Gamma_{R}$.

	Next, let us consider the unramified character $\eta_{\unrami}$ which takes values in $1+p\ZZ_p$ and trivialised by an element $\alpha \in 1+p\widehat{R^{\unrami}}$ (see Proposition \ref{prop:onedim_unramrep_struct}).
	Set $T = \ZZ_p(\eta_{\unrami}) = \ZZ_p e$, such that $g(e) = \eta_{\unrami}(g)e$.
	We have
	\begin{equation*}
		\mbfd^+\big(\ZZ_p(\eta_{\unrami})\big) = \big(\mbfa^+ \otimes_{\ZZ_p} \ZZ_p(\eta_{\unrami})\big)^{H_{R}} = \mbfa_R^+ \alpha e.
	\end{equation*}
	So we take $\mbfn\big(\ZZ_p(\eta_{\unrami})\big) = \mbfd^+\big(\ZZ_p(\eta_{\unrami})\big) = \mbfa_R^+ \alpha e$.
	This clearly satisfies the conditions of Definition \ref{defi:wach_reps}.
	Also, we have that $\mbfn(\QQ_p(\eta_{\unrami})) = \mbfd^+(\QQ_p(\eta_{\unrami}))$.
	On the other hand, we have
	\begin{align*}
		\pazo \mbfd_{\crys}\big(\QQ_p(\eta_{\unrami})\big) &= \big(\pazo \mbfb_{\crys}(\overline{R}) \otimes_{\QQ_p} \QQ_p(\eta_{\unrami})\big)^{G_{R}}\\
		&= \big\{b \otimes e, \hspace{1mm} \textrm{with} \hspace{1mm} b \in \pazo\mbfb_{\crys}(\overline{R}) \hspace{1mm} \textrm{such that} \hspace{1mm} g(b) = \eta_{\unrami}(g)b\big\} = R\big[\tfrac{1}{p}\big] \alpha e.
	\end{align*}
	Therefore, we obtain
	\begin{equation*}
		\big(\pazo \mbfa_{R,\varpi}^{\textpd} \otimes_{\mbfa_R^+} \mbfn\big(\QQ_p(\eta_{\unrami})\big)\big)^{\Gamma_{R}} = R\big[\tfrac{1}{p}\big] \alpha e = \big(\pazo \mbfb_{\crys}(\overline{R}) \otimes_{\QQ_p} \QQ_p(\eta_{\unrami})\big)^{G_{R}}.
	\end{equation*}
	Clearly, the natural maps
	\begin{equation*}
		\pazo \mbfa_{R,\varpi}^{\textpd} \otimes_{R} \pazo \mbfd_{\crys}\big(\QQ_p(\eta_{\unrami})\big) \lisomorphic \pazo \mbfa_{R,\varpi}^{\textpd} \otimes_{R} \big(\pazo \mbfa_{R,\varpi}^{\textpd} \otimes_{\mbfa_R^+} \mbfn\big(\QQ_p(\eta_{\unrami})\big)\big)^{\Gamma_{R}} \isomorphic \pazo \mbfa_{R,\varpi}^{\textpd} \otimes_{\mbfa_R^+} \mbfn\big(\QQ_p(\eta_{\unrami})\big),
	\end{equation*}
	are isomorphisms compatible with Frobenius, filtration and the action of $\Gamma_{R}$.
	
	Finally, let $T = \ZZ_p(n) = \ZZ_p e_n$ such that $g(e_n) = \chi(g)^n e_n$, then $V = \QQ_p \otimes_{\ZZ_p} T$ is a crystalline representation.
	In this case, we can take $\mbfn\big(\ZZ_p(n)\big) = \mbfa_R^+ \pi^{-n} e_n$.
	Note that for $n \leq 0$, we have that $\mbfn\big(\ZZ_p(n)\big) / \varphi^{\ast}\big(\mbfn\big(\ZZ_p(n)\big)\big)$ is killed by $q^{-n}$, where $q = \frac{\varphi(\pi)}{\pi}$.
	It can easily be verified that $\Gamma_R$ acts trivially modulo $\pi$ on $\mbfn(T)$.
	So, we set $\mbfn\big(\QQ_p(n)\big) = \mbfb_{R}^+ \pi^{-n} e_n$.
	Similarly,
	\begin{equation*}
		\pazo \mbfd_{\crys}\big(\QQ_p(n)\big) = \big(\pazo \mbfb_{\crys}(\overline{R}) \otimes_{\QQ_p} \QQ_p(n)\big)^{G_{R}} = R\big[\tfrac{1}{p}\big] t^{-n} e_n,
	\end{equation*}
	and $\big(\pazo \mbfa_{R,\varpi}^{\textpd} \otimes_{\mbfa_R^+} \mbfn\big(\QQ_p(n)\big)\big)^{\Gamma_{R}} = R\big[\frac{1}{p}\big] t^{-n} e_n = \pazo \mbfd_{\crys}\big(\QQ_p(n)\big)$ compatible with Frobenius, filtration and connection on each side.
	Finally, the map
	\begin{align*}
		\pazo \mbfa_{R,\varpi}^{\textpd} \otimes_{R} \pazo \mbfd_{\crys}\big(\QQ_p(n)\big) &\longrightarrow \pazo \mbfa_{R,\varpi}^{\textpd} \otimes_{\mbfa_R^+} \mbfn\big(\QQ_p(n)\big)\\
		t^{-n}e_n &\longmapsto \tfrac{\pi^n}{t^n}\pi^{-n}e_n.
	\end{align*}
	is trivially an isomorphism compatible with Frobenius, filtration and the action of $\Gamma_{R}$, since $\frac{\pi^n}{t^n} \in \pazo \mbfa_{R,\varpi}^{\textpd}$ are units for $n \in \ZZ$ (see Lemma \ref{lem:t_over_pi_unit}).
	This proves the lemma.
\end{proof}

\begin{rem}\label{rem:onedim_crysrep_wachmod_intiso}
	Note that for $T = \ZZ_p\big(\eta_{\fini}\eta_{\unrami}\big)$ or $\ZZ_p(n)$, we even have an isomorphism on the integral level
	\begin{equation*}
\pazo \mbfa_{R,\varpi}^{\textpd} \otimes_{R} \big(\pazo \mbfa_{R,\varpi}^{\textpd} \otimes_{\mbfa_R^+} \mbfn(T)\big)\big)^{\Gamma_{R}} \isomorphic \pazo \mbfa_{R,\varpi}^{\textpd} \otimes_{\mbfa_R^+} \mbfn(T).
	\end{equation*}
\end{rem}

\cleardoublepage

\section{Relative Fontaine-Laffaille modules}\label{sec:fontaine_laffaile_to_wach}

In this section we will consider relative Fontaine-Laffaille data and construct Wach modules given such data.
Carrying out such a process would involve starting with a module over $R$ and constructing modules over the ring $\mbfa_{R, \varpi}^{\textpd}$ and $\mbfa_{R, \varpi}^+$, and finally descending over to the ring $\mbfa_{R}^+$.

Explicitly, we will work with objects in the category $\MF_{[0, p-2], \free}(R, \Phi, \partial)$, defined by \cite[\S 4]{tsuji-ainf-genrep} as a full subcategory of the abelian category $\mathfrak{MF}_{[0, p-2], \free}^{\nabla}(R)$ which was introduced by Faltings in \cite[\S II]{faltings-crystalline}.
In particular,
\begin{defi}\label{defi:rel_fontaine_laffaille}
	Define the category of \textit{free relative Fontaine-Laffaille} modules of level $[0, p-2]$, denoted by $\MF_{[0, p-2], \free}(R, \Phi, \partial)$, as follows:\\
	An object with weights in the interval $[0, p-2]$ is a quadruple $(M, \Fil^{\bullet} M, \partial, \Phi)$ such that,
	\begin{enumromanup}
	\item $M$ is a free $R\textrm{-module}$ of finite rank.

	\item $M$ is equipped with a decreasing filtration $\{\Fil^k M\}_{k \in \ZZ}$ by finite $R\textrm{-submodules}$ with $\Fil^0 M = M$ and $\Fil^{s+1} M = 0$ such that $\gr^k_{\Fil} M$ is a finite free $R\textrm{-module}$ for every $k \in \ZZ$.

	\item The connection $\partial : M \rightarrow M \otimes_{R} \Omega^1_{R}$ is $p\textrm{-adically}$ quasi-nilpotent and integrable, and satisfies Griffiths transversality with respect to the filtration, i.e. $\partial(\Fil^k M) \subset \Fil^{k-1} M \otimes_{R} \Omega^1_{R}$ for $k \in \ZZ$.	
	
	\item Let $(\varphi^{\ast}(M), \varphi^{\ast}(\partial))$ denote the pullback of $(M, \partial)$ by $\varphi : R \rightarrow R$, and equip it with a decreasing filtration $\Fil^k_p(\varphi^{\ast}(M)) = \sum_{i \in \NN} p^{[i]} \varphi^{\ast}(\Fil^{k-i} M)$ for $k \in \ZZ$.
		We suppose that there is an $R\textrm{-linear}$ morphism $\Phi : \varphi^{\ast}(M) \rightarrow M$ such that $\Phi$ is compatible with connections, $\Phi\big(\Fil^k_p(\varphi^{\ast}(M))\big) \subset p^k M$ for $0 \leq k \leq s$, and $\sum_{k=0}^{s} p^{-k}\Phi\big(\Fil^k_p(\varphi^{\ast}(M))\big) = M$.
		We denote the composition $M \rightarrow \varphi^{\ast}(M) \xrightarrow{\Phi} M$ by $\varphi$.
	\end{enumromanup}
	A morphism between two objects of the category $\MF_{[0, p-2], \free}(R, \Phi, \partial)$ is a continuous $R\textrm{-linear}$ map compatible with the homomorphism $\Phi$, the connection $\partial$ and filtration on each side.
\end{defi}

\begin{nota}
	By a slight abuse of notations, we will denote $(M, \Fil^k M, \partial, \Phi) \in \MF_{[0, p-2], \free}(R, \Phi, \partial)$ by $M$ and say that it is of level $[0, p-2]$.
\end{nota}

To an object $M \in \MF_{[0, p-2], \free}(R, \Phi, \partial)$, we associate a $\ZZ_p\textrm{-module}$ as
\begin{equation}
	T_{\crys}^{\ast}(M) := \Hom_{R, \hspace{0.3mm}\Fil, \hspace{0.3mm}\varphi, \hspace{0.3mm}\partial}(M, \pazo \mbfa_{\crys}(\overline{R})),
\end{equation}
i.e. $R\textrm{-linear}$ maps from $M$ to $\pazo \mbfa_{\crys}(\overline{R})$ compatible with Frobenius, filtration and connection, where we have $\varphi : M \rightarrow \varphi^{\ast}(M) \xrightarrow{\Phi} M$.

\begin{prop}\phantomsection\label{prop:tcrysast_fl}
	\begin{enumromanup}
	\item For a free Fontaine-Laffaille module $M$ of level $[0, p-2]$, the $\ZZ_p\textrm{-module}$ $T_{\crys}^{\ast}(M)$ is a free module of rank $=\textup{rk}_{R} M$ equipped with a continuous action of $G_{R}$.
		Further, the $\padic$ representation $V_{\crys}^{\ast}(M) := \QQ_p \otimes_{\ZZ_p} T_{\crys}^{\ast}(M)$ is a crystalline representation of $G_{R}$ with Hodge-Tate weights in the interval $[0, p-2]$.

	\item The contravariant $\ZZ_p\textrm{-linear}$ functor
		\begin{equation*}
			T_{\crys}^{\ast} : \MF_{[0, p-2], \free}(R, \Phi, \partial) \longrightarrow \Rep_{\ZZ_p, \free}(G_{R}),
		\end{equation*}
		is fully faithful.
		Here $\Rep_{\ZZ_p, \free}(G_{R})$ denotes the category of finite free $\ZZ_p\textrm{-modules}$ equipped with a continuous action of $G_{R}$.
	\end{enumromanup}
\end{prop}
\begin{proof}
	The claim in (i) follows from \cite[Theorem 2.4]{faltings-crystalline} and \cite[Proposition 66]{tsuji-ainf-genrep}.
	Further, the claim in (ii) follows from \cite[Theorem 2.4]{faltings-crystalline} and \cite[Theorem 77]{tsuji-ainf-genrep}.
\end{proof}

\begin{defi}\label{defi:tcrys_m}
	Let $M$ be a free relative Fontaine-Laffaille module of level $[0, p-2]$, and set 
	\begin{equation*}
		T_{\crys}(M) := \Hom_{\ZZ_p}(T_{\crys}^{\ast}(M), \ZZ_p),
	\end{equation*}
	which is a free $\ZZ_p\textrm{-module}$ of rank $=\textup{rk}_{R} M$, admitting a continuous action of $G_{R}$.
\end{defi}

The main result of this section is as follows:
\begin{thm}\label{thm:fl_to_wach}
	For a free relative Fontaine-Laffaille module $M$ over $R$ of level $[0, p-2]$, the associated representation $V_{\crys}(M) := \QQ_p \otimes_{\ZZ_p} T_{\crys}(M)$ is a positive finite $q\textrm{-height}$ representation (in the sense of Definition \ref{defi:wach_reps}).
\end{thm}

The proof crucially exploits the computation of Fontaine \cite{fontaine-corps-des-periodes}, Wach \cite{wach-cristallines-torsion} and Tsuji \cite{tsuji-ainf-genrep}.
It follows in three steps:
First, starting with a Fontaine-Laffaille module, we obtain an $\mbfa_{R, \varpi}^{\textpd}\textrm{-module}$ using formal consequences of crystalline site for maps $\theta : \mbfa_{R, \varpi}^{\textpd} \twoheadrightarrow R[\varpi]$, and $\theta_{R} : \pazo \mbfa_{R, \varpi}^{\textpd} \twoheadrightarrow R[\varpi]$ (see Proposition \ref{prop:arpd_mod_fl_data}, we also give an alternate proof of the proposition).
Next, we exploit equivalence of categories in Theorem \ref{thm:arplus_arpd_cat_equiv} obtained by scalar extension along the maps $\mbfa_{R, \varpi}^{\textpd} \twoheadrightarrow \mbfa_{R, \varpi}^{\textpd} / I^{(p-1)} \mbfa_{R, \varpi}^{\textpd} \lisomorphic \mbfa_{R, \varpi}^+ / I^{(p-1)} \mbfa_{R, \varpi}^+ \twoheadleftarrow \mbfa_{R, \varpi}^+$.
This gives us an $\mbfa_{R, \varpi}^+\textrm{-module}$ with precise description of the Frobenius and the action of $\Gamma_{R}$ (see Proposition \ref{prop:fl_data_arplus_mod}).
Finally, we descend over to the ring $\mbfa_{R}^+$ by exploiting the Frobenius and $\Gamma_{R}\textrm{-action}$, thus obtaining a Wach module over $\mbfa_{R}^+$ and proving the theorem (see \S \ref{subsubsec:obtain_wach_mod}).

For clarity of exposition and notational convenience in explaining the result of the first step, we start with preliminaries on some ideals of $\mbfa_{R, \varpi}^+$ and $\mbfa_{R, \varpi}^{\textpd}$ (appearing in the second step in the paragraph above) which will help us in proving categorical equivalence between certain modules over the concerned rings.

\subsection{Some ideals of \texorpdfstring{$\mbfa_{R, \varpi}^+$}{--} and \texorpdfstring{$\mbfa_{R, \varpi}^{\textpd}$}{--}}\label{subsec:ideals_arplus_arpd}

In this section, we will collect some technical results about the rings $\mbfa_{R, \varpi}^+$ and $\mbfa_{R, \varpi}^{\textpd}$ and some of their ideals.
The results are motivated by the corresponding results over $\mbfa_{\inf}(\overline{R})$ and $\mbfa_{\crys}(\overline{R})$ and their respective ideals, studied in \cite[\S 5]{fontaine-corps-des-periodes}.

\begin{lem}\phantomsection\label{lem:arplus_pa_complete}
	Let $a \in \mbfa_{R, \varpi}^+$ such that $\mbfa_{R, \varpi}^+ / p \mbfa_{R, \varpi}^+$ is $a\textrm{-torsion}$ free and $a\textrm{-adically}$ complete.
	Then,
	\begin{enumromanup}
	\item $\mbfa_{R, \varpi}^+$ is $(p, a)\textrm{-adically}$ complete.

	\item For $n \in \NN$, the rings $\mbfa_{R, \varpi}^+ / a^n \mbfa_{R, \varpi}^+$ are $p\textrm{-torsion}$ free and $\padic$ally complete.
		
	\item For $n \in \NN$, $\mbfa_{R, \varpi}^+$ and $\mbfa_{R, \varpi}^+ / p^n \mbfa_{R, \varpi}^+$ are $a\textrm{-torsion}$ free and $a\textrm{-adically}$ complete.

	\item The $(p, a)\textrm{-adic}$ topology coincides with $(p, \pi_m)\textrm{-adic}$ topology.
	\end{enumromanup}
\end{lem}
\begin{proof}
	As $\mbfa_{R, \varpi}^+$ is a flat $\ZZ_p\textrm{-algebra}$, claims (i), (ii) and (iii) follow from \cite[Lemma 2]{tsuji-ainf-genrep}.
	The last claim follows from \cite[Lemma 1]{tsuji-ainf-genrep} and the fact that $\mbfa_{R, \varpi}^+ \subset \mbfa_{\inf}(\overline{R})$, where the former ring is equipped with the induced topology.
\end{proof}

For $n \in \NN$, let us write $n = (p-1)f(n) + r(n)$, with $r(n), f(n) \in \NN$ and $0 \leq r(n) < p-1$.
Let $t^{\{n\}} := \frac{t^n}{p^{f(n)} f(n)!}$ (resp. $t^{\{n\}} := \frac{t^n}{p^n n!}$ if $p=2$).
\begin{lem}
	We have $t^{p-1} \in p \mbfa_{R, \varpi}^{\textpd}$, therefore $t^{\{n\}} \in \mbfa_{R, \varpi}^{\textpd}$.
\end{lem}
\begin{proof}
	Note that we have $q = \frac{\varphi(\pi)}{\pi} = p \varphi\big(\frac{\pi}{t}\big)\frac{t}{\pi}$.
	Since $\frac{t}{\pi}$ is a unit in $\mbfa_{R, \varpi}^{\textpd}$ (see Lemma \ref{lem:t_over_pi_unit}), we get that $q$ and $p$ are associates in $\mbfa_{R, \varpi}^{\textpd}$.
	But also, $q = \frac{\varphi(\pi)}{\pi} = \pi^{p-1} + p(\pi^{p-2} + \cdots + 1)$, i.e. $\pi^{p-1} \in p \mbfa_{R, \varpi}^{\textpd}$.
	Again, using Lemma \ref{lem:t_over_pi_unit}, we get that $t^{p-1} \in p \mbfa_{R, \varpi}^{\textpd}$.
\end{proof}
Note that we also have $\pi = \exp(t) - 1 = \sum_{n \geq 1} \frac{t^n}{n!} = \sum_{n \geq 1} c_n t^{\{n\}}$, where $c_n = \frac{p^{f(n)}f(n)!}{n!}$ (resp. $c_n = 2^n$ if $p=2$) such that $c_n \rightarrow 0$ as $n \rightarrow +\infty$ (see \cite[\S 5.2.4]{fontaine-corps-des-periodes}).
Let 
\begin{equation*}
	\Lambda := \big\{\sum_{n \in \NN} a_n t^{\{n\}} \hspace{1mm} \textrm{with} \hspace{1mm} a_n \in O_F \hspace{1mm} \textrm{such that} \hspace{1mm} a_n = 0 \hspace{1mm} \textrm{if} \hspace{1mm} (p-1) \nmid \hspace{1mm} n \textrm{ (resp. } 2 \nmid n \textrm{ if } p=2\textrm{)}\big\}
\end{equation*}
be a ring and let $z = \sum_{a \in \FF_p} [\varepsilon]^{[a]}$ (resp. $z = [\varepsilon] + [\varepsilon]^{-1}$ if $p=2$) and $\pi_0 = z - p$, then we have $\pi_0 = (p-1) \sum_{n \geq 1, \hspace{0.5mm} p-1 | n} \tfrac{t^n}{n!} \in \Lambda$ (resp. $\pi_0 = 2 \sum_{n \geq 1, \hspace{0.5mm} 2 | n} \tfrac{t^n}{n!} \in \Lambda$ if $p=2$).
Further, we have that $\pi_0 \in p\Lambda$ (resp. $\pi_0 \in 8\Lambda$ if $p=2$) and there exists $\upsilon \in \Lambda^{\times}$ such that $\frac{\pi_0}{p} = v \frac{t^{p-1}}{p}$ (resp. $\frac{\pi_0}{8} = v \frac{t^{2}}{8}$ if $p=2$), see \cite[\S 5.2.5]{fontaine-corps-des-periodes}.

Next, recall that the filtration on $\mbfa_{\crys}(\overline{R})$ is given as $\Fil^k \mbfa_{\crys}(\overline{R}) = \langle\xi^{[n]}, \hspace{1mm} n \geq k\rangle \subset \mbfa_{\crys}(\overline{R})$, for $k \in \NN$ (see \S \ref{subsec:relative_crystalline_period_rings}).
The filtration on $\mbfa_{\inf}(\overline{R})$ is defined as the induced filtration, i.e. $\Fil^k \mbfa_{\inf}(\overline{R}) = \Fil^k \mbfa_{\crys}(\overline{R}) \cap \mbfa_{\inf}(\overline{R}) = \xi^k \mbfa_{\inf}(\overline{R})$.
Similarly, the filtration on $\mbfa_{R, \varpi}^{\textpd}$ is again given by divided powers of $\xi$, i.e. $\Fil^k \mbfa_{R, \varpi}^{\textpd} = \langle\xi^{[n]}, \hspace{1mm} n \geq k\rangle \subset \mbfa_{R, \varpi}^{\textpd}$, for $k \in \NN$ (see Definition \ref{defi:filtration_vanishing_varpi}).
The filtration on $\mbfa_{R, \varpi}^+$ is defined as the induced filtration, i.e. $\Fil^k \mbfa_{R, \varpi}^+ = \Fil^k \mbfa_{R, \varpi}^{\textpd} \cap \mbfa_{R, \varpi}^+ = \xi^k \mbfa_{\inf}(\overline{R})$.

Now, for $k \in \NN$ let us define an ideal of $\mbfa_{\inf}(\overline{R})$ as 
\begin{equation*}
	I^{(k)} \mbfa_{\inf}(\overline{R}) = \{x \in \mbfa_{\inf}(\overline{R}) \hspace{1mm} \textrm{such that} \hspace{1mm} \varphi^n(x) \in \Fil^k \mbfa_{\inf}(\overline{R}) \hspace{1mm} \textrm{for} \hspace{1mm} n \in \NN\}.
\end{equation*}
Similarly, we can define respective ideals $I^{(k)} \mbfa_{\crys}(\overline{R}) \subset \mbfa_{\crys}(\overline{R})$, $I^{(k)} \mbfa_{R, \varpi}^+ \subset \mbfa_{R, \varpi}^+$ and $I^{(k)} \mbfa_{R, \varpi}^{\textpd} \subset \mbfa_{R, \varpi}^{\textpd}$.
Since the natural map $\mbfa_{R, \varpi}^+ \rightarrow \mbfa_{\inf}(\overline{R})$ is flat and we have $\mbfa_{R, \varpi}^+ = \mbfa_{\inf}(\overline{R}) \cap \mbfa_{R, \varpi} \subset W(\CC(\overline{R})^{\flat})$, we obtain that

\begin{lem}\phantomsection\label{lem:ideal_generator}
	\begin{enumromanup}
	\item The ideal $I^{(k)} \mbfa_{R, \varpi}^+$ is a principal ideal generated by $\pi^k$.

	\item The element $\pi_0$ is a generator of $I^{(p-1)} \mbfa_{R, \varpi}^+$ (resp. $I^{(2)} \mbfa_{R, \varpi}^+$ if $p=2$).

	\item Let $S_0 = W[[\pi_0]]$ then there exists a unit $u \in S_0$ such that $\varphi(\pi_0) = u \pi_0 z^{p-1}$ (resp. $\varphi(\pi_0) = u \pi_0 z^2$ if $p=2$).
	\end{enumromanup}
\end{lem}
\begin{proof}
	We only show the case $p \neq 2$, the claims for $p=2$ follow analogously.
	\begin{enumromanup}
	\item From the definitions it is clear that $I^{(k)} \mbfa_{\inf}(\overline{R}) \cap \mbfa_{R, \varpi}^+ = I^{(k)} \mbfa_{R, \varpi}^+$, where we take the intersection inside $\mbfa_{\inf}(\overline{R})$.
		Now, from \cite[\S 5.1.3, Proposition]{fontaine-corps-des-periodes} we have that $I^{(k)} \mbfa_{\inf}(\overline{R}) = \pi^k \mbfa_{\inf}(\overline{R})$.
		Since the map $\mbfa_{R, \varpi}^+ \rightarrow \mbfa_{\inf}(\overline{R})$ is flat and $\mbfa_{R, \varpi}^+ = \mbfa_{\inf}(\overline{R}) \cap \mbfa_{R, \varpi}$, we obtain that $I^{(k)} \mbfa_{R, \varpi}^+ = I^{(k)} \mbfa_{\inf}(\overline{R}) \cap \mbfa_{R, \varpi}^+ = \pi^k \mbfa_{\inf}(\overline{R}) \cap \mbfa_{R, \varpi}^+ = \pi^k \mbfa_{R, \varpi}^+$.

	\item Since $\pi_0 \in \mbfa_{R, \varpi}^+$, the map $\mbfa_{R, \varpi}^+ \rightarrow \mbfa_{\inf}(\overline{R})$ is flat and $\mbfa_{R, \varpi}^+ = \mbfa_{\inf}(\overline{R}) \cap \mbfa_{R, \varpi}$, we have $\pi_0\mbfa_{\inf}(\overline{R}) \cap \mbfa_{R, \varpi}^+ = \pi_0 \mbfa_{R, \varpi}^+$.
		Now, from \cite[\S 5.2.6, Proposition (i)]{fontaine-corps-des-periodes} we have that $I^{(p-1)} \mbfa_{\inf}(\overline{R}) = \pi_0 \mbfa_{\inf}(\overline{R})$.
		So we obtain that $I^{(p-1)} \mbfa_{R, \varpi}^+ = I^{(p-1)} \mbfa_{\inf}(\overline{R}) \cap \mbfa_{R, \varpi}^+ = \pi_0 \mbfa_{\inf}(\overline{R}) \cap \mbfa_{R, \varpi}^+ = \pi_0 \mbfa_{R, \varpi}^+$.

	\item This follows from \cite[\S 5.2.6, Proposition (ii)]{fontaine-corps-des-periodes}.
	\end{enumromanup}
\end{proof}

\begin{prop}\label{prop:arpd_alt_desc}
	The continuous morphism of $\mbfa_{R, \varpi}^+\textrm{-algebras}$
	\begin{align*}
		\alpha : \mbfa_{R, \varpi}^+ \hspace{0.5mm}\widehat{\otimes}_{S_0} \Lambda &\longrightarrow \mbfa_{R, \varpi}^{\textpd}\\
				\sum_{n \in \NN} a_n \otimes \big(\tfrac{\pi_0}{p}\big)^{[n]} &\longmapsto \sum_{n \in \NN} a_n \big(\tfrac{\pi_0}{p}\big)^{[n]},
	\end{align*}
	is an isomorphism.
\end{prop}
\begin{proof}
	The proof follows in a manner similar to the proof of \cite[\S 5.2.7, Th\'eor\`eme]{fontaine-corps-des-periodes}.
	We will only show the case $p \neq 2$, the claim for $p=2$ follows analogously.

	The homomorphism $\alpha$ in the claim is well defined and continuous since $\frac{\pi_0}{p} \in \Fil^1 \mbfa_{R, \varpi}^{\textpd}$.
	So we are left to show that $\alpha$ is an isomorphism.
	Since the source and targets are $\padic$ally complete $p\textrm{-torsion}$-free rings, it is enough to show that $\alpha$ is an isomorphism modulo $p$.

	Let $z_1 = \varphi^{-1}(z) \in \mbfa_{R, \varpi}^+$.
	Note that $\mbfa_{R, \varpi}^{\textpd}$ modulo $p$ is the divided power envelope of $\mbfe_{R, \varpi}^+$ with respect to the ideal generated by $\overline{z_1} \equiv \overline{\xi} \mod p$.
	Therefore, it is a free module over $\mbfe_{R, \varpi}^+ / \overline{z_1}^p$ with basis the images of $z_1^{[pn]}$, or equivalently $\big(\frac{z_1^p}{p}\big)^{[n]}$.
	From Lemma \ref{lem:ideal_generator} (iii), we have that $\varphi(\pi_0) = u \pi_0 z^{p-1}$, with $u \in S_0^{\times}$.
	Therefore, $\pi_0 = \varphi^{-1}(u) \varphi^{-1}(\pi_0) z_1^{p-1} = \varphi^{-1}(u)(z_1 - p)z_1^{p-1}$, which implies that $\mbfe_{R, \varpi}^+ / \overline{z_1}^p = \mbfe_{R, \varpi}^+ / \overline{\pi}_0$ and $\mbfa_{R, \varpi}^{\textpd}$ modulo $p$ is a free module over $\mbfe_{R, \varpi}^+ / \overline{\pi}_0$ with basis the images of $\big(\frac{\pi_0}{p}\big)^{[n]}$.
	Since it is immediate that the same is true for $\mbfa_{R, \varpi}^+ \otimes_{S_0} \Lambda$ modulo $p$, we get the claim.
\end{proof}

\begin{lem}\label{lem:ikarpd_alt_desc}
	For $k \in \NN$ the ideal $I^{(k)} \mbfa_{R, \varpi}^{\textpd}$ is a divided power ideal which is the associated $\mbfa_{R, \varpi}^+\textrm{-submodule}$ of $\mbfa_{R, \varpi}^{\textpd}$ generated by $t^{\{n\}}$ for $n \geq k$.
\end{lem}
\begin{proof}
	The proof follows in a manner similar to the proof of \cite[\S 5.3.5, Proposition]{fontaine-corps-des-periodes}.
	Let $J^{(k)}$ be the $\mbfa_{R, \varpi}^+\textrm{-submodule}$ of $\mbfa_{R, \varpi}^{\textpd}$ generated by $t^{\{n\}}$ for $n \geq k$.
	It is straightforward to check that $J^{(k)} \subset I^{(k)}$, and $J^{(k)}$ is a divided power ideal.
	Thus it remains to show that $I^{(k)} \subset J^{(k)}$.
	We will show this by induction on $k$.
	The case $k = 0$ is trivial.

	Now suppose $k \geq 1$ and $x \in I^{(k)}$.
	The induction hypothesis allows us to write $x = \sum_{n \geq k-1} a_n t^{\{n\}}$ where $a_n \in \mbfa_{R, \varpi}^+$ goes to $0$ as $n \rightarrow +\infty$.
	If $b = a_{n-1}$, we have $a = bt^{\{n-1\}} + a\prm$ where $a\prm \in J^{(k)} \subset I^{(k)}$, thus $b t^{\{k-1\}} \in I^{(k)}$.
	But $\varphi^s(b t^{\{k-1\}}) = p^{(k-1)s} \varphi^s(b) t^{\{k-1\}} = c_{k, s} \varphi^s(b) t^{\{k-1\}}$, where $c_{k, s}$ is a nonzero rational number.
	Since $t^{k-1} \in \Fil^{k-1} \mbfa_{R, \varpi}^{\textpd} \setminus \Fil^k \mbfa_{R, \varpi}^{\textpd}$, one has $b \in I^{(1)}\mbfa_{R, \varpi}^{\textpd} \cap \mbfa_{R, \varpi}^+$, which is the principal ideal generated by $\pi$.
	Thus $bt^{\{k-1\}}$ belongs to an ideal of $\mbfa_{R, \varpi}^{\textpd}$ generated by $\pi t^{\{k-1\}}$.
	But $\frac{t}{\pi} \in \mbfa_{R, \varpi}^{\textpd}$ is a unit (see Lemma \ref{lem:t_over_pi_unit}).
	Hence, $bt^{\{k-1\}}$ belongs to an ideal generated by $t \cdot t^{\{k-1\}}$, which is contained in $J^{(k)}$.
\end{proof}

Following is an immediate consequence of Lemma \ref{lem:ikarpd_alt_desc}:
\begin{cor}\label{cor:ik_arpd_1}
	For $k \in \NN$, consider the homomorphism $\mbfa_{R, \varpi}^+ \rightarrow I^{(k)} \mbfa_{R, \varpi}^{\textpd}$ sending $x \mapsto x \cdot t^{\{k\}}$.
	Then, the induced map $\mbfa_{R, \varpi}^+ / I^{(1)} \mbfa_{R, \varpi}^+ \rightarrow I^{(k)} \mbfa_{R, \varpi}^{\textpd} / I^{(k+1)} \mbfa_{R, \varpi}^{\textpd}$ is bijective.
\end{cor}

Now, from \cite[\S 5.3.5, Proposition]{fontaine-corps-des-periodes}, we have a natural isomorphism $\mbfa_{\inf}(\overline{R}) / I^{(k)} \mbfa_{\inf}(\overline{R}) \isomorphic \mbfa_{\crys}(\overline{R}) / I^{(k)} \mbfa_{\crys}(\overline{R})$, for $0 \leq k \leq p-1$.
A similar statement is true in our setting:
\begin{prop}\label{prop:arplus_mod_arpd_mod_iso}
	For $k \in \NN$, the rings $\mbfa_{R, \varpi}^+ / I^{(k)}\mbfa_{R, \varpi}^+$ and $\mbfa_{R, \varpi}^{\textpd} / I^{(k)}\mbfa_{R, \varpi}^{\textpd}$ are $p\textrm{-torsion}$ free.
	Moreover, if $0 \leq k \leq p-1$, then the natural map $\mbfa_{R, \varpi}^+ / I^{(k)} \mbfa_{R, \varpi}^+ \rightarrow \mbfa_{R, \varpi}^{\textpd} / I^{(k)} \mbfa_{R, \varpi}^{\textpd}$ is an isomorphism.
\end{prop}
\begin{proof}
	The proof follows from arguments similar to the proof of \cite[\S 5.3.5, Proposition]{fontaine-corps-des-periodes}.
	First, note that for every $k \in \NN$, $\mbfa_{R, \varpi}^{\textpd} / \Fil^k \mbfa_{R, \varpi}^{\textpd}$ is torsion free.
	Further, the kernel of the map
	\begin{align*}
		\mbfa_{R, \varpi}^{\textpd} &\longrightarrow (\mbfa_{R, \varpi}^{\textpd} / \Fil^k \mbfa_{R, \varpi}^{\textpd})^{\NN} \\
		x &\longmapsto (\varphi^n(x) \mod \Fil^k \mbfa_{R, \varpi}^{\textpd})_{k \in \NN},
	\end{align*}
	is $I^{(k)} \mbfa_{R, \varpi}^{\textpd}$.
	Therefore, $\mbfa_{R, \varpi}^+ / I^{(k)} \mbfa_{R, \varpi}^+ \rightarrowtail \mbfa_{R, \varpi}^{\textpd} / I^{(k)} \mbfa_{R, \varpi}^{\textpd} \rightarrowtail (\mbfa_{R, \varpi}^{\textpd} / \Fil^k \mbfa_{R, \varpi}^{\textpd})^{\NN}$, which implies that the former two rings are torsion free.

	Now from Proposition \ref{prop:arpd_alt_desc} and Lemma \ref{lem:ikarpd_alt_desc}, it follows that as $\mbfa_{R, \varpi}^+\textrm{-module}$, $\mbfa_{R, \varpi}^{\textpd} / I^{(k)} \mbfa_{R, \varpi}^{\textpd}$ is generated by the images of $\big(\frac{\pi_0}{p}\big)^{[n]}$ for $0 \leq (p-1)n < k$.
	For $0 \leq k \leq p-1$, we have that $\big(\frac{\pi_0}{p}\big)^{[n]} \in \mbfa_{R, \varpi}^+$, hence we get the claim.
\end{proof}

Next, we mention a lemma useful for the proof of Proposition \ref{prop:arplus_arpd_cond_equiv}.
\begin{lem}\label{lem:ik_arpd_pcomplete}
	\begin{enumromanup}
	\item For $0 \leq k < j$, we have that $I^{(k)} \mbfa_{R, \varpi}^{\textpd} / I^{(j)} \mbfa_{R, \varpi}^{\textpd}$ is $p\textrm{-torsion}$ free.

	\item For $k \in \NN$, we have that $I^{(k)} \mbfa_{R, \varpi}^{\textpd}$ is $\padic$ally complete.
	\end{enumromanup}
\end{lem}
\begin{proof}
	\begin{enumromanup}
	\item The proof is similar to the proof of \cite[Lemma A3.19 (1)]{tsuji-comparison}.
		Let $x \in I^{(k)} \mbfa_{R, \varpi}^{\textpd}$ and assume that $px \in I^{(j)} \mbfa_{R, \varpi}^{\textpd}$.
		Then $p\varphi^i(x) \in \Fil^j \mbfa_{R, \varpi}^{\textpd}$ for all $i \in \NN$.
		Since $\mbfa_{R, \varpi}^{\textpd} / \Fil^j \mbfa_{R, \varpi}^{\textpd} \subset \mbfa_{\crys}(\overline{R}) / \Fil^j \mbfa_{\crys}(\overline{R})$ is $p\textrm{-torsion}$ free (see \cite[Lemma A2.11 (2)]{tsuji-comparison}), we get that $\varphi^i(x) \in \Fil^j \mbfa_{R, \varpi}^{\textpd}$ for all $i \in \NN$, i.e. $x \in I^{(j)} \mbfa_{R, \varpi}^{\textpd}$.

	\item The proof is similar to the proof of \cite[Lemma A3.27]{tsuji-comparison}.
		We will prove the statement by induction on $k$.
		For $k = 0$, the statement is trivial by the definition of $\mbfa_{R, \varpi}^{\textpd}$.
		Next, from part (i) and Corollary \ref{cor:ik_arpd_1}, we have that $I^{(k)} \mbfa_{R, \varpi}^{\textpd} / I^{(k+1)} \mbfa_{R, \varpi}^{\textpd}$ is $p\textrm{-torsion}$ free and $\padic$ally complete.
		Therefore, we obtain exact sequences
		\begin{equation*}
			0 \longrightarrow \lim_n \big(I^{(k+1)} \mbfa_{R, \varpi}^{\textpd} \otimes \ZZ/p^n\ZZ\big) \longrightarrow \lim_n \big(I^{(k)} \mbfa_{R, \varpi}^{\textpd} \otimes \ZZ/p^n\ZZ\big) \longrightarrow I^{(k)} \mbfa_{R, \varpi}^{\textpd} / I^{(k+1)} \mbfa_{R, \varpi}^{\textpd} \longrightarrow 0.
		\end{equation*}
		The statement now follows by induction on $k$.
	\end{enumromanup}
\end{proof}

\subsection{Equivalence of categories}

In \cite{tsuji-ainf-genrep}, Tsuji has established a relationship between free relative Fontaine-Laffaille modules (see Definition \ref{defi:rel_fontaine_laffaille}) and $\mbfa_{\inf}(\overline{R})\textrm{-representations}$ as well as $\mbfa_{\crys}(\overline{R})\textrm{-representations}$ of $G_{R}$ (in a precise functorial manner).
Tsuji's computations are motivated by computations of Wach in \cite{wach-cristallines-torsion} for the arithmetic case.

Recall from \S \ref{subsec:ideals_arplus_arpd} that for $k \in \NN$ we have the ideal 
\begin{equation*}
	I^{(k)} \mbfa_{\inf}(\overline{R}) = \{x \in \mbfa_{\inf}(\overline{R}) \hspace{1mm} \textrm{such that} \hspace{1mm} \varphi^n(x) \in \Fil^k \mbfa_{\inf}(\overline{R}) \hspace{1mm} \textrm{for} \hspace{1mm} n \in \NN\}.
\end{equation*}
Similarly, we can define respective ideals $I^{(k)} \mbfa_{\crys}(\overline{R}) \subset \mbfa_{\crys}(\overline{R})$, $I^{(k)} \mbfa_{R, \varpi}^+ \subset \mbfa_{R, \varpi}^+$ and $I^{(k)} \mbfa_{R, \varpi}^{\textpd} \subset \mbfa_{R, \varpi}^{\textpd}$.
Given a free Fontaine-Laffaille module, in \cite[\S 5]{tsuji-ainf-genrep} Tsuji functorially obtains an $\mbfa_{\crys}(\overline{R})\textrm{-module}$ (in a manner similar to Proposition \ref{prop:arpd_mod_fl_data}).
Further, he exploits the isomorphism $\mbfa_{\inf}(\overline{R}) / I^{(p-1)} \mbfa_{\inf}(\overline{R}) \isomorphic \mbfa_{\crys}(\overline{R}) / I^{(p-1)} \mbfa_{\crys}(\overline{R})$, to construct an $\mbfa_{\inf}(\overline{R})\textrm{-representation}$ of $G_{R}$.
The last step is carried out by establishing certain equivalence of categories.
Tsuji's computations are general and follows from certain assumptions on the structure of the rings and modules, one is studying.
In this section, we will recall and verify those assumptions in our case, which would help us in establishing equivalence between several categories (see Theorem \ref{thm:arplus_arpd_cat_equiv}).

Let $A = \mbfa_{R, \varpi}^+$, $\mbfa_{R, \varpi}^+ / I^{(p-1)} \mbfa_{R, \varpi}^+$, $\mbfa_{R, \varpi}^{\textpd}$, or $\mbfa_{R, \varpi}^{\textpd} / I^{(p-1)} \mbfa_{R, \varpi}^{\textpd}$.
\begin{lem}\label{lem:q_nonzerodivisor}
	Let $q = \frac{\varphi(\pi)}{\pi} \in A$, then $q$ is a non-zero-divisor in $A$.
\end{lem}
\begin{proof}
	For $A = \mbfa_{R, \varpi}^+$ and $\mbfa_{R, \varpi}^{\textpd}$, the claim follows from the definitions.
	Next, note that we have $q = \frac{\varphi(\pi)}{\pi} = \pi^{p-1} + p u \in \mbfa_{R, \varpi}^+$ for some unit $u \in \mbfa_{R, \varpi}^+$, in particular, $q \equiv pu \mod \pi^{p-1}$.
	Now since $I^{(p-1)} \mbfa_{R, \varpi}^+ = \pi^{p-1} \mbfa_{R, \varpi}^+$ by Lemma \ref{lem:ideal_generator} (ii), we obtain that $q$ and $p$ are associates in $\mbfa_{R, \varpi}^+ / I^{(p-1)} \mbfa_{R, \varpi}^+ \isomorphic \mbfa_{R, \varpi}^{\textpd} / I^{(p-1)} \mbfa_{R, \varpi}^{\textpd}$.
	Since $\mbfa_{R, \varpi}^+ / I^{(p-1)} \mbfa_{R, \varpi}^+$ and $\mbfa_{R, \varpi}^{\textpd} / I^{(p-1)} \mbfa_{R, \varpi}^{\textpd}$ are $p\textrm{-torsion}$ free by Proposition \ref{prop:arplus_mod_arpd_mod_iso}, we get the claim.
\end{proof}

Next, note that we have $\Fil^0 A = A$ and $\Fil^i A \cdot \Fil^j A \subset \Fil^{i+j} A$ for $i, j \in \ZZ$, and $\varphi(\Fil^k A) \subset q^k A$ for $k \in \NN$.
In particular, we see that our choice of $A$ and $q$ satisfies \cite[Condition 39]{tsuji-ainf-genrep}.
\begin{defi}\label{defi:mf_free_a}
	Define the category $\MF^{q}_{[0, p-2], \free}(A, \varphi, \Gamma_{R})$ as follows:
	An object is a triplet $(N, \Fil^k N, \varphi)$ such that,
	\begin{enumromanup}
	\item $N$ is a free $A\textrm{-module}$ of rank $h$.

	\item The filtration $\Fil^k N$ is decreasing, and there exists an $A\textrm{-basis}$ $\{e_1, \ldots, e_h\}$ of $N$ and $k_1, \ldots, k_h \in \NN_{\leq p-2}$ such that $\Fil^k N = \sum_{i=1}^h \Fil^{k-k_i} A e_i$ for $0 \leq k \leq p-2$.

	\item A Frobenius-semilinear endomorphism $\varphi : N \rightarrow N$ such that $\varphi(\Fil^k N) \subset q^k N$ for $0 \leq k \leq p-2$, and $\sum_{k=0}^{p-2}A \cdot q^{-k}\varphi(\Fil^k N) = N$.

	\item $N$ is equipped with a continuous action of $\Gamma_{R}$ such that $\Fil^k N$ is stable under this action, and the endomorphism $\varphi$ commutes with the action of $\Gamma_{R}$.
	\end{enumromanup}
	A morphism between two objects of the category $\MF^{q}_{[0, p-2], \free}(A, \varphi, \Gamma_{R})$ is a continuous $A\textrm{-linear}$ morphism commuting with the endomorphism $\varphi$ and the action of $\Gamma_{R}$ on each side.
\end{defi}

\begin{nota}
	By a slight abuse of notations, we will denote $(N, \Fil^k N, \partial, \Phi) \in \MF^{q}_{[0, p-2], \free}(A, \varphi, \Gamma_{R})$ by $N$ and say that it has filtration of level $[0, p-2]$.
\end{nota}

\begin{rem}\label{rem:q_p_associates}
	In $\mbfa_{R, \varpi}^{\textpd}$, note that we can write $q = p\frac{t}{\pi}\varphi\big(\frac{\pi}{t}\big)$, and since $\frac{t}{\pi}$ is a unit in $\mbfa_{R, \varpi}^{\textpd}$ (see Lemma \ref{lem:t_over_pi_unit}), we obtain that $q$ and $p$ are associates in $\mbfa_{R, \varpi}^{\textpd}$.
	Therefore, for $A = \mbfa_{R, \varpi}^{\textpd}$ in Definition \ref{defi:mf_free_a}, we can replace $q$ by $p$.
	Further, since $q = \pi^{p-1} + pu$ for $u \in (\mbfa_{R, \varpi}^+)^{\times}$ and $\pi^{p-1}$ generates $I^{(p-1)} \mbfa_{R, \varpi}^+$ (see Lemma \ref{lem:ideal_generator} (ii)), we obtain that $q \equiv pu \mod I^{(p-1)} \mbfa_{R, \varpi}^+$, i.e.\ $q$ and $p$ are associates in $\mbfa_{R, \varpi}^+ / I^{(p-1)} \mbfa_{R, \varpi}^+ \isomorphic \mbfa_{R, \varpi}^{\textpd} / I^{(p-1)} \mbfa_{R, \varpi}^{\textpd}$.
	Therefore, for $A = \mbfa_{R, \varpi}^+ / I^{(p-1)} \mbfa_{R, \varpi}^+$ in Definition \ref{defi:mf_free_a}, we can replace $q$ by $p$, and similarly for $\mbfa_{R, \varpi}^{\textpd} / I^{(p-1)} \mbfa_{R, \varpi}^{\textpd}$.
\end{rem}

\begin{lem}[{\cite[Lemma 41]{tsuji-ainf-genrep}}]\label{lem:filmod_strong_div}
	Let $(N, \Fil^k N)$ be as in Definition \ref{defi:mf_free_a} (i), (ii).
	Then a Frobenius-semilinear endomorphism $\varphi : N \rightarrow N$ satisfies the conditions in Definition \ref{defi:mf_free_a} (iii) if and only if $\varphi(e_i) \in q^{k_i} N$ for $1 \leq i \leq h$ and $\{q^{-k_1}\varphi(e_1), \ldots, q^{-k_h}\varphi(e_h)\}$ is an $A\textrm{-basis}$ of $N$.
\end{lem}
\begin{proof}
	Let us assume that $(N, \Fil^k N)$ satisfies the condition in Definition \ref{defi:mf_free_a} (iii).
	Then, since $e_i \in \Fil^{k_i} N$, we have $\varphi(e_i) \in q^{k_i} N$ for $1 \leq i \leq h$.
	Now for $0 \leq k \leq p-2$, we have 
	\begin{equation*}
		\varphi(\Fil^{k-k_i} A e_i) = \varphi(\Fil^{k-k_i} A) \varphi(e_i) \subset q^k A \cdot q^{-k_i} \varphi(e_i) \subset q^k N.
	\end{equation*}
	Therefore, from the identity $\sum_{k=0}^{p-2}A \cdot q^{-k}\varphi(\Fil^k N) = N$, we obtain that $\{q^{-k_1} \varphi(e_1), \ldots, q^{-k_h} \varphi(e_h)\}$ generate $N$ as an $A\textrm{-module}$.
	Since $N$ is free of rank $h$ over $A$, we get that $\{q^{-k_1} \varphi(e_1), \ldots, q^{-k_h} \varphi(e_h)\}$ is indeed a basis.

	Conversely, assume that $\varphi(e_i) \in q^{k_i} N$ for $1 \leq i \leq h$ and $\{q^{-k_1} \varphi(e_1), \ldots, q^{-k_h} \varphi(e_h)\}$ form an $A\textrm{-basis}$ of $N$.
	Then, from Definition \ref{defi:mf_free_a} (ii), we have 
	\begin{equation*}
		\varphi(\Fil^k N) = \varphi\big(\sum_{i=1}^h \Fil^{k-k_i} A e_i\big) \subset \sum_{i=1}^h q^{k-k_i} A \varphi(e_i) = q^k \sum_{i=1}^h A \cdot q^{-k_i} \varphi(e_i) = q^k N.
	\end{equation*}
	Further, since $\{q^{-k_1} \varphi(e_1), \ldots, q^{-k_h} \varphi(e_h)\} \in \sum_{k=0}^{p-2}A \cdot q^{-k}\varphi(\Fil^k N)$, we obtain the last equality in Definition \ref{defi:mf_free_a} (iii).
\end{proof}

Now we introduce some necessary conditions in order to adapt Tsuji's results from \cite[\S 4 - \S 8]{tsuji-ainf-genrep}.
\begin{cond}\label{cond:essential_cond_equiv}
	Let $A = \mbfa_{R, \varpi}^+$, $\mbfa_{R, \varpi}^{\textpd}$ and $q = \frac{\varphi(\pi)}{\pi} \in A$.
	Consider the projection map $A \twoheadrightarrow A/J$ for some ideal $J \subset A$ and assume that
	\begin{enumromanup}
	\item The ideal $J$ is contained in the Jacobson radical of $A$, and $J \subset \Fil^{p-2} A$.
		Moreover, $\varphi(J) \subset J$ and $\varphi(J) \subset q^{p-1} A$.
		Further, the ideal $J$ is preserved under the action of $\Gamma_{R}$.

	\item The ideal $J$ is closed as a submodule of $A$.

	\item There exists a decreasing sequence of ideals $\cdots \subset H_{n+1} \subset H_n \subset \cdots \subset H_0 \subset A$ for $n \in \NN$, such that $H_n$ form a fundamental system of neighborhoods of $0$ in $A$, the homomorphism $A \rightarrow \lim_n A/H_n$ is an isomorphism, and $q^{-(p-1)} \varphi(H_n \cap J) \subset H_n \cap J$ for every $n \in \NN$.

	\item The image of $q$ in $A/J$ is a non-zero-divisor.
		Moreover, the sequence $\prod_{k=0}^n \varphi^k(q) \in A$ converges to $0$ as $n \rightarrow +\infty$.

	\item The homomorphism $\varphi : A \rightarrow A$ is continuous and multiplication by $q$ induces a homeomorphism $A \rightarrow qA$, where the latter is equipped with the induced topology.
	\end{enumromanup}
\end{cond}

\begin{prop}\phantomsection\label{prop:arplus_arpd_cond_equiv}
	\begin{enumromanup}
	\item Let $A = \mbfa_{R, \varpi}^+$ with $J = I^{(p-1)} \mbfa_{R, \varpi}^+$, and $H_n = p^n \mbfa_{R, \varpi}^+ + \pi^{n+p-1} \mbfa_{R, \varpi}^+$.
		Then $\mbfa_{R, \varpi}^+$ satisfies Condition \ref{cond:essential_cond_equiv}.

	\item Let $A = \mbfa_{R, \varpi}^{\textpd}$ with $J = I^{(p-1)} \mbfa_{R, \varpi}^{\textpd}$, and $H_n = p^n \mbfa_{R, \varpi}^{\textpd}$.
		Then $\mbfa_{R, \varpi}^{\textpd}$ satisfies Condition \ref{cond:essential_cond_equiv}.
	\end{enumromanup}
\end{prop}
\begin{proof}
	The proof follows in a manner similar to the proof of \cite[Proposition 59]{tsuji-ainf-genrep}.
	\begin{enumromanup}
	\item The ring $\mbfa_{R, \varpi}^+$ is $\pi\textrm{-adically}$ complete, and since $I^{(p-1)} \mbfa_{R, \varpi}^+ = \pi^{p-1} \mbfa_{R, \varpi}^+ \subset \pi \mbfa_{R, \varpi}^+$, we see that $I^{(p-1)} \mbfa_{R, \varpi}^+$ is contained in the Jacobson radical of $\mbfa_{R, \varpi}^+$.
		Moreover, we have $\varphi(\pi^{p-1}\mbfa_{R, \varpi}^+) \subset q^{p-1}\pi^{p-1} \mbfa_{R, \varpi}^+ \subset \pi^{p-1} \mbfa_{R, \varpi}^+$, therefore $I^{(p-1)} \mbfa_{R, \varpi}^+ \subset \Fil^{p-2} \mbfa_{R, \varpi}^+$.
		It is clear that $I^{(p-1)} \mbfa_{R, \varpi}^+$ is stable under the action of $\Gamma_{R}$.
		Therefore, Condition \ref{cond:essential_cond_equiv} (i) is satisfied.

		Now we have $H_n = p^n \mbfa_{R, \varpi}^+ + \pi^{n+p-1} \mbfa_{R, \varpi}^+$ for $n \in \NN$, which is a fundamental system of neighborhoods of $0 \in \mbfa_{R, \varpi}^+$ and $\mbfa_{R, \varpi}^+ = \lim_n \mbfa_{R, \varpi}^+ / H_n$ (see Lemma \ref{lem:arplus_pa_complete}).
		Further, since $\mbfa_{R, \varpi}^+ / I^{(p-1)} \mbfa_{R, \varpi}^+$ is $p\textrm{-torsion}$ free, we obtain that $H_n \cap I^{(p-1)} \mbfa_{R, \varpi}^+ = (p^n \mbfa_{R, \varpi}^+ + \pi^{n+p-1} \mbfa_{R, \varpi}^+) \cap I^{(p-1)} \mbfa_{R, \varpi}^+ = p^n I^{(p-1)} \mbfa_{R, \varpi}^+ + \pi^n I^{(p-1)} \mbfa_{R, \varpi}^+$.
		The Condition \ref{cond:essential_cond_equiv} (iii) now follows from this.
		Moreover, $I^{(p-1)}\mbfa_{R, \varpi}^+$ is a free $\mbfa_{R, \varpi}^+\textrm{-module}$ of rank 1, so it follows that $J$ is a closed submodule of $\mbfa_{R, \varpi}^+$ by Lemma \ref{lem:arplus_pa_complete} (i) \& (iv), verifying Condition \ref{cond:essential_cond_equiv} (ii).
		
		Next, from Lemma \ref{lem:q_nonzerodivisor}, it follows that $q$ is a non-zero-divisor in $\mbfa_{R, \varpi}^+ / I^{(p-1)} \mbfa_{R, \varpi}^+$.
		Further, for $k \in \NN$, we have $\varphi^k(q) = \varphi^{k+1}(\xi) \in \varphi^{k+1}(\Fil^1 \mbfa_{R, \varpi}^+) \subset \varphi^{k+1}(p \mbfa_{R, \varpi}^+ + \pi_1 \mbfa_{R, \varpi}^+) \subset p \mbfa_{R, \varpi}^+ + \pi \mbfa_{R, \varpi}^+$.
		Therefore, $\prod_{k=0}^n \varphi^k(q)$ converges to $0$ as $n \rightarrow + \infty$, and Condition \ref{cond:essential_cond_equiv} (iv) has been verified.

		By the definition of $\varphi$ in \S \ref{subsec:cyclotomic_embeddings}, we see that it is continuous.
		Further, from Lemma \ref{lem:arplus_pa_complete} (iii), it follows that $\mbfa_{R, \varpi}^+ / q\mbfa_{R, \varpi}^+$ is $p\textrm{-torsion}$ free.
		Therefore, we have $(p^n \mbfa_{R, \varpi}^+ + q^{n+1} \mbfa_{R, \varpi}^+) \cap q \mbfa_{R, \varpi}^+ = p^n(q \mbfa_{R, \varpi}^+) + q^n(q \mbfa_{R, \varpi}^+)$.
		By Lemma \ref{lem:arplus_pa_complete} (i), it follows that $\mbfa_{R, \varpi}^+ \xrightarrow{\times q} q\mbfa_{R, \varpi}^+$ is a homeomorphism, which verifies Condition \ref{cond:essential_cond_equiv} (v).

	\item Note that we have $q = p\varphi\big(\frac{\pi}{t}\big)\frac{t}{\pi}$, which implies that $q$ and $p$ are associates in $\mbfa_{R, \varpi}^{\textpd}$ (see Lemma \ref{lem:t_over_pi_unit}).
		Therefore, it is enough to verfiy Condition \ref{cond:essential_cond_equiv}, with $q$ replaced by $p$ everywhere.

		We have $I^{(p-1)} \mbfa_{R, \varpi}^{\textpd} \subset \Fil^1 \mbfa_{R, \varpi}^{\textpd} + p \mbfa_{R, \varpi}^{\textpd}$, $\mbfa_{R, \varpi}^{\textpd}$ is $\padic$ally complete and $\Fil^1 \mbfa_{R, \varpi}^{\textpd} / p\Fil^1 \mbfa_{R, \varpi}^{\textpd}$ is a nil ideal of $\mbfa_{R, \varpi}^{\textpd} / p\mbfa_{R, \varpi}^{\textpd}$.
		Therefore, $I^{(p-1)} \mbfa_{R, \varpi}^{\textpd}$ is contained in the Jacobson radical of $\mbfa_{R, \varpi}^{\textpd}$.
		Further, by definitions we have $I^{(p-1)} \mbfa_{R, \varpi}^{\textpd} \subset \Fil^{p-2} \mbfa_{R, \varpi}^{\textpd}$ and $\varphi\big(I^{(p-1)} \mbfa_{R, \varpi}^{\textpd}\big) \subset I^{(p-1)} \mbfa_{R, \varpi}^{\textpd}$.
		Also, $\varphi\big(I^{(p-1)} \mbfa_{R, \varpi}^{\textpd}\big) \subset q^{p-1} \mbfa_{R, \varpi}^{\textpd} = p^{p-1} \mbfa_{R, \varpi}^{\textpd}$.
		It is clear that $I^{(p-1)} \mbfa_{R, \varpi}^{\textpd}$ is stable under the action of $\Gamma_{R}$.
		Therefore, Condition \ref{cond:essential_cond_equiv} (i) is satisfied.

		Next, we know that $\mbfa_{R, \varpi}^{\textpd}$ is $\padic$ally complete and $\mbfa_{R, \varpi}^{\textpd} / I^{(p-1)} \mbfa_{R, \varpi}^{\textpd}$ is $p\textrm{-torsion}$ free by Proposition \ref{prop:arplus_mod_arpd_mod_iso}, therefore $p^n \mbfa_{R, \varpi}^{\textpd} \cap I^{(p-1)} \mbfa_{R, \varpi}^{\textpd} = p^n I^{(p-1)} \mbfa_{R, \varpi}^{\textpd}$.
		This gives us Condition \ref{cond:essential_cond_equiv} (iii).
		Further, $I^{(p-1)} \mbfa_{R, \varpi}^{\textpd}$ is $\padic$ally complete by Lemma \ref{lem:ik_arpd_pcomplete} (ii), so we get Condition \ref{cond:essential_cond_equiv} (ii).

		Condition \ref{cond:essential_cond_equiv} (iv) \& (v) follow trivially from the fact that $\mbfa_{R, \varpi}^{\textpd} = \lim_n \mbfa_{R, \varpi}^{\textpd} / p^n \mbfa_{R, \varpi}^{\textpd}$.
	\end{enumromanup}
\end{proof}

Finally, we come to the main result of this section.
Note that the categories $\MF^p$ below are defined by combining Definition \ref{defi:mf_free_a} and Remark \ref{rem:q_p_associates}.
\begin{thm}\label{thm:arplus_arpd_cat_equiv}
	The natural maps $\mbfa_{R, \varpi}^+ \twoheadrightarrow \mbfa_{R, \varpi}^+ / I^{(p-1)} \mbfa_{R, \varpi}^+ \isomorphic \mbfa_{R, \varpi}^{\textpd} / I^{(p-1)}\mbfa_{R, \varpi}^{\textpd} \twoheadleftarrow \mbfa_{R, \varpi}^{\textpd}$, induce equivalence of categories:
	\begin{align*}
		\MF_{[0, p-2], \free}^p(\mbfa_{R, \varpi}^{\textpd}, \varphi, \Gamma_{R}) &\xrightarrow[\hspace{1mm} \sim \hspace{1mm}]{\hspace{1mm} (1) \hspace{1mm}} \MF_{[0, p-2], \free}^{p}(\mbfa_{R, \varpi}^{\textpd} / I^{(p-1)}\mbfa_{R, \varpi}^{\textpd}, \varphi, \Gamma_{R})\\
		&\xleftarrow[\hspace{1mm} \sim \hspace{1mm}]{\hspace{1mm} (2) \hspace{1mm}} \MF_{[0, p-2], \free}^{p}(\mbfa_{R, \varpi}^+ / I^{(p-1)}\mbfa_{R, \varpi}^+, \varphi, \Gamma_{R})\\
		&\xleftarrow[\hspace{1mm} \sim \hspace{1mm}]{\hspace{1mm} (3) \hspace{1mm}} \MF_{[0, p-2], \free}^q(\mbfa_{R, \varpi}^+, \varphi, \Gamma_{R}).
	\end{align*}
\end{thm}
\begin{proof}
	The natural projection map $\mbfa_{R, \varpi}^+ \twoheadrightarrow \mbfa_{R, \varpi}^+ / I^{(p-1)} \mbfa_{R, \varpi}^+$ is compatible with Frobenius and the action of $\Gamma_{R}$ and we have $q \equiv pu \mod I^{(p-1)} \mbfa_{R, \varpi}^+$ for $u \in (\mbfa_{R, \varpi}^+)^{\times}$ (see also Remark \ref{rem:q_p_associates}), i.e.\ $q$ and $p$ are associates in $\mbfa_{R, \varpi}^+ / I^{(p-1)} \mbfa_{R, \varpi}^+$.
	Further, $\mbfa_{R, \varpi}^+$ satisfies Condition \ref{cond:essential_cond_equiv}.
	Therefore, from \cite[Proposition 56]{tsuji-ainf-genrep}, we obtain that the functor in (3) is an equivalence of categories.

	Next, from Proposition \ref{prop:arplus_mod_arpd_mod_iso}, we have an isomorphism of rings $\mbfa_{R, \varpi}^+ / I^{(p-1)} \mbfa_{R, \varpi}^+ \sim \mbfa_{R, \varpi}^{\textpd} / I^{(p-1)}\mbfa_{R, \varpi}^{\textpd}$, compatible with Frobenius and the action of $\Gamma_{R}$.
	Therefore, we obtain that the functor in (2) is an equivalence of catgeories.

	Finally, the natural projection map $\mbfa_{R, \varpi}^{\textpd} \twoheadrightarrow \mbfa_{R, \varpi}^{\textpd} / I^{(p-1)} \mbfa_{R, \varpi}^{\textpd}$ is compatible with Frobenius and the action of $\Gamma_{R}$ and we have $q \equiv pu \mod I^{(p-1)} \mbfa_{R, \varpi}^+$ for $u \in (\mbfa_{R, \varpi}^+)^{\times}$ (see also Remark \ref{rem:q_p_associates}), i.e.\ $q$ and $p$ are associates in $\mbfa_{R, \varpi}^+ / I^{(p-1)} \mbfa_{R, \varpi}^+ \isomorphic \mbfa_{R, \varpi}^{\textpd} / I^{(p-1)} \mbfa_{R, \varpi}^{\textpd}$.
	Further, $\mbfa_{R, \varpi}^{\textpd}$ satisfies Condition \ref{cond:essential_cond_equiv}.
	Therefore, from \cite[Proposition 56]{tsuji-ainf-genrep}, we obtain that the functor in (1) is an equivalence of categories.
\end{proof}

\subsection{Wach modules from Fontaine-Laffaille data}

In this section, we will work with objects in the category $\MF_{[0, p-2], \free}(R, \Phi, \partial)$ (see Definition \ref{defi:rel_fontaine_laffaille}) and obtain Wach modules over $\mbfa_{R}^+$ (see Definition \ref{defi:rel_wach_mods}).
In \S \ref{subsubsec:fl_to_arplus_mod}, starting with a Fontaine-Laffaille module, we will first obtain a free module over $\mbfa_{R, \varpi}^+$ with desired properties and in \S \ref{subsubsec:obtain_wach_mod} we will descend over to $\mbfa_R^+$.
Note that in \S \ref{subsubsec:fl_to_arplus_mod}, we will first establish a mod $p^n$ statement (see \eqref{eq:flmod_wachpdmod_n}) and as a consequence deduce a $\padic$ statement (see Proposition \ref{prop:arpd_mod_fl_data}).
However, it is possible to prove the $\padic$ statement directly (see another proof of Proposition \ref{prop:arpd_mod_fl_data}).
Readers interested only in the $\padic$ statement can directly skip to Proposition \ref{prop:arpd_mod_fl_data}.

\subsubsection{From Fontaine-Laffaille modules to \texorpdfstring{$\mbfa_{R, \varpi}^+$}{--}-modules}\label{subsubsec:fl_to_arplus_mod}

Following \cite[\S 4]{tsuji-ainf-genrep}, for $n \in \NN_{>0}$ we set $X_n = \Spec(R / p^n)$ and $\Sigma_n = \Spec(O_F / p^n)$ and consider the big crystalline sites $\CRYS(X_n, \Sigma_n)$ and $\CRYS(X_1, \Sigma_n)$, and the respective topos $(X_n / \Sigma_n)_{\CRYS}$ and $(X_1 / \Sigma_n)_{\CRYS}$, with the PD-ideal $(p(O_F / p^n), [\hspace{0.5mm}])$.
Let $F_{\Sigma_n} : \Sigma_n \rightarrow \Sigma_n$ denote a lifting of the absolute Frobenius of $\Sigma_1$, such that it is a PD-morphism with repsect to the PD-structure.
The absolute Frobenius $F_{X_1}$ of $X_1$ and $F_{\Sigma_n}$ define a morphism of PD-ringed topos $F_{X_1 / \Sigma_n, \CRYS} : (X_1 / \Sigma_n)_{\CRYS} \rightarrow (X_1 / \Sigma_n)_{\CRYS}$.

Let $(M, \Fil^{\bullet} M, \partial, \Phi) \in \MF_{[0, p-2], \free}(R, \Phi, \partial)$ be a free relative Fontaine-Laffaille module (see Definition \ref{defi:rel_fontaine_laffaille}), and let $(M_n, \Fil^{\bullet} M_n, \partial, \Phi)$ denote its modulo $p^n$ reduction.
Then, by \cite[Definition 26, Theorems 17 \& 29]{tsuji-ainf-genrep} this data corresponds to a quasi-coherent filtered crystal $(\pazf_n, \Fil^{\bullet} \pazf_n)$ on $\CRYS(X_n / \Sigma_n)$.
Similarly, by \cite[Definition 26, Theorems 22 \& 29]{tsuji-ainf-genrep} this data also corresponds to a quasi-coherent crystal $\pazg_n$ on $\CRYS(X_1 / \Sigma_n)$.
The reduction modulo $p^n$ of $\Phi : \varphi^{\ast} M \rightarrow M$ equip $\pazg_n$ with a morphism $\Phi_{\pazg_n} : F^{\ast}_{X_1/\Sigma_n, \CRYS}(\pazg_n) \rightarrow \pazg_n$.
Further, for the morphism of ringed topos $i_{n, \CRYS} : (X_1 / \Sigma_n)_{\CRYS} \rightarrow (X_n / \Sigma_n)_{\CRYS}$ induced by the closed immersion $i_n : X_1 \rightarrow X_n$ over $\textup{id}_{\Sigma_n}$, we have $i_{n, \CRYS}^{\ast}(\pazf_n) = \pazg_n$ (see \cite[Propositions 25 \& 32]{tsuji-ainf-genrep}).
Moreover, we have similar statements for the morphism of ringed topos induced by $X_n \rightarrow X_{n+1}$ and $\Sigma_n \rightarrow \Sigma_{n+1}$.

Now, for $n \in \NN_{> 0}$ let $X_n\prm := \Spec(R[\varpi] / p^n)$, $D_n := \Spec(\mbfa_{R, \varpi}^{\textpd} / p^n)$ and $F_{D_n} : D_n \rightarrow D_n$ be the lifting of the absolute Frobenius on $D_1$ defined by $\varphi$ of $\mbfa_{R, \varpi}^{\textpd} / p^n$.
We have the surjective map $\theta : \mbfa_{R, \varpi}^+ \twoheadrightarrow R[\varpi]$.
So taking mod $p^n$ reduction, we obtain an embedding $X_n\prm \rightarrowtail \Spec(\mbfa_{R, \varpi}^+ / p^n)$ and taking divided power envelope, we obtain a closed immersion $X_n\prm \rightarrowtail D_n$ (resp. $X_1\prm \rightarrowtail D_n$) which can naturally be regarded as an object of the site $\CRYS(X_n / \Sigma_n)$ (resp. $\CRYS(X_1 / \Sigma_n)$), endowed with a right action of $\Gamma_{R}$.

\begin{defi}
	Define an $\mbfa_{R, \varpi}^{\textpd} / p^n\textrm{-module}$ as $N^{\textpd}_n := \Gamma(X_n\prm \rightarrowtail D_n, \pazf_n) \isomorphic \Gamma(X_1\prm \rightarrowtail D_n, \pazg_n)$.
\end{defi}

The right action of $\Gamma_{R}$ on $D_n$ induces a left action on $N^{\textpd}_n$.
The filtration on $\pazf_n$ induces a filtration by $\mbfa_{R, \varpi}^{\textpd} / p^n\textrm{-submodules}$ on $N^{\textpd}_n$, which is stable under the $\Gamma_{R}\textrm{-action}$.
Then $N^{\textpd}_n$ is a finite free filtered $\mbfa_{R, \varpi}^{\textpd} / p^n\textrm{-module}$ of level $[0, p-2]$ (see \cite[Lemma 20]{tsuji-ainf-genrep}).
The Frobenius $\Phi_{\pazg_n}$ of $\pazg_n$  and the lifting of Frobenius $F_{D_n}$ on $D_n$ define a semilinear $\Gamma_{R}\textrm{-equivariant}$ endomorphism of $\Gamma(X_1\prm \rightarrowtail D_n, \pazg_n)$ and hence that of $N^{\textpd}_n$ as $\Gamma(X_1\prm \rightarrowtail D_n, \pazg_n) \rightarrow \Gamma(X_1\prm \rightarrowtail D_n, F_{X_1, \CRYS}^{\ast}\pazg_n) \xrightarrow{\Phi_{\pazg_n}} \Gamma(X_1\prm \rightarrowtail D_n, \pazg_n)$, where the first homomorphism is induced by $F_{X_1\prm}$ and $F_{D_n}$.

Now, let $[\hspace{0.5mm}]$ denote the PD-structure on the ideal $p(\mbfa_{R, \varpi}^{\textpd} / p^n) + \Fil^1 \mbfa_{R, \varpi}^{\textpd} / p^n$ of $\mbfa_{R, \varpi}^{\textpd} / p^n$.
Then we have the big crystalline sites $\CRYS(X_n\prm / D_n)$ and $\CRYS(X_1\prm / D_n)$, and the respective topos $(X_n\prm / D_n)_{\CRYS}$ and $(X_1\prm / D_n)_{\CRYS}$ with the PD-ideal $(p(\mbfa_{R, \varpi}^{\textpd} / p^n) + \Fil^1 \mbfa_{R, \varpi}^{\textpd} / p^n, [\hspace{0.5mm}])$ of $\mbfa_{R, \varpi}^{\textpd} / p^n$.
By taking the pullback of $(\pazf_n, \Fil^{\bullet} \pazf_n)$ (resp. $\pazg_n$) under the morphism of ringed topos $(X_n\prm / D_n)_{\CRYS} \rightarrow (X_n / \Sigma_n)_{\CRYS}$ (resp. $(X_1\prm / D_n)_{\CRYS} \rightarrow (X_1 / \Sigma_n)_{\CRYS}$), we obtain a quasi-coherent filtered crystal $(\pazf_n\prm, \Fil^{\bullet} \pazf_n\prm)$ (resp. a quasi-coherent crystal $\pazg_n\prm$ with a morphism $\Phi_{\pazg_n\prm} : F_{X_1\prm / D_n, \CRYS}^{\ast}(\pazg_n\prm) \rightarrow \pazg_n\prm$), endowed with compatible $\Gamma_{R}\textrm{-action}$.
Since $X_n\prm \rightarrowtail D_n$ (resp. $X_1\prm \rightarrowtail D_n$) is a final object of $\CRYS(X_n\prm / D_n)$ (resp. $\CRYS(X_1\prm / D_n)$), we have canonical $\mbfa_{R, \varpi}^{\textpd} / p^n\textrm{-linear}$ ismorphisms $N^{\textpd}_n \isomorphic \Gamma((X_n\prm / D_n)_{\CRYS}, \pazf_n\prm) \isomorphic \Gamma((X_1\prm / D_n)_{\CRYS}, \pazg_n\prm)$ compatible with supplementary structures (see \cite[p. 188-189]{tsuji-ainf-genrep}).

Next, for $n \in \NN_{> 0}$, similar to above let $E_n := \Spec(\pazo\mbfa_{R, \varpi}^{\textpd} / p^n)$ and $F_{E_n} : E_n \rightarrow E_n$ be the lifting of the absolute Frobenius on $E_1$ defined by $\varphi$ of $\pazo\mbfa_{R, \varpi}^{\textpd} / p^n$.
We have the surjective map $\theta_{R} : R \otimes_{\ZZ} \mbfa_{R, \varpi}^+ \twoheadrightarrow R[\varpi]$.
So taking mod $p^n$ reduction, we have an embedding $X_n\prm \rightarrowtail \Spec(R \otimes_{\ZZ} \mbfa_{R, \varpi}^+ / p^n)$ and taking divided power envelope, we obtain a closed immersion $X_n\prm \rightarrowtail E_n$ (resp. $X_1\prm \rightarrowtail E_n$) which can naturally be regarded as an object of the site $\CRYS(X_n\prm / D_n)$ (resp. $\CRYS(X_1\prm / D_n)$), endowed with a right action of $\Gamma_{R}$.

\begin{defi}
	Define an $\pazo \mbfa_{R, \varpi}^{\textpd} / p^n\textrm{-module}$ as $\pazo N^{\textpd}_n := \Gamma(X_n\prm \rightarrowtail E_n, \pazf_n\prm) \isomorphic \Gamma(X_1\prm \rightarrowtail E_n, \pazg_n\prm)$.
\end{defi}

The right action of $\Gamma_{R}$ on $E_n$ induces a left action on $\pazo N^{\textpd}_n$.
The filtration on $\pazf_n\prm$ induces a filtration by $\pazo \mbfa_{R, \varpi}^{\textpd} / p^n\textrm{-submodules}$ on $\pazo N^{\textpd}_n$, which is stable under the $\Gamma_{R}\textrm{-action}$.
Then $\pazo N^{\textpd}_n$ is a finite free filtered $\pazo \mbfa_{R, \varpi}^{\textpd} / p^n\textrm{-module}$ of level $[0, p-2]$ (see \cite[Lemma 20]{tsuji-ainf-genrep}).
Further, by \cite[Theorem 29, Proposition 32]{tsuji-ainf-genrep} $\pazo N^{\textpd}_n$ is equipped with a $\Gamma_{R}\textrm{-equivariant}$ integrable connection compatible with the connection on $\pazo \mbfa_{R, \varpi}^{\textpd} / p^n$ and satisfying Griffiths transversality with the respect to the filtration.
Moreover, this the $\Gamma_{R}\textrm{-action}$ and connection are compatible with the respective structures on $\Gamma(X_1\prm \rightarrowtail E_n, \pazg_n\prm)$ (see \cite[Propositions 25 \& 32]{tsuji-ainf-genrep}).
The Frobenius $\Phi_{\pazg_n\prm}$ of $\pazg_n\prm$  and the lifting of Frobenius $F_{E_n}$ on $E_n$ define a semilinear $\Gamma_{R}\textrm{-equivariant}$ endomorphism $\varphi$ of $\Gamma(X_1\prm \rightarrowtail E_n, \pazg_n\prm)$ and hence that of $\pazo N^{\textpd}_n$.
Further, the Frobenius-semilinear endomorphism $\varphi$ commutes with the connection on $\pazo N^{\textpd}_n$.

From \cite[Proposition 4.1.4]{berthelot-cohomologie-cristalline} and \cite[Theorem 7.1]{berthelot-ogus-crystalline}, we have a description of the global sections of a crystal in terms of horizontal sections of the corresponding module with an integrable connection on the PD-envelope of an embedding into a smooth scheme.
In other words, we have an $\mbfa_{R, \varpi}^{\textpd} / p^n\textrm{-linear}$ isomorphism
\begin{equation*}
	N^{\textpd}_n \isomorphic \big(\pazo N^{\textpd}_n\big)^{\partial=0},
\end{equation*}
compatible with filtration, Frobenius and the action of $\Gamma_{R}$ on each side (see \cite[p. 190]{tsuji-ainf-genrep}).
Since $X_n\prm \rightarrowtail D_n$ (resp. $X_1\prm \rightarrowtail D_n$) is a final object of $\CRYS(X_n\prm / D_n)$ (resp. $\CRYS(X_1\prm / D_n)$), we obtain a canonical $\pazo \mbfa_{R, \varpi}^{\textpd} / p^n\textrm{-linear}$ isomorphism
\begin{equation*}
	\pazo \mbfa_{R, \varpi}^{\textpd} / p^n \otimes_{\mbfa_{R, \varpi}^{\textpd} / p^n} N^{\textpd}_n \isomorphic \pazo N^{\textpd}_n,
\end{equation*}
compatible with Frobenius, filtration, connection and the action of $\Gamma_{R}$ on each side.
Here the connection on the tensor product on the left is given as $\partial_{\pazo \mbfa_{R, \varpi}^{\textpd}} \otimes 1$.
Moreover, from \cite[Propositions 24, 25 \& 32]{tsuji-ainf-genrep}, we obtain an $\pazo \mbfa_{R, \varpi}^{\textpd} / p^n\textrm{-linear}$ ismorphism
\begin{equation*}
	\pazo \mbfa_{R, \varpi}^{\textpd} / p^n \otimes_{R / p^n} M / p^n \isomorphic \pazo N^{\textpd}_n,
\end{equation*}
compatible with Frobenius, filtration, connection and the action of $\Gamma_{R}$ on each side (see \cite[p. 191]{tsuji-ainf-genrep}).
Here the connection on the tensor product on the left is given as $\partial_{\pazo \mbfa_{R, \varpi}^{\textpd}} \otimes 1 + 1 \otimes \partial_M$.
Combining the two isomorphisms above, we obtain an $\pazo \mbfa_{R, \varpi}^{\textpd} / p^n\textrm{-linear}$ isomorphism
\begin{equation}\label{eq:flmod_wachpdmod_n}
	\pazo \mbfa_{R, \varpi}^{\textpd} / p^n \otimes_{\mbfa_{R, \varpi}^{\textpd} / p^n} N^{\textpd}_n \isomorphic \pazo \mbfa_{R, \varpi}^{\textpd} / p^n \otimes_{R / p^n} M / p^n,
\end{equation}
compatible with Frobenius, filtration, connection and the action of $\Gamma_{R}$ on each side.
Therefore, we also have an $\mbfa_{R, \varpi}^{\textpd} / p^n\textrm{-linear}$ ismorphism 
\begin{equation*}
	N^{\textpd}_n \isomorphic \big(\pazo \mbfa_{R, \varpi}^{\textpd} / p^n \otimes_{R / p^n} M / p^n\big)^{\partial=0}
\end{equation*}
compatible with Frobenius, filtration and the action of $\Gamma_{R}$ on each side.

\begin{defi}
	Define an $\mbfa_{R, \varpi}^{\textpd}\textrm{-module}$ as $N^{\textpd}(M) := \lim_n N^{\textpd}_n$, equipped with a semilinear and continuous action of $\Gamma_{R}$, a filtration given as $\Fil^k N^{\textpd}(M) := \lim_n \Fil^k N^{\textpd}_n$, which is stable under the action of $\Gamma_{R}$, and a Frobenius-semilinear $\Gamma_{R}\textrm{-equivariant}$ endomorphism $\varphi$.
\end{defi}

Passing to the limit in \eqref{eq:flmod_wachpdmod_n} we obtain an $\pazo \mbfa_{R, \varpi}^{\textpd}\textrm{-linear}$ isomorphism
\begin{equation*}
	\pazo \mbfa_{R, \varpi}^{\textpd} \otimes_{\mbfa_{R, \varpi}^{\textpd}} N^{\textpd}(M) \isomorphic \pazo \mbfa_{R, \varpi}^{\textpd} \otimes_{R} M,
\end{equation*}
compatible with Frobenius, filtration, connection and the action of $\Gamma_{R}$ on each side.
Therefore, we have the following conclusion:
\begin{prop}\label{prop:arpd_mod_fl_data}
	Let $M$ be a free relative Fontaine-Laffaille module.
	Then 
	\begin{equation*}
		N^{\textpd}(M) := \big(\pazo \mbfa_{R, \varpi}^{\textpd} \otimes_{R} M\big)^{\partial = 0},
	\end{equation*}
	is a finite free $\mbfa_{R, \varpi}^{\textpd}\textrm{-module}$ equipped with a decreasing filtration of level $[0, p-2]$, a Frobenius-semilinear endomorphism $\varphi : N^{\textpd}(M) \rightarrow N^{\textpd}(M)$ and a continuous action of $\Gamma_{R}$ on each side.
	In particular, $N^{\textpd}(M) \in \MF_{[0, p-2], \free}^p(\mbfa_{R, \varpi}^{\textpd}, \varphi, \Gamma_{R})$.
	Further, we have a natural isomorphism 
	\begin{equation}\label{eq:flmod_wachpdmod}
		\pazo \mbfa_{R, \varpi}^{\textpd} \otimes_{\mbfa_{R, \varpi}^{\textpd}} N^{\textpd}(M) \isomorphic \pazo \mbfa_{R, \varpi}^{\textpd} \otimes_{R} M,
	\end{equation}
	compatible with the Frobenius, filtration, connection and the action of $\Gamma_{R}$ on each side.
\end{prop}
\begin{proof}[Another proof of Proposition \ref{prop:arpd_mod_fl_data}]
	Let us consider the injective map $R \rightarrow \mbfa_{R, \varpi}^{\textpd}$ sending $X_i \rightarrow [X_i^{\flat}]$.
	\begin{lem}\label{lem:arpd_mod_fl_data}
		We have an $\mbfa_{R, \varpi}^{\textpd}\textrm{-linear}$ isomorphism $\mbfa_{R, \varpi}^{\textpd} \otimes_R M \isomorphic \big(\pazo \mbfa_{R, \varpi}^{\textpd} \otimes_{R} M\big)^{\partial = 0}$.
	\end{lem}
	\begin{proof}
		Let $J = ([X_1^{\flat}]-X_1, \ldots, [X_d^{\flat}]-X_d)\pazo \mbfa_{R, \varpi}^{\textpd}$ and let $J^{[n]}$ denote its $n\textrm{-th}$ divided power for $n \geq 1$.
		We have the projection map,
		\begin{equation*}
			\pazo \mbfa_{R, \varpi}^{\textpd} \otimes_R M \longrightarrow \mbfa_{R, \varpi}^{\textpd} \otimes_R M,
		\end{equation*}
		via the map $X_i \mapsto [X_i^{\flat}]$ and the kernel is given as $J^{[1]}\pazo \mbfa_{R, \varpi}^{\textpd} \otimes_R M$.
		Moreover, we have an $\mbfa_{R, \varpi}^{\textpd}\textrm{-linear}$ section of the projection above given as
		\begin{align}\label{eq:arpd_section}
			\begin{split}
				\mbfa_{R, \varpi}^{\textpd} \otimes_R M &\longrightarrow \pazo \mbfa_{R, \varpi}^{\textpd} \otimes_R M\\
				1 \otimes d &\longmapsto \sum_{\smbfk \in \NN^d} \prod_{i=1}^d \partial_i^{k_i}(d) \prod_{i=1}^d ([X_i^{\flat}] - X_i)^{[k_i]}.
			\end{split}
		\end{align}
		Note that the image of the section lies in $(\pazo \mbfa_{R, \varpi}^{\textpd} \otimes_R M)^{\partial=0}$.
		Now let $Q = J^{[1]} \pazo \mbfa_{R, \varpi}^{\textpd} \otimes_R M$ and $Q' = (\pazo \mbfa_{R, \varpi}^{\textpd} \otimes_R M)/(\pazo \mbfa_{R, \varpi}^{\textpd} \otimes_R M)^{\partial=0}$ and we consider the following diagram with exact rows (the top row is split exact)
		\begin{center}
			\begin{tikzcd}
				0 \arrow[r] & \mbfa_{R, \varpi}^{\textpd} \otimes_R M \arrow[r] \arrow[d] & \pazo \mbfa_{R, \varpi}^{\textpd} \otimes_R M \arrow[r] \arrow[d, equal] & Q \arrow[r] \arrow[d] & 0\\
				0 \arrow[r] & (\pazo \mbfa_{R, \varpi}^{\textpd} \otimes_R M)^{\partial=0} \arrow[r] & \pazo \mbfa_{R, \varpi}^{\textpd} \otimes_R M \arrow[r] & Q' \arrow[r] & 0.
			\end{tikzcd}
		\end{center}
		Note that the left vertical arrow is an injection and the right vertical arrow is a surjection.
		To get that the left vertical arrow is a bijection we need to show that the right vertical arrow is an injection.
		We have 
		\begin{equation*}
			(J^{[1]} \pazo \mbfa_{R, \varpi}^{\textpd} \otimes_R M)^{\partial = 0} \subset \big(J^{[1]} \pazo \mbfb_{\crys}(\overline{R}) \otimes_{R\big[\frac{1}{p}\big]} \pazo \mbfd_{\crys}(V)\big)^{\partial=0} = \big(J^{[1]} \pazo \mbfb_{\crys}(\overline{R}) \otimes_{\QQ_p} V\big)^{\partial=0},
		\end{equation*}
		where $V = V_{\crys}(M)$ is crystalline (see Proposition \ref{prop:tcrysast_fl} (i)) and $J^{[1]} \pazo \mbfb_{\crys}(\overline{R}) \subset \pazo \mbfb_{\crys}(\overline{R})$ is the divided power ideal generated by $([X_1^{\flat}]-X_1, \ldots, [X_d^{\flat}]-X_d)$.
		Then it easily follows that $(J^{[1]} \pazo \mbfb_{\crys}(\overline{R}) \otimes_{\QQ_p} V\big)^{\partial=0} = 0$ and we conclude that $\mbfa_{R, \varpi}^{\textpd} \otimes_R M \isomorphic \big(\pazo \mbfa_{R, \varpi}^{\textpd} \otimes_{R} M\big)^{\partial = 0}$.
	\end{proof}

	From the identification $N^{\textpd}(M) = \big(\pazo \mbfa_{R, \varpi}^{\textpd} \otimes_{R} M\big)^{\partial = 0} \isomorphic \mbfa_{R, \varpi}^{\textpd} \otimes_R M$ (where the rightmost term is equipped with a $\Gamma_R\textrm{-action}$ as in Remark \ref{rem:arpd_mod_fl_data_gamma_act}), it easily follows that $N^{\textpd}(M) \in \MF_{[0, p-2], \free}^p(\mbfa_{R, \varpi}^{\textpd}, \varphi, \Gamma_{R})$.
	Next, we can $\pazo \mbfa_{R, \varpi}^{\textpd}\textrm{-linearly}$ extend the map in \eqref{eq:arpd_section} to obtain
	\begin{align}\label{eq:oarpd_iso}
		\begin{split}
			\pazo \mbfa_{R, \varpi}^{\textpd} \otimes_{\mbfa_{R, \varpi}^{\textpd}} (\mbfa_{R, \varpi}^{\textpd} \otimes_R M) &\longrightarrow \pazo \mbfa_{R, \varpi}^{\textpd} \otimes_R M\\
			1 \otimes d &\longmapsto \sum_{\smbfk \in \NN^d} \prod_{i=1}^d \partial_i^{k_i}(d) \prod_{i=1}^d ([X_i^{\flat}] - X_i)^{[k_i]}.
		\end{split}
	\end{align}
	We equip the left term with a $\Gamma_R\textrm{-action}$ as in Remark \ref{rem:arpd_mod_fl_data_gamma_act}.
	Choosing a basis of $M$ it is easy to see that the determinant of the map in \eqref{eq:oarpd_iso} is invertible in $\pazo \mbfa_{R, \varpi}^{\textpd}$, i.e. the map \eqref{eq:oarpd_iso} is bijective.
	Moreover, it is compatible with Frobenius, filtration, connection and the action of $\Gamma_R$.
	Now we have a natural injective map
	\begin{equation*}
		\pazo \mbfa_{R, \varpi}^{\textpd} \otimes_{\mbfa_{R, \varpi}^{\textpd}} N^{\textpd}(M) \longrightarrow \pazo \mbfa_{R, \varpi}^{\textpd} \otimes_{R} M,
	\end{equation*}
	compatible with the Frobenius, filtration, connection and the action of $\Gamma_{R}$ on each side.
	The map above is bijective because of the following commutative diagram
	\begin{center}
		\begin{tikzcd}
			\pazo \mbfa_{R, \varpi}^{\textpd} \otimes_{\mbfa_{R, \varpi}^{\textpd}} N^{\textpd}(M) \arrow[r, rightarrowtail] \arrow[d, "\wr"] & \pazo \mbfa_{R, \varpi}^{\textpd} \otimes_{R} M \arrow[d, equal]\\
			\pazo \mbfa_{R, \varpi}^{\textpd} \otimes_{\mbfa_{R, \varpi}^{\textpd}} (\mbfa_{R, \varpi}^{\textpd} \otimes_R M) \arrow[r, "\sim"] & \pazo \mbfa_{R, \varpi}^{\textpd} \otimes_{R} M,
		\end{tikzcd}
	\end{center}
	where the the bottom horizontal arrow is the isomorphism in \eqref{eq:oarpd_iso}.
	This concludes the proof.
\end{proof}

\begin{rem}\label{rem:arpd_mod_fl_data_gamma_act}
	Using the $\mbfa_{R, \varpi}^{\textpd}\textrm{-linear}$ isomorphism in Lemma \ref{lem:arpd_mod_fl_data}, $\mbfa_{R, \varpi}^{\textpd} \otimes_R M \isomorphic \big(\pazo \mbfa_{R, \varpi}^{\textpd} \otimes_{R} M\big)^{\partial = 0}$, we can describe the action of $\Gamma_R$ on the left term explitcitly.
	The action can be given by the formula $g(a \otimes d) = g(a) \otimes \sum_{\smbfk \in \NN^d} \prod_{i=1}^d \partial_i^{k_i}(d) \prod_{i=1}^d (g([X_i^{\flat}]) - [X_i^{\flat}])^{[k_i]}$, for $g \in \Gamma_R$.
\end{rem}

\begin{lem}
	Let $N^{\textpd}(M)$ as in Proposition \ref{prop:arpd_mod_fl_data}.
	Then, the action of $\Gamma_{R, \varpi}$ is trivial on $N^{\textpd}(M) / \pi N^{\textpd}(M)$, whereas $\Gamma_{R} / \Gamma_{R, \varpi}$ acts trivially over $N^{\textpd}(M) / \pi_m N^{\textpd}(M)$.
\end{lem}
\begin{proof}
	This follows from the $\Gamma_{R}\textrm{-equivariant}$ isomorphism in \eqref{eq:flmod_wachpdmod} (or from Lemma \ref{lem:arpd_mod_fl_data} and Remark \ref{rem:arpd_mod_fl_data_gamma_act}) and the action of $\Gamma_{R}$ on $\pazo \mbfa_{R, \varpi}^{\textpd}$ (see Lemma \ref{lem:gamma_inv_oarpd} (i)).
\end{proof}

\begin{prop}\label{prop:fl_data_arpd_ff}
	The functor
	\begin{align*}
		N^{\textpd} : \MF_{[0, p-2], \free}(R, \Phi, \partial) &\longrightarrow \MF_{[0, p-2], \free}^p(\mbfa_{R, \varpi}^{\textpd}, \varphi, \Gamma_{R})\\
				M &\longmapsto N^{\textpd}(M) = \big(\pazo \mbfa_{R, \varpi}^{\textpd} \otimes_{R} M\big)^{\partial=0},
	\end{align*}
	is fully faithful.
\end{prop}
\begin{proof}
	By taking $\Gamma_R\textrm{-invariants}$ in \eqref{eq:flmod_wachpdmod}, we obtain an $R\textrm{-linear}$ isomorphism $\big(\pazo \mbfa_{R, \varpi}^{\textpd} \otimes_{\mbfa_{R, \varpi}^{\textpd}} N^{\textpd}(M)\big)^{\Gamma_R} \isomorphic M$ compatible with Frobenius, filtration, connection on each side, and functorial in $M$.
\end{proof}

Having obtained a finite free module with desired structures over the ring $\mbfa_{R, \varpi}^{\textpd}$, we will now pass to the ring $\mbfa_{R, \varpi}^+$.
Let $M \in \MF_{[0, p-2], \free}(R, \Phi, \partial)$ and $N^{\textpd}(M) \in \MF_{[0, p-2], \free}^p(\mbfa_{R, \varpi}^{\textpd}, \varphi, \Gamma_{R})$ the $\mbfa_{R, \varpi}^{\textpd}\textrm{-module}$ obtained under the functor of Proposition \ref{prop:fl_data_arpd_ff}.

Next, from Theorem \ref{thm:arplus_arpd_cat_equiv}, we have an equivalence of catgeories $\MF_{[0, p-2], \free}^p(\mbfa_{R, \varpi}^{\textpd}, \varphi, \Gamma_{R}) \isomorphic \MF_{[0, p-2], \free}^q(\mbfa_{R, \varpi}^+, \varphi, \Gamma_{R})$ sending $N^{\textpd} \mapsto N^+$.
Combining this with Propositions \ref{prop:arpd_mod_fl_data} \& \ref{prop:fl_data_arpd_ff}, we obtain:
\begin{prop}\label{prop:fl_data_arplus_mod}
	The functor
	\begin{align*}
		N^+ : \MF_{[0, p-2], \free}(R, \Phi, \partial) &\longrightarrow \MF_{[0, p-2], \free}^q(\mbfa_{R, \varpi}^+, \varphi, \Gamma_{R})\\
				M &\longmapsto N^+(M),
	\end{align*}
	is fully faithful.
	Further, for $M$ and $N^+(M)$ as above, we have a natural isomorphism 
	\begin{equation}\label{eq:flmod_wachplusmod}
		\pazo \mbfa_{R, \varpi}^{\textpd} \otimes_{\mbfa_{R, \varpi}^+} N^+(M) \isomorphic \pazo \mbfa_{R, \varpi}^{\textpd} \otimes_{R} M,
	\end{equation}
	compatible with the Frobenius, filtration, connection and the action of $\Gamma_{R}$ on each side.
\end{prop}

\begin{lem}\label{lem:gamma_action_nr}
	Let $N^+(M)$ as in Proposition \ref{prop:fl_data_arplus_mod}.
	Then, the action of $\Gamma_{R, \varpi}$ is trivial on $N^+(M) / \pi N^+(M)$, whereas $\Gamma_{R} / \Gamma_{R, \varpi}$ acts trivially over $N^+(M) / \pi_m N^+(M)$.
\end{lem}
\begin{proof}
	This follows from the $\Gamma_{R}\textrm{-equivariant}$ isomorphism in \eqref{eq:flmod_wachplusmod} and the action of $\Gamma_{R}$ on $\pazo \mbfa_{R, \varpi}^{\textpd}$ (see Lemma \ref{lem:gamma_inv_oarpd} (i)).
\end{proof}

\subsubsection{Obtaining Wach modules}\label{subsubsec:obtain_wach_mod}

For the rest of this section we will fix $m = 1$ (fix $m=2$ if $p=2$), i.e. we take $K = F(\zeta_p)$ (take $K = F(\zeta_{p^2})$ if $p=2$).
Consider the localization $S = \mbfa_{R, \varpi}^+\big[\frac{1}{\pi_1}\big]$.
Let $M$ and $M\prm$ be free relative Fontaine-Laffaille modules and $N^+(M)$ and $N^+(M\prm)$ the respective $\mbfa_{R, \varpi}^+\textrm{-modules}$ obtained by the functor in Proposition \ref{prop:fl_data_arplus_mod}.

\begin{lem}\label{lem:localize_ff_gamma}
	The natural map
	\begin{equation}\label{eq:localize_ff}
		\Hom_{\mbfa_{R, \varpi}^+, \hspace{0.3mm}\Gamma_{R}}(N^+(M), N^+(M\prm)) \isomorphic \Hom_{S, \hspace{0.3mm}\Gamma_{R}}\big(N^+(M)\big[\tfrac{1}{\pi_1}\big], N^+(M\prm)\big[\tfrac{1}{\pi_1}\big]\big),
	\end{equation}
	is bijective.
\end{lem}
\begin{proof}
	As we are working with free modules and the morphism of rings $\mbfa_{R, \varpi}^+ \rightarrow \mbfa_{R, \varpi}^+\big[\frac{1}{\pi_1}\big] = S$ is flat, we obtain that \eqref{eq:localize_ff} is injective.
	To check surjectivity, let $f : N^+(M)\big[\frac{1}{\pi_1}\big] \rightarrow N^+(M\prm)\big[\frac{1}{\pi_1}\big]$ be an $S\textrm{-linear}$ and $\Gamma_{R}\textrm{-equivariant}$ morphism.
	We need to show that $f(N^+(M)) \subset N^+(M\prm)$.
	Assume $f(N^+(M)) \subset \pi_1^{-k} N^+(M\prm)$ for $k \in \NN$, and consider the reduction of $f$ modulo $\pi$, which is again $\Gamma_{R}\textrm{-equivariant}$.
	Now from Lemma \ref{lem:gamma_action_nr}, we have that $\Gamma_{R}$ acts trivially over $N^+(M) / \pi_1 N^+(M)$, whereas the action of $\Gamma_{R}$ is non-trivial over $\pi_1^{-k}N^+(M\prm) / \pi_1^{-k+1} N^+(M\prm)$ for $k \neq 0$ (the action of $\gamma_0 \in \Gamma_K$ is non-trivial for $k \neq 0$).
	Hence, we must have $k = 0$, i.e. $f(N^+(M)) \subset N^+(M\prm)$, which shows the claim.
\end{proof}

Now note that we have a morphism $\varphi : S = \mbfa_{R, \varpi}^+\big[\frac{1}{\pi_1}\big] \rightarrow \mbfa_{R, \varpi}^+\big[\frac{1}{\pi}\big]$.
The respective Frobenius-semilinear endomorphisms $\varphi$ on $N^+(M)$ and $N^+(M')$ induce semilinear morphisms $\varphi : N^+(M)\big[\frac{1}{\pi_1}\big] \rightarrow N^+(M)\big[\frac{1}{\pi}\big]$ and $\varphi: N^+(M')\big[\frac{1}{\pi_1}\big] \rightarrow N^+(M')\big[\frac{1}{\pi}\big]$.
Now let $f \in \Hom_{S, \hspace{0.3mm}\Gamma_{R}}\big(N^+(M)\big[\tfrac{1}{\pi_1}\big], N^+(M\prm)\big[\tfrac{1}{\pi_1}\big]\big)$ be a morphism, such that the following diagram commutes
\begin{center}
	\begin{tikzcd}
		N^+(M)\big[\tfrac{1}{\pi_1}\big] \arrow[r, "f"] \arrow[d, "\varphi"] & N^+(M')\big[\tfrac{1}{\pi_1}\big] \arrow[d, "\varphi"]\\
		N^+(M)\big[\tfrac{1}{\pi}\big] \arrow[r, "f"] & N^+(M')\big[\tfrac{1}{\pi}\big],
	\end{tikzcd}
\end{center}
where the bottom horizontal arrow is well-defined due to Lemma \ref{lem:localize_ff_gamma}.
We will call such a morphism $f$ to be $(\varphi, \Gamma_R)\textrm{-equivariant}$.


\begin{lem}\label{lem:localize_ff}
	The natural map
	\begin{equation*}
		\Hom_{\mbfa_{R, \varpi}^+, \hspace{0.3mm}\varphi, \hspace{0.3mm}\Gamma_{R}}(N^+(M), N^+(M\prm)) \isomorphic \Hom_{S, \hspace{0.3mm}\varphi, \hspace{0.3mm}\Gamma_{R}}\big(N^+(M)\big[\tfrac{1}{\pi_1}\big], N^+(M\prm)\big[\tfrac{1}{\pi_1}\big]\big),
	\end{equation*}
	is bijective.
\end{lem}

\begin{proof}[\textbf{Proof of Theorem \ref{thm:fl_to_wach}}]
	Let $M \in \MF_{[0, p-2], \free}(R, \Phi, \partial)$ and let $N^+(M)$ denote the $\mbfa_{R, \varpi}^+\textrm{-module}$ obtained from $M$ from the functor of Proposition \ref{prop:fl_data_arplus_mod}.
	We will show that a basis of $N^+(M)$ descends over to $\mbfa_{R}^+$.

	In the notation of Definition \ref{defi:mf_free_a}, let $\{e_1, \ldots, e_h\}$ denote an $\mbfa_{R, \varpi}^+\textrm{-basis}$ of $N^+(M)$.
	Then from Lemma \ref{lem:filmod_strong_div}, we have that $\{q^{-k_1}\varphi(e_1), \ldots, q^{-k_h}\varphi(e_h)\}$ is also an $\mbfa_{R, \varpi}^+\textrm{-basis}$ of $N^+(M)$.
	Without loss of generality, we may further assume that $k_h \leq k_{h-1} \leq \cdots \leq k_1$.
	Let us set $s := k_1$, so we get that $N^+(M) / \varphi^{\ast}(N^+(M))$ is killed by $q^s$ and $s \in \NN$ is the smallest such number.

	Let $D(M) := N^+(M)\big[\frac{1}{\pi_1}\big]^{\wedge}$, where ${ }^{\wedge}$ denotes the $\padic$ completion.
	Then $D(M)$ is an \'etale $(\varphi, \Gamma_{R, \varpi})\textrm{-module}$ over $\mbfa_{R, \varpi} = \mbfa_{R, \varpi}^+\big[\frac{1}{\pi_1}\big]^{\wedge}$, free of rank $h$.
	Further, combining Lemma \ref{lem:localize_ff} with \cite[Theorem 8.14]{matsumura-reid-commutative} (the $\padic$ completion, i.e. $\mbfa_{R, \varpi}^+\big[\frac{1}{\pi_1}\big] \rightarrow \mbfa_{R, \varpi}$ is faithfully flat since $p$ is in the Jacobson radical of $\mbfa_{R, \varpi}^+\big[\frac{1}{\pi}\big]$), we obtain that the functor 
	\begin{align*}
		\MF_{[0, p-2], \free}(R, \Phi, \partial) &\longrightarrow (\varphi, \Gamma_R)\textup{-Mod}_{\mbfa_{R, \varpi}}^{\etale}\\
				M &\longmapsto N^+(M)\big[\tfrac{1}{\pi_1}\big]^{\wedge},
	\end{align*}
	is fully faithful.
	
	Now, from Proposition \ref{prop:tcrysast_fl} and Definition \ref{defi:tcrys_m} we have that $T := T_{\crys}(M)$ is a free $\ZZ_p\textrm{-representation}$ of $G_{R}$.
	Considering $T$ as a representation of $G_{R, \varpi}$, we have the associated $(\varphi, \Gamma_{R, \varpi})\textrm{-module}$ $\mbfd_{R, \varpi}(T)$ over $\mbfa_{R, \varpi}$.
	By the full faithfullness of the functor above and equivalence of categories in \eqref{eq:phi_gamma_equiv}, we conclude that $D(M) \isomorphic \mbfd_{R, \varpi}(T)$ as \'etale $(\varphi, \Gamma_{R, \varpi})\textrm{-module}$ over $\mbfa_{R, \varpi}$.
	Also, we have $\varphi(\mbfd_{R, \varpi}(T)) \subset \mbfd(T)$, where the latter module is an étale $(\varphi, \Gamma_{R})\textrm{-module}$ over $\mbfa_{R}$, free of rank $h$.
	
	Next, let $N := N^+(M) \cap \mbfd(T)$ where we take the intersection inside $\mbfd_{R, \varpi}(T)$.
	Note that $N$ is equipped with a Frobenius-semilinear endomorphism $\varphi$ and it is stable under the action of $\Gamma_{R}$.
	We claim that
	\begin{lem}
		The elements $\{q^{-k_1}\varphi(e_1), \ldots, q^{-k_h}\varphi(e_h)\}$ form a basis of $N$.
	\end{lem}
	\begin{proof}
		Let us set $N\prm := \sum_{i=1}^h \mbfa_{R}^+ q^{-k_i} \varphi(e_i)$
		Since $q^{-k_i}\varphi(e_i) \in N^+(M) \cap \mbfd(T) = N$, we have $N\prm \subset N$.
		This also implies that $\varphi(e_i) \in q^{k_i} N$.
		Extending scalars along the faithfully flat morphism of rings $\mbfa_{R}^+ \rightarrow \mbfa_{R, \varpi}^+$, we get that $N^+(M) = \mbfa_{R, \varpi}^+ \otimes_{\mbfa_{R}^+} N\prm \subset \mbfa_{R, \varpi}^+ \otimes_{\mbfa_{R}^+} N \subset N^+(M)$.
		Therefore, $\mbfa_{R, \varpi}^+ \otimes_{\mbfa_{R}^+} N\prm \isomorphic \mbfa_{R, \varpi}^+ \otimes_{\mbfa_{R}^+} N$.
		But since the map $\mbfa_{R}^+ \rightarrow \mbfa_{R, \varpi}^+$ is faithfully flat, we obtain that $N\prm \isomorphic N$.
	\end{proof}
	
		We will now verify the conditions of Definition \ref{defi:wach_reps} for $V = \QQ_p \otimes_{\ZZ_p} T$.
		Since $V$ arises from a Fontaine-Laffaille module of level $[0, p-2]$, we have that $V$ is crystalline with non-positive Hodge-Tate weights.
		We have that $N$ is a free $\mbfa_{R}^+\textrm{-module}$ of rank $h$ stable under $\varphi$ and $\Gamma_{R}$, and such that $N \subset \mbfd^+(T)$ as well as $\mbfa_{R} \otimes_{\mbfa_{R}^+} N \isomorphic \mbfd(T)$.
		Next, we want to show that $q^s N \subset \varphi^{\ast}(N)$ as $\mbfa_{R}^+\textrm{-modules}$, where $s = k_1$.
		Since $\mbfa_{R}^+ \rightarrow \mbfa_{R, \varpi}^+$ is faithfully flat, it is equivalent to showing that $q^s \mbfa_{R, \varpi}^+ \otimes_{\mbfa_{R}^+} N \subset \mbfa_{R, \varpi}^+ \otimes_{\mbfa_{R}^+} \varphi^{\ast}(N)$.
		But the latter inclusion can be re-expressed as $q^s N^+(M) \subset \varphi^{\ast}(N^+(M))$ as $\mbfa_{R, \varpi}^+\textrm{-modules}$, which was established above by showing that $N^+(M) / \varphi^{\ast}(N^+(M))$ is killed by $q^s$.
		Therefore, we conclude that $N / \varphi^{\ast}(N)$ is killed by $q^s$ and $s \in \NN$ is the smallest such number.

		Next, we look at the action of $\Gamma_{R}$ over $N$.
		Recall from \S \ref{subsec:relative_phi_gamma_mod} that we have $\{\gamma_0, \gamma_1, \ldots, \gamma_d\}$ as topological generators of $\Gamma_{R, \varpi}$, where $\gamma_0$ is a lift of a topological generator of $\Gamma_K$.
		The action of $\gamma_j$ on the basis elements of $N^+(M)$ can be given as
		\begin{equation*}
			\gamma_j(e_i) = e_i + \pi x_{i, j} \hspace{2mm} \textrm{for} \hspace{2mm} 1 \leq i \leq h, \hspace{1mm} 0 \leq j \leq d \hspace{2mm} \textrm{and} \hspace{2mm} x_{i, j} \in \mbfa_{R, \varpi}^+.
		\end{equation*}
		Since $\varphi$ is $\Gamma_{R}\textrm{-equivariant}$, we get that $\gamma_j(\varphi(e_i)) = \varphi(e_i) + q\pi \varphi(x_{i, j})$, where $\varphi(x_{i, j}) \in \varphi(N^+(M)) \subset N^+(M) \cap \mbfd(T) = N$.
		Now $\varphi(e_i) \in q^{k_i} N$, so we must have that $q \pi \varphi(x_{i, j}) \in q^{k_i}N \cap q\pi N = q^{k_i}\pi N \subset N$, for $1 \leq i \leq h$ and $0 \leq j \leq d$.
		Therefore, we get that 
		\begin{equation*}
			\gamma_j(q^{-k_i} \varphi(e_i)) \equiv q^{-k_i} \varphi(e_i) \mod \pi N \hspace{2mm} \textrm{for} \hspace{2mm} 1 \leq j \leq d.
		\end{equation*}
		For $j = 0$, recall that $\gamma_0(\pi) = \chi(\gamma_0) \pi u$ for some unit $u \in 1 + \pi \mbfa_{R}^+$.
		Therefore, we have $\gamma_0(q) = q \varphi(u) u^{-1}$ and $\gamma_0(q^{-1}) = q^{-1} \varphi(u^{-1}) u$.
		So we obtain 
		\begin{equation*}
			\gamma_0(q^{-k_i} \varphi(e_i)) = \gamma_0(q^{-k_i}) \gamma_0(\varphi(e_i)) = q^{-k_i} \varphi(u^{-k_i}) u^{k_i}(\varphi(e_i) + q \pi \varphi(x_{i, j})) \equiv q^{-k_i} \varphi(e_i) \mod \pi N.
		\end{equation*}
		Finally, let $g \in \Gamma_{R}$ be a lift of a generator $\overline{g} \in \Gamma_{R} / \Gamma_{R, \varpi}$, a finite group of order $p-1$.
		Then we have $g(e_i) = e_i + \pi_1 y_i$ for $1 \leq i \leq h$ and $y_i \in N^+(M)$.
		Since $\varphi$ is $\Gamma_{R}\textrm{-equivariant}$, we get that $g(\varphi(e_i)) = \varphi(e_i) + \pi \varphi(y_i)$, where $\varphi(y_i) \in \varphi(N^+(M)) \subset N^+(M) \cap \mbfd(T) = N$.
		Now $\varphi(e_i) \in q^{k_i} N$, so we must have that $\pi \varphi(y_i) \in q^{k_i}N \cap \pi N = q^{k_i}\pi N \subset N$, for $1 \leq i \leq h$.
		Further, we know that $g(\pi) = \chi(g)\pi v$ for some unit $v \in 1 + \pi \mbfa_R^+$, which gives us that $g(q) = q \varphi(v) v^{-1}$.
		Therefore, $g(q^{-k_i}\varphi(e_i)) = q^{-k_i}\varphi(u^{-k_i}) u^{k_i} (\varphi(e_i) + \pi \varphi(y_i)) \equiv q^{-k_i}\varphi(e_i) \mod \pi N$, for $1 \leq i \leq h$.
		Hence, $\Gamma_{R}$ acts trivially over $N / \pi N$.
	
		Setting $\mbfn(T) := N$, we see that conditions of Definition \ref{defi:wach_reps} have been satisfied.
		In particular, $V$ is a positive finite $q\textrm{-height}$ representation.
\end{proof}

\cleardoublepage



\nocite{*}
\phantomsection
\printbibliography[heading=bibintoc, title={References}]

\Addresses

\end{document}